\documentclass[a4paper,reqno,11pt]{amsart}
\usepackage[margin=1in]{geometry}
\usepackage{amsmath, amssymb,amsthm,url,mathrsfs,graphicx,amscd,amsfonts}
\usepackage[mathscr]{eucal}
\usepackage{mathtools}
\usepackage{hyperref}
\usepackage[all]{xy}

\usepackage[utf8]{inputenc}
\usepackage{graphicx,todonotes,hyperref}
\usepackage{epstopdf}
\usepackage{enumitem}
\usepackage{xcolor}
\usepackage{xspace}
\usepackage{tikz-cd}
\usepackage{multirow}

\newtheorem*{acknowledgements*}{Acknowledgements}

\newcommand{\colim}{\mathit{colim}}

\newcommand{\Ker}{\mbox{Ker}}

\allowdisplaybreaks
\theoremstyle{definition}
\newtheorem{theorem}{Theorem}[section]

\newtheorem{lemma}[theorem]{Lemma}

\newtheorem*{theorem*}{Theorem}
\newtheorem*{corollary*}{Corollary}
\newtheorem*{prop*}{Proposition}
\newtheorem*{rmk*}{Remark}
\newtheorem{proposition}[theorem]{Proposition}
\newtheorem{corollary}[theorem]{Corollary}
\newtheorem{definition}[theorem]{Definition}
\newtheorem{remark}[theorem]{Remark}
\newtheorem*{thma}{Theorem A}
\newtheorem*{thmb}{Theorem B}
\newtheorem*{thmc}{Theorem C}

\begin{document}
		%\title{A Diffeomorphism Classification of Smooth Structures and Tangential Homotopy Types of $\mathbb{C}P^m$ for $5 \le m \le 8$}
        \title[A Diffeomorphism Classification of $\mathbb{C}P^m$]%
{Diffeomorphism Classification of Smooth Structures and Tangential Homotopy Types of $\mathbb{C}P^m$ for $5 \le m \le 8$}
		\vspace{2cm}

		\author{Ramesh Kasilingam}
		
		\email{rameshkasilingam.iitb@gmail.com  ; rameshk@iitm.ac.in  }
		\address{Department of Mathematics,
			Indian Institute Of Technology, Chennai-600036, India}
		
		\date{}
		%\paragraph{Classification.}
		%57R55; 57R50.
		
		%\maketitle
\begin{abstract}
This paper provides a diffeomorphism classification of smooth manifolds homeomorphic 
to the complex projective space $\mathbb{C}P^m$ for $m \in \{5, 6, 7, 8\}$. 
The classification is obtained by computing the group of concordance classes 
of smooth structures on $\mathbb{C}P^m$ and determining the orbit space under 
the action induced by the group of self-homeomorphisms. Using these 
computations in conjunction with the tangential surgery exact sequence and 
techniques from stable homotopy theory, we determine the diffeomorphism 
classes of smooth manifolds within the tangential homotopy type of $\mathbb{C}P^m$ 
for $4 \le m \le 8$. We also investigate the relationship between these two 
classification problems by studying the natural map from the homeomorphism type 
to the tangential homotopy type. As a consequence, we prove that for $m = 4$, 
there exists a unique smooth manifold, up to diffeomorphism, that is tangentially 
homotopy equivalent to $\mathbb{C}P^4$ but not homeomorphic to it. 
Furthermore, for $m = 8$, there exist exactly two pairwise non-diffeomorphic 
smooth manifolds that are tangentially homotopy equivalent to $\mathbb{C}P^8$ 
but not homeomorphic to it.
\end{abstract}
\maketitle
\section{Introduction}

A smooth homotopy complex projective $m$-space is a closed smooth manifold that is homotopy equivalent to the complex projective space $\mathbb{C}P^m$. One approach to the study of such manifolds is through the classification of differentiable free $\mathbb{S}^{1}$-actions on homotopy spheres (see, for example, \cite{Hsi66, HH64, Bro68, MY66, MY67, MY68, Kaw70}). Later, D.~Sullivan \cite{Sul67} classified PL homotopy complex projective spaces as an application of his characteristic variety theorem. In the smooth category, the surgery-theoretic diffeomorphism classification for $3 \leq m \leq 6$ was obtained by Brumfiel \cite{Bru68, Bru71}. For the specific cases $m=3$ and $m=4$, the diffeomorphism classification of smooth manifolds homeomorphic to $\mathbb{C}P^{m}$ was established in \cite{MY66, Kaw68} (see also \cite{Ram16}).

In this paper, we classify all smooth manifolds homeomorphic to $\mathbb{C}P^{m}$ for $m = 5, 6, 7,$ and $8$, up to both (unoriented) diffeomorphism and orientation-preserving diffeomorphism (see Theorems~\ref{main3} and \ref{main4}). This classification is achieved by computing the group $\mathcal{C}(\mathbb{C}P^{m})$ of concordance classes of smooth structures on $\mathbb{C}P^{m}$ (cf.~\cite[Definition 2.2]{Ram16}) and by explicitly determining the action of the self-homeomorphism group on $\mathcal{C}(\mathbb{C}P^{m})$ for $5 \le m \le 8$ (see Theorem~\ref{main2} and Theorem~\ref{conju-cp8}(1)). As a consequence of these classifications, we establish the following:
\begin{thma}\indent \label{A}
\begin{enumerate}
    \item Up to diffeomorphism, there are exactly seven exotic smooth structures on $\mathbb{C}P^{5}$.
    \item For $m=6,7,8$, there are exactly three exotic smooth structures on $\mathbb{C}P^{m}$, up to diffeomorphism.
\end{enumerate}
\end{thma}
Furthermore, for $m=7$ and $8$, we compute the smooth tangential structure set $\mathcal{S}^t_{\mathrm{Diff}}(\mathbb{C}P^{m})$ of $\mathbb{C}P^{m}$, which consists of equivalence classes $[M, f, \widehat{f}]$ of triples $(M, f, \widehat{f})$, where $f \colon M \to \mathbb{C}P^{m}$ is a homotopy equivalence and $\widehat{f} \colon \nu_M \to \nu_{\mathbb{C}P^{m}}$ is a stable bundle map covering $f$ (cf.~\cite[p.~103]{CH15}). This computation uses the tangential surgery exact sequence together with explicit calculations of the groups $[\mathbb{C}P^{m}, SF]$ (see Proposition~\ref{sur-obst} and Theorem~\ref{main5}). 
In addition, we determine the number of smooth homotopy complex projective $m$-spaces with fixed Pontryagin classes, up to diffeomorphism, for $4 \le m \le 8$ (see Proposition~\ref{num}). We also determine the explicit action of the group $\epsilon^{t}(\mathbb{C}P^{m})$ of tangential self-equivalences of $\mathbb{C}P^{m}$ (i.e., bundle maps of the stable normal bundle of $\mathbb{C}P^{m}$ up to homotopy) on the structure set $\mathcal{S}^t_{\mathrm{Diff}}(\mathbb{C}P^{m})$ (see Theorem~\ref{main6}). This yields a diffeomorphism classification of all smooth manifolds in the tangential homotopy type of $\mathbb{C}P^{m}$ (see also Theorem~\ref{difftang}). As a consequence, we obtain the following result.
\begin{thmb}\label{B}
Let $4 \leq m \leq 8$ and let $\eta^{t} \colon \mathcal{S}^t_{\mathrm{Diff}}(\mathbb{C}P^{m}) \longrightarrow [\mathbb{C}P^{m}, SF]$ be the tangential normal invariant. Suppose that $[M, f, \widehat{f}]$ and $[N, g, \widehat{g}] \in \mathcal{S}^t_{\mathrm{Diff}}(\mathbb{C}P^{m})$.
\begin{itemize}
\item[(1)] For $m=4, 5, 6, 8$, the smooth manifolds $M$ and $N$ are diffeomorphic if and only if 
$\eta^{t}([M, f, \widehat{f}]) = \pm \eta^{t}([N, g, \widehat{g}])$.
\item[(2)] For $m=7$, the smooth manifolds $M$ and $N$ are diffeomorphic if and only if 
$\eta^{t}([M, f, \widehat{f}]) = \eta^{t}([N, g, \widehat{g}])$.
\end{itemize}
\end{thmb}
(See Theorem~\ref{Main7}.)
We further show that there exists a well-defined injective map 
\[
E \colon \mathcal{C}(\mathbb{C}P^{m}) \to \mathcal{S}^t_{\mathrm{Diff}}(\mathbb{C}P^{m}).
\]
Moreover, for the same range of $m$, we identify the action of $\epsilon^{t}(\mathbb{C}P^{m})$ on the image of $E$ (see Theorem~\ref{difftang}). This allows us to establish the following relationship between homeomorphism and tangential homotopy equivalence.
\begin{thmc}\indent 
\begin{enumerate}
    \item[(1)] For $m=5, 6, 7$, a closed smooth manifold $M$ is homeomorphic to $\mathbb{C}P^{m}$ if and only if $M$ is tangentially homotopy equivalent to $\mathbb{C}P^{m}$. 
    \item[(2)] For $m=4$, there exists a unique smooth manifold (up to diffeomorphism) that is tangentially homotopy equivalent to $\mathbb{C}P^{4}$ but is not homeomorphic to $\mathbb{C}P^{4}$.
    \item[(3)] For $m=8$, there are exactly two smooth manifolds (up to diffeomorphism) that are tangentially homotopy equivalent to $\mathbb{C}P^{8}$ but are not homeomorphic to $\mathbb{C}P^{8}$.
\end{enumerate}
\end{thmc}
(See Corollary~\ref{TE}.) 
In the proofs of Theorem~A and Theorem~B, we primarily use Lemma~I.9 of Brumfiel~\cite{Bru71} and Kawakubo’s theorem~\cite[Theorem~1]{Kaw68}. Since proofs of both results are not readily available in the literature, we include them in this paper (see Lemma~\ref{lem:Bru-seq} and Theorem~\ref{kawa}).
\subsection*{Notation}
Throughout the paper, we use the following conventions and notations.
\begin{itemize}
 \item $\mathbb{Z}_{n}\{\alpha\}$ denotes the cyclic group of order $n$ generated by~$\alpha$.
 If $A$ is an abelian group, then $A_{(2)}$ denotes its localization at the prime $2$, that is,
 \[
 A_{(2)} = A \otimes \mathbb{Z}\Bigl[\tfrac13, \tfrac15, \ldots \Bigr].
 \]
 The symbol $\cong$ denotes an isomorphism of groups.

 \item Depending on the context, we use either $\operatorname{Im}(f)$ or $\operatorname{Im}(G\xrightarrow{f}H)$ for the image of a group homomorphism $f\colon G\longrightarrow H$, and either $\operatorname{Ker}(f)$ or $\operatorname{Ker}(G\xrightarrow{f}H)$ for its kernel.

 \item $\mathbb{S}^{n}$ denotes the standard $n$-sphere in Euclidean space $\mathbb{R}^{n+1}$. We write $\mathbb{C}P^{n}/\mathbb{C}P^{k}$ for the stunted complex projective space, and for all $n \ge k$, we denote by
\[
q : \mathbb{C}P^{n} \longrightarrow \mathbb{C}P^{n}/\mathbb{C}P^{k} \quad \text{and} \quad i : \mathbb{C}P^{k} \hookrightarrow \mathbb{C}P^{n}
\]
the quotient map and the standard inclusion, respectively. We denote by
\[
p : \mathbb{S}^{2n+1} \longrightarrow \mathbb{C}P^{n}
\]
the Hopf fibration. The composition $\mathbb{S}^{2n+1} \xrightarrow{p} \mathbb{C}P^{n} \xrightarrow{q} \mathbb{C}P^{n}/\mathbb{C}P^{k}$ is denoted by $\varphi_{n+1}$.

 \item The notation $[X,Y]$ denotes the set of homotopy classes of maps $X\longrightarrow Y$.
 If $Y$ is an infinite loop space, then $[-,Y]$ is regarded as a contravariant functor
 from the homotopy category of CW–complexes to the category of abelian groups.
 For a map $f:X\longrightarrow Y$, we often write $f\in [X,Y]$ for its homotopy class when no confusion can arise.
 For a map (or homotopy class) $f$, we denote by $C_{f}$ the homotopy cofibre of $f$,
 and we refer to the homotopy cofibre simply as the cofibre.
 The suspension of a space \(X\) is denoted by \(\Sigma X\).
 \item For a space $X$, we write $X_{(p)}$ for its localization at the prime $p$.
 If $f \colon X \longrightarrow Y$ is a continuous map, we denote by
 \[
 f_{(p)} \colon X_{(p)} \longrightarrow Y_{(p)}
 \]
 the map induced by localization at~$p$.
 We write $X \simeq Y$ to indicate an (unstable) homotopy equivalence of spaces, and
 $X \simeq_{(2)} Y$ to denote a homotopy equivalence after localization at~$p$, that is,
 $X_{(p)} \simeq Y_{(p)}$ (see \cite{Sul67, Sul05}).

 \item If \(X = X_1 \vee X_2\) is the wedge of two CW–subcomplexes, then for \(k = 1,2\) we denote by
 \[
 p_k : X_1 \vee X_2 \longrightarrow X_k
 \]
 the canonical collapse map that sends the complementary subcomplex \(X_{3-k}\) to the wedge point.

 \item For a closed connected oriented $n$–manifold $M$, we denote by
 \[
 f_{M}:M\longrightarrow \mathbb{S}^n
 \]
 a fixed degree–one map; for example, choose an embedded disk
 $\mathbb{D}^n\hookrightarrow M$ and let
 $f_{M}: M\longrightarrow M/(M-\operatorname{int}(\mathbb{D}^n))\simeq \mathbb{S}^n$
 be the quotient map.
 For a CW complex \(X=X^{k-1}\cup_{\alpha}\mathbb{D}^{k}\), where $X^{k-1}$ is the $(k-1)$–skeleton
 and $\alpha:\mathbb{S}^{k-1}\longrightarrow X^{k-1}$ is the attaching map, we denote also by
 \[
 f_{X} : X \longrightarrow \mathbb{S}^{k}
 \]
 the map obtained by collapsing the \((k-1)\)–skeleton to a point.
 \item If $N$ is a closed oriented smooth manifold, the symbol $-N$ denotes the same underlying topological manifold as $N$, but endowed with the opposite orientation. 
\item We denote by
\[
\begin{aligned}
 O   &= \underset{n\longrightarrow \infty}{\colim}\,O_n, \quad
 Top = \underset{n\longrightarrow \infty}{\colim}\,Top_n,\\
 PL  &= \underset{n\longrightarrow \infty}{\colim}\,PL_n, \quad
 F   = \underset{n\longrightarrow \infty}{\colim}\,F_n, \quad
 SF  = \underset{n\longrightarrow \infty}{\colim}\,SF_n
\end{aligned}
\]
 the direct limits of the groups of orthogonal transformations, self-homeomorphisms of $\mathbb{R}^n$
 fixing the origin, piecewise linear homeomorphisms of $\mathbb{R}^n$ fixing the origin,
 basepoint–preserving self–homotopy equivalences of $\mathbb{S}^n$, and basepoint–preserving
 degree–one self–homotopy equivalences of $\mathbb{S}^n$, respectively
 (see \cite{KL66,LR65,MM79}).
 Let $F/O$ be the homotopy fibre of the canonical map $BO \longrightarrow BF$ between the classifying spaces
 for stable vector bundles and stable spherical fibrations
 (see \cite[\S2, \S3]{MM79} and \cite[p.113]{Wal99}),
 and let $\operatorname{Top}/O$ be the homotopy fibre of the canonical map $BO \longrightarrow BTop$ between the classifying
 spaces for stable vector bundles and stable topological $\mathbb{R}^n$–bundles
 (see \cite[Theorem~10.1, Essay~IV]{KS77}).
 Similarly, let $PL/O$ be the homotopy fibre of the map $BO \longrightarrow BPL$ between the classifying spaces
 for stable vector bundles and stable piecewise–linear $\mathbb{R}^n$–bundles
 (see \cite[p.~92]{Lan00}).
 \item The symbols for the generators of the stable homotopy groups of spheres $\pi_n^{s}$ and for the stable Toda brackets that we use are taken from \cite{Tod62}, \cite[Theorem~B, p.~271]{Muk66}, and \cite{Rav03}, and are mainly drawn from \cite{Tod62}.
 \item For each $n\ge 1$ we identify $[\mathbb{S}^n,SF]\cong \pi_n^{s},$ (\cite[Corollary 3.8, p.~48]{MM79})
and we freely pass between these notations when convenient.
 \item For a real vector bundle $\xi$ over a smooth manifold, we write $p_i(\xi)$ for its $i$–th Pontryagin class.
\end{itemize}
\section{The Group of Concordance Classes of Smooth Structures}\label{2}
We recall some terminology from \cite{KM63}.

\begin{definition}\rm
\begin{itemize}
\item[(a)] A homotopy $m$–sphere $\Sigma^m$ is an oriented smooth closed manifold homotopy equivalent to the standard unit sphere $\mathbb{S}^m \subset \mathbb{R}^{m+1}$.

\item[(b)] A homotopy $m$–sphere $\Sigma^m$ is said to be \emph{exotic} if it is not diffeomorphic to $\mathbb{S}^m$.
\end{itemize}
\end{definition}

\begin{definition}\label{KM}\rm
The $m$–th group of smooth homotopy spheres $\Theta_m$ is defined as follows.
Elements are oriented $h$–cobordism classes $[\Sigma]$ of homotopy $m$–spheres $\Sigma$, where $\Sigma$ and $\Sigma^{\prime}$ are called oriented $h$–cobordant if there is an oriented $h$–cobordism $(W, \partial_0 W, \partial_1 W)$ together with orientation–preserving diffeomorphisms $\Sigma\longrightarrow \partial_0W$ and $(-\Sigma^{\prime})\longrightarrow \partial_1W$ (here $(-\Sigma^{\prime})$ is obtained from $\Sigma^{\prime}$ by reversing the orientation). The addition is given by the connected sum $\#$, the zero element is represented by $[\mathbb{S}^m]$, and the inverse of $[\Sigma]$ is given by $[-\Sigma]$.
Kervaire and Milnor \cite{KM63} showed that each $\Theta_m$ is a finite abelian group for $m\geq 1$; in particular, $\Theta_{m}\cong \mathbb{Z}_2$ for $m=8,14,16$, and $\Theta_{10}\cong \mathbb{Z}_6$.
For $m\geq 5$, the $h$–cobordism theorem \cite{Sma62} implies that $\Theta_m$ can be identified with the set of all oriented diffeomorphism classes of smooth structures on $\mathbb{S}^m$.
\end{definition}

Recall that elements of the group $\mathcal{C}(M)$ of concordance classes of smooth structures on a closed smooth manifold $M^m$ are represented by pairs $(N,f)$, where $f : N \longrightarrow M^m$ is a homeomorphism. The concordance class of $(N, f)$ is denoted by $[N,f]$, and the class $[M^m,\mathrm{Id}]$ of the identity $\mathrm{Id} : M^m \longrightarrow M^m$ is regarded as the base point of $\mathcal{C}(M)$.
There is a canonical homeomorphism 
\[
h_{\Sigma} : M^{m} \# \Sigma^{m} \longrightarrow M^{m}
\]
which agrees with the identity outside the homotopy sphere $\Sigma^{m}$ and is well-defined up to topological concordance.
We denote the class in $\mathcal{C}(M)$ of $(M^{m} \# \Sigma^{m}, h_{\Sigma})$ by $[M^{m} \# \Sigma^{m}]$ (note that $[M^{m} \# \mathbb{S}^{m}]$ is the class of $(M^{m}, \mathrm{Id})$).

The key to analyzing $\mathcal{C}(M)$ is the following result.

\begin{theorem}[Kirby and Siebenmann {\cite[p.~194]{KS77}}]\label{kirby}
Let $M$ be a smooth manifold of dimension at least $5$. Then there is a bijection
\[
\mathcal{C}(M) \cong [M, \operatorname{Top}/O]
\]
which takes the base point $[M,\mathrm{Id}]$ to the homotopy class of the constant map.
\end{theorem}
Recall that \(\operatorname{Top}/O\) has an infinite loop space structure \cite{BV73}, so the set \([M, \operatorname{Top}/O]\) inherits an abelian group structure, and the bijection in Theorem~\ref{kirby} endows \(\mathcal{C}(M)\) with the structure of an abelian group, with \([M, \mathrm{Id}]\) serving as the identity element. Let $f_{M} : M^{m} \longrightarrow \mathbb{S}^{m}$ be a degree–one map, well-defined up to homotopy.
Composition with $f_{M}$ induces a homomorphism
\[
f_{M}^{*} : [\mathbb{S}^{m}, \operatorname{Top}/O] \longrightarrow [M^{m}, \operatorname{Top}/O].
\]
Under the identifications
\[
\Theta_{m} = [\mathbb{S}^{m}, \operatorname{Top}/O] \quad \text{and} \quad 
\mathcal{C}(M^{m}) = [M^{m}, \operatorname{Top}/O],
\]
provided by Theorem~\ref{kirby}, the map $f_{M}^{*}$ sends 
\[
[\Sigma^{m}] \longmapsto [M^{m} \# \Sigma^{m}].
\]

The kernel
\[
\operatorname{Ker}\!\left( [\mathbb{S}^{m}, \operatorname{Top}/O] 
\xrightarrow{\, f_{M}^{*} \,} 
[M^{m}, \operatorname{Top}/O] \right)
\]
can be identified with a subgroup of $\Theta_{m}$, called the \emph{concordance inertia group} of $M$. It consists of those homotopy spheres $\Sigma \in \Theta_{m}$ for which the pairs $(M, \mathrm{Id})$ and $(M \# \Sigma, \mathrm{Id})$ are concordant.
We recall some results about complex projective spaces that will be used later. The spaces $SF$, $PL$, $PL/O$, $\operatorname{Top}/O$, and $F/O$ are $H$-spaces, and there exist $H$-space maps
\[
\phi : SF \longrightarrow F/O,\qquad \beta:PL\longrightarrow PL/O,\qquad j_{PL}:PL\longrightarrow F,\qquad \psi : \operatorname{Top}/O \longrightarrow F/O
\]
(see \cite{BV73}, \cite[p.~83]{MM79}) such that
\begin{equation}\label{equ1}
\phi_{*} : [\mathbb{C}P^{m}, SF] \longrightarrow [\mathbb{C}P^{m}, F/O]
\end{equation}
is a monomorphism for all $m \ge 1$ (\cite[Corollary~2.1]{Bru68}, \cite[Lemma~2.7]{BK18}), and
\begin{equation}\label{equ2}
\psi_{*} : [\mathbb{C}P^{m}, \operatorname{Top}/O] \longrightarrow [\mathbb{C}P^{m}, F/O]
\end{equation}
is also a monomorphism for all $m \ge 1$ (\cite[Lemma~2.6]{BK18}). Applying \([\mathbb{C}P^{m}, -]\) to the fibration
\[
\Omega(\operatorname{Top}/PL) \longrightarrow PL/O \xrightarrow{\pi} \operatorname{Top}/O \longrightarrow \operatorname{Top}/PL,
\]
and using the result \(\operatorname{Top}/PL \simeq K(\mathbb{Z}_{2},3)\) (\cite[p.~251, Theorem~5.5]{KS77}), the fact that \[[\mathbb{C}P^{m}, \Omega(F/PL)] = 0,\] and noting that the map \(\Omega(\operatorname{Top}/PL) \longrightarrow PL/O\) factors through \(\Omega(F/PL)\), it follows that the map
\begin{equation}\label{equ21}
F_{*} : [\mathbb{C}P^{m}, PL/O] \longrightarrow [\mathbb{C}P^{m}, \operatorname{Top}/O]
\end{equation}
is an isomorphism. By \cite[Proposition~2.4]{Bru68}, the map
\begin{equation}\label{equ22}
\beta_{*}:[\mathbb{C}P^{m}, PL] \longrightarrow [\mathbb{C}P^{m}, PL/O]
\end{equation}
is also an isomorphism, and the map
\begin{equation}\label{equ23}
(j_{PL})_{*}=J_{PL}:[\mathbb{C}P^{m}, PL] \longrightarrow [\mathbb{C}P^{m}, F]
\end{equation}
is a monomorphism. By \eqref{equ21}, \eqref{equ22}, and \eqref{equ23}, the group $[\mathbb{C}P^{m}, \operatorname{Top}/O]$ can be regarded as a subgroup of $[\mathbb{C}P^{m}, SF]$.

Brumfiel \cite[p.~400]{Bru71}, \cite[p.~12, Corollary~2.3]{Bru68} showed that there is a splitting
\begin{equation}\label{equ3}
[\mathbb{C}P^{m-1}, F/O] \cong \mathbb{Z}^{\left\lfloor \frac{m-1}{2} \right\rfloor} \oplus [\mathbb{C}P^{m-1}, SF],
\end{equation}
where
\[
\mathbb{Z}^{\left\lfloor \frac{m-1}{2} \right\rfloor} \subset \operatorname{Im}\left([\mathbb{C}P^{m}, F/O] \xrightarrow{i^{*}} [\mathbb{C}P^{m-1}, F/O]\right),
\]
and the torsion subgroup of $[\mathbb{C}P^{m}, F/O]$ is identified with $[\mathbb{C}P^{m}, SF]$. In \cite[\S1]{Bru71} and \cite[\S1]{Bru68}, Brumfiel defined an invariant
\[
d: [\mathbb{C}P^{m}, F/O] \longrightarrow \Theta_{2m+1}
\]
such that the composition
\[
[\mathbb{C}P^{m}, F/O] \xrightarrow{\;d\;} \Theta_{2m+1} \xrightarrow{\;\psi_{*}\;} \pi_{2m+1}(F/O) \cong \pi^{s}_{2m+1}/\operatorname{Im}(J) = \operatorname{Coker}(J_{2m+1})
\]
coincides with the map
\[
[\mathbb{C}P^{m}, F/O] \xrightarrow{\, p^{*}\,} [\mathbb{S}^{2m+1}, F/O]
\]
induced by the Hopf fibration $p : \mathbb{S}^{2m+1} \longrightarrow \mathbb{C}P^{m}$ (see \cite[Proposition 3.1]{Bru71}). In particular, the restriction
\[
d : [\mathbb{C}P^{m}, \operatorname{Top}/O] \longrightarrow \Theta_{2m+1}
\]
coincides with
\begin{equation}\label{equ41}
[\mathbb{C}P^{m}, \operatorname{Top}/O] \xrightarrow{\, p^{*}\,} [\mathbb{S}^{2m+1}, \operatorname{Top}/O].
\end{equation}
Furthermore, Brumfiel proved (\cite[p.~392, Corollary~5.5(i)]{Bru71} and \cite[Proposition 3.3, p.~18]{Bru68}) that for $m = 4k - 1$ ($k \ge 1$),
\begin{equation}\label{coker}
d\big([\mathbb{C}P^{m}, SF]\big) \subseteq \operatorname{Coker}(J_{2m+1}) \subset \Theta_{2m+1} = \mathit{bP}_{2m+2} \oplus \operatorname{Coker}(J_{2m+1}).
\end{equation}
We now introduce some notation that will be used throughout the paper. Let $C_{\alpha} = X \cup_{\alpha} CY$ be the mapping cone of a map $\alpha \colon Y \rightarrow X$. Denote by $i \colon X \rightarrow C_{\alpha}$ the inclusion and by $f_{C_{\alpha}} \colon C_{\alpha} \rightarrow \Sigma Y$ the quotient map that collapses $X$ to a point. Consider elements $\beta \in [X, Z]$ and $\gamma \in [W, Y]$ such that $\beta \circ \alpha$ and $\alpha \circ \gamma$ are null-homotopic, where $Z$ and $W$ are arbitrary spaces. We denote by $\overline{\beta} \in [C_{\alpha}, Z]$ an extension of $\beta$ satisfying $i^*(\overline{\beta}) = \beta$, where $i^* \colon [C_{\alpha}, Z] \rightarrow [X, Z]$. Similarly, let $\widetilde{\gamma} \in [\Sigma W, C_{\alpha}]$ be a coextension of $\gamma$ satisfying $(f_{C_{\alpha}})_{*}\left(\widetilde{\gamma}\right) = \Sigma \gamma$, where $(f_{C_{\alpha}})_{*} \colon [\Sigma W, C_{\alpha}] \rightarrow [\Sigma W, \Sigma Y]$. 

We shall fix the following notation: let $q \colon X^{m} \to X^{m}/X^{m-1} \cong \mathbb{S}^{m}$ be a quotient map that collapses the $(m-1)$-skeleton of the CW-complex $X$ to a point. If $\alpha \in \pi_{m}(W)$, then
\[
(\alpha)_{m} \in [X^{m}, W]
\]
denotes the element represented by the composition $X^{m} \xrightarrow{\, q \,} \mathbb{S}^{m} \xrightarrow{\, \alpha \,} W$. For $k \ge m$, a chosen extension of $(\alpha)_{m}$ along the restriction map $i^{*} \colon [X^{k}, W] \longrightarrow [X^{m}, W]$ is denoted by $\left(\overline{(\alpha)_{m}}\right)_{k}$.
We also recall the following result.
\begin{proposition}[{\cite[p.\,190, Proof of (5.1)]{Mos68}}]\label{St}
Let \(p:\mathbb{S}^{2m+1} \longrightarrow \mathbb{C}P^{m}\) be the Hopf fibration and \(f_{ \mathbb{C}P^{m}}:\mathbb{C}P^{m}\longrightarrow \mathbb{S}^{2m}\) be the degree-one map. Then the composite \(f_{\mathbb{C} P^{m}} \circ p: \mathbb{S}^{2m+1}\longrightarrow \mathbb{S}^{2m}\) is null-homotopic if and only if \(m\) is even.
\end{proposition}

In this paper, we mainly analyze the following long exact sequence associated to the cofiber sequence
\[
\mathbb{S}^{2m-1} \xrightarrow{p} \mathbb{C}P^{m-1} \xrightarrow{i} \mathbb{C}P^{m} \xrightarrow{f_{\mathbb{C}P^{m}}} \mathbb{S}^{2m},
\]
namely,
\begin{equation}\label{longG}
\cdots \longrightarrow [\mathbb{S}^{2m}, X] \xrightarrow{\, f^{*}_{\mathbb{C}P^{m}}\,} [\mathbb{C}P^{m}, X] \xrightarrow{\, i^{*}\,} [\mathbb{C}P^{m-1}, X] \xrightarrow{\, p^{*}\,} [\mathbb{S}^{2m-1}, X] \longrightarrow \cdots
\end{equation}
where \(X = \operatorname{Top}/O\) or \(SF\). The following lemma appears in \cite[Lemma~I.9]{Bru71} without proof. Note that the group structure of \([\mathbb{C}P^{4}, SF]\) is incorrectly stated there as \(\mathbb{Z}_{2} \oplus \mathbb{Z}_{2}\). We provide the corrected computations below.
\begin{lemma}[{\cite[Lemma~I.9]{Bru71}}]\label{lem:Bru-seq}
Let \( p : \mathbb{S}^{2m+1} \longrightarrow \mathbb{C}P^{m} \) be the Hopf fibration, and let \( f_{\mathbb{C}P^{m}} : \mathbb{C}P^{m} \longrightarrow \mathbb{S}^{2m} \) be the degree-one map. The following hold:
\begin{itemize}
    \item[(i)] \( [\mathbb{C}P^{2}, SF] = 0 \).
    
    \item[(ii)] The map \( f_{\mathbb{C}P^{3}}^{*} : [\mathbb{S}^{6}, SF] \longrightarrow [\mathbb{C}P^{3}, SF] \cong \mathbb{Z}_{2}\{(\nu^{2})_{6}\} \) is an isomorphism, and the induced map \( p^{*} : [\mathbb{C}P^{3}, SF] \longrightarrow [\mathbb{S}^{7}, SF] \) is trivial.
    
    \item[(iii)] There exists a non-split exact sequence
    \[
    0 \longrightarrow \mathbb{Z}_{2}\{x\} \xrightarrow{f_{\mathbb{C}P^{4}}^*} [\mathbb{C}P^{4}, SF] \cong \mathbb{Z}_{4}\{\overline{(\nu^{2})_{6}}\} \xrightarrow{i^*} [\mathbb{C}P^{3}, SF] \cong \mathbb{Z}_{2}\{(\nu^{2})_{6}\} \longrightarrow 0,
    \]
    where \( x \in \{\epsilon, \epsilon+\eta \circ \sigma\} \subset [\mathbb{S}^{8}, SF] \) and \( 2\overline{(\nu^{2})_{6}}=(x)_{8} \). Furthermore,
    \[
    \operatorname{Im}\left([\mathbb{C}P^{4}, SF] \xrightarrow{p^*} [\mathbb{S}^{9}, SF]\right) \cong \mathbb{Z}_{2}\{\nu^{3}\} \quad \text{and} \quad \operatorname{Ker}\left([\mathbb{C}P^{4}, SF] \xrightarrow{p^*} [\mathbb{S}^{9}, SF]\right) \cong \mathbb{Z}_{2}\{(x)_{8}\}.
    \]
    
    \item[(iv)] There exists a split exact sequence
    \[
    0 \longrightarrow [\mathbb{S}^{10}, SF] \xrightarrow{f_{\mathbb{C}P^{5}}^*} [\mathbb{C}P^{5}, SF] \xrightarrow{i^*} \mathbb{Z}_{2}\{(\epsilon)_{8}\} \longrightarrow 0,
    \]
    where \( \mathbb{Z}_{2}\{(\epsilon)_{8}\} \subset [\mathbb{C}P^{4}, SF] \). Furthermore,
    \[
    \operatorname{Im}\left([\mathbb{C}P^{5}, SF] \xrightarrow{p^*} [\mathbb{S}^{11}, SF]\right) \cong \mathbb{Z}_{2}\{\eta^2 \circ \mu\}
    \]
    and
    \[
    \operatorname{Ker}\left([\mathbb{C}P^{5}, SF] \xrightarrow{p^*} [\mathbb{S}^{11}, SF]\right) \cong \mathbb{Z}_{2}\{\left(\overline{(\epsilon)_{8}}\right)_{10}\} \oplus \mathbb{Z}_{3}\{(\beta_{1})_{10}\}.
    \]
    
    \item[(v)] There exists a split exact sequence
    \[
    0 \longrightarrow [\mathbb{C}P^{6}, SF] \xrightarrow{i^*} [\mathbb{C}P^{5}, SF] \xrightarrow{p^*} \mathbb{Z}_{2}\{\eta^2 \circ \mu\} \longrightarrow 0,
    \]
    where \( \mathbb{Z}_{2}\{\eta^2 \circ \mu\} \subset [\mathbb{S}^{11}, SF] \). Furthermore, the map \( p^{*} : [\mathbb{C}P^{6}, SF] \longrightarrow [\mathbb{S}^{13}, SF] \cong \mathbb{Z}_{3}\{\alpha_1 \circ \beta_1\} \) is surjective, with kernel
    \[
    \operatorname{Ker}\left([\mathbb{C}P^{6}, SF] \xrightarrow{p^*} [\mathbb{S}^{13}, SF]\right) \cong \mathbb{Z}_{2}\{\left(\overline{(\epsilon)_{8}}\right)_{12}\}.
    \]
\end{itemize}
\end{lemma}

\begin{proof}
We often use the exact sequence \eqref{longG} for \( X=SF \) and \( 2 \leq m \leq 6 \).

For the case \( m=2 \), since \( [\mathbb{S}^{4}, SF] = 0 \) and the map \( \eta^{*} : [\mathbb{S}^{2}, SF] \to [\mathbb{S}^{3}, SF] \) sends \( \eta^{2} \) to \( \eta^{3} = 4\nu \), which has order \( 2 \) (see \cite[Theorem 19.1]{Tod62}), it follows immediately from \eqref{longG} that \( [\mathbb{C}P^{2}, SF] = 0 \), which completes the proof of Part (i).

For the case \( m=3 \), it follows from Part (i) and the long exact sequence \eqref{longG}
\[
\dots \to [\Sigma \mathbb{C}P^{2}, SF] \xrightarrow{(\Sigma p)^{*}} [\mathbb{S}^{6}, SF] \xrightarrow{f_{\mathbb{C}P^{3}}^{*}} [\mathbb{C}P^{3}, SF] \to 0,
\]
where \( \Sigma p = \pm i \circ \nu' \) by \cite[(13), p. 189]{Muk82}, with \( i: \mathbb{S}^{3} \hookrightarrow \Sigma \mathbb{C}P^{2} \) and \( \nu' \in \pi_{6}(\mathbb{S}^{3}) \) such that \( \nu' \) is stably homotopic to \( 2\nu \in \pi_{3}^{s} \) (see \cite[(5.5), p. 42]{Tod62}). Since \( SF \) is an infinite loop space, the induced map \( (\nu')^{*}: [\mathbb{S}^{3}, SF] \to [\mathbb{S}^{6}, SF] \) can be identified with the map \( (2\nu)^{*}: \pi_{3}^{s} \to \pi_{6}^{s} \), which is the trivial map. This implies that the map \( (\Sigma p)^{*}: [\Sigma \mathbb{C}P^{2}, SF] \to [\mathbb{S}^{6}, SF] \) is trivial; hence, from the exact sequence above, \( f_{\mathbb{C}P^{3}}^{*}: [\mathbb{S}^{6}, SF] \to [\mathbb{C}P^{3}, SF] \) is an isomorphism, which completes the first assertion of Part (ii).
Furthermore, by Proposition \ref{St}, the composition \( \mathbb{S}^{7} \xrightarrow{p} \mathbb{C}P^{3} \xrightarrow{f_{\mathbb{C}P^{3}}} \mathbb{S}^{6} \) is represented by \( \eta \in \pi_{1}^{s} \). Consequently, the induced composition \( p^{*} \circ f_{\mathbb{C}P^{3}}^{*}: [\mathbb{S}^{6}, SF] \to [\mathbb{S}^{7}, SF] \) sends \( \nu^{2} \) to \( \eta \circ \nu^{2} = 0 \). Therefore, by the first assertion, the map \( p^{*}: [\mathbb{C}P^{3}, SF] \to [\mathbb{S}^{7}, SF] \) is trivial, which completes the proof of Part (ii). 

We now turn to the case \( m=4 \). Consider the following commutative diagram:
\begin{equation}\label{long-cp4}
\begin{xy}
\xymatrix{
[\Sigma \mathbb{C}P^3, SF] \ar[r]^{(\Sigma p)^{*}} & [\mathbb{S}^8, SF] \ar[r] & [\mathbb{C}P^4, SF] \ar[r]^{i^{*}} & [\mathbb{C}P^3, SF] \ar[r] & 0 \\
[\Sigma \mathbb{C}P^3, O] \ar[r]^{(\Sigma p)^{*}} \ar[u]^{J} & [\mathbb{S}^8, O] \ar[r] \ar[u]^{J} & 0 = [\mathbb{C}P^4, O] \ar[u]^{J} & &  
}
\end{xy}
\end{equation}
where the vertical maps are induced by the stable \(J\)-homomorphism \(O \longrightarrow SF\), the first row follows from the exact sequence \eqref{longG} and Part (ii), and the bottom map
\[
[\Sigma \mathbb{C}P^3, O] \longrightarrow [\mathbb{S}^8, O]
\] 
is surjective. Furthermore, the map \(J \colon [\mathbb{S}^8, O] \longrightarrow [\mathbb{S}^8, SF]\) is injective \cite{Ada66}, and its image is generated by the class \(\eta \circ \sigma = \epsilon + \bar{\nu}\) (see \cite[Theorem~1.1.13, p.~5]{Rav03} and \cite[Theorem 14.1, p.~190]{Tod62}). Since the first vertical map
\[
J : [\Sigma \mathbb{C}P^3, O] \longrightarrow [\Sigma \mathbb{C}P^3, SF]
\]
is surjective, as \([\Sigma \mathbb{C}P^3, F/O] = 0\), it follows from the commutativity of Diagram \eqref{long-cp4} that the image of the map
\[
[\Sigma \mathbb{C}P^3, SF] \xrightarrow{(\Sigma p)^{*}} [\mathbb{S}^8, SF]
\]
is generated by \(\eta \circ \sigma \in [\mathbb{S}^8, SF]\). Applying this to the first row of Diagram \eqref{long-cp4}, we obtain the short exact sequence
\begin{equation}\label{cp4-short}
0 \longrightarrow \mathbb{Z}_2\{x\} \xrightarrow{f_{\mathbb{C}P^4}^{*}} [\mathbb{C}P^4, SF] \xrightarrow{i^{*}} [\mathbb{C}P^3, SF]\cong \mathbb{Z}_{2}\{(\nu^2)_{6}\} \longrightarrow 0,
\end{equation}
where \( x \in \{\epsilon, \epsilon + \eta \circ \sigma\} \subset [\mathbb{S}^8, SF] \). We now show that this sequence does not split. Note that
\[
\mathbb{C}P^{4}/\mathbb{C}P^{2} \simeq \mathbb{S}^{6} \cup_{\eta} \mathbb{D}^{8} \cong \Sigma^{4}\mathbb{C}P^{2}.
\]
Applying the functor \([-, SF]\) to the cofiber sequence
\[
\mathbb{S}^{7} \xrightarrow{\eta} \mathbb{S}^{6} \xrightarrow{\Sigma^{4}i_{\mathbb{C}}} \Sigma^{4}\mathbb{C}P^{2} \xrightarrow{\Sigma^{4}f_{\mathbb{C}P^{2}}} \mathbb{S}^{8},
\]
where \(i_{\mathbb{C}}:\mathbb{S}^{2}\hookrightarrow \mathbb{C}P^{2}\) is the inclusion and \(f_{\mathbb{C}P^{2}}:\mathbb{C}P^{2}\to \mathbb{S}^{4}\) is the collapse map, we use the facts that \(\eta^{*}:[\mathbb{S}^{6},SF]\to [\mathbb{S}^{7},SF]\) is trivial and that the image of \(\eta^{*}:[\mathbb{S}^{7},SF]\to [\mathbb{S}^{8},SF]\) is \(\mathbb{Z}_{2}\{\eta \circ \sigma\}\) (since \(\eta \circ \nu^2=0\) and \(\eta \circ \sigma\neq 0\); see \cite{Tod62}). This yields the following short exact sequence:
\begin{equation}\label{exct-cp4}
0 \longrightarrow \mathbb{Z}_{2}\{x\} \xrightarrow{(\Sigma^{4}f_{\mathbb{C}P^{2}})^*} [\Sigma^{4}\mathbb{C}P^{2}, SF] \xrightarrow{(\Sigma^{4}i_{\mathbb{C}})^*} [\mathbb{S}^{6}, SF] \cong \mathbb{Z}_{2}\{\nu^{2}\} \longrightarrow 0,
\end{equation}
where \(x\in \{\epsilon, \epsilon+\eta \circ \sigma\}\). Let \(\overline{\nu^2}\in [\Sigma^{4}\mathbb{C}P^{2}, SF]\) be an extension of \(\nu^{2}\in [\mathbb{S}^{6}, SF]\) along the map \((\Sigma^{4}i_{\mathbb{C}})^{*}\). Then
\begin{equation}\label{todaeq1}
\begin{aligned}
2\,\overline{\nu^{2}} &= \overline{\nu^{2}} \circ (2\,\Sigma^{4} \iota_{\mathbb{C}}) \\
&= \overline{\nu^{2}} \circ (\Sigma^{4} i_{\mathbb{C}} \circ \overline{2 \iota_{6}} + \widetilde{2 \iota_{7}} \circ \Sigma^{4} f_{\mathbb{C}P^{2}}) \quad \text{by \cite[Corollary~2.6(i)]{KMNST01}} \\
&= \nu^{2} \circ \overline{2 \iota_{6}} + \overline{\nu^{2}} \circ \widetilde{2 \iota_{7}} \circ \Sigma^{4} f_{\mathbb{C}P^{2}}, \quad \text{since } (\Sigma^{4} i_{\mathbb{C}})^{*}(\overline{\nu^{2}}) = \nu^{2}.
\end{aligned}
\end{equation}
Applying Proposition~2.7(2) of \cite{KMNST01} with \(\alpha = \nu^{2}\), \(\beta = 2 \iota_{6}\), \(\gamma = \eta\), and \(p = \Sigma^{4} f_{\mathbb{C}P^{2}}\), we obtain \(\nu^{2} \circ \overline{2 \iota_{6}} \in \langle \nu^{2}, 2 \iota_{6}, \eta \rangle \circ \Sigma^{4} f_{\mathbb{C}P^{2}}\). Similarly, by taking \(\alpha = \nu^{2}\), \(\beta = \eta\), and \(\gamma = 2 \iota_{7}\) in Proposition~2.7(1) of \cite{KMNST01}, we have \(\overline{\nu^{2}} \circ \widetilde{2 \iota_{7}} \in \langle \nu^{2}, \eta, 2 \iota_{7} \rangle\). Since \(\langle \nu^{2}, \eta, 2 \rangle = 0\) and \(\langle \nu^{2}, 2, \eta \rangle = \{\epsilon, \epsilon + \eta \circ \sigma\}\) by \cite[Theorem 2.1(ii), p.~68]{Muk69}, it follows that
\[
2\overline{\nu^2} = \nu^{2} \circ \overline{2 \iota_{6}} = x \circ \Sigma^{4} f_{\mathbb{C}P^{2}} = (\Sigma^{4}f_{\mathbb{C}P^{2}})^{*}(x),
\]
which is non-zero in \([\Sigma^{4}\mathbb{C}P^{2}, SF]\) by \eqref{exct-cp4}. This shows that any extension \(\overline{\nu^{2}}\) of \(\nu^{2}\) has order \(4\). Hence, by \eqref{exct-cp4},
\[
[\mathbb{C}P^{4}/\mathbb{C}P^{2}, SF] \cong [\Sigma^{4}\mathbb{C}P^{2}, SF] \cong \mathbb{Z}_4\{\overline{\nu^{2}}\},
\]
where \(2\overline{\nu^{2}} = (\Sigma^{4}f_{\mathbb{C}P^{2}})^{*}(x)\). Since the map \(q^{*}:[\mathbb{C}P^{4}/\mathbb{C}P^{2}, SF]\to [\mathbb{C}P^{4}, SF]\) is surjective (as \([\mathbb{C}P^{2}, SF]=0\) by Part (i)) and \([\mathbb{C}P^{4}, SF]\) is a group of order 4 by \eqref{cp4-short}, \(q^*\) is an isomorphism. Thus, \([\mathbb{C}P^{4}, SF] \cong \mathbb{Z}_4\{\overline{\nu^{2}}\}\), which implies that the sequence \eqref{cp4-short} does not split. It remains to determine the kernel and image of the map \(p^{*}:[\mathbb{C}P^{4}, SF]\to [\mathbb{S}^{9},SF]\). Since, by Proposition \ref{St}, the composition \(\mathbb{S}^{9}\xrightarrow{p}\mathbb{C}P^{4}\xrightarrow{f_{\mathbb{C}P^{4}}}\mathbb{S}^{8}\) is null-homotopic, the induced composition \(p^{*}\circ f_{\mathbb{C}P^{4}}^{*}\) is trivial, and thus \(p^*\) sends \(2\overline{\nu^{2}}=(x)_{8}\) to zero. The map \(p^{*}\) fits into the following commutative diagram:
\begin{equation}\label{hopdcomm}
\begin{xy}
\xymatrix{
[\mathbb{C}P^4, SF] \ar[r]^{p^*} & [\mathbb{S}^{9}, SF] \\
[\mathbb{C}P^4 / \mathbb{C}P^2, SF] \ar[u]_{\cong}^{q^*} \ar[ur]^{p^{*}\circ q^*} &  
}
\end{xy}
\end{equation}
where \(q \circ p : \mathbb{S}^{9} \to \mathbb{C}P^4 / \mathbb{C}P^2\) is the attaching map for the $10$-cell in \(\mathbb{C}P^5 / \mathbb{C}P^2\), denoted by \(\varphi_{5}\). Since the composition of \(\varphi_{5}\) with the collapsing map to \(\mathbb{S}^{8}\) is null-homotopic, \(\varphi_{5}\) factors through the bottom cell \(\mathbb{S}^{6}\). Thus, \(\varphi_{5}\) is homotopic to the composition
\[
\mathbb{S}^{9} \xrightarrow{\lambda \nu} \mathbb{S}^{6} \xrightarrow{i} \mathbb{C}P^4 / \mathbb{C}P^2,
\]
where \(\lambda\) is odd and \(\nu \in \pi_3^s{}_{(2)}\). This follows because the Steenrod operation \(Sq^{4}\) is an isomorphism on \(H^{6}(\mathbb{C}P^5 / \mathbb{C}P^2; \mathbb{Z}_{2})\), while the Steenrod power operation \(\mathcal{P}^1\) is trivial on \(H^{6}(\mathbb{C}P^5 / \mathbb{C}P^2; \mathbb{Z}_{3})\).

Therefore, the induced map satisfies
\[
\varphi_{5}^{*} = (\lambda \nu)^{*} \circ i^{*},
\]
where \(i^{*}: [\mathbb{C}P^4 / \mathbb{C}P^2, SF] \to [\mathbb{S}^{6}, SF]\) is surjective by \eqref{exct-cp4}, and \((\lambda \nu)^{*}: [\mathbb{S}^{6}, SF] \to [\mathbb{S}^{9}, SF]\) sends \(\nu^{2}\) to \(\lambda \nu^{3}\), which is nonzero since \(\lambda\) is odd. Consequently, the image of \(\varphi_{5}^{*}\) is generated by \(\nu^{3}\).

Since \(q^{*}\) is an isomorphism, Diagram~\eqref{hopdcomm} implies that the generator \(\overline{\nu^{2}}\) of \([\mathbb{C}P^{4}, SF]\) maps nontrivially to \(\nu^{3}\) under \(p^{*}\). This shows that the kernel of \(p^{*}\) is generated by \((x)_{8}\), and the image of \(p^{*}\) is generated by \(\nu^{3}\), completing the proof of Part~(iii).

For the case \(m=5\), we first prove that the map \(f_{\mathbb{C}P^5}^{*}:[\mathbb{S}^{10}, SF] \cong \mathbb{Z}_{2}\{\eta \circ \mu\} \oplus \mathbb{Z}_{2}\{\beta_1\} \to [\mathbb{C}P^5, SF]\) is injective. Since the composition \(\mathbb{S}^{11} \xrightarrow{p} \mathbb{C}P^{5} \xrightarrow{f_{\mathbb{C}P^{5}}} \mathbb{S}^{10}\) is homotopic to \(\eta \in \pi_{1}^s\) by Proposition \ref{St}, the induced map \( p^{*} \circ f_{\mathbb{C}P^{5}}^{*}: [\mathbb{S}^{10}, SF] \to [\mathbb{S}^{11}, SF] \) sends \( \eta \circ \mu \) to \( \eta^{2} \circ \mu \neq 0\). This implies that \(f_{\mathbb{C}P^{5}}^{*}\) maps \(\eta \circ \mu\) non-trivially. To show that \(f_{\mathbb{C}P^{5}}^{*}\) is injective, it suffices to prove that the generator \(\beta_1 \in \mathbb{Z}_{3}\) maps non-trivially, which is equivalent to showing the map \((\Sigma p)_{(3)}^{*}:[\Sigma \mathbb{C}P^{4}, SF_{(3)}] \to [\Sigma \mathbb{S}^{9}, SF_{(3)}] \cong \mathbb{Z}_{3}\) is the trivial map by the exact sequence \eqref{longG}, where \((\Sigma p)_{(3)}: \Sigma \mathbb{S}^{9} \to \Sigma \mathbb{C}P^{4}\) is the 3-local map induced by the suspension of the Hopf fibration \(p: \mathbb{S}^{9} \to \mathbb{C}P^{4}\). It follows from \cite[Lemma 11.7(ii), p. 195]{Muk82} that \((\Sigma p)_{(3)}\) is stably homotopic to \(3\gamma\) for some \(\gamma \in \pi_{10}^{s}(\Sigma \mathbb{C}P^{4})\). Therefore, the induced map \(((\Sigma p)_{(3)})^{*} = 3(\gamma)^{*}: [\Sigma \mathbb{C}P^{4}, SF_{(3)}] \to \mathbb{Z}_{3}\) is trivial, which proves the injectivity of the map \(f_{\mathbb{C}P^{5}}^{*}\). Applying this and Part (iii) to the exact sequence \eqref{longG}, we obtain 
\begin{equation}\label{cp5-short}
0 \longrightarrow [\mathbb{S}^{10}, SF]\cong \mathbb{Z}_{6} \xrightarrow{f_{\mathbb{C}P^5}^{*}} [\mathbb{C}P^5, SF] \xrightarrow{i^{*}} \mathbb{Z}_{2}\{(x)_{8}\} \longrightarrow 0,
\end{equation}
where \(\mathbb{Z}_{2}\{(x)_{8}\} \subset [\mathbb{C}P^4, SF]\) and \(x \in \{\epsilon, \epsilon+\eta \circ \sigma\}\). We now show that this sequence splits. Consider the long exact sequence associated to the cofibre sequence \(\mathbb{C}P^{3}/\mathbb{C}P^{2} \xrightarrow{i} \mathbb{C}P^{5}/\mathbb{C}P^{2} \xrightarrow{q} \mathbb{C}P^{5}/\mathbb{C}P^{3} \simeq \mathbb{S}^{10} \vee \mathbb{S}^{8}\):
\begin{equation}\label{quotie-cp5}
\dots \to [\mathbb{S}^{7}, SF] \xrightarrow{\delta^{*} = ((\delta_{r_1})^{*}, (\delta_{r_2})^{*})} [\mathbb{S}^{10} \vee \mathbb{S}^{8}, SF] \xrightarrow{q^{*}} [\mathbb{C}P^{5}/\mathbb{C}P^{2}, SF] \xrightarrow{i^{*}} [\mathbb{C}P^{3}/\mathbb{C}P^{2}, SF],
\end{equation}
where \(\mathbb{C}P^{5}/\mathbb{C}P^{3} \simeq \mathbb{S}^{10} \vee \mathbb{S}^{8}\) by Proposition \ref{St}, and \(\delta_{r_1}, \delta_{r_2}\) are the restrictions of the connecting map \(\delta: \mathbb{S}^{10} \vee \mathbb{S}^{8} \to \mathbb{S}^{7}\). The cofibre of \(\delta\) is \(\Sigma \mathbb{C}P^{5}/\mathbb{C}P^{2}\); thus, \(\delta_{r_1}\) and \(\delta_{r_2}\) are the attaching maps of the suspensions of the 10-cell and 8-cell in \(\Sigma \mathbb{C}P^{5}/\mathbb{C}P^{2}\). As shown in the proof of Part (iii), \(\delta_{r_1}\) is homotopic to \(\eta\) and \(\delta_{r_2}\) is homotopic to \(\lambda \nu\), where \(\lambda\) is odd. Since \(\eta \circ \sigma \neq 0\) and \(\nu \circ \sigma = 0\) \cite{Tod62}, the image of the map \(\eta^{*}: [\mathbb{S}^{7}, SF] \to [\mathbb{S}^{8}, SF]\) is generated by \(\eta \circ \sigma\), whereas the map \((\lambda \nu)^{*}: [\mathbb{S}^{7}, SF] \to [\mathbb{S}^{10}, SF]\) is trivial. Consequently, the image of the map \(\delta^{*}: [\mathbb{S}^{7}, SF] \to [\mathbb{S}^{10}, SF] \oplus [\mathbb{S}^{8}, SF]\) is \(\mathbb{Z}_{2}\{\eta \circ \sigma\}\). 

On the other hand, the map \(i^{*}: [\mathbb{C}P^{5}/\mathbb{C}P^{2}, SF] \to [\mathbb{C}P^{3}/\mathbb{C}P^{2}, SF]\) factors through \([\mathbb{C}P^{4}/\mathbb{C}P^{2}, SF]\), and the generator \(\overline{\nu^{2}}\) of \([\mathbb{C}P^{4}/\mathbb{C}P^{2}, SF]\) maps non-trivially to \(\nu^{3}\) under the map \(\varphi_{5}^{*}: [\mathbb{C}P^{4}/\mathbb{C}P^{2}, SF] \to [\mathbb{S}^{9}, SF]\) induced by the top-cell attaching map \(\varphi_{5}\) (see Part (iii)). This implies that the generator \(\overline{\nu^{2}}\) cannot admit an extension in \([\mathbb{C}P^{5}/\mathbb{C}P^{2}, SF]\). Consequently, the composition \(i^{*}: [\mathbb{C}P^{5}/\mathbb{C}P^{2}, SF] \to [\mathbb{C}P^{4}/\mathbb{C}P^{2}, SF] \to [\mathbb{S}^{6}, SF] \cong \mathbb{Z}_{2}\{\nu^2\}\) is the trivial map. Combining these observations with the exact sequence \eqref{quotie-cp5}, we find that the map \(q^{*}\) induces an isomorphism \([\mathbb{S}^{10}, SF] \oplus \mathbb{Z}_{2}\{x\} \cong [\mathbb{C}P^{5}/\mathbb{C}P^{2}, SF]\), where \(x \in \{\epsilon, \epsilon + \eta \circ \sigma\}\). Since the map \(q^{*}: [\mathbb{C}P^{5}/\mathbb{C}P^{2}, SF] \to [\mathbb{C}P^{5}, SF]\) is surjective as \([\mathbb{C}P^{2}, SF] = 0\) by Part (i) and the group \([\mathbb{C}P^{5}, SF]\) is of order \(12\) by \eqref{cp5-short}, this implies that \([\mathbb{C}P^{5}, SF] \cong \mathbb{Z}_{6} \oplus \mathbb{Z}_{2}\). This completes the proof of the splitting of the short exact sequence \eqref{cp5-short}. Therefore, we write
\[
[\mathbb{C}P^{5}, SF] \cong \mathbb{Z}_{2}\{(\eta \circ \mu)_{10}\} \oplus \mathbb{Z}_{3}\{(\beta_{1})_{10}\} \oplus \mathbb{Z}_{2}\{\left(\overline{(\epsilon)_{8}}\right)_{10}\}.
\]
We now compute the image and kernel of the map \(p^{*}: [\mathbb{C}P^{5}, SF] \to [\mathbb{S}^{11}, SF]\). As established previously, the composition \( p^{*} \circ f_{\mathbb{C}P^{5}}^{*}: [\mathbb{S}^{10}, SF] \to [\mathbb{S}^{11}, SF] \) maps \( \eta \circ \mu \) to \( \eta^{2} \circ \mu \neq 0\) and \( \beta_1 \) to \( \eta \circ \beta_1 = 0 \). Since \( f_{\mathbb{C}P^{5}}^{*} \) is injective, it follows that \( \mathbb{Z}_{3}\{(\beta_{1})_{10}\} \) is contained in the kernel of \( p^{*} \). We prove that the generator \( \left(\overline{(\epsilon)_{8}}\right)_{10} \) is also in the kernel of $p^{*}$. Since \( \left(\overline{(\epsilon)_{8}}\right)_{10} \) is an extension of \( (\epsilon)_{8} = f_{\mathbb{C}P^{4}}^{*}(\epsilon) \), which generates the kernel of \( [\mathbb{C}P^4, SF] \xrightarrow{p^*} [\mathbb{S}^9, SF] \) by Part (iii), the composite 
\begin{equation}\label{trivia-cp3}
i^{*}: [\mathbb{C}P^5, SF] \to [\mathbb{C}P^3, SF]
\end{equation}
is the trivial map. This implies that the map \(q^{*}: [\mathbb{C}P^{5}/\mathbb{C}P^{3}, SF] \to [\mathbb{C}P^{5}, SF]\) is surjective. Note that the composition \( \varphi_{6}: \mathbb{S}^{11} \xrightarrow{p} \mathbb{C}P^{5} \xrightarrow{q} \mathbb{C}P^{5}/\mathbb{C}P^{3} \simeq \mathbb{S}^{10} \vee \mathbb{S}^{8} \) restricted to the first factor \( (\varphi_{6})_{1} \) is \( \eta \), and the restriction to the second factor \( (\varphi_{6})_{2}: \mathbb{S}^{11} \to \mathbb{S}^{8} \) is \( 2t\nu + \alpha_{1} \in \pi_{3}^{s} \) for some \( t \), since \( Sq^{4} \) is trivial on \( H^{8}(\mathbb{C}P^{6}/\mathbb{C}P^{3}; \mathbb{Z}_{2}) \) and \( \mathcal{P}^{1} \) is non-trivial on \( H^{8}(\mathbb{C}P^{6}/\mathbb{C}P^{3}; \mathbb{Z}_{3}) \). Therefore, the induced map \( (\varphi_{6})^{*} = p^{*} \circ q^{*} = \eta^{*} \oplus (2t\nu + \alpha_{1})^{*}: [\mathbb{S}^{10}, SF] \oplus [\mathbb{S}^{8}, SF] \to [\mathbb{S}^{11}, SF] \). As \( \nu \circ \pi_{8}^{s} = 0 \) and \( \alpha_1 \circ \pi_{8}^{s} = 0 \) \cite[Theorem 14.1]{Tod62}, we have \( (2t\nu + \alpha_{1})^{*} = 0 \); hence, the image of the composite \( p^{*} \circ q^{*} \) is generated by \( \eta^{2} \circ \mu \). Consequently, the image of \( p^{*} \) is generated by \( \eta^{2} \circ \mu \) because \( q^{*} \) is surjective. This also implies that \( \mathbb{Z}_{2}\{\left(\overline{(\epsilon)_{8}}\right)_{10}\} \) is contained in the kernel of \( p^{*} \), and thus
\[
\operatorname{Ker} \left( [\mathbb{C}P^5, SF] \xrightarrow{p^*} [\mathbb{S}^{11}, SF] \right) \cong \mathbb{Z}_{3}\{(\beta_{1})_{10}\} \oplus \mathbb{Z}_{2}\{\left(\overline{(\epsilon)_{8}}\right)_{10}\},
\]
which completes the proof of Part (iv). 

Finally, consider the case \( m=6 \). Using \( \pi_{12}(SF) = 0 \) and the kernel of \( p^{*} \) given by Part (iv) in the exact sequence \eqref{longG}, we obtain the first assertion of Part (v). It remains to determine the image and kernel of \( p^{*}:[\mathbb{C}P^6, SF]\to [\mathbb{S}^{13}, SF] \). Since \( [\mathbb{S}^{13}, SF] \cong \mathbb{Z}_{3}\{\alpha_1\beta_1\} \), it suffices to work 3-locally. Note that \( \mathbb{C}P^{6}/\mathbb{C}P^{3} \simeq_{(3)} \mathbb{C}P^{5}/\mathbb{C}P^{3} \cup_{\varphi_{6}} \mathbb{D}^{12} \simeq_{(3)} \mathbb{S}^{8} \cup_{(\varphi_{6})_{2}} \mathbb{D}^{12} \vee \mathbb{S}^{10} \), where \( \varphi_{6}: \mathbb{S}^{11} \xrightarrow{((\varphi_{6})_{1}, (\varphi_{6})_{2})} \mathbb{S}^{10} \vee \mathbb{S}^{8} \), the component \( (\varphi_{6})_{1}: \mathbb{S}^{11} \to \mathbb{S}^{10} \) is null-homotopic (3-locally), and \( (\varphi_{6})_{2}: \mathbb{S}^{11} \to \mathbb{S}^{8} \) is homotopic to \( \alpha_1 \in \pi_{3}^{s} \) because the Steenrod power operation \( \mathcal{P}^{1} \) is non-trivial on \( H^{8}(\mathbb{C}P^{6}/\mathbb{C}P^{3}; \mathbb{Z}_{3}) \). Since the composite \( \varphi_{7}: \mathbb{S}^{13} \xrightarrow{p} \mathbb{C}P^{6} \xrightarrow{q} \mathbb{C}P^{6}/\mathbb{C}P^{3} \simeq_{(3)} \mathbb{S}^{8} \cup_{\alpha_{1}} \mathbb{D}^{12} \vee \mathbb{S}^{10} \) is the attaching map of the 14-cell in \( \mathbb{C}P^{7}/\mathbb{C}P^{3} \) and \( \mathcal{P}^{1} \) is non-trivial on \( H^{10}(\mathbb{C}P^{7}/\mathbb{C}P^{3}; \mathbb{Z}_{3}) \), the component \( (\varphi_{7})_{2}: \mathbb{S}^{13} \to \mathbb{S}^{10} \) of \( \varphi_{7} \) is \( \alpha_1 \in \pi_{3}^{s} \). As \( (\alpha_{1})^{*}: [\mathbb{S}^{10}, SF_{(3)}] \to [\mathbb{S}^{13}, SF_{(3)}] \) takes \( \beta_1 \) to \( \alpha_1 \circ \beta_1 \) and is thus an isomorphism, the induced map \( \varphi_{7}^{*} = p^{*} \circ q^{*}: [\mathbb{C}P^{6}/\mathbb{C}P^{3}, SF] \to [\mathbb{S}^{13}, SF] \) is surjective; hence, \( p^{*}: [\mathbb{C}P^{6}, SF] \to [\mathbb{S}^{13}, SF] \) is also surjective. On the other hand, by the first assertion of Part (v), \( [\mathbb{C}P^{6}, SF] \cong \mathbb{Z}_{3}\{\left(\overline{(\beta_{1})_{10}}\right)_{12}\} \oplus \mathbb{Z}_{2}\{\left(\overline{(\epsilon)_{8}}\right)_{12}\} \). Using this identification, we conclude that the kernel of \( p^{*}: [\mathbb{C}P^{6}, SF] \to [\mathbb{S}^{13}, SF] \) is \( \mathbb{Z}_{2}\{\left(\overline{(\epsilon)_{8}}\right)_{12}\} \). This completes the proof of Part (v).
\end{proof}
Recall that the map
\[
f^{*}_{\mathbb{C}P^{m}} : [\mathbb{S}^{2m}, \operatorname{Top}/O] \longrightarrow [\mathbb{C}P^{m}, \operatorname{Top}/O]
\]
is given by the assignment
\[
[\Sigma^{2m}] \longmapsto [\mathbb{C}P^{m} \# \Sigma^{2m}].
\]
It follows from \cite[Remark 4.3]{Ram15} that the kernel
\[
\operatorname{Ker}\!\left(
[\mathbb{S}^{2m}, \operatorname{Top}/O]
\xrightarrow{\, f^{*}_{\mathbb{C}P^{m}}\,}
[\mathbb{C}P^{m}, \operatorname{Top}/O]
\right)
\]
can be identified with the inertia group of $\mathbb{C}P^{m}$, namely, the subgroup of $\Theta_{2m}$ consisting of all homotopy $2m$-spheres $\Sigma^{2m}$ such that $\mathbb{C}P^{m} \# \Sigma^{2m}$ is orientation-preservingly diffeomorphic to $\mathbb{C}P^{m}$. In \cite[Theorem~1]{Kaw68}, Kawakubo asserted that the inertia group of $\mathbb{C}P^{m}$ is trivial for $m \le 8$, although a complete proof does not seem to appear in the literature. As this result is used frequently in the sequel, we include a proof here.
\begin{theorem}[Kawakubo’s theorem {\cite{Kaw68}}]\label{kawa}
If $m \leq 8$, the map
\[
f^{*}_{\mathbb{C}P^{m}} : [\mathbb{S}^{2m}, \operatorname{Top}/O] \longrightarrow [\mathbb{C}P^{m}, \operatorname{Top}/O]
\]
is injective.
\end{theorem}
For the proof of Theorem~\ref{kawa} and for later use, we recall some results.
By \cite[Theorem~5.18]{MM79}, there is a (local) splitting of infinite loop spaces
\[
F/O_{(2)} \simeq BSO_{(2)} \times \operatorname{cok} J_{(2)},
\]
where \(\operatorname{cok} J_{(2)}\) is defined in \cite[Definition~5.16]{MM79} as the fibre of a map
\(F/O_{(2)} \longrightarrow BSO_{(2)}\).
This is itself an infinite loop space whose homotopy groups are isomorphic to the quotient
of the cokernel of the \(J\)-homomorphism, that is,
\[
\pi_*\bigl(\operatorname{cok} J_{(2)}\bigr) \cong \operatorname{Coker}(J_*) / \operatorname{Tors}(KO_*),
\]
(see \cite[p.~9]{CSS18} and \cite[Remark~11.43]{LM24}).
Here \(\operatorname{Tors}(KO_*) \cong \mathbb{Z}_2\) for \(* \equiv 1,2 \pmod{8}\), and is zero otherwise.
We now prove the following lemma.
\begin{lemma}\label{twonew}
\leavevmode
\begin{enumerate}
    \item Let $n > 4$. The map
    \[
        i^* \colon [\mathbb{C}P^n, SF] \longrightarrow [\mathbb{C}P^3, SF]
    \]
    is trivial.

    \item Let $n=8$ or $9$, and consider the mapping cone
    \[
        C_{g} = \mathbb{S}^n \cup_{g} \mathbb{D}^{n+4}
    \]
    where the attaching map
    \[
        g \colon \mathbb{S}^{n+3} \longrightarrow \mathbb{S}^n
    \]
    is represented by the class $2\nu \in \pi_{n+3}(\mathbb{S}^n) \cong \mathbb{Z}_8\{\nu\} \oplus \mathbb{Z}_3\{\alpha_1\}$.

    Let
    \[
      f_{C_{g}} \colon \mathbb{S}^n \cup_{g} \mathbb{D}^{n+4} \longrightarrow \mathbb{S}^{n+4}
    \]
    be the map which collapses $\mathbb{S}^n$ to a point, and let $\iota \colon \mathbb{S}^n \hookrightarrow \mathbb{S}^n \cup_{g} \mathbb{D}^{n+4}$ be the inclusion. Consider a map
    \[
        T \colon \mathbb{S}^{n+7} \longrightarrow \mathbb{S}^n \cup_{g} \mathbb{D}^{n+4}
    \]
    such that the composite
    \[
\mathbb{S}^{n+7} \xrightarrow{\quad T \quad} \mathbb{S}^n \cup_{g} \mathbb{D}^{n+4} \xrightarrow{\; f_{C_{g}} \;} \mathbb{S}^{n+4}
\]
    is represented by the class $\nu \in \pi_{n+7}(\mathbb{S}^{n+4})$. Then

    \begin{enumerate}
        \item[(a)] The map
        \[
            \iota^* \colon [\mathbb{S}^8 \cup_{g} \mathbb{D}^{12}, SF] \longrightarrow [\mathbb{S}^8, SF]
        \]
        induced by the inclusion $\iota \colon \mathbb{S}^8 \hookrightarrow \mathbb{S}^8 \cup_{g} \mathbb{D}^{12}$ is an isomorphism.

        \item[(b)] The map
        \[
            \iota^* \colon [\mathbb{S}^9 \cup_{g} \mathbb{D}^{13}, \operatorname{cok} J_{(2)}] \longrightarrow [\mathbb{S}^9, \operatorname{cok} J_{(2)}]
        \]
        is an isomorphism.

        \item[(c)] The kernel of 
        \[
            T^* \colon [\mathbb{S}^8 \cup_{g} \mathbb{D}^{12}, SF] \cong \mathbb{Z}_{2}\{\epsilon\} \oplus \mathbb{Z}_{2}\{\bar{\nu}\} \longrightarrow [\mathbb{S}^{15}, SF]
        \]
        is generated by the class $\epsilon+\bar{v}=\eta \circ \sigma$, and its image is generated by the class $\eta \circ \kappa \in \pi_{15}(SF) \cong \pi_{15}^s$.

        \item[(d)] The map
        \[
            T^* \colon [\mathbb{S}^9 \cup_{g} \mathbb{D}^{13}, \operatorname{cok} J_{(2)}] \longrightarrow [\mathbb{S}^{16}, \operatorname{cok} J_{(2)}]
        \]
        is trivial.
    \end{enumerate}
\end{enumerate}
\end{lemma}
\begin{proof}
Since the composite
\[
[\mathbb{C}P^5, SF] \xrightarrow{i^*} 
[\mathbb{C}P^4, SF] \xrightarrow{i^*} 
[\mathbb{C}P^3, SF]
\]
is trivial by \eqref{trivia-cp3}, and since the inclusion
\[
\mathbb{C}P^3 \hookrightarrow \mathbb{C}P^n \quad (n \geq 5)
\]
factors through \(\mathbb{C}P^5\), the proof of {\rm(1)} is complete.
Let $n=8$ or $9$. For the proof of $(2)(a)$ and $(2)(b)$, apply $[-, X]$ to the cofibre sequence
\[
    \mathbb{S}^{n+3} \xrightarrow{g} \mathbb{S}^n \xrightarrow{\iota} \mathbb{S}^n \cup_{g} \mathbb{D}^{n+4} \xrightarrow{f_{C_{g}}} \mathbb{S}^{n+4}
\]
where $X = SF$ or $\operatorname{cok} J_{(2)}$. This yields the long exact sequence
\[
    [\mathbb{S}^{n+4}, X] \xrightarrow{f^{*}_{C_{g}}} [\mathbb{S}^n \cup_{g} \mathbb{D}^{n+4}, X] \xrightarrow{\iota^*} [\mathbb{S}^n, X] \xrightarrow{(2\nu)^*} [\mathbb{S}^{n+3}, X].
\]
Using the facts that $\pi_{12}^s = 0$, $\pi_8^s = \mathbb{Z}_2 \oplus \mathbb{Z}_2$, $[\mathbb{S}^9, \operatorname{cok} J_{(2)}] \cong \mathbb{Z}_2$~\cite[Table A3.3]{Rav03}, and $[\mathbb{S}^{13}, \operatorname{cok} J_{(2)}] = 0$ in the above exact sequence, we obtain the statements of $(2)(a)$ and $(2)(b)$. We now turn to the proof of \((2)(c)\).
By Part \((2)(a)\) the map
\[
\iota^* : [\mathbb{S}^8 \cup_{g} \mathbb{D}^{12}, SF]
   \longrightarrow [\mathbb{S}^8, SF] \cong \mathbb{Z}_2\{\epsilon\} \oplus \mathbb{Z}_2\{\bar{\nu}\}
\]
is an isomorphism.
By assumption, the composite
\[
\mathbb{S}^{15} \xrightarrow{T} \mathbb{S}^8 \cup_{g} \mathbb{D}^{12}
   \xrightarrow{f_{C_{g}}} \mathbb{S}^{12}
\]
represents the element \(\nu \in \pi_{15}(\mathbb{S}^{12})\), and the attaching map
\(g : \mathbb{S}^{11} \longrightarrow \mathbb{S}^8\) represents \(2\nu\).
Let \(\overline{x} : \mathbb{S}^8 \cup_{g} \mathbb{D}^{12} \longrightarrow SF\) be the extension of
\(x : \mathbb{S}^8 \longrightarrow SF\) along the map \(\iota^{*}\). Applying Proposition~2.7(1) of \cite{KMNST01} with
\(\alpha = x\), \(\beta = 2\nu\) and \(\gamma = \nu\), so that the coextension of \(\nu\) is
\(\widetilde{\nu} = T\), we obtain
\[
T^{*}(\overline{x}) = \overline{x} \circ \widetilde{\nu}
\in \langle x, 2\nu, \nu \rangle \subseteq [\mathbb{S}^{15}, SF],
\]
since \(2\nu \circ x = 0\) in \(\pi_{11}^s\) for every \(x \in \pi_8^s\).
It follows from \cite[pp.~110--111]{Tod62} that the Toda brackets
\[
\langle \epsilon, 2\nu, \nu \rangle,\qquad
\langle \bar{\nu}, 2\nu, \nu \rangle,\qquad
\langle \epsilon + \bar{\nu}, 2\nu, \nu \rangle
\]
are all singletons, with
\[
\langle \epsilon, 2\nu, \nu \rangle
  = -\,\langle \bar{\nu}, 2\nu, \nu \rangle
  = \eta \circ \kappa,
\qquad
\langle \epsilon + \bar{\nu}, 2\nu, \nu \rangle = 0 \in \pi_{15}^s.
\]
Therefore \(T^{*}(\overline{x})\) is either \(0\) or \(\eta \circ \kappa\), and in fact
\(T^{*}(\overline{x}) = \eta \circ \kappa\) occurs for both \(x = \epsilon\) and
\(x = \bar{\nu}\).
Hence
\[
\operatorname{Im}\!\left(
  [\mathbb{S}^8 \cup_{g} \mathbb{D}^{12}, SF]\cong \mathbb{Z}_2\{\epsilon\} \oplus \mathbb{Z}_2\{\bar{\nu}\}
    \xrightarrow{\,T^*\,}
  [\mathbb{S}^{15}, SF]
\right)
\cong \mathbb{Z}_2\{\eta \circ \kappa\},
\] and \[
\operatorname{Ker}\!\left(
  [\mathbb{S}^8 \cup_{g} \mathbb{D}^{12}, SF]\cong \mathbb{Z}_2\{\epsilon\} \oplus \mathbb{Z}_2\{\bar{\nu}\}
    \xrightarrow{\,T^*\,}
  [\mathbb{S}^{15}, SF]
\right)
\cong \mathbb{Z}_2\{\eta \circ \sigma\},
\] where $\epsilon+\bar{\nu}=\eta \circ \sigma$ by \cite[Theorem 14.1, p.190]{Tod62}. This proves \((2)(c)\). Finally, consider the case \(n=9\).
We want to show that the map
\[
    T^* :
    [\mathbb{S}^9 \cup_{g} \mathbb{D}^{13}, \operatorname{cok} J_{(2)}]
    \longrightarrow
    [\mathbb{S}^{16}, \operatorname{cok} J_{(2)}]
\]
is trivial. Since
\[
   \iota^* :
   [\mathbb{S}^9 \cup_{g} \mathbb{D}^{13}, \operatorname{cok} J_{(2)}]
   \longrightarrow
   [\mathbb{S}^9, \operatorname{cok} J_{(2)}]
\]
is an isomorphism by \((2)(b)\),
\(
[\mathbb{S}^9, \operatorname{cok} J_{(2)}]
   \cong \mathbb{Z}_2\{\nu^3\}
\)
\cite[Lemma~3.8]{BKS25}, and
\(
\pi_{16}(\operatorname{cok} J_{(2)}) \cong \mathbb{Z}_2\{\eta^*\}
\)
\cite[Table~A3.3]{Rav03}, Proposition~2.7(1) of \cite{KMNST01}, applied as in the proof of \((2)(c)\), yields
\[
  T^{*}(\overline{\nu^3})
    = \overline{\nu^3}\circ \widetilde{\nu}
    \in \langle \nu^3, 2\nu, \nu \rangle
    \subseteq [\mathbb{S}^{16}, SF],
\]
where \(\overline{\nu^3}\) is the extension of \(\nu^3\) and \(\widetilde{\nu} = T\) is the coextension of \(\nu\).

The indeterminacy of this Toda bracket is
\[
    \pi_{7}^s \circ \nu^3 + \nu \circ \pi_{13}^s = 0,
\]
so \(\langle \nu^3, 2\nu, \nu \rangle\) consists of a single element.
By \cite[(3.5), p.~33]{Tod62}, we have
\[
    \langle \nu^3, 2\nu, \nu \rangle
    \subseteq \langle \nu^2, 2\nu^2, \nu \rangle,
\]
and \(\langle \nu^2, 2\nu^2, \nu \rangle = 0\), since \(2\nu^2 = 0\) and the indeterminacy of this bracket is \(0\).
Therefore
\[
    T^*(\overline{\nu^3}) = 0,
\]
which implies that
\[
    T^* :
    [\mathbb{S}^9 \cup_{g} \mathbb{D}^{13}, \operatorname{cok} J_{(2)}]
    \longrightarrow
    [\mathbb{S}^{16}, \operatorname{cok} J_{(2)}]
\]
is the trivial map, as required.

\end{proof}
\subsection*{Proof of Theorem~\ref{kawa}} \rm

For $n=1,2,3$ and $6$, we have
\[
\Theta_{2n} \cong [\mathbb{S}^{2n}, \operatorname{Top}/O] = 0,
\]
so the claim holds in these cases. For $n=4$ and $5$, the proof follows from \cite[Lemma~3.17]{FJ94}. Consider now the case $n = 7$ and the commutative diagram
\begin{equation}\label{kappa1}
\xymatrix{
    \Theta_{14} \cong [\mathbb{S}^{14}, \operatorname{Top}/O] \ar[rr]^{f^*_{\mathbb{C}P^7}} \ar[d]_{\psi_*} && [\mathbb{C}P^{7}, \operatorname{Top}/O] \ar[d]^{\psi_*} \\
    [\mathbb{S}^{14}, F/O] \ar[rr]^{f^*_{\mathbb{C}P^7}} \ar@{<-}[d]_{\phi_*} && [\mathbb{C}P^{7}, F/O] \ar@{<-}[d]^{\phi_*} \\
    [\mathbb{S}^{14}, SF] \ar[rr]^{f^*_{\mathbb{C}P^7}} && [\mathbb{C}P^{7}, SF]
}
\end{equation}
where $\phi_* : [\mathbb{S}^{14}, SF] \cong \mathbb{Z}_2\{\kappa\} \oplus \mathbb{Z}_2\{\sigma^2\} \to [\mathbb{S}^{14}, F/O]$ is an isomorphism, since $[\mathbb{S}^{14}, SO] \cong [\mathbb{S}^{15}, BSO] = 0$. The image of the map 
\begin{equation}\label{kappagenrator}
\psi_* : [\mathbb{S}^{14}, \operatorname{Top}/O] \cong \mathbb{Z}_2 \longrightarrow [\mathbb{S}^{14}, F/O] 
\end{equation}
is generated by $\phi_{*}(\kappa)$, as both $\sigma^2$ and $\kappa + \sigma^2$ are Kervaire invariant one elements \cite{Bro69} (see, for example, \cite[Figure 2.1]{BHHM20} and \cite[Remark 4.5]{Xu16}). 
Applying (\ref{equ1}) and (\ref{equ2}) to Diagram \eqref{kappa1}, to show that the map $f^*_{\mathbb{C}P^7} : [\mathbb{S}^{14}, \operatorname{Top}/O] \to [\mathbb{C}P^7, \operatorname{Top}/O]$ is injective, it suffices to show that $f^*_{\mathbb{C}P^7}: [\mathbb{S}^{14}, SF] \to [\mathbb{C}P^7, SF]$ maps $\kappa$ nontrivially. From Proposition~\ref{St}, we note that the composite $\mathbb{S}^{15} \xrightarrow{p} \mathbb{C}P^7 \xrightarrow{f_{\mathbb{C}P^7}} \mathbb{S}^{14}$ is homotopic to $\eta \in \pi_{15}(\mathbb{S}^{14})$, which yields the following commutative diagram:
\begin{equation}\label{kappa2}
\xymatrix{
[\mathbb{S}^{14}, SF] \ar[dr]_{\eta^{*}} \ar[rr]^{f^{*}_{\mathbb{C}P^{7}}} && 
[\mathbb{C}P^{7}, SF] \ar[dl]^{p^{*}} \\
& [\mathbb{S}^{15}, SF] &
}
\end{equation}
It follows from \cite[p.~190]{Tod62} that $\eta \circ \kappa \neq 0$. Therefore, the commutativity of Diagram \eqref{kappa2} implies that $f^*_{\mathbb{C}P^7} : [\mathbb{S}^{14}, SF] \to [\mathbb{C}P^7, SF]$ sends $\kappa$ to a nontrivial element. This completes the proof for the case $n = 7$. We now turn to the case $n=8$. Consider the following commutative diagram :
\[
\xymatrix{
[\Sigma \mathbb{C}P^3, \Omega G/\operatorname{Top}] \ar[r] \ar[d] 
  & [\mathbb{C}P^8 / \mathbb{C}P^3, \Omega G/\operatorname{Top}] \ar[d] \\
[\Sigma \mathbb{C}P^3, \operatorname{Top}/O] \ar[r] \ar[d] 
  & [\mathbb{C}P^{8}/\mathbb{C}P^{3}, \operatorname{Top}/O] \ar@{->>}[r]^{q^{*}}
  & [\mathbb{C}P^8, \operatorname{Top}/O] \ar[r]^{i^{*}} & [\mathbb{C}P^3, \operatorname{Top}/O] \\
[\Sigma \mathbb{C}P^3, F/O]
}
\]
where the rows are part of the long exact sequence induced by the cofibration $$\mathbb{C}P^3 \xrightarrow{i} \mathbb{C}P^8 \xrightarrow{q} \mathbb{C}P^8 / \mathbb{C}P^3,$$ the vertical arrows are part of the long exact sequence induced by the fibration $$\Omega G/\operatorname{Top} \longrightarrow \operatorname{Top}/O \longrightarrow F/O,$$ and the surjectivity of \(q^*\) follows from the fact that $[\mathbb{C}P^3, \operatorname{Top}/O]=0$. Note that $$[\mathbb{C}P^8 / \mathbb{C}P^3, \Omega G/\operatorname{Top}]=0$$ as $[\mathbb{C}P^k / \mathbb{C}P^3, \Omega G/\operatorname{Top}]=0$ for all $k\geq 4$, and $[\Sigma \mathbb{C}P^3, F/O]=0$ as $[\mathbb{S}^7, F/O]=0$ and $[\Sigma \mathbb{C}P^2, F/O]=0$. Therefore, from the exactness of the first column and the commutativity of the above diagram, the map $[\Sigma \mathbb{C}P^3, \Omega G/\operatorname{Top}]\longrightarrow 
[\Sigma \mathbb{C}P^3, \operatorname{Top}/O]$ is surjective and hence the map $[\Sigma \mathbb{C}P^3, \operatorname{Top}/O]\longrightarrow [\mathbb{C}P^8 / \mathbb{C}P^3, \operatorname{Top}/O]$ is trivial. This, in turn, shows that the middle horizontal map 
\begin{equation}\label{quo}
q^*:[\mathbb{C}P^8 / \mathbb{C}P^3, \operatorname{Top}/O]\longrightarrow [\mathbb{C}P^8, \operatorname{Top}/O]
\end{equation}
is an isomorphism. Since the map
$f_{\mathbb{C}P^8} : \mathbb{C}P^8 \longrightarrow \mathbb{S}^{16} $
factors through the quotient map
$ q : \mathbb{C}P^8 \longrightarrow \mathbb{C}P^8 / \mathbb{C}P^3,$ we obtain the commutative diagram
\[
\xymatrix{
[\mathbb{S}^{16}, \operatorname{Top}/O] \ar[r]^{f_{\mathbb{C}P^8}^*} \ar[d]_{f_{\mathbb{C}P^8 / \mathbb{C}P^3}^*} 
& [\mathbb{C}P^8, \operatorname{Top}/O] \\
[\mathbb{C}P^8 / \mathbb{C}P^3, \operatorname{Top}/O] \ar[ur]^{\cong}_{q^*}
&
}
\]
where
$f_{\mathbb{C}P^8 / \mathbb{C}P^3} :
\mathbb{C}P^8 / \mathbb{C}P^3 \longrightarrow \mathbb{S}^{16}$
is the quotient map onto the top cell, obtained by collapsing the \(15\)\nobreakdash-skeleton of \(\mathbb{C}P^8 / \mathbb{C}P^3\) to a point. Therefore, instead of proving that
\[
f_{\mathbb{C}P^8}^* : [\mathbb{S}^{16}, \operatorname{Top}/O] \longrightarrow [\mathbb{C}P^8, \operatorname{Top}/O]
\]
is injective, it is enough to prove that the map
\[
f_{\mathbb{C}P^8 / \mathbb{C}P^3}^* : [\mathbb{S}^{16}, \operatorname{Top}/O] \longrightarrow [\mathbb{C}P^8 / \mathbb{C}P^3, \operatorname{Top}/O]
\]
is injective. Note that
\[
\mathbb{C}P^8 / \mathbb{C}P^3\simeq \mathbb{C}P^7 / \mathbb{C}P^3\cup_{\varphi_{8}} \mathbb{D}^{16}
\quad\text{and}\quad
\mathbb{C}P^7 / \mathbb{C}P^3\simeq \mathbb{C}P^6 / \mathbb{C}P^3\cup_{\varphi_{7}} \mathbb{D}^{14},
\]
where the top cell attaching map
\[
\varphi_{n} : \mathbb{S}^{2n-1} \longrightarrow \mathbb{C}P^{n-1} / \mathbb{C}P^k
\]
is the composite
\[
\mathbb{S}^{2n-1} \xrightarrow{p}  \mathbb{C}P^{n-1} \xrightarrow{q} \mathbb{C}P^{n-1} / \mathbb{C}P^k,
\]
for \(n=7,8\) and \(k=3\). It follows from Proposition~\ref{St} that the composite of \(\varphi_{7}\) with the map to the top cell
\[
\mathbb{S}^{13}\xrightarrow{\varphi_{7}} \mathbb{C}P^6 /\mathbb{C}P^3 \longrightarrow \mathbb{S}^{12}
\]
is null-homotopic. Therefore the \(14\)\nobreakdash-cell attaching map descends to a map
\[
\mathbb{S}^{13}\longrightarrow \mathbb{C}P^5 / \mathbb{C}P^3\simeq \mathbb{S}^{8} \vee \mathbb{S}^{10}
\]
and is given by the composite
\[
\mathbb{S}^{13} \xrightarrow{\bar{\varphi}} \mathbb{S}^{10}\hookrightarrow \mathbb{S}^{8} \vee \mathbb{S}^{10}
  \hookrightarrow \mathbb{C}P^6 / \mathbb{C}P^3,
\]
as \(\pi_{13}(\mathbb{S}^{8}) = 0\). Note that this composition, followed by the map \(\mathbb{C}P^6 / \mathbb{C}P^3\longrightarrow \mathbb{C}P^6 / \mathbb{C}P^4\), is the top cell attaching map of
\(\mathbb{C}P^{7} / \mathbb{C}P^{4}\), and the Steenrod operation \(\operatorname{Sq}^{4}\) on
\(H^{10}(\mathbb{C}P^{7} / \mathbb{C}P^{4},\mathbb{Z}_2)\) is trivial. Hence the factor map
\(\bar{\varphi}\) lies in
\(\pi_{13}(\mathbb{S}^{10})\cong \pi^{s}_{3}\) and can be written as
\(\bar{\varphi}=2 \lambda \nu+l \alpha_1\).
From Table~5.3 and the proof of Proposition~5.2 in \cite[p.~185 and p.~190]{Mos68}, it follows that \(\lambda=0\).
Since the Steenrod power operation \(P^{1}\) detects \(\alpha_1\) in the stable homotopy groups of
spheres and $\mathcal{P}^{1}$ is an isomorphism on
\(H^{10}(\mathbb{C}P^{7} / \mathbb{C}P^{4},\mathbb{Z}_3)\), it follows that
\(l\neq 0\in \mathbb{Z}_3\).
Therefore, \(\bar{\varphi}=l \alpha_1 \in \pi_{13}(\mathbb{S}^{10})\).
This implies that 
\begin{equation}\label{spil1}
\mathbb{C}P^{7} / \mathbb{C}P^{3} \simeq_{(2)}
\mathbb{C}P^{6} / \mathbb{C}P^{3} \,\vee\, \mathbb{S}^{14},
\end{equation}
and hence
\begin{equation}\label{spil11}
\mathbb{C}P^{8} / \mathbb{C}P^{3}
  \simeq_{(2)} \bigl(\mathbb{C}P^{6} / \mathbb{C}P^{3} \vee \mathbb{S}^{14}\bigr)\cup_{\varphi_8} \mathbb{D}^{16},
\end{equation}
where
\[
\varphi_8:\mathbb{S}^{15}\xrightarrow{p} \mathbb{C}P^{7} \xrightarrow{q} \mathbb{C}P^{7} / \mathbb{C}P^{3}\simeq
\mathbb{C}P^{6} / \mathbb{C}P^{3} \,\vee\, \mathbb{S}^{14}
\]
is given by \(\varphi_8=((\varphi_8)_1, (\varphi_8)_2)\), with first component
\((\varphi_8)_1: \mathbb{S}^{15}\longrightarrow \mathbb{C}P^{6} / \mathbb{C}P^{3}\) and second component
\((\varphi_8)_2=\eta \in \pi_{15}(\mathbb{S}^{14})\) by Proposition~\ref{St}. Applying \([-\,,\operatorname{Top}/O]\) to the cofibre sequence
\[
\mathbb{S}^{15} \xrightarrow{\varphi_{8} = ((\varphi_{8})_1,\eta)} \mathbb{C}P^{6} / \mathbb{C}P^{3} \vee \mathbb{S}^{14}
   \longrightarrow \mathbb{C}P^{8} / \mathbb{C}P^{3},
\]
we obtain the exact sequence
\[
\cdots \longrightarrow
[\Sigma \mathbb{C}P^{6} / \mathbb{C}P^{3}, \operatorname{Top}/O]
   \xrightarrow{\;(\Sigma (\varphi_{8})_{1})^{*}\oplus \eta^*\;}
[\mathbb{S}^{16}, \operatorname{Top}/O]\xrightarrow{f_{\mathbb{C}P^{8}/\mathbb{C}P^{3}}^{*}}
[\mathbb{C}P^{8} / \mathbb{C}P^{3}, \operatorname{Top}/O] \longrightarrow \cdots .
\]
where
\[
\eta^{*} : [\mathbb{S}^{15}, \operatorname{Top}/O]
   \longrightarrow [\mathbb{S}^{16}, \operatorname{Top}/O]
\]
is the trivial map by \cite[Lemma~3.1]{BKS25}. From this exact sequence, to prove that the map \(f_{\mathbb{C}P^{8}/\mathbb{C}P^{3}}^{*}\) is injective, we show that the map 
\begin{equation}\label{conntcp7}
(\Sigma (\varphi_{8})_{1})^{*}:[\Sigma \mathbb{C}P^{6} / \mathbb{C}P^{3}, \operatorname{Top}/O]\longrightarrow [\mathbb{S}^{16}, \operatorname{Top}/O]
\end{equation}
is trivial, which concludes the result for the case $n=8$. As \([\mathbb{S}^{16}, \operatorname{Top}/O]\cong \mathbb{Z}_2\), we now work \(2\)\nobreakdash-locally. Consider the following commutative diagram
\begin{equation}\label{equ7}
\begin{gathered}
\xymatrix{
[\Sigma\mathbb{C}P^{6}/\mathbb{C}P^{3}, \operatorname{Top}/O_{(2)}]
   \ar[rr]^-{\;(\Sigma (\varphi_{8})_{1})^{*}\;} \ar[d]_{\,(\operatorname{Pr}\circ \psi_{(2)})_{*}\,} &&
[\mathbb{S}^{16}, \operatorname{Top}/O_{(2)}]
   \ar[d]^{\,(\operatorname{Pr}\circ \psi_{(2)})_{*}\,} \cong \mathbb{Z}_{2} \\
[\Sigma \mathbb{C}P^{6}/\mathbb{C}P^{3}, \operatorname{cok} J_{(2)}]
   \ar[rr]^-{\;(\Sigma (\varphi_{8})_{1})^{*}\;} &&
[\mathbb{S}^{16}, \operatorname{cok} J_{(2)}]
}
\end{gathered}
\end{equation}
where
\[
\psi_{(2)} : \operatorname{Top}/O_{(2)} \longrightarrow F/O_{(2)}
\]
is the canonical map localized at the prime \(2\) and
\[
\operatorname{Pr} : F/O_{(2)} \longrightarrow \operatorname{cok} J_{(2)}
\]
is the projection onto \(\operatorname{cok} J_{(2)}\). By the results of Kervaire and Milnor
\cite{KM63}, the composite
\[
(\operatorname{Pr}\circ \psi_{(2)})_{*} :
[\mathbb{S}^{16}, \operatorname{Top}/O_{(2)}]
   \longrightarrow [\mathbb{S}^{16}, \operatorname{cok} J_{(2)}]
\]
is a bijection, and its image is generated by the class \([\eta^{*}]\)
(see \cite[Table~A3.3]{Rav03}). So, from Diagram~\ref{equ7}, we reduce the claim to showing that the map
\[
(\Sigma(\varphi_{8})_{1})^{*}: [\Sigma \mathbb{C}P^{6}/\mathbb{C}P^{3}, \operatorname{cok} J_{(2)}]\longrightarrow 
[\mathbb{S}^{16}, \operatorname{cok} J_{(2)}]
\]
is trivial. Note that \(\mathbb{C}P^{6}/\mathbb{C}P^{3}\simeq_{(2)} (\mathbb{S}^8\vee \mathbb{S}^{10})\cup_{\varphi_{6}} \mathbb{D}^{12}\), where the map \(\varphi_{6}:\mathbb{S}^{11}\longrightarrow \mathbb{S}^8\vee \mathbb{S}^{10}\) is \(\eta\in \pi_1^s\) on the first factor by Proposition~\ref{St} and \(2\nu\in \pi_3^s\) on the second factor by \cite[Proposition~5.2, Table 5.3]{Mos68}, using the fact that the Steenrod operation \(\operatorname{Sq}^{4}\) on
\(H^{8}(\mathbb{C}P^{6} / \mathbb{C}P^{3}, \mathbb{Z}_2)\) is trivial. 
Consider the homotopy commutative diagram of homotopy cofibrations
\begin{equation}\label{equ8}
\begin{gathered}
\xymatrix{
\Sigma\mathbb{S}^{11}
   \ar[rr]^-{\;\Sigma \varphi_{6}\;} \ar[d]_{\operatorname{id}} &&
\Sigma(\mathbb{S}^{8}\vee \mathbb{S}^{10})
   \ar[rr]^-{\;\Sigma i\;} \ar[d]_{\Sigma p_{1}} &&
\Sigma(\mathbb{C}P^{6}/\mathbb{C}P^{3})
   \ar[rr]^-{\;\Sigma f_{\mathbb{C}P^{6}/\mathbb{C}P^{3}}\;} \ar[d]_{\Sigma \Psi} &&
\Sigma\mathbb{S}^{12} \ar[d]^{\operatorname{id}} \\
\Sigma\mathbb{S}^{11}
   \ar[rr]^-{\;\Sigma(2\nu)\;} &&
\Sigma\mathbb{S}^{8}
   \ar[rr]^-{\;\Sigma \iota\;} &&
\Sigma(\mathbb{S}^{8}\cup_{2\nu}\mathbb{D}^{12})
   \ar[rr]^-{\;\Sigma q\;} &&
\Sigma\mathbb{S}^{12}
}
\end{gathered}
\end{equation}
where \(\Psi : \mathbb{C}P^{6}/\mathbb{C}P^{3} \longrightarrow \mathbb{S}^{8}\cup_{2\nu}\mathbb{D}^{12}\) satisfies
\(q \circ \Psi \simeq f_{\mathbb{C}P^{6}/\mathbb{C}P^{3}}\). Now apply \([-\,,\operatorname{cok} J_{(2)}]\) to the above homotopy commutative diagram. We obtain the induced square
\begin{equation}\label{equ9}
\begin{gathered}
\xymatrix{
[\Sigma(\mathbb{C}P^{6}/\mathbb{C}P^{3}),\operatorname{cok} J_{(2)}]
   \ar[d]_{\,(\Sigma i)^{*}\,} &
&
[\Sigma(\mathbb{S}^{8}\cup_{2\nu}\mathbb{D}^{12}),\operatorname{cok} J_{(2)}]
   \ar[ll]_-{\;(\Sigma\Psi)^{*}\;} \ar[d]^{\,(\Sigma\iota)^{*}\,} \\
[\Sigma(\mathbb{S}^{8}\vee \mathbb{S}^{10}),\operatorname{cok} J_{(2)}]
&
&
[\Sigma\mathbb{S}^{8},\operatorname{cok} J_{(2)}]
   \ar[ll]_-{\;(\Sigma p_{1})^{*}\;}^-{\;\cong\;}
}
\end{gathered}
\end{equation}
is commutative, and the first vertical map \((\Sigma i)^{*}\) is an isomorphism, since
\(\pi_{12}^s=0\) and \(\pi_{13}(\operatorname{cok} J_{(2)})=0\), while the second vertical map
\((\Sigma\iota)^{*}\) is also an isomorphism by Lemma~\ref{twonew}(2)(b). Hence the horizontal map
\[
(\Sigma\Psi)^{*}: [\Sigma(\mathbb{S}^{8}\cup_{2\nu}\mathbb{D}^{12}),\operatorname{cok} J_{(2)}] \longrightarrow
[\Sigma(\mathbb{C}P^{6}/\mathbb{C}P^{3}),\operatorname{cok} J_{(2)}]
\]
is an isomorphism. Moreover, from Diagram~\ref{equ8} we observe that
\(\Sigma q \circ \Sigma \Psi \simeq \Sigma f_{\mathbb{C}P^{6}/\mathbb{C}P^{3}}\), and hence
\(\Sigma q \circ \Sigma \Psi \circ \Sigma (\varphi_{8})_{1} \simeq
\Sigma f_{\mathbb{C}P^{6}/\mathbb{C}P^{3}} \circ \Sigma (\varphi_{8})_{1}\).
If the composite
\[
\Sigma f_{\mathbb{C}P^{6}/\mathbb{C}P^{3}} \circ \Sigma (\varphi_{8})_{1} :
\mathbb{S}^{16}\xrightarrow{\Sigma(\varphi_{8})_{1}}\Sigma \mathbb{C}P^{6}/\mathbb{C}P^{3}\xrightarrow{\Sigma f_{\mathbb{C}P^{6}/\mathbb{C}P^{3}}} \mathbb{S}^{13}
\]
is homotopic to \(\nu\), then the map
\begin{equation}\label{nuhomo}
\Sigma q \circ \bigl(\Sigma \Psi \circ \Sigma (\varphi_{8})_{1}\bigr)
\end{equation}
is also homotopic to \(\nu\). Hence, by applying Lemma~\ref{twonew}(2)(d) with
\(T=\Sigma \Psi \circ \Sigma (\varphi_{8})_{1} \in \pi_{16}(\mathbb{S}^{9}\cup_{2\nu}\mathbb{D}^{13})\),
we see that the induced map \(\bigl(\Sigma \Psi \circ \Sigma (\varphi_{8})_{1}\bigr)^{*}\) is trivial.
This means that
\[
(\Sigma (\varphi_{8})_{1})^{*}:
[\Sigma(\mathbb{C}P^{6}/\mathbb{C}P^{3}),\operatorname{cok} J_{(2)}]
\longrightarrow
[\mathbb{S}^{16},\operatorname{cok} J_{(2)}]
\]
is trivial. This completes the proof of the result for the case \(n=8\).
Thus it remains to prove that the composite
\[
\mathbb{S}^{16}\xrightarrow{\Sigma(\varphi_{8})_{1}}
\Sigma (\mathbb{C}P^{6}/\mathbb{C}P^{3})
\xrightarrow{\Sigma f_{\mathbb{C}P^{6}/\mathbb{C}P^{3}}} \mathbb{S}^{13}
\]
is homotopic to \(\nu\). Note that the cofiber of the composite
\[
\mathbb{S}^{15} \xrightarrow{\varphi_{8}=((\varphi_{8})_{1},\eta)}
\mathbb{C}P^{6}/\mathbb{C}P^{3}\vee \mathbb{S}^{14}
\simeq_{(2)} \mathbb{C}P^{7}/\mathbb{C}P^{3}
\xrightarrow{q}\mathbb{C}P^{7}/\mathbb{C}P^{5}\simeq \mathbb{S}^{12} \vee \mathbb{S}^{14} 
\]
is \(\mathbb{C}P^{8}/\mathbb{C}P^{5}\). As \(\operatorname{Sq}^{4}\) on
\(H^{12}(\mathbb{C}P^{8}/\mathbb{C}P^{5};\mathbb{Z}_{2})\) is nontrivial, it follows again from
\cite[Proposition~5.2,Table 5.3]{Mos68} that the \(16\)\nobreakdash-cell in \(\mathbb{C}P^{8}/\mathbb{C}P^{5}\) attaches onto the \(12\)\nobreakdash-cell by the map \(\nu\). Hence the composite 
\[
\mathbb{S}^{15}\xrightarrow{(\varphi_{8})_{1}} \mathbb{C}P^{6}/\mathbb{C}P^{3}
\xrightarrow{f_{\mathbb{C}P^{6}/\mathbb{C}P^{3}}} \mathbb{C}P^{6}/\mathbb{C}P^{5}\simeq \mathbb{S}^{12}
\]
is represented by \(\nu\). This completes the required claim. \hfill $\square$ \\
To compute the group \(\mathcal{C}(\mathbb{C}P^{m})\), we will mainly use Proposition~\ref{St} together with Theorem~\ref{kawa}.

\begin{theorem}\label{main}
\begin{itemize}
\item[(i)] There is a split short exact sequence
\[
0 \longrightarrow \Theta_{10} \longrightarrow \mathcal{C}(\mathbb{C}P^{5})
   \longrightarrow \mathcal{C}(\mathbb{C}P^{4}) \longrightarrow 0,
\]
where \(\Theta_{10}\cong \mathbb{Z}_{6}\), and the natural map
\(f^{*}_{\mathbb{C}P^{4}}:\Theta_{8}\cong \mathbb{Z}_{2} \longrightarrow \mathcal{C}(\mathbb{C}P^{4})\)
is bijective.

\item[(ii)] There is a split short exact sequence
\[
0 \longrightarrow \mathcal{C}(\mathbb{C}P^{6}) \longrightarrow
\mathcal{C}(\mathbb{C}P^{5}) \longrightarrow \operatorname{Im}(p^{*}) \longrightarrow 0,
\]
where $\mathcal{C}(\mathbb{C}P^{6})\cong \mathbb{Z}_{6}$ and the map
\(p^{*}: \mathcal{C}(\mathbb{C}P^{5}) \longrightarrow \mathbb{Z}_2 \subset \mathcal{C}(\mathbb{S}^{11})\)
is induced by the Hopf fibration \(p:\mathbb{S}^{11}\longrightarrow \mathbb{C}P^{5}\).

\item[(iii)] There is a split short exact sequence
\[
0 \longrightarrow \Theta_{14} \longrightarrow \mathcal{C}(\mathbb{C}P^{7})
   \longrightarrow \operatorname{Ker}(p^{*}) \longrightarrow 0,
\]
where \(\Theta_{14}\cong \mathbb{Z}_{2}\), and the map $p^{*}: \mathcal{C}(\mathbb{C}P^{6})\longrightarrow \mathcal{C}(\mathbb{S}^{13})\cong \mathbb{Z}_3$ induced by the Hopf fibration $p:\mathbb{S}^{13}\longrightarrow \mathbb{C}P^{6}$ is surjective.
\end{itemize}
\end{theorem}
\begin{proof}
Since the induced map $f^{*}_{\mathbb{C} P^{4}}: [\mathbb{S}^{8}, \operatorname{Top}/O] \to [\mathbb{C} P^{4}, \operatorname{Top}/O]$ is an isomorphism by \cite[Theorem~2.3]{Ram15} and $[\mathbb{S}^{8}, \operatorname{Top}/O] \cong \Theta_8 \cong \mathbb{Z}_2$, the nontrivial element in $[\mathbb{C} P^{4}, \operatorname{Top}/O]$ is represented by a composite map
\[
g: \mathbb{C} P^{4} \xrightarrow{\,f_{\mathbb{C} P^{4}}\,} \mathbb{S}^{8} \xrightarrow{\,\Sigma\,} \operatorname{Top}/O,
\]
where $\Sigma: \mathbb{S}^{8} \to \operatorname{Top}/O$ represents the exotic $8$-sphere in $\Theta_8$. Therefore, the effect of the induced map $p^{*}: [\mathbb{C} P^{4}, \operatorname{Top}/O] \to [\mathbb{S}^{9}, \operatorname{Top}/O]$ on the homotopy class $[g]$ is represented by the composite
\[
\mathbb{S}^9 \xrightarrow{\,p\,} \mathbb{C} P^{4} \xrightarrow{\,f_{\mathbb{C} P^{4}}\,} \mathbb{S}^{8} \xrightarrow{\,\Sigma\,} \operatorname{Top}/O.
\]
Since $f_{\mathbb{C} P^{4}} \circ p: \mathbb{S}^9 \to \mathbb{S}^{8}$ is null-homotopic by Proposition~\ref{St}, it follows that $p^{*}\bigl([g]\bigr) = 0$; hence, $p^{*}: [\mathbb{C} P^{4}, \operatorname{Top}/O] \to [\mathbb{S}^{9}, \operatorname{Top}/O]$ is the trivial map. Consequently, from the exact sequence \eqref{longG} and Theorem~\ref{kawa}, we obtain the following short exact sequence:
\[
0 \longrightarrow \Theta_{10} \longrightarrow [\mathbb{C}P^5, \operatorname{Top}/O] \longrightarrow [\mathbb{C}P^4, \operatorname{Top}/O] \longrightarrow 0.
\]
This sequence splits, as follows from Lemma~\ref{lem:Bru-seq}(iv) by identifying $[\mathbb{C}P^5, \operatorname{Top}/O]$ as a subgroup of $[\mathbb{C}P^5, SF]$. This completes the proof of Part \textup{(i)}. For Part \textup{(ii)}, we consider the exact sequence \eqref{longG} with $m=6$. Since
\[
[\mathbb{S}^{12}, \operatorname{Top}/O] \cong \Theta_{12} = 0,
\]
the induced map
\[
i^{*}: [\mathbb{C} P^{6}, \operatorname{Top}/O] \longrightarrow [\mathbb{C} P^{5}, \operatorname{Top}/O]
\]
is injective. Now consider the map
\[
\psi_*: [\mathbb{C} P^{6}, \operatorname{Top}/O] \longrightarrow [\mathbb{C} P^{6}, F/O],
\]
induced by the natural map $\psi: \operatorname{Top}/O \to F/O$, which is injective by \eqref{equ2}. In view of \eqref{equ3} and Lemma~\ref{lem:Bru-seq}(v), the group
\[
[\mathbb{C} P^{6}, F/O] \cong \mathbb{Z}^3 \oplus \mathbb{Z}_2 \oplus \mathbb{Z}_3,
\]
whereas, by Part \textup{(i)},
\[
[\mathbb{C} P^{5}, \operatorname{Top}/O] \cong \mathbb{Z}_6 \oplus \mathbb{Z}_2.
\]
Therefore, the map
\[
i^{*}: [\mathbb{C} P^{6}, \operatorname{Top}/O] \longrightarrow [\mathbb{C} P^{5}, \operatorname{Top}/O]
\]
is not surjective. Applying these observations to the exact sequence \eqref{longG}, and using the fact that 
$[\mathbb{S}^{11}, \operatorname{Top}/O] \cong \mathbb{Z}_{992}$, we obtain a short exact sequence
\[
0 \longrightarrow [\mathbb{C}P^6, \operatorname{Top}/O] \longrightarrow [\mathbb{C}P^5, \operatorname{Top}/O] \xrightarrow{p^{*}} \operatorname{Im}(p^{*}) \cong \mathbb{Z}_2 \longrightarrow 0,
\]
where
\[
p^{*}: [\mathbb{C}P^5, \operatorname{Top}/O] \longrightarrow [\mathbb{S}^{11}, \operatorname{Top}/O] \cong \Theta_{11} = \mathbb{Z}_{992}
\]
is induced by the Hopf fibration \(p: \mathbb{S}^{11} \to \mathbb{C}P^{5}\). The splitting of this sequence follows from the decomposition
\[
[\mathbb{C} P^{5}, \operatorname{Top}/O] \cong \mathbb{Z}_{2} \oplus \mathbb{Z}_{3} \oplus \mathbb{Z}_2,
\]
given in Part~\textup{(i)}, which completes the proof of Part~\textup{(ii)}. We finally turn to the proof of Part (iii). We first prove that the map 
$p^{*}: [\mathbb{C}P^{6}, \operatorname{Top}/O] \longrightarrow [\mathbb{S}^{13}, \operatorname{Top}/O]$ 
is surjective. Consider the following commutative diagram:
\begin{equation}\label{digram3}
\begin{CD}
0 @. 0 @. 0 @. 0 \\
@VVV @VVV @VVV @VVV \\
\mathbb{Z}_2 \cong [\mathbb{S}^{14}, \operatorname{Top}/O] @>f_{\mathbb{C} P^{7}}>> [\mathbb{C}P^{7}, \operatorname{Top}/O] @>i^{*}>> [\mathbb{C}P^{6}, \operatorname{Top}/O] @>p^{*}>> [\mathbb{S}^{13}, \operatorname{Top}/O] \cong \mathbb{Z}_3 \\
@VV\psi_*V @VV\psi_*V @VV\psi_*V @V\cong V\psi_*V \\
\mathbb{Z}_2 \oplus \mathbb{Z}_2 \cong [\mathbb{S}^{14}, F/O] @>f_{\mathbb{C} P^{7}}>> [\mathbb{C}P^{7}, F/O] @>i^{*}>> [\mathbb{C}P^{6}, F/O] @>p^{*}>> [\mathbb{S}^{13}, F/O] \\
@A\cong A\phi_*A @AA\phi_*A @AA\phi_*A @A \cong A\phi_*A \\
\mathbb{Z}_2 \oplus \mathbb{Z}_2 \cong [\mathbb{S}^{14}, SF] @>f_{\mathbb{C} P^{7}}>> [\mathbb{C}P^{7}, SF] @>i^{*}>> [\mathbb{C}P^{6}, SF] @>p^{*}>> [\mathbb{S}^{13}, SF] \\
@AAA @AAA @AAA @AAA \\
0 @. 0 @. 0 @. 0 \\
\end{CD}
\end{equation}
where the rows are part of the long exact sequences induced by the cofiber sequence
\[
\mathbb{S}^{13} \xrightarrow{p} \mathbb{C}P^{6} \xrightarrow{i} \mathbb{C}P^{7} \xrightarrow{f_{\mathbb{C}P^{7}}} \mathbb{S}^{14},
\]
and the vertical maps $\psi_*: [\mathbb{S}^{13}, \operatorname{Top}/O] \longrightarrow [\mathbb{S}^{13}, F/O]$ and $\phi_*: [\mathbb{S}^{n}, SF] \longrightarrow [\mathbb{S}^{n}, F/O]$ for $n=13,14$ are isomorphisms by the Kervaire--Milnor braid (see, for example, \cite[p.~93]{Lev85} and \cite[p.~463]{LM24}), using the facts that $\Theta_{13} \cong \mathbb{Z}_3$, $[\mathbb{S}^{13}, F/O] \cong \mathbb{Z}_3$, $[\mathbb{S}^{13}, G/\operatorname{Top}] = 0$, $[\mathbb{S}^{13}, SO] = 0$, and $[\mathbb{S}^{14}, SO] = 0$. 
Since $[\mathbb{C}P^{6}, \operatorname{Top}/O] \cong [\mathbb{C}P^{6}, SF] \cong \mathbb{Z}_2 \oplus \mathbb{Z}_3$ by Part (ii) and Lemma \ref{lem:Bru-seq}(v), it follows that the image of the vertical map $\psi_*: [\mathbb{C} P^{6}, \operatorname{Top}/O] \longrightarrow [\mathbb{C} P^{6}, F/O]$ is $\phi_*([\mathbb{C}P^{6}, SF])$. Therefore, a simple diagram chase in the right two columns of Diagram \eqref{digram3} shows that, to prove that the map $p^{*}: [\mathbb{C}P^{6}, \operatorname{Top}/O] \longrightarrow [\mathbb{S}^{13}, \operatorname{Top}/O]$ is surjective, it suffices to prove the surjectivity of the map $p^{*}: [\mathbb{C}P^{6}, SF] \longrightarrow [\mathbb{S}^{13}, SF]$, which follows from \cite[Lemma I.9(iv)]{Bru71}. It then follows from \eqref{longG} and Theorem~\ref{kawa} with $m=7$ that there is a short exact sequence 
\begin{equation}\label{short-7}
0 \longrightarrow [\mathbb{S}^{14}, \operatorname{Top}/O] \longrightarrow [\mathbb{C}P^{7}, \operatorname{Top}/O] \longrightarrow \operatorname{Ker}(p^{*}) \cong \mathbb{Z}_2 \longrightarrow 0.
\end{equation}

We now show that this sequence splits. Note that the group $[\mathbb{C}P^{7}, \operatorname{Top}/O]$ is $2$-torsion, so it suffices to work $2$-locally. It follows from \eqref{spil1} and the proof of Theorem~\ref{kawa} in the case $n=8$ that
\[
\mathbb{C}P^{7} / \mathbb{C}P^{3} \simeq_{(2)} \mathbb{C}P^{6} / \mathbb{C}P^{3} \vee \mathbb{S}^{14}
\quad \text{and} \quad
\mathbb{C}P^{6} / \mathbb{C}P^{3} \simeq_{(2)} (\mathbb{S}^{8} \vee \mathbb{S}^{10}) \cup_{\varphi_{6}} \mathbb{D}^{12},
\]
where $\varphi_{6} = (2\nu, \eta)$. Hence,
\[
[\mathbb{C}P^{7} / \mathbb{C}P^{3}, \operatorname{Top}/O_{(2)}] \cong [\mathbb{C}P^{6} / \mathbb{C}P^{3}, \operatorname{Top}/O_{(2)}] \oplus [\mathbb{S}^{14}, \operatorname{Top}/O_{(2)}].
\]
Moreover, the map
\begin{equation}\label{cp6-1}
i^{*}: [\mathbb{C}P^{6} / \mathbb{C}P^{3}, \operatorname{Top}/O_{(2)}] \longrightarrow [\mathbb{S}^8, \operatorname{Top}/O_{(2)}]
\end{equation}
is an isomorphism. This follows by applying $[-, \operatorname{Top}/O_{(2)}]$ to the cofiber sequence
\[
\mathbb{S}^{11} \xrightarrow{(2\nu, \eta)} \mathbb{S}^{8} \vee \mathbb{S}^{10} \longrightarrow \mathbb{C}P^{6} / \mathbb{C}P^{3},
\]
using the facts that $[\mathbb{S}^{12}, \operatorname{Top}/O_{(2)}] = 0$ and that the map $\eta^*: [\mathbb{S}^{10}, \operatorname{Top}/O_{(2)}] \longrightarrow [\mathbb{S}^{11}, \operatorname{Top}/O_{(2)}]$ is injective by \cite[Lemma~3.1]{BKS25}. Therefore,
\begin{equation}\label{SF-sum}
[\mathbb{C}P^{7} / \mathbb{C}P^{3}, \operatorname{Top}/O_{(2)}] \cong \mathbb{Z}_2 \oplus \mathbb{Z}_2.
\end{equation}
Note that the map $q^*: [\mathbb{C}P^{7} / \mathbb{C}P^{3}, \operatorname{Top}/O_{(2)}] \longrightarrow [\mathbb{C}P^{7}, \operatorname{Top}/O_{(2)}]$ is surjective, since $[\mathbb{C}P^{3}, \operatorname{Top}/O_{(2)}] = 0$. Combined with \eqref{short-7} and \eqref{SF-sum}, this implies that the short exact sequence \eqref{short-7} splits. This proves Part \textup{(iii)}.
\end{proof}
We now compute the group \([\mathbb{C}P^{m}, SF]\) for \(m=7\) and \(8\), which will be used later.

\begin{proposition}\label{stab}
There is a split short exact sequence
\[
0\longrightarrow [\mathbb{S}^{14}, SF]\xrightarrow{f^{*}_{\mathbb{C}P^{7}}}
[\mathbb{C}P^{7},SF]\longrightarrow \operatorname{Ker}(p^{*})\longrightarrow 0,
\]
where \([\mathbb{S}^{14}, SF]\cong \mathbb{Z}_2\oplus \mathbb{Z}_2\), and the map
\[
p^{*}:[\mathbb{C}P^{6},SF]\cong \mathbb{Z}_2\oplus \mathbb{Z}_3
   \longrightarrow [\mathbb{S}^{13}, SF]\cong \mathbb{Z}_3,
\]
induced by the Hopf fibration \(p:\mathbb{S}^{13}\longrightarrow \mathbb{C}P^{6}\), is surjective.
\end{proposition}
\begin{proof}
From the surgery exact sequences of \(\mathbb{C}P^{7}\) and \(\mathbb{S}^{14}\), we obtain the
commutative diagram (see \cite[Lemma~3.4]{Cro10})
\begin{equation}\label{digram31}
\begin{CD}
0@>>> \Theta_{14} @>\psi_*>> \pi_{14}(F/O)\cong \mathbb{Z}_2\oplus \mathbb{Z}_2
   @>s_{\mathbb{S}^{14}}>>  L_{14}(e)\cong \mathbb{Z}_2\\
@VV=V            @VVf_{\mathbb{C}P^{7}}^{*}V             @VVf_{\mathbb{C}P^{7}}^*V
   @VV=V\\
0@>>> \mathcal{S}_{\rm Diff}(\mathbb{C}P^{7})
   @>\eta_{\mathbb{C}P^{7}}>>    [\mathbb{C}P^{7}, F/O]
   @>s_{\mathbb{C}P^{7}}>>   L_{14}(e)\cong \mathbb{Z}_2 .
\end{CD}
\end{equation}
where $\mathcal{S}_{\rm Diff}(\mathbb{C}P^{m})$ is the smooth structure set of $\mathbb{C}P^{m}$ (see e.g. \cite{Bro72, Ran92, LM24, Wal99}) and the map $f_{\mathbb{C}P^{7}}^{*}: \Theta_{14}\longrightarrow \mathcal{S}_{\rm Diff}(\mathbb{C}P^{7})$ given by $[\Sigma]\mapsto [\mathbb{C}P^{7}\#\Sigma]$ is injective by \cite[Remark 4.3]{Ram15} and Theorem \ref{kawa}. A diagram chase in \eqref{digram31} shows that the map
\(f_{\mathbb{C}P^{7}}:[\mathbb{S}^{14},F/O]\longrightarrow[\mathbb{C}P^{7},F/O]\) is injective.
Together with the commutative diagram \eqref{digram3}, this implies that the induced map
\[
f_{\mathbb{C}P^{7}}^{*}:[\mathbb{S}^{14},SF]\longrightarrow[\mathbb{C}P^{7},SF]
\]
is also injective. We also note that the map
\[
p^{*} : [\mathbb{C}P^{6}, SF] \longrightarrow [\mathbb{S}^{13}, SF]
\]
is surjective with kernel isomorphic to \(\mathbb{Z}_2\) by Lemma~\ref{lem:Bru-seq}(v). Applying these facts to the long exact sequence appearing in the last row of \eqref{digram3}, we obtain a short exact sequence
\begin{equation}\label{spil2}
0\longrightarrow [\mathbb{S}^{14},SF]\cong \mathbb{Z}_2\oplus \mathbb{Z}_2
\longrightarrow [\mathbb{C}P^{7},SF]\longrightarrow \operatorname{Ker}(p^{*})\cong \mathbb{Z}_2
\longrightarrow 0.
\end{equation}
Next, by repeating the arguments used in the proof of Theorem~\ref{main}\,(iii), with
\(\operatorname{Top}/O\) replaced by \(SF\), we obtain
\[
[\mathbb{C}P^{7} / \mathbb{C}P^{3}, SF_{(2)}]\cong
[\mathbb{C}P^{6} / \mathbb{C}P^{3}, SF_{(2)}]\oplus
[\mathbb{C}P^{7} / \mathbb{C}P^{6}, SF_{(2)}]\cong
[\mathbb{C}P^{6} / \mathbb{C}P^{3}, SF_{(2)}]\oplus
[\mathbb{S}^{14}, SF_{(2)}].
\]
and the map
\[
i^*:[\mathbb{C}P^{6} / \mathbb{C}P^{3}, SF_{(2)}]\longrightarrow
[\mathbb{S}^{8}, SF_{(2)}]
\]
is an isomorphism, since \(\eta^{2}\circ \mu\neq 0\) (see \cite[p.~190]{Tod62}). Therefore,
\begin{align}\label{SF-spilit}
[\mathbb{C}P^{7} / \mathbb{C}P^{3}, SF_{(2)}]
&\cong [\mathbb{S}^{8}, SF_{(2)}]\oplus [\mathbb{S}^{14}, SF_{(2)}] \notag\\[0.3em]
&\cong 
\mathbb{Z}_2\{\left (\overline{(\epsilon)_{8}}\right )_{14}\}
\oplus 
\mathbb{Z}_2\{\left (\overline{(\eta \circ \sigma)_{8}}\right )_{14}\}
\oplus
\mathbb{Z}_2\{(\sigma^{2})_{14}\}
\oplus
\mathbb{Z}_2\{(\kappa)_{14}\}.
\end{align}
where $\left (\overline{(\epsilon)_{8}}\right )_{14}$ and $\left (\overline{(\eta \circ \sigma)_{8}}\right )_{14}$ are the extensions of $(\epsilon)_{8}=f^{*}_{\mathbb{C}P^{4} / \mathbb{C}P^{3}}(\epsilon)$ and $(\eta \circ \sigma)_{8}=f^{*}_{\mathbb{C}P^{4} / \mathbb{C}P^{3}}(\eta \circ \sigma)$, respectively, along the map $i^{*}:[\mathbb{C}P^{7} / \mathbb{C}P^{3}, SF_{(2)}]\longrightarrow [\mathbb{C}P^{6} / \mathbb{C}P^{3}, SF_{(2)}]\cong [\mathbb{S}^{8}, SF_{(2)}]$. 
Furthermore, the induced map
\begin{equation}\label{cp7-q}
q^*:[\mathbb{C}P^{7} / \mathbb{C}P^{3}, SF_{(2)}]\longrightarrow
[\mathbb{C}P^{7}, SF_{(2)}]  
\end{equation}
is surjective by Lemma~\ref{twonew}\,(1). Combining this with \eqref{spil2} and
\eqref{SF-spilit}, we obtain a splitting of the short exact sequence \eqref{spil2}.
This completes the proof.
\end{proof}
\begin{theorem}\label{eight}
\begin{itemize}
\item[(i)] There is a short exact sequence
\[
0\longrightarrow \mathbb{Z}_2\{(z)_{16}\}\longrightarrow
[\mathbb{C}P^{8},SF]\longrightarrow \operatorname{Ker}(p^{*})\longrightarrow 0,
\]
where $(z)_{16}=f_{\mathbb{C}P^{8}}^{*}(z)$ for $z\in \{\eta^*, \eta^*+\eta \circ \rho\}$ and the map
$f_{\mathbb{C}P^{8}}^{*}:[\mathbb{S}^{16}, SF]\cong \mathbb{Z}_2\{\eta^*\}\oplus \mathbb{Z}_2\{\eta \circ \rho\}
   \longrightarrow \mathbb{Z}_2 \subset [\mathbb{C}P^{8},SF]$
is induced by the degree-one map $f_{\mathbb{C}P^{8}}:\mathbb{C}P^{8}\longrightarrow \mathbb{S}^{16}$, and the map
$$p^{*}:[\mathbb{C}P^{7},SF]\cong \mathbb{Z}_2 \oplus \mathbb{Z}_2
\oplus \mathbb{Z}_2 \longrightarrow \mathbb{Z}_2\{\eta\circ \kappa\}\subset [\mathbb{S}^{15}, SF]$$ is induced by the Hopf
fibration $p:\mathbb{S}^{15}\longrightarrow \mathbb{C}P^{7}$.
\item[(ii)] There is a short exact sequence
\[
0\longrightarrow \Theta_{16}\xrightarrow{f_{\mathbb{C}P^{8}}^{*}} \mathcal{C}(\mathbb{C}P^{8})
   \longrightarrow \operatorname{Ker}(p^{*})\longrightarrow 0,
\]
where \(\Theta_{16}\cong \mathbb{Z}_{2}\), and the map
\(p^{*}: \mathcal{C}(\mathbb{C}P^{7})\longrightarrow \mathbb{Z}_{2} \subset \mathcal{C}(\mathbb{S}^{15})\)
is induced by the Hopf fibration \(p:\mathbb{S}^{15}\longrightarrow \mathbb{C}P^{7}\).
\end{itemize}
\end{theorem}
\begin{proof}
The key step is to determine the images of the maps
\[
f_{\mathbb{C}P^{8}}^{*}:[\mathbb{S}^{16},SF]\longrightarrow [\mathbb{C}P^{8},SF]
\quad\text{and}\quad
p^{*} : [\mathbb{C}P^{7},SF]\longrightarrow [\mathbb{S}^{15},SF].
\]
We first show that the image of
\[
p^{*} : [\mathbb{C}P^{7},SF]\longrightarrow [\mathbb{S}^{15},SF]
\]
is \(\mathbb{Z}_2\{\eta \circ \kappa\}\). Since \([\mathbb{S}^{15},SF]\cong \pi_{15}^s
 =\mathbb{Z}_{32}\{\rho\}\oplus \mathbb{Z}_2\{\eta \circ \kappa\}
 \oplus \mathbb{Z}_3\{\alpha_4\}\oplus \mathbb{Z}_5\{\alpha_{2,5}\}\)
(\cite[p.~189]{Tod62}) and
\([\mathbb{C}P^{7},SF]\cong \mathbb{Z}_2\oplus \mathbb{Z}_2\oplus \mathbb{Z}_2\) by
Proposition~\ref{stab}, it suffices to work \(2\)\nobreakdash-locally.

Consider the commutative triangle
\begin{equation}\label{cp7-trai}
\xymatrix{
[\mathbb{C}P^{7},SF_{(2)}] \ar[r]^-{p^{*}} &
[\mathbb{S}^{15},SF_{(2)}]  \\
[\mathbb{C}P^{7}/\mathbb{C}P^{3},SF_{(2)}] \ar[ur]^{\varphi_{8}^{*}} \ar[u]^{q^{*}} &
}
\end{equation}
where \(q:\mathbb{C}P^{7}\longrightarrow \mathbb{C}P^{7}/\mathbb{C}P^{3}\) is the quotient map and
\(\varphi_{8}:\mathbb{S}^{15}\longrightarrow \mathbb{C}P^{7}/\mathbb{C}P^{3}\) is the attaching map of the top cell of \(\mathbb{C}P^{8}/\mathbb{C}P^{3}\). By Lemma~\ref{twonew}(1), the induced map
\[
q^*:[\mathbb{C}P^{7}/\mathbb{C}P^{3},SF_{(2)}]\longrightarrow [\mathbb{C}P^{7},SF_{(2)}]
\]
is surjective. Thus, from Diagram \eqref{cp7-trai}, to identify \(\operatorname{Im}(p^{*})\), it is enough to determine the image of
\[
\varphi_{8}^{*} : [\mathbb{C}P^{7}/\mathbb{C}P^{3},SF_{(2)}]\longrightarrow
[\mathbb{S}^{15},SF_{(2)}].
\]
Recall that \(\mathbb{C}P^{7}/\mathbb{C}P^{3}\simeq_{(2)} \mathbb{C}P^{6}/\mathbb{C}P^{3}\vee \mathbb{S}^{14}\),
and let
\(p_1:\mathbb{C}P^{6}/\mathbb{C}P^{3}\vee \mathbb{S}^{14}\longrightarrow \mathbb{C}P^{6}/\mathbb{C}P^{3}\),
\(p_2:\mathbb{C}P^{6}/\mathbb{C}P^{3}\vee \mathbb{S}^{14}\longrightarrow \mathbb{S}^{14}\)
be the canonical projections. As in \eqref{spil11}, the map \(\varphi_{8}\) is determined by its
components
\[
p_1\circ \varphi_{8} = (\varphi_{8})_{1}
\quad\text{and}\quad
p_2\circ \varphi_{8} = \eta.
\]
Next, we use the commutative diagram
\begin{equation}\label{kappa}
\xymatrix@C=4.5em@R=4.5em{
&
[\mathbb{S}^{14},\,SF_{(2)}]
  \ar[d]^{\,p_2^{*}\,}
  \ar[dr]^{\,\eta^{*}\,}
& \\
[\mathbb{C}P^{7}/\mathbb{C}P^{3},\,SF_{(2)}]
  \ar@{=}[r]
&
[\mathbb{C}P^{6}/\mathbb{C}P^{3},\,SF_{(2)}]\;\oplus\;[\mathbb{S}^{14},\,SF_{(2)}]
  \ar[r]^-{\,\varphi_{8}^{*}\,}
&
[\mathbb{S}^{15},\,SF_{(2)}]
\\
[\mathbb{S}^{8}\cup_{2\nu}\mathbb{D}^{12},\,SF_{(2)}]
  \ar[r]_{\cong}^{\,\Psi^{*}\,}
&
[\mathbb{C}P^{6}/\mathbb{C}P^{3},\,SF_{(2)}]
  \ar[u]^{\,p_1^{*}\,}
  \ar[ur]_{\,(\varphi_8)_{1}^{*}\,}
&
}
\end{equation}
where \(\Psi:\mathbb{C}P^{6}/\mathbb{C}P^{3}\longrightarrow \mathbb{S}^{8}\cup_{2\nu}\mathbb{D}^{12}\) is the map given in
\eqref{equ8} and $\iota^{*}:[\mathbb{S}^{8}\cup_{2\nu}\mathbb{D}^{12},\,SF_{(2)}]\to [\mathbb{S}^{8},\,SF_{(2)}]\cong \mathbb{Z}_{2}\{\epsilon\}\oplus \mathbb{Z}_{2}\{\eta\circ \sigma\}$ is an isomorphism by Lemma \ref{twonew}(2)(a). The lower horizontal isomorphism
\[
\Psi^{*}:[\mathbb{S}^{8}\cup_{2\nu}\mathbb{D}^{12},\,SF_{(2)}]\longrightarrow
[\mathbb{C}P^{6}/\mathbb{C}P^{3},\,SF_{(2)}]
\]
is obtained by applying the functor \([-\,,\,SF_{(2)}]\) to the homotopy commutative
diagram of homotopy cofibrations \eqref{equ8} and using the fact that
\[
\eta^*:[\mathbb{S}^{10},\,SF_{(2)}]\cong \mathbb{Z}_2\{\eta\circ \mu\}
\longrightarrow [\mathbb{S}^{11},\,SF_{(2)}]
\]
is non-trivial. From the arguments given for \eqref{nuhomo}, the composite
\[
\mathbb{S}^{15}\xrightarrow{(\varphi_8)_{1}}
\mathbb{C}P^{6}/\mathbb{C}P^{3}\xrightarrow{\Psi}
\mathbb{S}^{8}\cup_{2\nu}\mathbb{D}^{12}\xrightarrow{q}\mathbb{S}^{12}
\]
is represented by \(\nu\in\pi_{3}^{s}\). Applying Lemma~\ref{twonew}(2)(c) with $T = \Psi \circ (\varphi_8)_{1}$, 
we find that the image of the induced map
\begin{equation}\label{imkereta}
\bigl(\Psi \circ (\varphi_8)_{1}\bigr)^{*} \colon 
[\mathbb{S}^{8} \cup_{2\nu} \mathbb{D}^{12}, SF_{(2)}] \cong 
\mathbb{Z}_{2}\{\epsilon\} \oplus \mathbb{Z}_{2}\{\eta \circ \sigma\} 
\longrightarrow [\mathbb{S}^{15}, SF_{(2)}]
\end{equation}
is precisely $\mathbb{Z}_2\{\eta \circ \kappa\}$, and its kernel is 
generated by $\eta \circ \sigma$. By the commutativity of the lower right
triangle in \eqref{kappa}, this implies that the image of the restriction
\begin{equation}\label{cp8-1}
(\varphi_8)^{*}_{[\mathbb{C}P^{6}/\mathbb{C}P^{3},SF_{(2)}]}:
[\mathbb{C}P^{6}/\mathbb{C}P^{3},SF_{(2)}]\longrightarrow [\mathbb{S}^{15},SF_{(2)}]
\end{equation}
is also \(\mathbb{Z}_2\{\eta\circ \kappa\}\).
On the other hand, using the relations \(\eta \circ \kappa \neq 0\) and
\(\eta \circ \sigma^{2} = 0\) (\cite[p.~190]{Tod62}), the upper triangle in \eqref{kappa}
shows that the restriction
\begin{equation}\label{cp8-2}
(\varphi_8)^{*}_{[\mathbb{S}^{14},SF_{(2)}]}:
[\mathbb{S}^{14},SF_{(2)}]\longrightarrow [\mathbb{S}^{15},SF_{(2)}]
\end{equation}
also has image \(\mathbb{Z}_2\{\eta\circ \kappa\}\). Combining the contributions from both
summands in the middle row of \eqref{kappa}, we conclude that
\[
\operatorname{Im}\bigl(
[\mathbb{C}P^{7}/\mathbb{C}P^{3},SF_{(2)}]\xrightarrow{\varphi_8^{*}} [\mathbb{S}^{15},SF_{(2)}]\bigr)
 = \mathbb{Z}_2\{\eta\circ \kappa\}.
\]
Hence, from \eqref{cp7-trai}, the image of the map
\begin{equation}\label{cp7-hop}
p^{*} : [\mathbb{C}P^{7},SF]\longrightarrow [\mathbb{S}^{15},SF]
\end{equation}
is \(\mathbb{Z}_2\{\eta\circ \kappa\}\). We now analyze the map
\[
f_{\mathbb{C}P^{8}}^{*}:[\mathbb{S}^{16},SF]\longrightarrow [\mathbb{C}P^{8},SF]
\]
by considering the following diagram, which is analogous to \eqref{digram3}, for the inclusion
\(i:\mathbb{C}P^{7}\hookrightarrow \mathbb{C}P^{8}\):
\begin{equation}\label{digram4}
\begin{CD}
0@.0@.0@>>>\mathit{bP}_{16}\\
@VVV @VVV  @VVV @VVV \\
\mathbb{Z}_2\cong [\mathbb{S}^{16}, \operatorname{Top}/O]
 @>f^{*}_{\mathbb{C}P^{8}}>>[\mathbb{C}P^{8},\operatorname{Top}/O]
 @>i^{*}>> [\mathbb{C}P^{7},\operatorname{Top}/O]
 @>p^{*}>> \mathit{bP}_{16}\oplus \operatorname{Coker}(J_{15})\\
@VV\psi_*V @VV\psi_*V  @VV\psi_*V  @VV\psi_*V  \\
\mathbb{Z}_2\oplus \mathbb{Z}\cong [\mathbb{S}^{16}, F/O]
 @>f^{*}_{\mathbb{C}P^{8}}>>[\mathbb{C}P^{8},F/O]
 @>i^{*}>> [\mathbb{C}P^{7},F/O]
 @>p^{*}>> \operatorname{Coker}(J_{15}) \\
@AA\phi_*A @AA\phi_*A  @AA\phi_*A  @AA\phi_*A \\
\mathbb{Z}_2\oplus \mathbb{Z}_2\cong[\mathbb{S}^{16}, SF]
 @>f^{*}_{\mathbb{C}P^{8}}>>[\mathbb{C}P^{8},SF]
 @>i^{*}>> [\mathbb{C}P^{7},SF]
 @>p^{*}>> [\mathbb{S}^{15}, SF]\\
@AAJ_{16}A @AAA  @AAA  @AAJ_{15}A  \\
\mathbb{Z}_2\cong[\mathbb{S}^{16}, SO]@>>>0@=0@>>>[\mathbb{S}^{15}, SO]\cong\mathbb{Z}\\
\end{CD}
\end{equation}
Here the map
\[
\psi_*: \mathit{bP}_{16}\oplus \operatorname{Coker}(J_{15})\cong [\mathbb{S}^{15}, \operatorname{Top}/O]
\longrightarrow \operatorname{Coker}(J_{15})
\]
is the Kervaire–Milnor map, which can be identified with the map
\([\mathbb{S}^{15}, \operatorname{Top}/O]\longrightarrow [\mathbb{S}^{15}, F/O]\) induced by
\(\operatorname{Top}/O\longrightarrow F/O\) (\cite{Bru68a}). Moreover,
\([\mathbb{C}P^{7},\operatorname{Top}/O]\cong \mathbb{Z}_2\oplus  \mathbb{Z}_2\) by
Theorem~\ref{main}(iii).

From the first column of Diagram \eqref{digram4} we observe that
\[
\operatorname{Im}\bigl([\mathbb{S}^{16}, \operatorname{Top}/O]
 \xrightarrow{\psi_*}[\mathbb{S}^{16}, F/O]\bigr)
=
\operatorname{Im}\bigl([\mathbb{S}^{16}, SF]
 \xrightarrow{\phi_*}[\mathbb{S}^{16}, F/O]\bigr),
\]
since the \(J\)\nobreakdash-homomorphism
\[
J_{16}:[\mathbb{S}^{16}, SO]\cong \mathbb{Z}_2\longrightarrow \mathbb{Z}_2\{\eta^*\} \oplus \mathbb{Z}_2\{\eta\circ \rho\}\cong [\mathbb{S}^{16}, SF]
\]
is injective, and its image is generated by the class \(\eta\circ \rho\) (\cite{Ada66}).
This observation, together with Theorem~\ref{kawa} (with \(m=8\)) and a diagram chase in the left two columns of Diagram \eqref{digram4}, implies that
\begin{equation}\label{cp8-deg}
\operatorname{Im}\bigl(
[\mathbb{S}^{16}, SF]\xrightarrow{\,f^{*}_{\mathbb{C}P^{8}}\,}[\mathbb{C}P^{8},SF]\bigr)
\cong \mathbb{Z}_2,
\end{equation}
generated by the element $(z)_{16}=f_{\mathbb{C}P^{8}}^{*}(z)$ for \(z\in \{\eta^{*}, \eta^{*}+\eta \circ \rho\}\).
Combining this with the image computed in \eqref{cp7-hop} of
\[
p^{*} : [\mathbb{C}P^{7},SF]\longrightarrow [\mathbb{S}^{15},SF]
\]
in the exact sequence \eqref{longG} with \(X = SF\), we obtain the required short exact sequence
\[
0\longrightarrow \mathbb{Z}_{2}\{(z)_{16}\}
\longrightarrow [\mathbb{C}P^{8}, SF]
\longrightarrow \mathbb{Z}_2\oplus \mathbb{Z}_2 \cong \operatorname{Ker}(p^{*})
\longrightarrow 0.
\] This finishes Part \textnormal{(i)}. For Part \textnormal{(ii)}, note from \eqref{coker} that the image of the induced map
\[
p^{*}: [\mathbb{C}P^{7}, \operatorname{Top}/O] \longrightarrow [\mathbb{S}^{15}, \operatorname{Top}/O] \cong \mathit{bP}_{16} \oplus \operatorname{Coker}(J_{15})
\]
is contained in $\mathbb{Z}_2\{\eta \circ \kappa\} = \operatorname{Coker}(J_{15})$. It follows from the arguments given in the proof of Theorem~\ref{kawa} for the case $n=7$, using Diagram \eqref{kappa1}, that 
\begin{equation}\label{equ10}
\phi_{*}\big( f^{*}_{\mathbb{C} P^{7}}(\kappa) \big ) = \psi_{*}\big( f^{*}_{\mathbb{C} P^{7}}(\Sigma^{14}) \big) \in [\mathbb{C} P^{7}, F/O],
\end{equation}
where $\Sigma^{14} \in [\mathbb{S}^{14}, \operatorname{Top}/O]$ is the exotic $14$-sphere and $f^{*}_{\mathbb{C} P^{7}}(\Sigma^{14}) \in [\mathbb{C}P^{7}, \operatorname{Top}/O]$. Furthermore, from Diagram \eqref{kappa2}, we have $p^*(f_{\mathbb{C} P^{7}}^*(x)) = \eta \circ \kappa \in [\mathbb{S}^{15}, SF]$, where $x \in \{\kappa, \kappa + \sigma^2\} \subset [\mathbb{S}^{14}, SF]$, and $p^{*}: [\mathbb{C}P^{7}, SF] \to [\mathbb{S}^{15}, SF]$ is the induced map. 
The latter implies that the composition
\[
[\mathbb{C}P^{7}, SF] \xrightarrow{p^{*}} [\mathbb{S}^{15}, SF] \xrightarrow{\phi_{*}} [\mathbb{S}^{15}, F/O] \cong \operatorname{Coker}(J_{15})
\] 
sends the element $f^{*}_{\mathbb{C} P^{7}}(x)$ to the class $\eta \circ \kappa$ in $\operatorname{Coker}(J_{15})$. By the commutativity of the third rectangle in the last two columns of Diagram \eqref{digram4}, this is equivalent to 
\[
p^{*} \circ \phi_{*} \big(f^{*}_{\mathbb{C} P^{7}}(x)\big ) = \eta \circ \kappa \in \operatorname{Coker}(J_{15}),
\] 
where $p^{*}: [\mathbb{C}P^{7}, F/O] \to [\mathbb{S}^{15}, F/O]$ is the induced map. Combining this with \eqref{equ10}, we obtain
\[
p^{*} \circ \psi_{*} \big( f^{*}_{\mathbb{C} P^{7}}(\Sigma^{14}) \big) = \eta \circ \kappa \in \operatorname{Coker}(J_{15}) \cong [\mathbb{S}^{15}, F/O].
\]
By the commutativity of the second rectangle in the last two columns of Diagram \eqref{digram4}, this is equivalent to
\[
\psi_{*} \circ p^* \big( f^{*}_{\mathbb{C} P^{7}}(\Sigma^{14}) \big) = \eta \circ \kappa \in \operatorname{Coker}(J_{15}) \cong [\mathbb{S}^{15}, F/O].
\]
This shows that the image of the map
\begin{equation}\label{kernelcp7-s15}
p^{*}: [\mathbb{C}P^{7}, \operatorname{Top}/O] \longrightarrow [\mathbb{S}^{15}, \operatorname{Top}/O] \cong \mathit{bP}_{16} \oplus \operatorname{Coker}(J_{15})
\end{equation}
is exactly $\mathbb{Z}_2\{\eta \circ \kappa\} = \operatorname{Coker}(J_{15})$. Since $[\mathbb{C}P^{7}, \operatorname{Top}/O] \cong \mathbb{Z}_2 \oplus \mathbb{Z}_2$, it follows that the kernel of $p^{*}: [\mathbb{C}P^{7}, \operatorname{Top}/O] \to [\mathbb{S}^{15}, \operatorname{Top}/O]$ is isomorphic to $\mathbb{Z}_2$. Using this fact along with Theorem~\ref{kawa} for $m=8$ in the exact sequence \eqref{longG}, we obtain the following short exact sequence:
\[
0 \longrightarrow [\mathbb{S}^{16}, \operatorname{Top}/O] \longrightarrow [\mathbb{C}P^{8}, \operatorname{Top}/O] \longrightarrow \operatorname{Ker}(p^{*}) \cong \mathbb{Z}_2 \longrightarrow 0.
\]
This completes the proof of Part \textnormal{(ii)}.
\end{proof}
We now discuss the splitness of both short exact sequences given in
Theorem~\ref{eight}. For this purpose, we first prove the following lemma.
\begin{lemma}\label{eight-sub}
\begin{itemize}
\item[(i)] There is a non-split short exact sequence
\[
0 \longrightarrow \mathbb{Z}_2\{(z)_{16}\} \longrightarrow
[\mathbb{C}P^{8}/\mathbb{C}P^{6}, SF]
 = \mathbb{Z}_4\{\left(\overline{(\sigma^{2})_{14}}\right)_{16}\}
 \xrightarrow{i^{*}} \mathbb{Z}_2\{(\sigma^2)_{14}\} \longrightarrow 0,
\]
where $\mathbb{Z}_2\{(\sigma^2)_{14}\} \subset [\mathbb{C}P^{7}/\mathbb{C}P^{6}, SF]$, \((z)_{16} = f_{\mathbb{C}P^{8}/\mathbb{C}P^{6}}^{*}(z)\) for \(z \in \{\eta^*, \eta^* + \eta \circ \rho\} \subset [\mathbb{S}^{16}, SF]\), and \((\sigma^2)_{14} = f^{*}_{\mathbb{C}P^{7}/\mathbb{C}P^{6}}(\sigma^2) \in [\mathbb{C}P^{7}/\mathbb{C}P^{6}, SF]\). Here, \(\left(\overline{(\sigma^{2})_{14}}\right)_{16}\) is an extension of \((\sigma^2)_{14}\) along the map \(i^{*} \colon [\mathbb{C}P^{8}/\mathbb{C}P^{6}, SF] \longrightarrow [\mathbb{C}P^{7}/\mathbb{C}P^{6}, SF]\) induced by the bottom cell inclusion \(\mathbb{S}^{14} \simeq \mathbb{C}P^{7}/\mathbb{C}P^{6} \longrightarrow \mathbb{C}P^{8}/\mathbb{C}P^{6}\), and \(f_{\mathbb{C}P^{8}/\mathbb{C}P^{6}} \colon \mathbb{C}P^{8}/\mathbb{C}P^{6} \longrightarrow \mathbb{S}^{16}\) is the collapse map.

\item[(ii)] There is a non-split short exact sequence
{\small
\[
\begin{aligned}
0 \longrightarrow \mathbb{Z}_2\{(z)_{16}\} \longrightarrow
[\mathbb{C}P^{8}/\mathbb{C}P^{3}, SF] \xrightarrow{i^{*}}
&\;\mathbb{Z}_2\{(\sigma^2)_{14}\} \oplus \mathbb{Z}_2\{\left(\overline{(\eta \circ \sigma)_8}\right)_{14}\} \\
&\;\oplus \mathbb{Z}_{2}\{(\kappa)_{14} + \left(\overline{(\epsilon)_8}\right)_{14}\} \longrightarrow 0,
\end{aligned}
\]}
where \((z)_{16} = f_{\mathbb{C}P^{8}/\mathbb{C}P^{3}}^{*}(z)\) for some \(z \in \{\eta^*, \eta^* + \eta \circ \rho\} \subset [\mathbb{S}^{16}, SF]\), and 
\[
\mathbb{Z}_2\{(\sigma^2)_{14}\} \oplus \mathbb{Z}_2\{\left(\overline{(\eta \circ \sigma)_8}\right)_{14}\} \oplus \mathbb{Z}_2\{(\kappa)_{14} + \left(\overline{(\epsilon)_8}\right)_{14}\} \subset [\mathbb{C}P^{7}/\mathbb{C}P^{3}, SF].
\]
In particular, 
\[
[\mathbb{C}P^{8}/\mathbb{C}P^{3}, SF] \cong \mathbb{Z}_{4}\{\left(\overline{(\sigma^{2})_{14}}\right)_{16}\} \oplus \mathbb{Z}_2\{\left(\overline{(\eta \circ \sigma)_8}\right)_{16}\} \oplus \mathbb{Z}_2\{\left(\overline{(\kappa)_{14} + \left(\overline{(\epsilon)_8}\right)_{14}}\right)_{16}\}.
\]
Here, \(\left(\overline{(\sigma^{2})_{14}}\right)_{16}\) is an extension of \((\sigma^{2})_{14} = f^{*}_{\mathbb{C}P^{7}/\mathbb{C}P^{3}}(\sigma^{2})\) along the map \(i^{*}\) induced by the inclusion
\[
i \colon \mathbb{C}P^{7}/\mathbb{C}P^{3} \simeq_{(2)} \mathbb{S}^{14} \vee \mathbb{C}P^{6}/\mathbb{C}P^{3} \longrightarrow \mathbb{C}P^{8}/\mathbb{C}P^{3},
\]
so that \(2\left(\overline{(\sigma^{2})_{14}}\right)_{16} = (z)_{16}\), and \(f_{\mathbb{C}P^{8}/\mathbb{C}P^{3}} \colon \mathbb{C}P^{8}/\mathbb{C}P^{3} \longrightarrow \mathbb{S}^{16}\) is the collapse map.
\end{itemize}
\end{lemma}
\begin{proof}
Observe that
\[
\mathbb{C}P^8 / \mathbb{C}P^6 \simeq \mathbb{C}P^7 / \mathbb{C}P^6 \cup_{\varphi_8} \mathbb{D}^{16},
\]
where the attaching map
\(\varphi_8 : \mathbb{S}^{15} \longrightarrow \mathbb{C}P^7 / \mathbb{C}P^6 \simeq \mathbb{S}^{14}\)
is represented by \(\eta \in \pi_{1}^s\) by Proposition~\ref{St}.
Consider the induced exact sequence for the cofibre
\[
\mathbb{S}^{15} \xrightarrow{\eta} \mathbb{S}^{14}
\longrightarrow \mathbb{S}^{14} \cup_{\eta} \mathbb{D}^{16}
\longrightarrow \mathbb{S}^{16}.
\]
\[
\longrightarrow [\mathbb{S}^{15}, SF] \xrightarrow{\;\eta^{*}\;}
[\mathbb{S}^{16}, SF] \xrightarrow{f_{\mathbb{C}P^8 / \mathbb{C}P^6}^{*}}
[\mathbb{S}^{14} \cup_{\eta} \mathbb{D}^{16}, SF] \xrightarrow{\;i^{*}\;}
[\mathbb{S}^{14}, SF] \xrightarrow{\;\eta^{*}\;} [\mathbb{S}^{15}, SF] \longrightarrow \cdots
\]
where the image of \(\eta^{*} : [\mathbb{S}^{15}, SF] \longrightarrow [\mathbb{S}^{16}, SF]\) is
\(\mathbb{Z}_{2}\{\eta\circ \rho\}\) and the kernel of
\(\eta^{*} : [\mathbb{S}^{14}, SF] \longrightarrow [\mathbb{S}^{15}, SF]\) is
\(\mathbb{Z}_{2}\{\sigma^{2}\}\) (\cite{Tod62}). So, the above sequence reduces to a short exact sequence
\begin{equation}\label{eta-seq}
0 \longrightarrow \mathbb{Z}_{2}\{(z)_{16}\}
\xrightarrow{\;i^{*}\;}
[\mathbb{S}^{14} \cup_{\eta} \mathbb{D}^{16}, SF]
\xrightarrow{\;j^{*}\;}
\mathbb{Z}_{2}\{\sigma^{2}\} \longrightarrow 0,
\end{equation}
where $(z)_{16}=f_{\mathbb{C}P^8 / \mathbb{C}P^6}^{*}(z)$ for \(z\in \{\eta^*, \eta^*+\eta \circ \rho\}\subset [\mathbb{S}^{16}, SF]\). We now show that this sequence does not split. Write
\[
\mathbb{C}P^8 / \mathbb{C}P^6 \simeq \mathbb{S}^{14} \cup_{\eta} \mathbb{D}^{16} = \Sigma^{12}\mathbb{C}P^{2}.
\]
Note that the quotient map \(f_{\mathbb{C}P^8 / \mathbb{C}P^6} : \mathbb{S}^{14} \cup_{\eta} \mathbb{D}^{16} \longrightarrow \mathbb{S}^{16}\) is homotopic to \(\Sigma^{12} f_{\mathbb{C}P^{2}}\), and the inclusion \(i : \mathbb{S}^{14} \hookrightarrow \mathbb{S}^{14} \cup_{\eta} \mathbb{D}^{16}\) is homotopic to the map \(\Sigma^{12} i_{\mathbb{C}}\), where
\[
f_{\mathbb{C}P^{2}} : \mathbb{C}P^{2} \longrightarrow \mathbb{S}^{4}
\]
is the degree-one map and \(i_{\mathbb{C}} : \mathbb{S}^2 \hookrightarrow \mathbb{C}P^{2}\) is the inclusion. Let \(\overline{\sigma^{2}} : \Sigma^{12}\mathbb{C}P^{2} \longrightarrow SF\) be an extension of \(\sigma^{2} : \mathbb{S}^{14} \longrightarrow SF\). Denoting by \(\iota_{\mathbb{C}} : \mathbb{C}P^{2} \longrightarrow \mathbb{C}P^{2}\) and \(\iota_{n} : \mathbb{S}^{n} \longrightarrow \mathbb{S}^{n}\) the identity maps, we have :
\begin{equation}\label{todaeq}
\begin{aligned}
2\,\overline{\sigma^{2}} &= \overline{\sigma^{2}} \circ (2\,\Sigma^{12} \iota_{\mathbb{C}}) \\
&= \overline{\sigma^{2}} \circ (\Sigma^{12} i_{\mathbb{C}} \circ \overline{2 \iota_{14}} + \widetilde{2 \iota_{15}} \circ \Sigma^{12} f_{\mathbb{C}P^{2}}) \quad \text{by \cite[Corollary~2.6(i)]{KMNST01}} \\
&= \sigma^{2} \circ \overline{2 \iota_{14}} + \overline{\sigma^{2}} \circ \widetilde{2 \iota_{15}} \circ \Sigma^{12} f_{\mathbb{C}P^{2}}, \quad \text{since } (\Sigma^{12} i_{\mathbb{C}})^{*}(\overline{\sigma^{2}}) = \sigma^{2}.
\end{aligned}
\end{equation}
Applying Proposition~2.7(2) of \cite{KMNST01} with $\alpha = \sigma^2$, $\beta = 2 \iota_{14}$, $\gamma = \eta$, and $p = \Sigma^{12} f_{\mathbb{C}P^{2}}$, we obtain 
\[
\sigma^{2} \circ \overline{2 \iota_{14}} \in \langle \sigma^{2}, 2 \iota_{14}, \eta \rangle \circ \Sigma^{12} f_{\mathbb{C}P^{2}}.
\]
Similarly, by taking $\alpha = \sigma^2$, $\beta = \eta$, and $\gamma = 2 \iota_{15}$ in Proposition~2.7(1) of \cite{KMNST01}, we have
\begin{equation}\label{toda-eta}
\overline{\sigma^{2}} \circ \widetilde{2 \iota_{15}} \in \langle \sigma^{2}, \eta, 2 \iota_{15} \rangle.
\end{equation}
It follows from \cite[PROOF OF (2), p.~279]{Muk66} that $\langle \sigma^{2}, \eta, 2 \iota_{15} \rangle = 0$ and
\[
\langle \sigma, 2\sigma, \eta \rangle = \langle \sigma^{2}, 2, \eta \rangle = \{\eta^{*}, \eta^{*} + \eta \circ \rho\}.
\]
Therefore, the map $\overline{\sigma^{2}} \circ \widetilde{2 \iota_{15}} \circ \Sigma^{12} f_{\mathbb{C}P^{2}}$ is null-homotopic. Thus, from \eqref{todaeq} and \eqref{toda-eta}, we obtain
\begin{equation}\label{toda-b}
2\,\overline{\sigma^{2}} = \eta^{*} \circ \Sigma^{12} f_{\mathbb{C}P^{2}} \quad \text{or} \quad (\eta^{*} + \eta \circ \rho) \circ \Sigma^{12} f_{\mathbb{C}P^{2}}.
\end{equation}
Since the induced map $(f_{\mathbb{C}P^8 / \mathbb{C}P^6})^{*} = (\Sigma^{12} f_{\mathbb{C}P^{2}})^{*} : [\mathbb{S}^{16}, SF] \longrightarrow [\mathbb{S}^{14} \cup_{\eta} \mathbb{D}^{16}, SF]$ maps both $\eta^{*}$ and $\eta^{*} + \eta \circ \rho$ to the same non-trivial element in $[\mathbb{S}^{14} \cup_{\eta} \mathbb{D}^{16}, SF]$ by \eqref{eta-seq}, it follows from \eqref{toda-b} that
\[
2\,\overline{\sigma^{2}} \neq 0 \in [\mathbb{S}^{14} \cup_{\eta} \mathbb{D}^{16}, SF].
\]
Thus, any extension $\overline{\sigma^{2}}$ of $\sigma^{2}$ has order $4$. It then follows from \eqref{eta-seq} that
\[
[\mathbb{C}P^{8}/\mathbb{C}P^{6}, SF] = \mathbb{Z}_4 \left\{ \left( \overline{(\sigma^{2})_{14}} \right)_{16} \right\},
\]
where $2(\overline{(\sigma^{2})_{14}})_{16} = (f_{\mathbb{C}P^{8}/\mathbb{C}P^{6}})^{*}(z)$ for $z \in \{\eta^{*}, \eta^{*} + \eta \circ \rho\}$ by \eqref{toda-b}. This completes the proof of Part~\textup{(i)}.
 For the proof of Part (ii), we use the following commutative diagram
\begin{equation}\label{cp8-ex}
\xymatrix{
[\Sigma \mathbb{C}P^{3}, SO] \ar[r]^{\;{\!\!\!\!\!\delta^{*}}\;} \ar[d]_{J} &
   [\mathbb{C}P^{8}/\mathbb{C}P^{3}, SO] \cong \mathbb{Z}_2 \ar[d]_{J} \\
[\Sigma \mathbb{C}P^{3}, SF] \ar[r]^{\;{\!\!\!\delta^{*}}\;} &
   [\mathbb{C}P^{8}/\mathbb{C}P^{3}, SF] \ar@{->>}[r]^{\;q^{*}\;} &
   [\mathbb{C}P^{8}, SF] \ar[r]^{\;i^{*}\;} &
   [\mathbb{C}P^{3}, SF]
}
\end{equation}
where the rows are the portions of
the long exact sequence induced by the cofibre sequence
\(
\mathbb{C}P^{3} \hookrightarrow \mathbb{C}P^{8}
\xrightarrow{q} \mathbb{C}P^{8}/\mathbb{C}P^{3}
\xrightarrow{\delta} \Sigma \mathbb{C}P^{3},
\)
the vertical arrows are induced by the stable \(J\)-homomorphism
\(SO \longrightarrow SF\), the surjectivity of the map
\(q^{*} : [\mathbb{C}P^{8}/\mathbb{C}P^{3}, SF] \longrightarrow [\mathbb{C}P^{3}, SF]\)
follows from Lemma~\ref{twonew}(1), and
\(
\widetilde{KO}^{-1}(\mathbb{C}P^{8}/\mathbb{C}P^{3})
\cong [\mathbb{C}P^{8}/\mathbb{C}P^{3}, SO] \cong \mathbb{Z}_{2}
\)
by \cite[Theorem~1.1]{NT19}. As the map
\(
J\colon[\Sigma\mathbb{C}P^{3},SO]\longrightarrow[\Sigma\mathbb{C}P^{3},SF]
\)
is surjective as $[\Sigma\mathbb{C}P^{3},F/O]=0$, the left rectangle in the above diagram shows that the image of
\(
\delta^{*}\colon[\Sigma\mathbb{C}P^{3},SF]\longrightarrow[\mathbb{C}P^{8}/\mathbb{C}P^{3},SF]
\)
has order at most two. On the other hand, the map \(\delta^{*}\) also fits into the following commutative diagram:
\begin{equation}\label{conn-8}
\xymatrix{
& [\Sigma \mathbb{C}P^{3}, SF]
   \ar[r]^{\;\delta^{*}\;} &
[\mathbb{C}P^{8}/\mathbb{C}P^{3}, SF] \ar[d]^{i^{*}} \\
[\Sigma(\mathbb{C}P^{4}/\mathbb{C}P^{1}), SF] \ar[r] &
[\Sigma(\mathbb{C}P^{3}/\mathbb{C}P^{1}), SF]
   \cong [\mathbb{S}^{7}, SF]
   \ar[u]^{q^{*}}
   \ar[r]^{\eta^{*}} &
[\mathbb{S}^{8}, SF] \cong [\mathbb{C}P^{4}/\mathbb{C}P^{3}, SF]
}
\end{equation}
where
\[
\mathbb{C}P^{4}/\mathbb{C}P^{1}
= \mathbb{C}P^{3}/\mathbb{C}P^{1} \cup_{\varphi_{4}} \mathbb{D}^{8},
\]
and the attaching map
\[
\varphi_{4} \colon \mathbb{S}^{7} \longrightarrow
\mathbb{C}P^{3}/\mathbb{C}P^{1} \simeq \mathbb{S}^{6}\vee \mathbb{S}^{4}
\]
restricts to \(\eta \in \pi_{1}^{s}\) on the first factor by
Proposition~\ref{St} and to \(\nu \in \pi_{3}^{s}\) on the second factor by
\cite[Proposition~5.2, Table~5.3]{Mos68}, and
\(Sq^{4}\) is an isomorphism on
\(H^{4}(\mathbb{C}P^{4}/\mathbb{C}P^{1};\mathbb{Z}_{2})\).
Moreover, the bottom row is induced by the cofibre sequence
\[
\mathbb{C}P^{3}/\mathbb{C}P^{1} \longrightarrow
\mathbb{C}P^{4}/\mathbb{C}P^{1} \longrightarrow
\mathbb{C}P^{4}/\mathbb{C}P^{3} \simeq \mathbb{S}^{8}
\xrightarrow{(\eta,\nu)}
\Sigma(\mathbb{C}P^{3}/\mathbb{C}P^{1}) \simeq
\mathbb{S}^{7}\vee \mathbb{S}^{5},
\]
and
\[
i\colon \mathbb{C}P^{4}/\mathbb{C}P^{3} \hookrightarrow
\mathbb{C}P^{8}/\mathbb{C}P^{3}
\]
is the inclusion of the bottom cell.
Since \(\eta^{*}(\sigma)=\eta \circ \sigma\neq 0\in [\mathbb{S}^{8},SF]\cong \pi_{8}^{S}\), the image of
\[
\eta^{*}\colon [\mathbb{S}^{7},SF]\longrightarrow [\mathbb{S}^{8},SF]
\]
is \(\mathbb{Z}_{2}\{\eta \circ \sigma\}.\)
By the commutativity of Diagram \eqref{conn-8}, this implies that the map
\[
\delta^{*}\colon [\Sigma\mathbb{C}P^{3},SF]\longrightarrow
[\mathbb{C}P^{8}/\mathbb{C}P^{3},SF]
\]
is non-zero, and hence
\begin{equation}\label{exrho}
\operatorname{Im}\bigl([\Sigma\mathbb{C}P^{3},SF]\xrightarrow{\;\delta^{*}\;}
[\mathbb{C}P^{8}/\mathbb{C}P^{3},SF]\bigr)\cong \mathbb{Z}_{2}\{\left(\overline{(\eta \circ \sigma)_{8}}\right)_{16}\},
\end{equation}
where \(\left(\overline{(\eta \circ \sigma)_{8}}\right)_{16}\) is an extension of \((\eta \circ \sigma)_{8}=f^{*}_{\mathbb{C}P^{4}/\mathbb{C}P^{3}}(\eta \circ \sigma)\) along the composite
\[
i^{*}\colon [\mathbb{C}P^{8}/\mathbb{C}P^{3},SF]\longrightarrow
[\mathbb{C}P^{7}/\mathbb{C}P^{3},SF]\longrightarrow
[\mathbb{C}P^{4}/\mathbb{C}P^{3},SF].
\]
Therefore, it follows from the bottom exact sequence in Diagram~\eqref{cp8-ex}, together with \eqref{exrho} and the surjectivity of the map $q^{*}\colon [\mathbb{C}P^{8}/\mathbb{C}P^{3}, SF] \longrightarrow [\mathbb{C}P^{8}, SF]$, that we obtain the short exact sequence
\begin{equation}\label{exact8-1}
    0 \longrightarrow \mathbb{Z}_{2}\left\{\left(\overline{(\eta \circ \sigma)_{8}}\right)_{16}\right\} \longrightarrow [\mathbb{C}P^{8}/\mathbb{C}P^{3}, SF] \xrightarrow{\,q^{*}\,} [\mathbb{C}P^{8}, SF] \longrightarrow 0.
\end{equation}
Note from Theorem~\ref{eight}(i) and the exact sequence \eqref{exact8-1} that the group $[\mathbb{C}P^{8}/\mathbb{C}P^{3}, SF]$ is $2$-torsion; thus, it suffices to work $2$-locally. Using the cofiber sequence
\[
\mathbb{S}^{15} \xrightarrow{\;\varphi_{8}=((\varphi_{8})_{1},(\varphi_{8})_{2})\;} \mathbb{C}P^{7}/\mathbb{C}P^{3} \simeq_{(2)} \mathbb{C}P^{6}/\mathbb{C}P^{3} \vee \mathbb{S}^{14} \longrightarrow \mathbb{C}P^{8}/\mathbb{C}P^{3} \longrightarrow \mathbb{S}^{16},
\]
we obtain the exact sequence
\begin{equation}\label{equ13}
\dots \longrightarrow [\mathbb{S}^{16}, SF_{(2)}] \xrightarrow{f_{\mathbb{C}P^{8}/\mathbb{C}P^{3}}^{*}} [\mathbb{C}P^{8}/\mathbb{C}P^{3}, SF_{(2)}] \xrightarrow{i^{*}} [\mathbb{C}P^{7}/\mathbb{C}P^{3}, SF_{(2)}] \xrightarrow{\varphi_{8}^{*}} [\mathbb{S}^{15}, SF_{(2)}].
\end{equation}
Recall from \eqref{SF-spilit} that
\[
\begin{aligned}
[\mathbb{C}P^{7} / \mathbb{C}P^{3}, SF_{(2)}] &\cong [\mathbb{C}P^{6}/\mathbb{C}P^{3}, SF_{(2)}] \oplus [\mathbb{S}^{14}, SF_{(2)}] \\
&\cong \mathbb{Z}_2\{\left(\overline{(\epsilon)_{8}}\right)_{14}\} \oplus \mathbb{Z}_2\{\left(\overline{(\eta \circ \sigma)_{8}}\right)_{14}\} \oplus \mathbb{Z}_2\{(\sigma^{2})_{14}\} \oplus \mathbb{Z}_2\{(\kappa)_{14}\},
\end{aligned}
\]
where $i^{*}: [\mathbb{C}P^{6}/\mathbb{C}P^{3}, SF_{(2)}] \cong \mathbb{Z}_2\{\left(\overline{(\epsilon)_{8}}\right)_{12}\} \oplus \mathbb{Z}_2\{\left(\overline{(\eta \circ \sigma)_{8}}\right)_{12}\} \to [\mathbb{S}^{8}, SF_{(2)}]$ is an isomorphism. It follows from \eqref{cp8-1} and \eqref{cp8-2} that the images of the components
\[
(\varphi_8)_{1}^{*}: [\mathbb{C}P^{6}/\mathbb{C}P^{3}, SF_{(2)}] \to [\mathbb{S}^{15}, SF_{(2)}] \quad \text{and} \quad (\varphi_8)_{2}^{*} = \eta^{*}: [\mathbb{S}^{14}, SF_{(2)}] \to [\mathbb{S}^{15}, SF_{(2)}]
\]
are both $\mathbb{Z}_2\{\eta \circ \kappa\}$. Furthermore, from \eqref{imkereta}, the kernel of the composite
\[
(\varphi_8)_{1}^{*}\circ \Psi^{*}: [\mathbb{S}^{8}\cup_{2\nu}\mathbb{D}^{12}, SF_{(2)}] \xrightarrow{\Psi^{*}} [\mathbb{C}P^{6}/\mathbb{C}P^{3}, SF_{(2)}] \xrightarrow{(\varphi_8)_{1}^{*}} [\mathbb{S}^{15}, SF_{(2)}]
\]
is generated by $\left(\overline{\eta \circ \sigma}\right)_{12}$. Here, we identify $[\mathbb{S}^{8}\cup_{2\nu}\mathbb{D}^{12}, SF_{(2)}] \cong \mathbb{Z}_2\{\left(\overline{\epsilon}\right)_{12}\} \oplus \mathbb{Z}_2\{\left(\overline{\eta \circ \sigma}\right)_{12}\}$, since the map $\iota^{*}: [\mathbb{S}^{8}\cup_{2\nu}\mathbb{D}^{12}, SF_{(2)}] \to [\mathbb{S}^{8}, SF_{(2)}]$ is an isomorphism.
Note from \eqref{equ8} that the isomorphism $\Psi^{*} \colon [\mathbb{S}^{8}\cup_{2\nu}\mathbb{D}^{12}, SF_{(2)}] \to [\mathbb{C}P^{6}/\mathbb{C}P^{3}, SF_{(2)}]$ (see Diagram \eqref{kappa}) fits into the following commutative diagram:
\begin{equation}\label{extensioncomm}
\begin{gathered}
\xymatrix{
[\mathbb{S}^{8}\cup_{2\nu}\mathbb{D}^{12}, SF_{(2)}] \ar[rr]^-{\Psi^{*}} \ar[d]_{\iota^{*}} && [\mathbb{C}P^{6}/\mathbb{C}P^{3}, SF_{(2)}] \ar[d]^{i^{*}} \\
[\mathbb{S}^{8}, SF_{(2)}] \ar[rr]^-{Id} && [\mathbb{S}^{8}, SF_{(2)}],
}
\end{gathered}
\end{equation}
where we identify $[\mathbb{S}^{8}\cup_{2\nu}\mathbb{D}^{12}, SF_{(2)}] \cong \mathbb{Z}_2\{\left(\overline{\epsilon}\right)_{12}\} \oplus \mathbb{Z}_2\{\left(\overline{\eta \circ \sigma}\right)_{12}\}$ and $[\mathbb{C}P^{6}/\mathbb{C}P^{3}, SF_{(2)}] \cong \mathbb{Z}_2\{\left(\overline{(\epsilon)_{8}}\right )_{12}\} \oplus \mathbb{Z}_2\{\left(\overline{(\eta \circ \sigma)_{8}}\right)_{12}\}$. The commutativity of this diagram ensures that $\Psi^{*}$ identifies the generators by sending $\left(\overline{\epsilon}\right)_{12}$ to $\left(\overline{(\epsilon)_{8}}\right)_{12}$ and $\left(\overline{\eta \circ \sigma}\right)_{12}$ to $\left(\overline{(\eta \circ \sigma)_{8}}\right)_{12}$. Since the kernel of the composite
\[
(\varphi_8)_{1}^{*}\circ \Psi^{*} \colon [\mathbb{S}^{8}\cup_{2\nu}\mathbb{D}^{12}, SF_{(2)}] \to [\mathbb{S}^{15}, SF_{(2)}]
\]
is generated by $\left(\overline{\eta \circ \sigma}\right)_{12}$, it follows that the kernel of the component
\[
(\varphi_8)_{1}^{*} \colon [\mathbb{C}P^{6}/\mathbb{C}P^{3}, SF_{(2)}] \to [\mathbb{S}^{15}, SF_{(2)}]
\]
is generated by $\left(\overline{(\eta \circ \sigma)_{8}}\right)_{12}$.
Combining these results, we find that
\begin{equation}\label{kernelofp}
\operatorname{Ker}\left( [\mathbb{C}P^{7}/\mathbb{C}P^{3}, SF_{(2)}] \xrightarrow{\varphi_{8}^{*}} [\mathbb{S}^{15}, SF_{(2)}] \right)\cong \mathbb{Z}_{2}\{(\sigma^{2})_{14}\} \oplus \mathbb{Z}_{2}\{\left(\overline{(\eta \circ \sigma)_{8}}\right)_{14}\} \oplus \mathbb{Z}_{2}\{(\kappa)_{14} + \left(\overline{(\epsilon)_{8}}\right)_{14}\},
\end{equation}
We now analyze the map
\(f^{*}_{\mathbb{C}P^{8}/\mathbb{C}P^{3}}\colon [\mathbb{S}^{16},SF_{(2)}]\longrightarrow
[\mathbb{C}P^{8}/\mathbb{C}P^{3},SF_{(2)}]\) in order to turn the long exact sequence \eqref{equ13}
into a short exact sequence. Observe from \eqref{exrho} that $\left(\overline{(\eta \circ \sigma)_{8}}\right)_{16}$ is an extension of an element in $[\mathbb{C}P^{7}/\mathbb{C}P^{3}, SF_{(2)}]$ along the map 
\[
i^{*} \colon [\mathbb{C}P^{8}/\mathbb{C}P^{3}, SF_{(2)}] \longrightarrow [\mathbb{C}P^{7}/\mathbb{C}P^{3}, SF_{(2)}].
\] 
We now show that 
\begin{equation}\label{ext-8}
i^{*}\left( \left(\overline{(\eta \circ \sigma)_{8}}\right)_{16} \right) = \left(\overline{(\eta \circ \sigma)_{8}}\right)_{14}.
\end{equation}
Consider the following commutative diagram:
\begin{equation}\label{cp8-conn}
\xymatrix@C=0.5em@R=1.2em{
& [\mathbb{C}P^{8}/\mathbb{C}P^3, SF_{(2)}] \ar[dr]^{i^{*}} & & \\
[\Sigma \mathbb{C}P^{3}, SF_{(2)}] \ar[ur]^{\delta^{*}} \ar[rr]^{\delta^{*}} && [\mathbb{C}P^6/\mathbb{C}P^3, SF_{(2)}] \oplus [\mathbb{S}^{14}, SF_{(2)}] \ar[d]^{i^{*}} & \\
[\Sigma(\mathbb{C}P^{3}/\mathbb{C}P^{1}), SF_{(2)}] \cong [\mathbb{S}^7, SF_{(2)}] \ar[u]^{q^{*}} \ar[rr]^{(q\circ \delta \circ i)^{*}=\eta^{*}} \ar[rru]^{(q\circ \delta)^{*}} && [\mathbb{S}^{8}, SF_{(2)}] \cong [\mathbb{C}P^{4}/\mathbb{C}P^3, SF_{(2)}] &
}
\end{equation}
where $[\mathbb{C}P^{7}/\mathbb{C}P^3, SF_{(2)}] \cong [\mathbb{C}P^6/\mathbb{C}P^3, SF_{(2)}] \oplus [\mathbb{S}^{14}, SF_{(2)}]$. The map $\delta \colon \mathbb{C}P^{n}/\mathbb{C}P^{3} \longrightarrow \Sigma \mathbb{C}P^{3}$ is the connecting map of the cofiber sequence $\mathbb{C}P^3 \hookrightarrow \mathbb{C}P^{n} \longrightarrow \mathbb{C}P^{n}/\mathbb{C}P^3$ for $n=7,8$. The commutativity of the top triangle follows from the naturality of these cofiber sequences, while the outer arrows commute by Diagram~\ref{conn-8}.

Let the restrictions of the composite
\[
q \circ \delta \colon \mathbb{C}P^6/\mathbb{C}P^3 \vee \mathbb{S}^{14} \simeq_{(2)} \mathbb{C}P^{7}/\mathbb{C}P^3 \xrightarrow{\delta} \Sigma \mathbb{C}P^3 \xrightarrow{q} \Sigma(\mathbb{C}P^{3}/\mathbb{C}P^{1}) \simeq \mathbb{S}^{5} \vee \mathbb{S}^{7}
\]
to the first and second wedge summands be denoted by $(q\circ \delta)_{r_1}$ and $(q\circ \delta)_{r_2}$, respectively. Since the bottom arrow satisfies $\eta^{*}(\sigma) = \eta \circ \sigma \neq 0 \in [\mathbb{S}^8, SF_{(2)}]$, and the image of $\delta^{*} \colon [\Sigma \mathbb{C}P^{3}, SF_{(2)}] \longrightarrow [\mathbb{C}P^{8}/\mathbb{C}P^3, SF_{(2)}]$ is $\mathbb{Z}_2\{\left(\overline{(\eta \circ \sigma)_{8}}\right)_{16}\}$ by \eqref{exrho}, Diagram~\eqref{cp8-conn} implies that the image of the middle map
\begin{equation}\label{imageextension}
\delta^{*} \colon [\Sigma \mathbb{C}P^{3}, SF_{(2)}] \longrightarrow [\mathbb{C}P^6/\mathbb{C}P^3, SF_{(2)}] \oplus [\mathbb{S}^{14}, SF_{(2)}]
\end{equation}
is $\mathbb{Z}_2\{y\}$, where $y$ is an extension of $(\eta \circ \sigma)_{8}$ along the map $i^{*} \colon [\mathbb{C}P^{7}/\mathbb{C}P^3, SF_{(2)}] \cong [\mathbb{C}P^6/\mathbb{C}P^3, SF_{(2)}] \oplus [\mathbb{S}^{14}, SF_{(2)}] \to [\mathbb{C}P^{4}/\mathbb{C}P^3, SF_{(2)}]$. Moreover, $\left(\overline{(\eta \circ \sigma)_{8}}\right)_{16}$ is an extension of $y$ along the map $i^{*} \colon [\mathbb{C}P^{8}/\mathbb{C}P^3, SF_{(2)}] \to [\mathbb{C}P^{7}/\mathbb{C}P^3, SF_{(2)}]$. 

From the middle commutative triangle in Diagram~\eqref{cp8-conn}, it follows that the generator $y = \left(\overline{(\eta \circ \sigma)_{8}}\right)_{14} + l$ for some $l \in [\mathbb{S}^{14}, SF_{(2)}]$, and that
\begin{equation}\label{image-ex}
y \in \operatorname{Im}\Bigl([\mathbb{S}^7, SF_{(2)}] \xrightarrow{(q\circ \delta)^{*} = (q\circ \delta)_{r_1}^{*} \oplus (q\circ \delta)_{r_2}^{*}} [\mathbb{C}P^6/\mathbb{C}P^3, SF_{(2)}] \oplus [\mathbb{S}^{14}, SF_{(2)}]\Bigr).
\end{equation}
Since $\left(\overline{(\eta \circ \sigma)_{8}}\right)_{16}$ is an extension of $y$ along $i^{*} \colon [\mathbb{C}P^{8}/\mathbb{C}P^3, SF_{(2)}] \longrightarrow [\mathbb{C}P^{7}/\mathbb{C}P^3, SF_{(2)}]$, it follows from \eqref{kernelofp} that $l=0$ or $l=(\sigma^2)_{14}$. We now prove that the latter case is not possible. 

Observe that the cofiber of the connecting map $\delta \colon \mathbb{C}P^6/\mathbb{C}P^3 \vee \mathbb{S}^{14} \simeq_{(2)} \mathbb{C}P^{7}/\mathbb{C}P^3 \longrightarrow \Sigma \mathbb{C}P^{3}$ is $\Sigma \mathbb{C}P^{7}$, so the cofiber of the composite
\[
q\circ \delta \colon \mathbb{C}P^6/\mathbb{C}P^3 \vee \mathbb{S}^{14} \xrightarrow{\;\delta\;} \Sigma \mathbb{C}P^{3} \xrightarrow{\;q\;} \Sigma(\mathbb{C}P^{3}/\mathbb{C}P^{1}) \simeq \mathbb{S}^{5} \vee \mathbb{S}^{7}
\]
is $\Sigma (\mathbb{C}P^{7}/\mathbb{C}P^{1})$. Hence, the restriction map $(q\circ \delta)_{r_2} \colon \mathbb{S}^{14} \longrightarrow \mathbb{S}^5 \vee \mathbb{S}^{7}$ is the suspension of the $14$-cell attaching map $\mathbb{S}^{13} \xrightarrow{(\beta,\gamma)} \mathbb{S}^{4} \vee \mathbb{S}^{6}$ onto the $6$-skeleton of $\mathbb{C}P^{7}/\mathbb{C}P^{1}$, where the maps $\beta \colon \mathbb{S}^{13} \to \mathbb{S}^{4}$ and $\gamma \colon \mathbb{S}^{13} \longrightarrow \mathbb{S}^{6}$ are the attaching maps onto the $4$-cell and $6$-cell in $\mathbb{C}P^{7}/\mathbb{C}P^{1}$, respectively. 

We now determine the map $\gamma \colon \mathbb{S}^{13} \longrightarrow \mathbb{S}^6$. Observe that the composite
\[
\mathbb{S}^{13} \xrightarrow{\;\varphi_{7}\;} \mathbb{C}P^{6}/\mathbb{C}P^{2} \xrightarrow{\;q\;} \mathbb{C}P^{6}/\mathbb{C}P^{3}
\]
is the $14$-cell attaching map of $\mathbb{C}P^{7}/\mathbb{C}P^{3}$ and is thus null-homotopic (after localizing at $2$) by \eqref{spil1}. Thus, the map $\varphi_{7} \colon \mathbb{S}^{13} \longrightarrow \mathbb{C}P^{6}/\mathbb{C}P^{2}$ factors through the map $\gamma \colon \mathbb{S}^{13} \longrightarrow \mathbb{S}^6$; that is, it is homotopic to the composite $\mathbb{S}^{13} \xrightarrow{\gamma} \mathbb{S}^6 \hookrightarrow \mathbb{C}P^{6}/\mathbb{C}P^{2}$. Note that $\gamma$ is stably of the form $\gamma = \lambda \sigma \in (\pi^{s}_{7})_{(2)} \cong \mathbb{Z}_{16}\{\sigma\}$. Applying these observations to Diagram 4.12 of \cite[p. 184]{Mos68} with $n=3$, $k=4$, and $t=1$ for the map $\gamma$, and using \cite[Proposition 5.6 and Table 5.7, p. 185]{Mos68}, shows that $\lambda=2$.

Hence, the map $\gamma \colon \mathbb{S}^{13} \longrightarrow \mathbb{S}^6$ is represented stably by $2\sigma \in (\pi^{s}_{7})_{(2)}$, and the map $(q\circ \delta)_{r_2} \colon \mathbb{S}^{14} \longrightarrow \mathbb{S}^{5} \vee \mathbb{S}^{7}$ is represented stably by $2\sigma$ on the second factor. Consequently, the induced map
\[
(q\circ \delta)_{r_2}^{*} \colon [\mathbb{S}^{5} \vee \mathbb{S}^{7}, SF_{(2)}] \cong [\mathbb{S}^{7}, SF_{(2)}] \longrightarrow [\mathbb{S}^{14}, SF_{(2)}]
\]
is identified with $(2\sigma)^{*} \colon [\mathbb{S}^{7}, SF_{(2)}] \longrightarrow [\mathbb{S}^{14}, SF_{(2)}]$, since $SF_{(2)}$ is an infinite loop space. As $(2\sigma)^{*} = 2(\sigma^{*})$ and $[\mathbb{S}^{14}, SF_{(2)}] \cong \mathbb{Z}_{2}\{\sigma^2\} \oplus \mathbb{Z}_{2}\{\kappa\}$, the map $(2\sigma)^{*} \colon [\mathbb{S}^{7}, SF_{(2)}] \longrightarrow [\mathbb{S}^{14}, SF_{(2)}]$ is trivial, and hence $(q\circ \delta)_{r_2}^{*} \colon [\mathbb{S}^{7}, SF_{(2)}] \longrightarrow [\mathbb{S}^{14}, SF_{(2)}]$ is the trivial map. Thus, the image
\begin{equation}\label{image-ex2}
\operatorname{Im}\Bigl([\mathbb{S}^7, SF_{(2)}] \xrightarrow{(q\circ \delta)^{*} = (q\circ \delta)_{r_1}^{*} \oplus (q\circ \delta)_{r_2}^{*}} [\mathbb{C}P^6/\mathbb{C}P^3, SF_{(2)}] \oplus [\mathbb{S}^{14}, SF_{(2)}]\Bigr) \subseteq [\mathbb{C}P^6/\mathbb{C}P^3, SF_{(2)}].
\end{equation}
Together with \eqref{image-ex}, this implies that $y \in [\mathbb{C}P^6/\mathbb{C}P^3, SF_{(2)}]$, and hence $l = 0 \in [\mathbb{S}^{14}, SF_{(2)}]$. Therefore, the generator $y = \left(\overline{(\eta \circ \sigma)_{8}}\right)_{14}$, and thus from \eqref{imageextension}, we have that
\begin{equation}\label{image-cp7}
\operatorname{Im}\Bigl( [\Sigma \mathbb{C}P^{3}, SF_{(2)}] \xrightarrow{\delta^{*}} [\mathbb{C}P^{7}/\mathbb{C}P^{3}, SF_{(2)}] \Bigr) \cong \mathbb{Z}_{2}\{\left(\overline{(\eta \circ \sigma)_{8}}\right)_{14}\}.
\end{equation}
Applying this to the commutativity of the top triangle in Diagram~\eqref{cp8-conn}, and using the fact that $\left(\overline{(\eta \circ \sigma)_{8}}\right)_{16}$ is an extension of $y$ along $i^{*}$, yields \eqref{ext-8}, as required. We now turn to determining the image of the map
\(f^{*}_{\mathbb{C}P^{8}/\mathbb{C}P^{3}} \colon [\mathbb{S}^{16}, SF_{(2)}] \longrightarrow [\mathbb{C}P^{8}/\mathbb{C}P^{3}, SF_{(2)}]\). 
Since this map factors through the quotient map \(q \colon \mathbb{C}P^{8} \longrightarrow \mathbb{C}P^{8}/\mathbb{C}P^{3}\), we have the following commutative diagram:
{\small
\begin{equation}\label{cp8-deg1}
\xymatrix{
[\mathbb{S}^{16}, SF_{(2)}] \ar[r]^{f_{\mathbb{C}P^8}^*} \ar[d]_{f_{\mathbb{C}P^8 / \mathbb{C}P^3}^*}  
& [\mathbb{C}P^8, SF_{(2)}] \\
[\mathbb{C}P^8 / \mathbb{C}P^3, SF_{(2)}] \ar[ur]_{q^*}
&
}
\end{equation}}
By \eqref{cp8-deg}, the map \(f_{\mathbb{C}P^{8}}^{*} \colon [\mathbb{S}^{16}, SF_{(2)}] \longrightarrow [\mathbb{C}P^{8}, SF_{(2)}]\) sends both \(\eta^{*}\) and \(\eta^{*} + \eta \circ \rho\) to the same nontrivial element, while \(f_{\mathbb{C}P^{8}}^{*}(\eta \circ \rho) = 0\). 

Diagram \eqref{cp8-deg1} implies that \(f_{\mathbb{C}P^{8}/\mathbb{C}P^{3}}^{*}\) maps both \(\eta^{*}\) and \(\eta^{*} + \eta \circ \rho\) nontrivially, and further satisfies \(q^{*} \circ f_{\mathbb{C}P^{8}/\mathbb{C}P^{3}}^{*}(\eta \circ \rho) = 0\). Thus,
\[
f_{\mathbb{C}P^{8}/\mathbb{C}P^{3}}^{*}(\eta \circ \rho) \in \operatorname{Ker}\Bigl( [\mathbb{C}P^{8}/\mathbb{C}P^{3}, SF_{(2)}] \xrightarrow{\,q^{*}\,} [\mathbb{C}P^{8}, SF_{(2)}] \Bigr),
\]
which is identified as \(\mathbb{Z}_{2}\{\left(\overline{(\eta \circ \sigma)_{8}}\right)_{16}\}\) by \eqref{exact8-1}. 

It follows that either \(f_{\mathbb{C}P^{8}/\mathbb{C}P^{3}}^{*}(\eta \circ \rho) = 0\) or \(f_{\mathbb{C}P^{8}/\mathbb{C}P^{3}}^{*}(\eta \circ \rho) = \left(\overline{(\eta \circ \sigma)_{8}}\right)_{16}\). By \eqref{ext-8}, the element \(\left(\overline{(\eta \circ \sigma)_{8}}\right)_{16}\) is an extension of \(\left(\overline{(\eta \circ \sigma)_{8}}\right)_{14}\) along the map 
\[
i^{*} \colon [\mathbb{C}P^{8}/\mathbb{C}P^{3}, SF_{(2)}] \longrightarrow [\mathbb{C}P^{7}/\mathbb{C}P^{3}, SF_{(2)}].
\]
However, \(i^{*} \circ f^{*}_{\mathbb{C}P^{8}/\mathbb{C}P^{3}} = (f \circ i)^{*} = 0\), which forces \(f_{\mathbb{C}P^{8}/\mathbb{C}P^{3}}^{*}(\eta \circ \rho) = 0\). 

Consequently, the map 
\[
f_{\mathbb{C}P^{8}/\mathbb{C}P^{3}}^{*} \colon [\mathbb{S}^{16}, SF_{(2)}] \longrightarrow [\mathbb{C}P^{8}/\mathbb{C}P^{3}, SF_{(2)}]
\]
maps both \(\eta^{*}\) and \(\eta^{*} + \eta \circ \rho\) to the same nontrivial element, and its image is \(\mathbb{Z}_2\{f_{\mathbb{C}P^{8}/\mathbb{C}P^{3}}^{*}(z)\}\) for \(z \in \{\eta^{*}, \eta^{*} + \eta \circ \rho\}\). Combining this with \eqref{kernelofp}, the long exact sequence \eqref{equ13} reduces to the following short exact sequence:
\begin{equation}\label{equ14}
\begin{aligned}
0 \longrightarrow \mathbb{Z}_2\{f_{\mathbb{C}P^{8}/\mathbb{C}P^{3}}^{*}(z)\}
&\longrightarrow [\mathbb{C}P^{8}/\mathbb{C}P^{3}, SF_{(2)}]
\xrightarrow{i^{*}} \mathbb{Z}_{2}\{(\sigma^{2})_{14}\} \oplus \mathbb{Z}_{2}\{\left(\overline{(\eta \circ \sigma)_{8}}\right)_{14}\} \\
&\qquad \oplus \mathbb{Z}_{2}\{(\kappa)_{14} + \left(\overline{(\epsilon)_{8}}\right)_{14}\} \longrightarrow 0,
\end{aligned}
\end{equation}
where \(z \in \{\eta^{*}, \eta^{*} + \eta \circ \rho\}\). Now, combining the short exact sequences \eqref{exact8-1} and \eqref{equ14}, we see that there are three possible group structures for \([\mathbb{C}P^{8}/\mathbb{C}P^{3}, SF_{(2)}]\) arising as extensions, namely
\begin{equation}\label{eq:CP8CP3-SF-decomp-1}
\begin{aligned}
[\mathbb{C}P^{8}/\mathbb{C}P^{3}, SF_{(2)}] \cong \mathbb{Z}_{4}\{w\} \oplus \mathbb{Z}_{2}\{\left(\overline{(\eta \circ \sigma)_{8}}\right)_{16}\} \oplus \mathbb{Z}_2\{\left(\overline{(\kappa)_{14} + \left(\overline{(\epsilon)_{8}}\right)_{14}}\right)_{16}\},
\end{aligned}
\end{equation}
where \(w = \left(\overline{(\sigma^{2})_{14}}\right)_{16}\) or \(\left(\overline{(\kappa)_{14} + \left(\overline{(\epsilon)_{8}}\right)_{14}}\right)_{16}\), satisfying
\begin{equation}\label{cp8-degre}
2w = f^{*}_{\mathbb{C}P^{8}/\mathbb{C}P^{3}}(z)
\end{equation}
for \(z \in \{\eta^{*}, \eta^{*} + \eta \circ \rho\}\); or
\begin{equation}\label{eq:CP8CP3-SF-decomp-2}
\begin{aligned}
[\mathbb{C}P^{8}/\mathbb{C}P^{3}, SF_{(2)}] \cong \;& \mathbb{Z}_{2}\{f^{*}_{\mathbb{C}P^{8}/\mathbb{C}P^{3}}(z)\} \oplus \mathbb{Z}_{2}\{\left(\overline{(\sigma^{2})_{14}}\right)_{16}\} \\
&\oplus \mathbb{Z}_{2}\{\left(\overline{(\eta \circ \sigma)_{8}}\right)_{16}\} \oplus \mathbb{Z}_2\{\left(\overline{(\kappa)_{14} + \left(\overline{(\epsilon)_{8}}\right)_{14}}\right)_{16}\},
\end{aligned}
\end{equation}
again for \(z \in \{\eta^{*}, \eta^{*} + \eta \circ \rho\}\).

We now show that the extension \eqref{eq:CP8CP3-SF-decomp-1} occurs for \(w = \left(\overline{(\sigma^{2})_{14}}\right)_{16}\), which completes the proof of Part \textup{(ii)}. Consider the following commutative diagram of short exact sequences:
{\small
\[
\xymatrix@C=1.2em@R=3.5em{
0 \ar[r] & \mathbb{Z}_2\{f_{\mathbb{C}P^{8}/\mathbb{C}P^{6}}^{*}(z)\} \ar[r] \ar[d]^{\mathrm{Id}} & [\mathbb{C}P^{8}/\mathbb{C}P^{6}, SF_{(2)}] \cong \mathbb{Z}_4\{\left(\overline{(\sigma^{2})_{14}}\right)_{16}\} \ar[r]^-{i^{*}} \ar[d]^{q^{*}} & \mathbb{Z}_2\{(\sigma^2)_{14}\} \ar[r] \ar[d]^{p_{2}^{*}} & 0 \\
0 \ar[r] & \mathbb{Z}_2\{f_{\mathbb{C}P^{8}/\mathbb{C}P^{3}}^{*}(z)\} \ar[r] & [\mathbb{C}P^{8}/\mathbb{C}P^{3}, SF_{(2)}] \ar[r]^-{i^{*}} & K \ar[r] & 0
}
\]}
where \(K = \mathbb{Z}_2\{(\sigma^{2})_{14}\} \oplus \mathbb{Z}_2\{\left(\overline{(\eta \circ \sigma)_{8}}\right)_{14}\} \oplus \mathbb{Z}_2\{(\kappa)_{14} + \left(\overline{(\epsilon)_{8}}\right)_{14}\}\). Here, \(\mathbb{Z}_2\{(\sigma^2)_{14}\} \subset [\mathbb{C}P^{7}/\mathbb{C}P^{6}, SF_{(2)}]\) and \(K \subset [\mathbb{C}P^{7}/\mathbb{C}P^{3}, SF_{(2)}]\). The first row is the short exact sequence given by Part~\textup{(i)}, and the second row is \eqref{equ14}. The map \(q \colon \mathbb{C}P^{8}/\mathbb{C}P^{3} \longrightarrow \mathbb{C}P^{8}/\mathbb{C}P^{6}\) is the natural collapse map, and the rightmost vertical map \(p_{2}^{*}\) is induced by the canonical collapse map
\[
p_{2} \colon \mathbb{C}P^{7}/\mathbb{C}P^{3} \simeq_{(2)} \mathbb{C}P^{6}/\mathbb{C}P^{3} \vee \mathbb{S}^{14} \longrightarrow \mathbb{S}^{14}.
\]
The commutativity of the left square follows from the fact that the composite \(f_{\mathbb{C}P^{8}/\mathbb{C}P^{6}} \circ q \colon \mathbb{C}P^{8}/\mathbb{C}P^{3} \longrightarrow \mathbb{S}^{16}\) is homotopic to the map \(f_{\mathbb{C}P^{8}/\mathbb{C}P^{3}} \colon \mathbb{C}P^{8}/\mathbb{C}P^{3} \longrightarrow \mathbb{S}^{16}\). The commutativity of the right square follows from the fact that the composite
\[
\mathbb{C}P^{7}/\mathbb{C}P^{3} \simeq_{(2)} \mathbb{C}P^{6}/\mathbb{C}P^{3} \vee \mathbb{S}^{14} \xrightarrow{p_{2}} \mathbb{S}^{14} \simeq \mathbb{C}P^{7}/\mathbb{C}P^{6} \hookrightarrow \mathbb{C}P^{8}/\mathbb{C}P^{6}
\]
is homotopic to the composite of the inclusion and the collapse map :
\[
\mathbb{C}P^{7}/\mathbb{C}P^{3} \hookrightarrow \mathbb{C}P^{8}/\mathbb{C}P^{3} \xrightarrow{q} \mathbb{C}P^{8}/\mathbb{C}P^{6}.
\]
Since \(p_{2}^{*}\) is injective on the summand \(\mathbb{Z}_{2}\{(\sigma^{2})_{14}\}\), a diagram chase shows that the middle vertical map \(q^{*}\) is injective. This implies that the extension \eqref{eq:CP8CP3-SF-decomp-1} must occur for \(w = \left(\overline{(\sigma^{2})_{14}}\right)_{16}\), as required.
\end{proof}
The following corollary gives explicit generators of the group $[\mathbb{C}P^{7}, SF]$.
\begin{corollary}\label{cp7-sf}
There is a split short exact sequence
\[
\begin{aligned}
0 \longrightarrow\; & \mathbb{Z}_{2}\{\left(\overline{(\eta \circ \sigma)_8}\right)_{14}\} \hookrightarrow [\mathbb{C}P^{7}/\mathbb{C}P^{3}, SF_{(2)}] \\
&\cong \mathbb{Z}_{2}\{\left(\overline{(\epsilon)_8}\right)_{14}\} \oplus \mathbb{Z}_{2}\{\left(\overline{(\eta \circ \sigma)_8}\right)_{14}\} \oplus \mathbb{Z}_{2}\{(\kappa)_{14}\} \oplus \mathbb{Z}_{2}\{(\sigma^{2})_{14}\} \\
&\xrightarrow{\,q^{*}\,} [\mathbb{C}P^{7}, SF_{(2)}] \longrightarrow 0.
\end{aligned}
\]
In particular, there is an isomorphism
\[
[\mathbb{C}P^{7}, SF] \cong \mathbb{Z}_{2}\{\left(\overline{(\epsilon)_8}\right)_{14}\} \oplus \mathbb{Z}_{2}\{(\kappa)_{14}\} \oplus \mathbb{Z}_{2}\{(\sigma^{2})_{14}\},
\]
where $(\kappa)_{14} = f^{*}_{\mathbb{C}P^{7}}(\kappa)$ and $(\sigma^{2})_{14} = f^{*}_{\mathbb{C}P^{7}}(\sigma^{2})$. The element $\left(\overline{(\epsilon)_8}\right)_{14}$ is an extension of $(\epsilon)_8 \in [\mathbb{C}P^4, SF]$ along the map $i^{*} \colon [\mathbb{C}P^7, SF] \to [\mathbb{C}P^4, SF]$, where $(\epsilon)_8 = f^{*}_{\mathbb{C}P^{4}}(\epsilon)$ for $\epsilon \in [\mathbb{S}^8, SF]$. Moreover,
\begin{equation}\label{kernelofp1}
\operatorname{Ker}\Bigl( [\mathbb{C}P^{7}, SF_{(2)}] \xrightarrow{\;p^{*}\;} [\mathbb{S}^{15}, SF_{(2)}] \Bigr) \cong \mathbb{Z}_{2}\{(\sigma^{2})_{14}\} \oplus \mathbb{Z}_{2}\{(\kappa)_{14} + \left(\overline{(\epsilon)_8}\right)_{14}\},
\end{equation}
where $p \colon \mathbb{S}^{15} \to \mathbb{C}P^7$ is the Hopf fibration.
\end{corollary}

\begin{proof}
Using \eqref{cp7-q} and \eqref{image-cp7}, the long exact sequence associated to the cofiber sequence
\[
\mathbb{C}P^{3} \hookrightarrow \mathbb{C}P^{7} \longrightarrow \mathbb{C}P^{7}/\mathbb{C}P^{3} \xrightarrow{\delta} \Sigma \mathbb{C}P^{3}
\]
implies the existence of a short exact sequence
\begin{equation}\label{short-cp7}
0 \longrightarrow \mathbb{Z}_{2}\{\left(\overline{(\eta \circ \sigma)_8}\right)_{14}\} \hookrightarrow [\mathbb{C}P^{7}/\mathbb{C}P^{3}, SF_{(2)}] \xrightarrow{q^{*}} [\mathbb{C}P^{7}, SF_{(2)}] \longrightarrow 0.
\end{equation}
By \eqref{SF-spilit}, we have
\[
[\mathbb{C}P^{7}/\mathbb{C}P^{3}, SF_{(2)}] \cong \mathbb{Z}_{2}\{\left(\overline{(\epsilon)_8}\right)_{14}\} \oplus \mathbb{Z}_{2}\{\left(\overline{(\eta \circ \sigma)_8}\right)_{14}\} \oplus \mathbb{Z}_{2}\{(\sigma^{2})_{14}\} \oplus \mathbb{Z}_{2}\{(\kappa)_{14}\}.
\]
Combining this with Proposition~\ref{stab}, it follows that the short exact sequence \eqref{short-cp7} splits, which proves the first assertion. Consequently, we have the isomorphism
\begin{equation}\label{cp7-sfl}
[\mathbb{C}P^{7}, SF_{(2)}] \cong \mathbb{Z}_{2}\{q^{*}\left(\left(\overline{(\epsilon)_8}\right)_{14}\right)\} \oplus \mathbb{Z}_{2}\{q^{*}\left((\kappa)_{14}\right)\} \oplus \mathbb{Z}_{2}\{q^{*}\left((\sigma^{2})_{14}\right)\}.
\end{equation}
To identify the generators of $[\mathbb{C}P^{7}, SF_{(2)}]$, first observe that for $y \in \{\kappa, \sigma^{2}\}$, we have
\begin{equation}\label{gen1}
q^{*}\left((y)_{14}\right) = q^{*} \circ f^{*}_{\mathbb{C}P^{7}/\mathbb{C}P^{3}}(y) = f^{*}_{\mathbb{C}P^{7}}(y).
\end{equation}
Next, consider the following commutative diagram:
{\small
\begin{equation}\label{cp7-quoti}
\xymatrix@C=4.4em@R=3.4em{
[\mathbb{C}P^{7}/\mathbb{C}P^{3}, SF_{(2)}] \ar[r]^-{q^{*}} \ar[d]^-{i^{*}} 
& [\mathbb{C}P^{7}, SF_{(2)}] \ar[d]^-{i^{*}} \\
[\mathbb{C}P^{4}/\mathbb{C}P^{3}, SF_{(2)}] \cong [\mathbb{S}^{8}, SF_{(2)}] \ar[r]^-{q^{*}} 
& [\mathbb{C}P^{4}, SF_{(2)}],
}
\end{equation}}
where the bottom horizontal map $q^{*} \colon [\mathbb{S}^{8}, SF_{(2)}] \longrightarrow [\mathbb{C}P^{4}, SF_{(2)}]$ is identified with the map $f^{*}_{\mathbb{C}P^{4}} \colon [\mathbb{S}^{8}, SF_{(2)}] \longrightarrow [\mathbb{C}P^{4}, SF_{(2)}]$. Thus, $f^{*}_{\mathbb{C}P^{4}}(\epsilon) = q^{*}(\epsilon) = (\epsilon)_{8}$ for $\epsilon \in [\mathbb{S}^{8}, SF_{(2)}]$. 

Since $\left(\overline{(\epsilon)_8}\right)_{14}$ is an extension of $(\epsilon)_8$ along the vertical map $$i^{*} \colon [\mathbb{C}P^{7}, SF_{(2)}] \longrightarrow [\mathbb{C}P^{4}/\mathbb{C}P^{3}, SF_{(2)}],$$ the commutativity of Diagram~\eqref{cp7-quoti} implies that $q^{*}\left(\left(\overline{(\epsilon)_8}\right)_{14}\right)$ is an extension of $(\epsilon)_{8}$ along the map $i^{*} \colon [\mathbb{C}P^{7}, SF_{(2)}] \longrightarrow [\mathbb{C}P^{4}, SF_{(2)}]$. We therefore set
\begin{equation}\label{qmaps}
q^{*}\left(\left(\overline{(\epsilon)_8}\right)_{14}\right) = \left(\overline{(\epsilon)_{8}}\right)_{14}.
\end{equation}
Applying this and \eqref{gen1} to \eqref{cp7-sfl}, we obtain
\begin{equation}\label{cp7-sfl1}
[\mathbb{C}P^{7}, SF_{(2)}] \cong \mathbb{Z}_{2}\{\left(\overline{(\epsilon)_{8}}\right)_{14}\} \oplus \mathbb{Z}_{2}\{f^{*}_{\mathbb{C}P^{7}}(\kappa)\} \oplus \mathbb{Z}_{2}\{f^{*}_{\mathbb{C}P^{7}}(\sigma^{2})\}.
\end{equation}
As the group $[\mathbb{C}P^{7}, SF]$ is $2$-torsion by Proposition~\ref{stab}, the second assertion of the corollary follows from \eqref{cp7-sfl1}. Finally, we compute the kernel of $p^{*} \colon [\mathbb{C}P^{7}, SF_{(2)}] \to [\mathbb{S}^{15}, SF_{(2)}]$. 

The top-cell attaching map $\varphi_{8} \colon \mathbb{S}^{15} \to \mathbb{C}P^{7}/\mathbb{C}P^{3}$ of $\mathbb{C}P^{8}/\mathbb{C}P^{3}$ is the composite of the Hopf fibration $p \colon \mathbb{S}^{15} \to \mathbb{C}P^{7}$ followed by the quotient map $q \colon \mathbb{C}P^{7} \to \mathbb{C}P^{7}/\mathbb{C}P^{3}$. Thus, $p^{*} \circ q^{*} = \varphi_{8}^{*}$. Since $q^{*}$ is surjective by \eqref{short-cp7}, it follows that $\operatorname{Ker}(p^{*}) = q^{*}(\operatorname{Ker}(\varphi_{8}^{*}))$. Applying this and \eqref{kernelofp} to the split short exact sequence given in the first assertion, we obtain
\[
\operatorname{Ker}\Bigl( [\mathbb{C}P^{7}, SF_{(2)}] \xrightarrow{\;p^{*}\;} [\mathbb{S}^{15}, SF_{(2)}] \Bigr) \cong \mathbb{Z}_{2}\{(\sigma^{2})_{14}\} \oplus \mathbb{Z}_{2}\{(\kappa)_{14} + \left(\overline{(\epsilon)_8}\right)_{14}\}.
\]
This completes the proof.
\end{proof}
\begin{theorem}\label{eight-sub2}
\begin{itemize}
\item[(i)] The short exact sequence in Theorem~\ref{eight}(i) does not split. In particular,
\[
[\mathbb{C}P^{8}, SF] \cong \mathbb{Z}_{4}\{\left(\overline{(\sigma^{2})_{14}}\right)_{16}\} \oplus \mathbb{Z}_2\{\left(\overline{(\kappa)_{14} + \left(\overline{(\epsilon)_{8}}\right)_{14}}\right)_{16}\},
\]
where \(\left(\overline{(\sigma^{2})_{14}}\right)_{16}\) and \(\left(\overline{(\kappa)_{14} + \left(\overline{(\epsilon)_{8}}\right)_{14}}\right)_{16}\) are extensions of \((\sigma^{2})_{14}\) and \((\kappa)_{14} + \left(\overline{(\epsilon)_{8}}\right)_{14}\), respectively, along the map
\[
i^{*}\colon [\mathbb{C}P^{8}, SF] \longrightarrow [\mathbb{C}P^{7}, SF],
\]
satisfying \(2\left(\overline{(\sigma^{2})_{14}}\right)_{16} = (z)_{16}\) with \((z)_{16} = f^{*}_{\mathbb{C}P^{8}}(z)\) for \(z \in \{\eta^{*}, \eta^{*} + \eta \circ \rho\} \subset [\mathbb{S}^{16}, SF]\).

\item[(ii)] The short exact sequence in Theorem~2.9(ii) splits. Moreover,
\[
[\mathbb{C}P^{8}, \operatorname{Top}/O] \cong \mathbb{Z}_{2}\{f_{\mathbb{C}P^{8}}^{*}(\Sigma^{16})\} \oplus \operatorname{Ker}(p^{*}),
\]
where \(\Sigma^{16} \in \Theta_{16} \cong \mathbb{Z}_2\) is the exotic \(16\)-sphere and \(\operatorname{Ker}(p^{*}) \cong \mathbb{Z}_2\), with the map
\[
p^{*}\colon [\mathbb{C}P^{7}, \operatorname{Top}/O] \cong \mathbb{Z}_{2} \oplus \mathbb{Z}_{2} \longrightarrow \mathbb{Z}_{2}\{\eta \circ \kappa\} = \operatorname{Coker}(J_{15}) \subset [\mathbb{S}^{15}, \operatorname{Top}/O]
\]
induced by the Hopf fibration \(p\colon \mathbb{S}^{15} \longrightarrow \mathbb{C}P^{7}\).
\end{itemize}
\end{theorem}

\begin{proof}
It follows from the short exact sequence \eqref{exact8-1} that the map
\[
q^{*}\colon [\mathbb{C}P^{8}/\mathbb{C}P^{3}, SF] \longrightarrow [\mathbb{C}P^{8}, SF]
\]
is surjective with kernel \(\mathbb{Z}_{2}\{\left(\overline{(\eta \circ \sigma)_8}\right)_{16}\}\). By Lemma~\ref{eight-sub}(ii), we have
\[
[\mathbb{C}P^{8}/\mathbb{C}P^{3}, SF] \cong \mathbb{Z}_{4}\{\left(\overline{(\sigma^{2})_{14}}\right)_{16}\} \oplus \mathbb{Z}_{2}\{\left(\overline{(\eta \circ \sigma)_8}\right)_{16}\} \oplus \mathbb{Z}_{2}\{\left(\overline{(\kappa)_{14} + \left(\overline{(\epsilon)_8}\right)_{14}}\right)_{16}\}.
\]
Hence, we obtain the isomorphism
\begin{equation}\label{cp8-gen}
[\mathbb{C}P^{8}, SF] \cong \mathbb{Z}_{4}\{q^{*}\left(\left(\overline{(\sigma^{2})_{14}}\right)_{16}\right)\} \oplus \mathbb{Z}_{2}\{q^{*}\left(\left(\overline{(\kappa)_{14} + \left(\overline{(\epsilon)_8}\right)_{14}}\right)_{16}\right)\}.
\end{equation}
Note from Lemma~\ref{eight-sub}(ii) that \(2\left(\overline{(\sigma^{2})_{14}}\right)_{16} = f^{*}_{\mathbb{C}P^{8}/\mathbb{C}P^{3}}(z)\) for \(z \in \{\eta^*, \eta^* + \eta \circ \rho\}\), and thus
\begin{equation}\label{ext2}
2q^{*}\left(\left(\overline{(\sigma^{2})_{14}}\right)_{16}\right) = q^{*}\left(f^{*}_{\mathbb{C}P^{8}/\mathbb{C}P^{3}}(z)\right) = f^{*}_{\mathbb{C}P^{8}}(z).
\end{equation}

To identify the generators in \eqref{cp8-gen}, consider the following commutative diagram:
{\small
\begin{equation}\label{cp8-quoti}
\xymatrix@C=4.4em@R=3.4em{
[\mathbb{C}P^{8}/\mathbb{C}P^{3}, SF] \ar[r]^-{q^{*}} \ar[d]^-{i^{*}} & [\mathbb{C}P^{8}, SF] \ar[d]^-{i^{*}} \\
[\mathbb{C}P^{7}/\mathbb{C}P^{3}, SF] \ar[r]^-{q^{*}} & [\mathbb{C}P^{7}, SF],
}
\end{equation}}
where \(\left(\overline{(\kappa)_{14} + \left(\overline{(\epsilon)_8}\right)_{14}}\right)_{16}\) and \(\left(\overline{(\sigma^{2})_{14}}\right)_{16}\) are extensions of \((\kappa)_{14} + \left(\overline{(\epsilon)_8}\right)_{14}\) and \((\sigma^{2})_{14} \in [\mathbb{C}P^{7}/\mathbb{C}P^{3}, SF]\), respectively, along the vertical map \(i^{*}\colon [\mathbb{C}P^{8}/\mathbb{C}P^{3}, SF] \to [\mathbb{C}P^{7}/\mathbb{C}P^{3}, SF]\). Here, \(f^{*}_{\mathbb{C}P^{7}/\mathbb{C}P^{3}}(\sigma^{2}) = (\sigma^{2})_{14}\). The commutativity of Diagram~\eqref{cp8-quoti}, combined with \eqref{qmaps} and \eqref{gen1}, implies that the image of \(\left(\overline{(\kappa)_{14} + \left(\overline{(\epsilon)_8}\right)_{14}}\right)_{16}\) under \(q^{*}\) is an extension of the element \((\kappa)_{14} + \left(\overline{(\epsilon)_8}\right)_{14} \in [\mathbb{C}P^{7}, SF]\) along the map \(i^{*}\colon [\mathbb{C}P^{8}, SF] \to [\mathbb{C}P^{7}, SF]\). We therefore set
\begin{equation}\label{ext1}
q^{*}\left(\left(\overline{(\kappa)_{14} + \left(\overline{(\epsilon)_8}\right)_{14}}\right)_{16}\right) = \left(\overline{(\kappa)_{14} + \left(\overline{(\epsilon)_{8}}\right)_{14}}\right)_{16},
\end{equation}
and similarly for the other generator:
\begin{equation}\label{ext3}
q^{*}\left(\left(\overline{(\sigma^{2})_{14}}\right)_{16}\right) = \left(\overline{f^{*}_{\mathbb{C}P^{7}}(\sigma^{2})}\right)_{16} = \left(\overline{(\sigma^{2})_{14}}\right)_{16}.
\end{equation}
Applying \eqref{ext2}, \eqref{ext1}, and \eqref{ext3} to \eqref{cp8-gen}, we obtain
\[
[\mathbb{C}P^{8}, SF] \cong \mathbb{Z}_{4}\{\left(\overline{(\sigma^{2})_{14}}\right)_{16}\} \oplus \mathbb{Z}_{2}\{\left(\overline{(\kappa)_{14} + \left(\overline{(\epsilon)_{8}}\right)_{14}}\right)_{16}\},
\]
where \(2\left(\overline{(\sigma^{2})_{14}}\right)_{16} = f^{*}_{\mathbb{C}P^{8}}(z)\) for \(z \in \{\eta^*, \eta^* + \eta \circ \rho\}\). This completes the proof of Part~\textup{(i)}.
For Part \textup{(ii)}, it follows from Theorem~\ref{eight}(ii) that the group $[\mathbb{C}P^{8}, \operatorname{Top}/O]$ is abelian of order $4$, and therefore is isomorphic to either $\mathbb{Z}_{2} \oplus \mathbb{Z}_{2}$ or $\mathbb{Z}_{4}$. Thus, we work $2$\nobreakdash-locally. In view of \eqref{equ21} and \eqref{equ22}, it suffices to show that $[\mathbb{C}P^{8}, PL_{(2)}] \cong \mathbb{Z}_{2} \oplus \mathbb{Z}_{2}$, which completes the proof of assertion \textup{(ii)}. To this end, we first analyze the generators of the images of the maps 
\[ J_{PL_{(2)}} \colon \pi_{14}(PL_{(2)}) \longrightarrow \pi_{14}(F_{(2)}) \quad \text{and} \quad J_{PL_{(2)}} \colon [\mathbb{C}P^{7}, PL_{(2)}] \longrightarrow [\mathbb{C}P^{7}, F_{(2)}] \] 
induced by the canonical map $j_{PL_{(2)}} \colon PL_{(2)} \longrightarrow F_{(2)}$. From the Kervaire–Milnor braid (see \cite[p.~463]{LM24} and \cite[p.~305]{Bru69}), and using the facts that $\pi_{k}(F/PL)=0$ for $k=9,15$, $\pi_{8}(F/PL)=\mathbb{Z}$, and $\pi_{k}(O)=0$ for $k=13,14$, we obtain that $J_{PL_{(2)}} \colon \pi_{14}(PL_{(2)}) \longrightarrow \pi_{14}(F_{(2)})$ is injective, $J_{PL_{(2)}} \colon \pi_{8}(PL_{(2)}) \longrightarrow \pi_{8}(F_{(2)})$ is an isomorphism, and $$\beta_{*} \colon \pi_{14}(PL_{(2)}) \longrightarrow \pi_{14}(PL/O_{(2)})$$ is also an isomorphism. This yields the commutative diagram:
{\small
\[
\xymatrix@C=4.4em@R=3.4em{
  \pi_{14}(PL_{(2)})
    \ar[r]^-{J_{PL_{(2)}}}
    \ar[d]_-{\beta_{*}}^-{\cong} &
  \pi_{14}(F_{(2)}) \cong \mathbb{Z}_2\{\kappa\} \oplus \mathbb{Z}_2\{\sigma^2\}
    \ar[d]^-{\varphi_{*}}_-{\cong} \\
  \Theta_{14} \cong \pi_{14}(PL/O_{(2)}) \cong \mathbb{Z}_{2}\{\Sigma^{14}\}
    \ar[r]^-{(\psi_{PL/O_{(2)}})_{*}} &
  \pi_{14}(F/O_{(2)}),
}
\]}
where $\Sigma^{14}$ is the exotic $14$\nobreakdash-sphere and $(\psi_{PL/O_{(2)}})_{*}(\Sigma^{14}) = \varphi_{*}(\kappa)$ (see \eqref{kappa1}). From this diagram, we obtain
\[
\operatorname{Im}\bigl( \pi_{14}(PL_{(2)}) \xrightarrow{J_{PL_{(2)}}} \pi_{14}(F_{(2)}) \bigr) \cong \mathbb{Z}_{2}\{\kappa\}.
\]
For the map $[\mathbb{C}P^{7}, PL_{(2)}] \longrightarrow [\mathbb{C}P^{7}, F_{(2)}]$, consider the commutative diagram:
{\small
\begin{equation}\label{equ-PL}
\xymatrix@C=1.2em@R=2.4em{
[\mathbb{C}P^{7}, PL_{(2)}]
\ar@{>->}[rr]^-{J_{PL_{(2)}}}
&&
[\mathbb{C}P^{7}, F_{(2)}] \cong
\mathbb{Z}_{2}\{\left(\overline{(\epsilon)_{8}}\right)_{14}\}
\oplus \mathbb{Z}_{2}\{(\kappa)_{14}\} \oplus \mathbb{Z}_{2}\{(\sigma^{2})_{14}\}
\\
[\mathbb{C}P^{7}/\mathbb{C}P^{3}, PL_{(2)}]
\ar@{>->}[rr]^-{J_{PL_{(2)}} \oplus J_{PL_{(2)}}}
\ar@{->>}[u]_-{q^{*}}
&&
[\mathbb{C}P^{7}/\mathbb{C}P^{3}, F_{(2)}] 
\ar@{->>}[u]^-{q^{*}}
}
\end{equation}}
The injectivity of $J_{PL_{(2)}} \colon [\mathbb{C}P^{7}, PL_{(2)}] \longrightarrow [\mathbb{C}P^{7}, F_{(2)}]$ follows from \eqref{equ23}, and the surjectivity of the vertical maps follows from $[\mathbb{C}P^{3}, PL]=0$ by \eqref{equ22} and Corollary~\ref{cp7-sf}. Moreover, 
\[ [\mathbb{C}P^{7}/\mathbb{C}P^{3}, PL_{(2)}] \cong [\mathbb{S}^{8}, PL_{(2)}] \oplus [\mathbb{S}^{14}, PL_{(2)}] \] 
by \eqref{spil1} and the fact that $i^{*} \colon [\mathbb{C}P^{6}/\mathbb{C}P^{3}, PL_{(2)}] \longrightarrow [\mathbb{S}^{8}, PL_{(2)}]$ is an isomorphism. The latter is obtained by repeating the arguments used for \eqref{cp6-1}, replacing $\operatorname{Top}/O_{(2)}$ with $PL_{(2)}$, together with the fact that the kernel of $\eta^{*} \colon [\mathbb{S}^{10}, PL_{(2)}] \longrightarrow [\mathbb{S}^{11}, PL_{(2)}]$ is zero (see the first paragraph of Lemma~3.1 in \cite[p.~11]{BKS25}). Since the image of the bottom horizontal map $(J_{PL_{(2)}} \oplus J_{PL_{(2)}})$ is $[\mathbb{S}^{8}, F_{(2)}] \oplus \mathbb{Z}_{2}\{\kappa\}$, and the kernel of $q^{*} \colon [\mathbb{C}P^{7}/\mathbb{C}P^{3}, F_{(2)}] \to [\mathbb{C}P^{7}, F_{(2)}]$ is $\mathbb{Z}_{2}\{\left(\overline{(\eta \circ \sigma)_8}\right)_{14}\}$ by Corollary~\ref{cp7-sf}, the commutativity of Diagram~\eqref{equ-PL} implies
\begin{equation}\label{cp7-pl}
\operatorname{Im}\bigl([\mathbb{C}P^{7}, PL_{(2)}] \xrightarrow{J_{PL_{(2)}}} [\mathbb{C}P^{7}, F_{(2)}]\bigr) \cong \mathbb{Z}_{2}\{\left(\overline{(\epsilon)_{8}}\right)_{14}\} \oplus \mathbb{Z}_{2}\{(\kappa)_{14}\}.
\end{equation}
Now apply this observation to the commutative diagram:
{\small
\begin{equation}\label{cp8-sf1}
\xymatrix@C=3.5em@R=3.6em{
[\mathbb{C}P^{8}, PL_{(2)}]
  \ar[rr]^-{J_{PL_{(2)}}}
  \ar[d]_-{i^{*}} &&
[\mathbb{C}P^{8}, F_{(2)}]
  \ar[d]^-{i^{*}} \\
[\mathbb{C}P^{7}, PL_{(2)}]
  \ar[rr]^-{J_{PL_{(2)}}} &&
\mathbb{Z}_{2}\{\left(\overline{(\epsilon)_{8}}\right)_{14}\} \oplus \mathbb{Z}_{2}\{(\kappa)_{14}\}
\subset [\mathbb{C}P^{7}, F_{(2)}],
}
\end{equation}}
where $[\mathbb{C}P^{8}, F_{(2)}] \cong \mathbb{Z}_{4}\{\left(\overline{(\sigma^{2})_{14}}\right)_{16}\} \oplus \mathbb{Z}_2\{\left(\overline{(\kappa)_{14} + \left(\overline{(\epsilon)_{8}}\right)_{14}}\right)_{16}\}$ by Part~\textup{(i)}, and the top horizontal map $J_{PL_{(2)}}$ is injective by \eqref{equ23}. If the group $[\mathbb{C}P^{8}, PL_{(2)}]$ contains an element of order $4$, then  
$J_{PL_{(2)}} \colon [\mathbb{C}P^{8}, PL_{(2)}] \to [\mathbb{C}P^{8}, F_{(2)}]$ 
maps onto a subgroup $\mathbb{Z}_{4}\{l\}$, where $l = \left(\overline{(\sigma^{2})_{14}}\right)_{16}$ 
or $l = \left(\overline{(\sigma^{2})_{14}}\right)_{16} + \left(\overline{(\kappa)_{14} + \left(\overline{(\epsilon)_{8}}\right)_{14}}\right)_{16}$. 
This would imply that the image of the composite 
\[
i^{*} \circ J_{PL_{(2)}} \colon [\mathbb{C}P^{8}, PL_{(2)}] \to [\mathbb{C}P^{7}, F_{(2)}] \cong \mathbb{Z}_{2}\{\left(\overline{(\epsilon)_{8}}\right)_{14}\} \oplus \mathbb{Z}_{2}\{(\sigma^{2})_{14}\} \oplus \mathbb{Z}_{2}\{(\kappa)_{14}\}
\]
contains $(\sigma^{2})_{14}$ or $(\sigma^{2})_{14} + (\kappa)_{14} + \left(\overline{(\epsilon)_{8}}\right)_{14}$. However, Diagram~\eqref{cp8-sf1} shows that this image must lie in the subgroup 
$\mathbb{Z}_{2}\{\left(\overline{(\epsilon)_{8}}\right)_{14}\} \oplus \mathbb{Z}_{2}\{(\kappa)_{14}\}$. 
This yields a contradiction, as $(\sigma^{2})_{14}$ (and its sum with other generators) is not 
contained in that subgroup. Consequently, $[\mathbb{C}P^{8}, PL_{(2)}]$ cannot contain an 
element of order $4$, which implies that the group is isomorphic to $\mathbb{Z}_{2} \oplus \mathbb{Z}_{2}$. Hence
\begin{equation}\label{pl8}
[\mathbb{C}P^{8}, PL_{(2)}] \cong \mathbb{Z}_{2}\{(z)_{16}\} \oplus \mathbb{Z}_{2}\{\left(\overline{(\kappa)_{14} + \left(\overline{(\epsilon)_{8}}\right)_{14}}\right)_{16}\} \subset [\mathbb{C}P^{8}, F_{(2)}],
\end{equation}
where $(z)_{16} = 2\left(\overline{(\sigma^{2})_{14}}\right)_{16}$ for $z \in \{\eta^{*}, \eta^{*} + \eta \circ \rho\} \subset [\mathbb{S}^{16}, SF_{(2)}]$. This completes the proof of the theorem.
\end{proof}
\begin{corollary}\label{cor-cp8}
The kernel of the map
\[
p^{*}\colon[\mathbb{C}P^{8},SF]\cong \mathbb{Z}_{4}\{\left(\overline{(\sigma^{2})_{14}}\right)_{16}\}
     \oplus \mathbb{Z}_2\{\left(\overline{(\kappa)_{14}+\left(\overline{(\epsilon)_{8}}\right)_{14}}\right)_{16}\}\longrightarrow [\mathbb{S}^{17},SF]
\] 
induced by the Hopf fibration \(p\colon\mathbb{S}^{17}\to \mathbb{C}P^{8}\) is
\(\mathbb{Z}_{4}\{\left(\overline{(\sigma^{2})_{14}}\right)_{16}\}\).
\end{corollary}
\begin{proof}
It is sufficient to work \(2\)-locally. Consider the following commutative diagram:
\begin{equation}\label{commus17}
\begin{tikzcd}[column sep=small, row sep=large, font=\small]
{\bigl[\mathbb{C}P^{8}/\mathbb{C}P^{3},SF_{(2)}\bigr]\cong
\mathbb{Z}_{4}\{\left(\overline{(\sigma^{2})_{14}}\right)_{16}\} 
\oplus 
\mathbb{Z}_2\{\left(\overline{(\eta\circ \sigma)}\right)_{16}\}
\oplus
\mathbb{Z}_2\{\left(\overline{(\kappa)_{14}+\left(\overline{(\epsilon)}\right)_{14}}\right)_{16}\}}
\arrow[r, "{\varphi_{9}^{*}}"]
\arrow[d, "{q^{*}}"']
&
{\bigl[\mathbb{S}^{17},SF_{(2)}\bigr]}
\\
{\bigl[\mathbb{C}P^{8},SF_{(2)}\bigr]\cong 
\mathbb{Z}_{4}\{\left(\overline{(\sigma^{2})_{14}}\right)_{16}\}
\oplus 
\mathbb{Z}_2\{\left(\overline{(\kappa)_{14}+\left(\overline{(\epsilon_{8})}\right)_{14}}\right)_{16}\}}
\arrow[ru, "{p^{*}}"']
&
\end{tikzcd}
\end{equation}
where the isomorphisms are provided by Lemma~\ref{eight-sub}(ii) and Theorem~\ref{eight-sub2}(i). In view of this diagram and \eqref{ext1} and \eqref{ext3}, to prove that
\[
\operatorname{Ker}\left([\mathbb{C} P^{8}, SF_{(2)}] \xrightarrow{p^{*}} [\mathbb{S}^{17},SF_{(2)}]\right) \cong \mathbb{Z}_{4}\{\left(\overline{(\sigma^{2})_{14}}\right)_{16}\},
\]
it suffices to show that the map \(\varphi_{9}^{*}\colon [\mathbb{C}P^{8}/\mathbb{C}P^{3},SF_{(2)}] \to [\mathbb{S}^{17},SF_{(2)}]\) sends the element \(\left(\overline{(\sigma^{2})_{14}}\right)_{16}\) to zero and maps \(\left(\overline{(\kappa)_{14}+\left(\overline{(\epsilon)_{8}}\right)_{14}}\right)_{16}\) nontrivially. 

By Proposition~\ref{St}, the composite
\[
\mathbb{S}^{17}\xrightarrow{\;\varphi_{9}\;} \mathbb{C}P^{8}/\mathbb{C}P^{3}\xrightarrow{f_{\mathbb{C}P^{8}/\mathbb{C}P^{3}}} \mathbb{S}^{16}
\]
is null-homotopic. Hence, \(\varphi_{9}\colon \mathbb{S}^{17}\to \mathbb{C}P^{8}/\mathbb{C}P^{3}\) is homotopic to the composite
\[
\mathbb{S}^{17}\xrightarrow{(\lambda_{1},\lambda_{2})} \mathbb{C}P^{7}/\mathbb{C}P^{3} \simeq_{(2)} \mathbb{C}P^{6}/\mathbb{C}P^{3} \vee \mathbb{S}^{14} \hookrightarrow \mathbb{C}P^{8}/\mathbb{C}P^{3},
\]
where \(\lambda_{2}\colon \mathbb{S}^{17}\to \mathbb{S}^{14}\) is the attaching map of the \(18\)-cell onto the \(14\)-cell of \(\mathbb{C}P^{9}/\mathbb{C}P^{6}\). This map is represented by \(l\nu \in \pi_{17}(\mathbb{S}^{14}) \cong (\pi_{3}^{s})_{(2)} \cong \mathbb{Z}_{8}\{\nu\}\) with \(l\) odd, since \(Sq^{4}\) on \(H^{14}(\mathbb{C}P^{9}/\mathbb{C}P^{6};\mathbb{Z}_{2})\) is an isomorphism. Moreover, the composite
\[
\mathbb{S}^{17}\xrightarrow{\lambda_{1}} \mathbb{C}P^{6}/\mathbb{C}P^{3}\xrightarrow{f_{\mathbb{C}P^{6}/\mathbb{C}P^{3}}} \mathbb{S}^{12}
\] 
is null-homotopic, so \(\lambda_{1}\colon \mathbb{S}^{17}\to \mathbb{C}P^{6}/\mathbb{C}P^{3}\) is homotopic to
\[
\mathbb{S}^{17}\xrightarrow{(t_{1},t_{2})} \mathbb{C}P^{5}/\mathbb{C}P^{3} \simeq \mathbb{S}^{8} \vee \mathbb{S}^{10} \hookrightarrow \mathbb{C}P^{6}/\mathbb{C}P^{3}
\]
for some maps \(t_{1}\colon \mathbb{S}^{17}\to \mathbb{S}^{8}\) and \(t_{2}\colon \mathbb{S}^{17}\to \mathbb{S}^{10}\). These identifications yield the commutative diagram:
{\small
\begin{equation}\label{hopindu}
\xymatrix@C=20pt{
\bigl[\mathbb{C}P^{8}/\mathbb{C}P^{3},SF_{(2)}\bigr]
  \ar[r]^{\varphi_{9}^{*}}
  \ar[d]_{i^{*}}
&
\bigl[\mathbb{S}^{17},SF_{(2)}\bigr]\cong \mathbb{Z}_{8}\oplus \mathbb{Z}_{2}\oplus \mathbb{Z}_{2}
\\
\bigl[\mathbb{C}P^{7}/\mathbb{C}P^{3},SF_{(2)}\bigr]\cong
\bigl[\mathbb{C}P^{6}/\mathbb{C}P^{3},SF_{(2)}\bigr]\oplus \bigl[\mathbb{S}^{14},SF_{(2)}\bigr]
  \ar[ur]^{\lambda_{1}^{*}\oplus \lambda_{2}^{*}}
  \ar[dr]_{i^{*}\oplus (\mathrm{Id}_{\mathbb{S}^{14}})^{*}}
& \\
& \bigl[\mathbb{S}^{8},SF_{(2)}\bigr]\oplus \bigl[\mathbb{S}^{10},SF_{(2)}\bigr]\oplus \bigl[\mathbb{S}^{14},SF_{(2)}\bigr]
  \ar[uu]_{t_{1}^{*}\oplus t_{2}^{*}\oplus \lambda_{2}^{*}}
}
\end{equation}}
Under the splitting of \([\mathbb{C}P^{8}/\mathbb{C}P^{3}, SF_{(2)}]\) given above, the element \(\left(\overline{(\sigma^{2})_{14}}\right)_{16}\) is an extension of \(\sigma^{2}\) from the \(\mathbb{S}^{14}\)-summand via \(i^{*}\). Since \(\lambda_{2}^{*}\) is given by multiplication by \(l\nu\) and \(\nu \circ \sigma^{2} = 0\), the image of \(\sigma^{2}\) in \([\mathbb{S}^{17}, SF_{(2)}]\) under \(\lambda_{2}^{*}\) is zero. Consequently, \(\varphi_{9}^{*}\) maps the extension \(\left(\overline{(\sigma^{2})_{14}}\right)_{16}\) to zero. 

On the other hand, \(\left(\overline{(\kappa)_{14}+\left(\overline{(\epsilon)_{8}}\right)_{14}}\right)_{16}\) is an extension of \(\kappa+\epsilon \in [\mathbb{S}^{8},SF_{(2)}] \oplus [\mathbb{S}^{14},SF_{(2)}]\) under \(i^{*}\colon[\mathbb{C}P^{8}/\mathbb{C}P^{3},SF_{(2)}] \to [\mathbb{C}P^{7}/\mathbb{C}P^{3},SF_{(2)}]\). The map \(t_{1}\colon \mathbb{S}^{17}\to \mathbb{S}^{8}\) is stably represented by an element of \(\pi_{9}^{s} \cong \mathbb{Z}_{2}\{\nu^{3}\} \oplus \mathbb{Z}_{2}\{\eta \circ \epsilon\} \oplus \mathbb{Z}_{2}\{\mu\}\). Using $\nu^{3}\circ \epsilon=0$, $\eta\circ \epsilon\circ \epsilon=0$, and $\mu\circ \epsilon=\eta^{2}\circ \rho\in \pi_{17}^{s}$ (see \cite[Theorem~14.1]{Tod62}), and the fact that $SF_{(2)}$ is an infinite loop space, the map $t_{1}^{*}\colon [\mathbb{S}^{8},SF_{(2)}]\to [\mathbb{S}^{17},SF_{(2)}]$ cannot send $\epsilon$ to $\nu\circ \kappa \in \pi_{17}^{s}$. In contrast, the element \(\kappa \in [\mathbb{S}^{14},SF_{(2)}]\) maps to \(\nu \circ \kappa\) under \(\lambda_{2}^{*}\colon [\mathbb{S}^{14},SF_{(2)}] \to [\mathbb{S}^{17},SF_{(2)}]\). Therefore, the direct sum map
\[ 
t_{1}^{*} \oplus t_{2}^{*} \oplus \lambda_{2}^{*} \colon [\mathbb{S}^{8},SF_{(2)}] \oplus [\mathbb{S}^{10},SF_{(2)}] \oplus [\mathbb{S}^{14},SF_{(2)}] \to [\mathbb{S}^{17},SF_{(2)}]
\]
maps the sum \(\kappa+\epsilon\) to either \(\nu \circ \kappa\) or \(\nu \circ \kappa + \eta^{2} \circ \rho\), depending on whether \(t_{1}^{*}(\epsilon)\) is \(0\) or \(\eta^{2} \circ \rho\). Hence, by the commutativity of Diagram \eqref{hopindu}, it follows that \(\varphi_{9}^{*} \colon [\mathbb{C}P^{8}/\mathbb{C}P^{3},SF_{(2)}] \to [\mathbb{S}^{17},SF_{(2)}]\) sends \(\left(\overline{(\kappa)_{14}+\left(\overline{(\epsilon)_{8}}\right)_{14}}\right)_{16}\) nontrivially to either \(\nu \circ \kappa\) or \(\nu \circ \kappa + \eta^{2} \circ \rho\). In light of Diagram \eqref{commus17}, the map 
\begin{equation}\label{kers17}
p^{*} \colon [\mathbb{C}P^{8},SF_{(2)}] \to [\mathbb{S}^{17},SF_{(2)}]
\end{equation}
sends the generator \(\left(\overline{(\kappa)_{14}+\left(\overline{(\epsilon)_{8}}\right)_{14}}\right)_{16}\) nontrivially to either \(\nu \circ \kappa\) or \(\nu \circ \kappa + \eta^{2} \circ \rho\), and maps \(\left(\overline{(\sigma^{2})_{14}}\right)_{16}\) to zero. Thus, the kernel of \(p^{*} \colon [\mathbb{C}P^{8},SF_{(2)}] \to [\mathbb{S}^{17},SF_{(2)}]\) is precisely \(\mathbb{Z}_{4}\{\left(\overline{(\sigma^{2})_{14}}\right)_{16}\}\). This completes the proof.
\end{proof}
\begin{remark}\label{group}\rm
Observe from \cite{KM63} that the Kervaire--Milnor map
\[
\psi_{*}\colon \pi_{n}(\operatorname{Top}/O)\longrightarrow 
\mathrm{Coker}(J_{n})
= \pi_{n}^{s}/\operatorname{Im}(J_{n})
\]
is an isomorphism onto its image for $n=8,10,14,$ and $16$.
Under this identification, we regard a generator of $\pi_{n}(\operatorname{Top}/O)$
as being represented by a class $[x]\in\mathrm{Coker}(J_n)$ with
$x\in \pi_n^s$, and we simply write this generator as $x$.
From \cite[Table~A.3.3]{Rav03} (see, for example, \cite[Figure 2.1]{BHHM20}), we have
\[
\pi_{8}(\operatorname{Top}/O)\cong \mathbb{Z}_{2}\{\epsilon\},\qquad
\pi_{10}(\operatorname{Top}/O)\cong \mathbb{Z}_{2}\{\eta\circ \mu\}\oplus
\mathbb{Z}_{3}\{\beta_{1}\},
\]
\[
\pi_{14}(\operatorname{Top}/O)\cong \mathbb{Z}_{2}\{\kappa\},\qquad
\pi_{16}(\operatorname{Top}/O)\cong \mathbb{Z}_{2}\{\eta^{*}\}.
\]

We now express the groups
\[
\mathcal{C}(\mathbb{C}P^{m})\cong [\mathbb{C}P^{m},\operatorname{Top}/O],
\qquad 5\le m\le 8,
\]
using the notation introduced in Section~\ref{2}.
Recall that if $\alpha\in \pi_{2k}(\operatorname{Top}/O)$, then
\[
(\alpha)_{2k}\in [\mathbb{C}P^{k},\operatorname{Top}/O]
\]
denotes the element represented by the composition
\[
\mathbb{C}P^{k}
\xrightarrow{\, f_{\mathbb{C}P^{k}} \,}
\mathbb{S}^{2k}
\xrightarrow{\, \alpha \,}
\operatorname{Top}/O.
\]
For $m\ge k$, a chosen extension of $(\alpha)_{2k}$ along the restriction map
\[
i^{*}\colon [\mathbb{C}P^{m},\operatorname{Top}/O]
\longrightarrow [\mathbb{C}P^{k},\operatorname{Top}/O]
\]
is denoted by $\left (\overline{(\alpha)_{2k}}\right )_{2m}$.
\begin{itemize}
\item[(A)] Since the natural map
\(f^{*}_{\mathbb{C} P^{4}}\colon[\mathbb{S}^{8},\operatorname{Top}/O]\longrightarrow
[\mathbb{C}P^{4},\operatorname{Top}/O]\)
is an isomorphism, \(\Theta_8\cong \mathbb{Z}_2\{\epsilon\}\), and
\(\Theta_{10}\cong \mathbb{Z}_2\{\eta \circ \mu\}\oplus \mathbb{Z}_3\{\beta_1\}\),
it follows from Theorem~\ref{main}(i) that
\[
\mathcal{C}(\mathbb{C}P^{5})
=
\mathbb{Z}_2\{(\eta \circ \mu)_{10}\}
\oplus
\mathbb{Z}_3\{(\beta_1)_{10}\}
\oplus
\mathbb{Z}_2\{\left (\overline{(\epsilon)_{8}}\right )_{10}\},
\]
and \(\mathcal{C}(\mathbb{C}P^{4})=\mathbb{Z}_{2}\{(\epsilon)_{8}\}\). By \eqref{equ41} and \cite[Proposition~8.12]{Bru68}, we have that
\[
\mathbb{Z}_3\{(\beta_1)_{10}\}
\oplus
\mathbb{Z}_2\{\left (\overline{(\epsilon)_{8}}\right )_{10}\}
=\Ker\left([\mathbb{C} P^{5}, \operatorname{Top}/O]
\stackrel{p^{*}}{\longrightarrow}
\mathbb{Z}_{2}\{496\Sigma_{M}^{11}\}\right),
\]
where \(\Sigma_{M}^{11}\) is the Milnor generator of
\(\mathit{bP}_{12}=[\mathbb{S}^{11},\operatorname{Top}/O]\).

\item[(B)] It follows from Theorem~\ref{main}(ii) that the image of
\(\mathcal{C}(\mathbb{C}P^{6})\) in \(\mathcal{C}(\mathbb{C}P^{5})\)
is 
\[
\Ker\left([\mathbb{C} P^{5}, \operatorname{Top}/O]
\stackrel{p^{*}}{\longrightarrow}
\mathbb{Z}_{2}\{496\Sigma_{M}^{11}\}\right).
\]
By Part~(A), this kernel is
\[
\mathbb{Z}_3\{(\beta_1)_{10}\}
\oplus
\mathbb{Z}_2\{\left (\overline{(\epsilon)_{8}}\right )_{10}\}.
\]
Hence,
\[
\mathcal{C}(\mathbb{C}P^{6})
=
\mathbb{Z}_3\{\left(\overline{(\beta_1)_{10}}\right )_{12}\}
\oplus
\mathbb{Z}_2\{\left (\overline{(\epsilon)_{8}}\right )_{12}\}.
\]
\item[(C)] Since \(\Theta_{14} \cong \mathbb{Z}_2\{\kappa\}\), it follows from Theorem~\ref{main}(iii) and Part~(B) that
\begin{equation}\label{cp7-top}
\mathcal{C}(\mathbb{C}P^{7}) = \mathbb{Z}_2\{(\kappa)_{14}\} \oplus \mathbb{Z}_2\{\left(\overline{(\epsilon)_{8}}\right)_{14}\},
\end{equation}
where \(\left(\overline{(\epsilon)_{8}}\right)_{14}\) is an extension of \(\left(\overline{(\epsilon)_{8}}\right)_{12}\), and
\[
\Ker\left([\mathbb{C} P^{6}, \operatorname{Top}/O] \xrightarrow{p^{*}} \mathbb{Z}_{3} \right) = \mathbb{Z}_2\{\left(\overline{(\epsilon)_{8}}\right)_{12}\}.
\]
\item[(D)] Note from the proof of Theorem~\ref{eight}(ii) and  \eqref{kernelcp7-s15} that the map
\[
p^{*}\colon [\mathbb{C}P^{7}, \operatorname{Top}/O] \longrightarrow [\mathbb{S}^{15}, \operatorname{Top}/O]
\]
sends the element \(f^{*}_{\mathbb{C}P^{7}}(\Sigma^{14})=(\kappa)_{14}\) to \(\eta\circ \kappa\), where \(\Sigma^{14}\) is the nontrivial element of \(\Theta_{14}\) corresponding to the class \(\kappa\) in \(\operatorname{Coker}(J_{14})\). Furthermore, 
\[
\operatorname{Im} \left([\mathbb{C}P^{7}, \operatorname{Top}/O] \xrightarrow{p^{*}} [\mathbb{S}^{15}, \operatorname{Top}/O]\right) \cong \mathbb{Z}_{2}\{\eta \circ \kappa\}.
\] 
By Part~(C), we have the splitting
\[
[\mathbb{C}P^{7}, \operatorname{Top}/O] = \mathbb{Z}_{2}\{(\kappa)_{14}\} \oplus \mathbb{Z}_{2}\{\left(\overline{(\epsilon)_{8}}\right)_{14}\}.
\]
Hence, the kernel of \(p^{*}\colon [\mathbb{C}P^{7}, \operatorname{Top}/O] \longrightarrow [\mathbb{S}^{15}, \operatorname{Top}/O]\) is generated by either \((\kappa)_{14}+\left(\overline{(\epsilon)_{8}}\right)_{14}\) or \(\left(\overline{(\epsilon)_{8}}\right)_{14}\).
Using \eqref{kernelofp1} and \eqref{cp7-hop}, we see that the map
\[
p^{*}\colon [\mathbb{C}P^{7},SF] \cong \mathbb{Z}_{2}\{\left(\overline{(\epsilon)_{8}}\right)_{14}\} \oplus \mathbb{Z}_{2}\{(\kappa)_{14}\} \oplus \mathbb{Z}_{2}\{(\sigma^{2})_{14}\} \longrightarrow [\mathbb{S}^{15},SF]
\]
sends the element \(\left(\overline{(\epsilon)_{8}}\right)_{14} \in [\mathbb{C}P^{7},SF]\) to \(\eta\circ \kappa \in [\mathbb{S}^{15},SF]\). Consequently, under the composite
\[
[\mathbb{C}P^{7},SF] \xrightarrow{p^{*}} [\mathbb{S}^{15},SF] \xrightarrow{\phi_{*}} \operatorname{Coker}(J_{15}) \cong [\mathbb{S}^{15},F/O],
\]
the element \(\left(\overline{(\epsilon)_{8}}\right)_{14}\) is mapped to the nontrivial class \(\eta\circ \kappa \in \operatorname{Coker}(J_{15})\). Using this fact, together with the commutativity of Diagram~\eqref{digram4}, and arguing as in the proof of Theorem~\ref{eight}(ii), we conclude that the map
\[
p^{*}\colon [\mathbb{C}P^{7},\operatorname{Top}/O] \longrightarrow \operatorname{Coker}(J_{15}) \subset [\mathbb{S}^{15},\operatorname{Top}/O]
\]
sends \(\left(\overline{(\epsilon)_{8}}\right)_{14}\) nontrivially.

Therefore, \((\kappa)_{14}+\left(\overline{(\epsilon)_{8}}\right)_{14}\) is the unique generator of the kernel of \(p^{*}\colon [\mathbb{C}P^{7},\operatorname{Top}/O] \longrightarrow [\mathbb{S}^{15},\operatorname{Top}/O]\). Applying this to Theorem~\ref{eight-sub2}(ii), we obtain
\[
[\mathbb{C}P^{8}, \operatorname{Top}/O] = \mathbb{Z}_{2}\{(\eta^{*})_{16}\} \oplus \mathbb{Z}_{2}\{\left(\overline{(\kappa)_{14}+\left(\overline{(\epsilon)_{8}}\right)_{14}}\right)_{16}\},
\]
where \(\left(\overline{(\kappa)_{14}+\left(\overline{(\epsilon)_{8}}\right)_{14}}\right)_{16}\) denotes an extension of \((\kappa)_{14}+\left(\overline{(\epsilon)_{8}}\right)_{14}\) along the induced map
\[
i^{*}\colon [\mathbb{C}P^{8}, \operatorname{Top}/O] \to [\mathbb{C}P^{7}, \operatorname{Top}/O].
\]
\end{itemize}
\end{remark}
\begin{lemma}\label{cor-cp8pl}
The kernel of the map
\[
p^{*}\colon[\mathbb{C}P^{8},\operatorname{Top}/O] \longrightarrow [\mathbb{S}^{17},\operatorname{Top}/O]
\] induced by the Hopf fibration \(p\colon\mathbb{S}^{17}\to \mathbb{C}P^{8}\) is
\(\mathbb{Z}_{2}\{(\eta^{*})_{16}\}\).
\end{lemma}
\begin{proof}
It follows from \eqref{equ21} and Remark \ref{group}(D) that there is an isomorphism
\[
[\mathbb{C}P^{8},PL/O] \cong [\mathbb{C}P^{8},\operatorname{Top}/O] \cong \mathbb{Z}_{2}\{(\eta^{*})_{16}\} \oplus \mathbb{Z}_2\{\left(\overline{(\kappa)_{14}+\left(\overline{(\epsilon)_{8}}\right)_{14}}\right)_{16}\}.
\] 
Combining this with the fact that $[\mathbb{S}^{17},PL/O] \cong [\mathbb{S}^{17},\operatorname{Top}/O]$, it suffices to show that the kernel of the induced map 
\[
p^{*}: [\mathbb{C}P^{8},PL/O_{(2)}] \to [\mathbb{S}^{17},PL/O_{(2)}]
\]
is precisely the summand $\mathbb{Z}_{2}\{(\eta^{*})_{16}\}$. Consider the following commutative square:
\begin{equation}\label{cp8-pl1}
\xymatrix@C=15pt@R=35pt{
[\mathbb{C}P^{8}, PL_{(2)}] \ar[rr]^-{J_{PL_{(2)}}} \ar[d]_-{p^{*}} && 
{ \begin{array}{c} [\mathbb{C}P^{8}, F_{(2)}] \cong \mathbb{Z}_{4}\{\left(\overline{(\sigma^{2})_{14}}\right)_{16}\} \\ \oplus\, \mathbb{Z}_2\{\left(\overline{(\kappa)_{14} + \left(\overline{(\epsilon)_{8}}\right)_{14}}\right)_{16}\} \end{array} } \ar[d]^-{p^{*}} \\
[\mathbb{S}^{17}, PL_{(2)}] \cong \mathbb{Z}_{2} \oplus [\mathbb{S}^{17}, F_{(2)}] \ar[rr]^-{J_{PL_{(2)}}} && [\mathbb{S}^{17}, F_{(2)}],
}
\end{equation}
where the group 
\[
[\mathbb{C}P^{8}, PL_{(2)}] \cong \mathbb{Z}_{2}\{2\left(\overline{(\sigma^{2})_{14}}\right)_{16}\} \oplus \mathbb{Z}_{2}\{\left(\overline{(\kappa)_{14} + \left(\overline{(\epsilon)_{8}}\right)_{14}}\right)_{16}\}
\] 
is given by \eqref{pl8}. Under this identification, the top horizontal map \( J_{PL_{(2)}} \) can be identified as an inclusion map. Furthermore, the bottom $J$-homomorphism is identified with the projection map, as established in \cite[Diagram B, p.~306]{Bru69}. It follows from Corollary \ref{cor-cp8} and \eqref{kers17} that the second vertical map $p^{*}: [\mathbb{C}P^{8}, F_{(2)}] \to [\mathbb{S}^{17}, SF_{(2)}]$ sends the element $\left(\overline{(\kappa)_{14} + (\overline{(\epsilon)_{8}})_{14}}\right)_{16}$ nontrivially to either $\nu \circ \kappa$ or $\nu \circ \kappa + \eta^{2} \circ \rho$. The commutativity of Diagram \eqref{cp8-pl1} then implies that the composition $J_{PL_{(2)}} \circ p^{*}:[\mathbb{C}P^{8}, PL_{(2)}] \to [\mathbb{S}^{17}, SF_{(2)}]$ sends the element $\left(\overline{(\kappa)_{14} + (\overline{(\epsilon)_{8}})_{14}}\right)_{16}$ to either $\nu \circ \kappa$ or $\nu \circ \kappa + \eta^{2} \circ \rho$. Applying these observations to the following commutative diagram:
\begin{equation}\label{cp8-plo1}
\resizebox{\textwidth}{!}{%
\xymatrix@C=2em@R=3.2em{
[\mathbb{C}P^{8},PL_{(2)}] \ar[rr]^-{\beta_{*}} \ar[d]_-{p^{*}} &&
\bigl[\mathbb{C}P^{8},PL/O_{(2)}\bigr]\cong \mathbb{Z}_{2}\{(\eta^{*})_{16}\} \oplus \mathbb{Z}_2\{\left(\overline{(\kappa)_{14}+\left(\overline{(\epsilon)_{8}}\right)_{14}}\right)_{16}\} \ar[d]^-{p^{*}} \\
[\mathbb{S}^{17},PL_{(2)}]\cong \mathbb{Z}_{2} \oplus [\mathbb{S}^{17}, SF_{(2)}] \ar[rr]^-{\beta_{*}} && [\mathbb{S}^{17},PL/O_{(2)}]\cong \mathbb{Z}_{2} \oplus [\mathbb{S}^{17}, SF_{(2)}]/\operatorname{Im}\,J,
}
} 
\end{equation}
the top horizontal map is an isomorphism by \eqref{equ22}, and the bottom map $\beta_{*}$ is identified as $\text{Id}_{\mathbb{Z}_{2}} \oplus \chi$ following \cite[Diagram B, p.~306]{Bru69}. Here, $\chi: \pi_{17}^s \to \pi_{17}^s/\operatorname{Im}\,J$ is the quotient map and $\operatorname{Im}\,J \cong \mathbb{Z}_{2}\{\eta^2 \circ \rho\}$. We thus find that $p^{*}: [\mathbb{C}P^{8}, PL/O_{(2)}] \to [\mathbb{S}^{17}, PL/O_{(2)}]$ sends $\left(\overline{(\kappa)_{14} + (\overline{(\epsilon)_{8}})_{14}}\right)_{16}$ to the element whose component in the quotient is $\nu \circ \kappa \in \pi_{17}^s/\operatorname{Im}\,J$. This element is nontrivial, and hence 
\begin{equation}\label{nontrivialge}
p^{*}\left(\left(\overline{(\kappa)_{14} + (\overline{(\epsilon)_{8}})_{14}}\right)_{16}\right) \neq 0 
\end{equation}
On the other hand, since the composition $\mathbb{S}^{17} \xrightarrow{p} \mathbb{C}P^8 \xrightarrow{f_{\mathbb{C}P^8}} \mathbb{S}^{16}$ is null-homotopic, the map $p^{*}$ sends the element $(\eta^{*})_{16} = f_{\mathbb{C}P^{8}}^{*}(\eta^{*})$ to zero. Therefore, the kernel of the induced map 
$p^{*}: [\mathbb{C}P^{8},PL/O_{(2)}] \to [\mathbb{S}^{17},PL/O_{(2)}]$ 
is precisely $\mathbb{Z}_{2}\{(\eta^{*})_{16}\}$, which completes the proof.
\end{proof}
\section{Diffeomorphism Classification of Manifolds Homeomorphic to \texorpdfstring{$\mathbb{C}P^{m}$ ($5 \le m \le 8$)}{}}\label{3}

In this section, we classify, up to (unoriented) diffeomorphism and orientation-preserving diffeomorphism, all closed smooth manifolds homeomorphic to $\mathbb{C}P^{m}$ for $5\le m\le 8$, using the computations from the previous section.

For a smooth manifold $M^{m}$ $(m\ge 5)$, the group of concordance classes of self-homeomorphisms (resp.\ orientation-preserving self-homeomorphisms) of $M$ is denoted by $MCG^{\pm}(M)$ (resp.\ $MCG^{+}(M)$). Let $\mathcal{M}^{+}_{\mathrm{Diff}}(M)$ denote the set of oriented diffeomorphism classes of smooth manifolds homeomorphic to $M$.

There is a natural forgetful map
\[
F_{\mathcal{C}} \colon \mathcal{C}(M) \longrightarrow \mathcal{M}^{+}_{\mathrm{Diff}}(M), 
\qquad [f \colon N \to M] \longmapsto [N],
\]
which descends to a bijection
\begin{equation}\label{forget}
\mathcal{C}(M)/MCG^{+}(M)\xrightarrow{\;\cong\;} \mathcal{M}^{+}_{\mathrm{Diff}}(M),
\end{equation}
where the group $MCG^{+}(M)$ acts on $\mathcal{C}(M)$ via post-composition:
\begin{equation}\label{act}
MCG^{+}(M)\times \mathcal{C}(M)\longrightarrow \mathcal{C}(M), 
\qquad \bigl([h],[N,f]\bigr)\longmapsto [N,h\circ f].
\end{equation}
Next, we recall a homotopy-theoretic reformulation of the action \eqref{act}.
Let $[M, \mathrm{Id}]$ denote the identity element in the group $\mathcal{C}(M)$ of concordance classes of smoothings. Given a smoothing $(N, f)$ representing the concordance class $[N, f] \in \mathcal{C}(M)$ and a self-homeomorphism $[h] \in MCG^{\pm}(M)$, the smoothing $(N, h \circ f)$ determines a new concordance class given by
\begin{equation}\label{for}
[N, h \circ f] = [M, h] + (h^{-1})^{*}([N, f]).
\end{equation}

Here we identify $\mathcal{C}(M) \cong [M, \operatorname{Top}/O]$. Under this identification, the map
\[
(h^{-1})^{*} : [M, \operatorname{Top}/O] \longrightarrow [M, \operatorname{Top}/O]
\]
is induced by the self-map $h^{-1} : M \to M$, and $(h^{-1})^{*}([N, f])$ denotes the pullback of the concordance class $[N, f]$. The operation $+$ denotes the abelian group structure on $[M, \operatorname{Top}/O]$ induced by the Whitney sum in $\operatorname{Top}/O$ (see \cite[pp.~143--144]{Sch71} or Proposition~3.10 of \cite{Cro10}, with $G/O$ replaced by $\operatorname{Top}/O$). We now compute the action \eqref{for} for $M=\mathbb{C}P^{m}$. By \cite[Theorem~8]{Sul67}, for $m>2$,
\begin{equation}\label{iso1}
MCG^{\pm}(\mathbb{C}P^{m})=\{[Id],[C]\}\cong \mathbb{Z}_{2},
\end{equation}
and
\begin{equation}\label{iso2}
MCG^{+}(\mathbb{C}P^{m})=
\begin{cases}
\{[Id],[C]\}, & \text{if $m$ is even},\\
\{[Id]\}, & \text{if $m$ is odd},
\end{cases}
\end{equation}
where $[Id]$ and $[C]$ denote the concordance classes of the identity map
$Id\colon \mathbb{C}P^{m}\to \mathbb{C}P^{m}$ and the complex conjugation
\[
C \colon \mathbb{C}P^{m} \to \mathbb{C}P^{m}, \qquad [z_{0}, z_{1}, \dots, z_{m}] \longmapsto [\overline{z}_{0}, \overline{z}_{1}, \dots, \overline{z}_{m}],
\]
respectively.

Since the conjugation map $C$ is a diffeomorphism, the smooth structures
$(\mathbb{C}P^{m},C)$ and $(\mathbb{C}P^{m},Id)$ are concordant. Therefore, by
\eqref{for}, the action of $[C]\in MCG^{+}(\mathbb{C}P^{m})$ on an arbitrary smoothing
$[N,f]\in [\mathbb{C}P^{m},\operatorname{Top}/O]$ is given by
\begin{equation}\label{act1}
[N,C\circ f] = C^{*}([N,f]).
\end{equation}
Hence, to compute $\mathcal{M}^{+}_{\mathrm{Diff}}(\mathbb{C}P^{m})$, it suffices to
determine the map
\[
C^{*}\colon [\mathbb{C}P^{m},\operatorname{Top}/O]\longrightarrow [\mathbb{C}P^{m},\operatorname{Top}/O].
\]
\begin{remark}\label{actions}
Let \(X=\operatorname{Top}/O\) (or \(X=SF\)). Recall that there is a natural map
\[
f^{*}_{\mathbb{C}P^{m}}\colon [\mathbb{S}^{2m}, X]
\longrightarrow
[\mathbb{C}P^{m}, X],
\]
where \(f_{\mathbb{C}P^{m}}\colon \mathbb{C}P^{m}\to \mathbb{S}^{2m}\) is the degree-one map,
and a map
\[
p^{*}\colon [\mathbb{C}P^{m}, X]
\longrightarrow
[\mathbb{S}^{2m+1}, X]
\]
induced by the Hopf fibration \(p\colon \mathbb{S}^{2m+1}\to \mathbb{C}P^{m}\).
We record the following observations.
\begin{itemize}
\item[(1)]
If \(C\) has degree \(1\), then the maps
\(f_{\mathbb{C}P^{m}}\circ C\) and \(f_{\mathbb{C}P^{m}}\) are homotopic. Hence
\[
(f_{\mathbb{C}P^{m}}\circ C)^{*}
= C^{*}\circ f^{*}_{\mathbb{C}P^{m}}
= f^{*}_{\mathbb{C}P^{m}},
\]
and therefore every element in the image of \(f^{*}_{\mathbb{C}P^{m}}\)
is fixed by \(C^{*}\).

\item[(2)]
If \(C\) has degree \(-1\), then \(f_{\mathbb{C}P^{m}}\circ C\) is homotopic to
\(\rho\circ f_{\mathbb{C}P^{m}}\), where
\(\rho\colon \mathbb{S}^{2m}\to \mathbb{S}^{2m}\) is a reflection. Consequently,
\(C^{*}\) sends \(f^{*}_{\mathbb{C}P^{m}}(\Sigma)\) to
\(-f^{*}_{\mathbb{C}P^{m}}(\Sigma)\).

\item[(3)]
Since the map \(C\) extends to \(\mathbb{C}P^{m+1}\) for all \(m\), the maps
\(C\circ p\) and \(p\) are homotopic. Consequently,
\[
p^{*}\circ C^{*} = p^{*}.
\]
Since
\[
C^{*}\colon [\mathbb{C}P^{m}, X]
\longrightarrow
[\mathbb{C}P^{m}, X]
\]
is an isomorphism, it follows that
\[
C^{*}\bigl(\operatorname{Ker}(p^{*})\bigr)
=
\operatorname{Ker}(p^{*}),
\]
where
\[
p^{*}\colon [\mathbb{C}P^{m}, X]
\longrightarrow
[\mathbb{S}^{2m+1}, X].
\]
\end{itemize}
\end{remark}
Since $[\mathbb{C}P^{3},\operatorname{Top}/O]=\{[\mathbb{C}P^{3}]\}$ and
$[\mathbb{C}P^{4},\operatorname{Top}/O]=\{[\mathbb{C}P^{4}\#\Sigma]\mid \Sigma\in\Theta_{8}\}
\cong \mathbb{Z}_{2}$ (see \cite[Theorem~2.3]{Ram16}), it follows that
$C^{*}$ is the identity map for $m=3,4$. For $5\leq m\leq 8$, we have the following.
\begin{theorem}\label{main2}
Let $C\colon \mathbb{C}P^{m}\to \mathbb{C}P^{m}$ be the complex conjugation map.
\begin{itemize}

\item[(i)]
The induced map
\[
C^{*}\colon 
\mathbb{Z}_{2}\{(\eta\circ\mu)_{10}\}
\oplus
\mathbb{Z}_{3}\{(\beta_{1})_{10}\}
\oplus
\mathbb{Z}_{2}\{\left(\overline{(\epsilon)_{8}}\right)_{10}\}
\longrightarrow
\mathbb{Z}_{2}\{(\eta\circ\mu)_{10}\}
\oplus
\mathbb{Z}_{3}\{(\beta_{1})_{10}\}
\oplus
\mathbb{Z}_{2}\{\left(\overline{(\epsilon)_{8}}\right)_{10}\},
\]
where
\[
[\mathbb{C}P^{5},\operatorname{Top}/O]
=
\mathbb{Z}_{2}\{(\eta\circ\mu)_{10}\}
\oplus
\mathbb{Z}_{3}\{(\beta_{1})_{10}\}
\oplus
\mathbb{Z}_{2}\{\left(\overline{(\epsilon)_{8}}\right)_{10}\},
\]
is given in coordinates by
\[
C^{*}(x,y,z)=(x,-y,z).
\]

\item[(ii)]
The induced map
\[
C^{*}\colon
\mathbb{Z}_{3}\{\left(\overline{(\beta_{1})_{10}}\right)_{12}\}
\oplus
\mathbb{Z}_{2}\{\left(\overline{(\epsilon)_{8}}\right)_{12}\}
\longrightarrow
\mathbb{Z}_{3}\{\left(\overline{(\beta_{1})_{10}}\right)_{12}\}
\oplus
\mathbb{Z}_{2}\{\left(\overline{(\epsilon)_{8}}\right)_{12}\},
\]
where
\[
[\mathbb{C}P^{6},\operatorname{Top}/O]
=
\mathbb{Z}_{3}\{\left(\overline{(\beta_{1})_{10}}\right)_{12}\}
\oplus
\mathbb{Z}_{2}\{\left(\overline{(\epsilon)_{8}}\right)_{12}\},
\]
is given by
\[
C^{*}(x,y)=(-x,y).
\]

\item[(iii)]
The induced map
\[
C^{*}\colon
\mathbb{Z}_{2}\{(\kappa)_{14}\}
\oplus
\mathbb{Z}_{2}\{\left(\overline{(\epsilon)_{8}}\right)_{14}\}
\longrightarrow
\mathbb{Z}_{2}\{(\kappa)_{14}\}
\oplus
\mathbb{Z}_{2}\{\left(\overline{(\epsilon)_{8}}\right)_{14}\},
\]
where 
\[
[\mathbb{C}P^{7},\operatorname{Top}/O]
=
\mathbb{Z}_{2}\{(\kappa)_{14}\}
\oplus
\mathbb{Z}_{2}\{\left(\overline{(\epsilon)_{8}}\right)_{14}\},
\]
is the identity map.
\end{itemize}
\end{theorem}
\begin{proof}
\noindent
(i)
By Remark~\ref{group}(A) and (B),
\[
[\mathbb{C}P^{5},\operatorname{Top}/O]
=
\mathbb{Z}_{2}\{(\eta\circ\mu)_{10}\}
\oplus
\mathbb{Z}_{3}\{(\beta_{1})_{10}\}
\oplus
\mathbb{Z}_{2}\{\left(\overline{(\epsilon)_{8}}\right)_{10}\},
\]
where $(\eta\circ\mu)_{10}$ and $(\beta_{1})_{10}$ lie in the image of
$f^{*}_{\mathbb{C}P^{5}}$, and
$\left(\overline{(\epsilon)_{8}}\right)_{10}$ lies in $$\Ker\left([\mathbb{C} P^{5}, \operatorname{Top}/O]
\stackrel{p^{*}}{\longrightarrow}
\mathbb{Z}_{2}\{496\Sigma_{M}^{11}\}\right).$$
By Remark~\ref{actions}(1)--(3),
$(\eta\circ\mu)_{10}$ and $\left(\overline{(\epsilon)_{8}}\right)_{10}$
are fixed by $C^{*}$, while
$(\beta_{1})_{10}$ is sent to its inverse.
This yields $C^{*}(x,y,z)=(x,-y,z)$.

\medskip
\noindent
(ii)
The conclusion follows from the commutative diagram
\[
\begin{CD}
0 @>>> [\mathbb{C}P^{6},\operatorname{Top}/O] @>i^{*}>> [\mathbb{C}P^{5},\operatorname{Top}/O] \\
@. @VV C^{*} V @VV C^{*} V \\
0 @>>> [\mathbb{C}P^{6},\operatorname{Top}/O] @>i^{*}>> [\mathbb{C}P^{5},\operatorname{Top}/O],
\end{CD}
\]
together with Part~(i) and Remark~\ref{group}(B).

\medskip
\noindent
(iii) By Remark~\ref{group}(C) and (D), \((\kappa)_{14}\) lies in the image of \(f^{*}_{\mathbb{C}P^{7}}\), and the sum \((\kappa)_{14} + \left(\overline{(\epsilon)_{8}}\right)_{14}\) is the unique generator of 
\[
\Ker\left([\mathbb{C}P^{7}, \operatorname{Top}/O] \xrightarrow{p^{*}} [\mathbb{S}^{15}, \operatorname{Top}/O] \right).
\]
Consequently, the result follows from Remark~\ref{actions}(1) and (3).
This completes the proof of the theorem.
\end{proof}
For analyzing the map $\widetilde{C} : [\mathbb{C}P^8, \operatorname{Top}/O] \to [\mathbb{C}P^8, \operatorname{Top}/O]$, we use the following observations and lemmas. Recall from Madsen and Milgram \cite[Theorem 5.18]{MM79} that there is a fibration
\[
\mathrm{cok}\, J_{(2)} \longrightarrow SF_{(2)} \longrightarrow \operatorname{Im}\, J_{(2)}
\]
and a homotopy equivalence $SF_{(2)} \cong \mathrm{cok}\, J_{(2)} \times \operatorname{Im}\, J_{(2)}$, where $\mathrm{cok}\, J_{(2)}$ is $5$-connected. 

In the following lemma, we use the cell structures of the stunted projective spaces $\mathbb{R}P^n/\mathbb{R}P^{n-2}$ and $\mathbb{R}P^n/\mathbb{R}P^{n-3}$ as given in \cite[p.~5]{GMS12}. Denote by $M(\mathbb{Z}_2, n) \simeq \Sigma^{n-1} \mathbb{R}P^2$ the Moore space of type $(\mathbb{Z}_2, n)$ for the cyclic group $\mathbb{Z}_2$ of order two, with $n \geq 2$. Then we have 
\[ 
[\mathbb{S}^{n+2}, M(\mathbb{Z}_2, n)] = \mathbb{Z}_4\{\widetilde{\eta}\} \quad \text{and} \quad [\mathbb{S}^{n+1}, M(\mathbb{Z}_2, n)] = \mathbb{Z}_2\{i \circ \eta\} 
\]
for $n \geq 3$, where the generator $\widetilde{\eta}$ is the map $\mathbb{S}^{n+2} \to M(\mathbb{Z}_2, n)$ such that the composition 
\[
\mathbb{S}^{n+2} \xrightarrow{\quad \widetilde{\eta} \quad} M(\mathbb{Z}_2, n) \xrightarrow{\quad q \quad} \mathbb{S}^{n+1}
\]
is represented by $\eta \in \pi_{n+2}(\mathbb{S}^{n+1}) \cong \pi_1^S$, and $i \circ \eta : \mathbb{S}^{n+1} \to M(\mathbb{Z}_2, n)$ is the composition 
\[
\mathbb{S}^{n+1} \xrightarrow{\quad \eta \quad} \mathbb{S}^n \xrightarrow{\quad i \quad} M(\mathbb{Z}_2, n).
\]
\begin{lemma}\label{teccp8}
Let 
\[ q_{n,k}: \mathbb{R}P^{n}/\mathbb{R}P^{k} \to \mathbb{R}P^{n}/\mathbb{R}P^{k+1} \] 
be the quotient map, and let 
\[ i: \mathbb{R}P^{k+1}/\mathbb{R}P^k \hookrightarrow \mathbb{R}P^n/\mathbb{R}P^k \] 
denote the inclusion map. Then:
\begin{enumerate} 
    \item Suppose $k+2\leq n$ with $k+3 \equiv 0 \pmod 4$.
    \begin{enumerate}
        \item The induced map 
        \[ i^*: [ \Sigma^2 \mathbb{R}P^n/\mathbb{R}P^k, G/\mathrm{Top}_{(2)} ] \to [ \Sigma^2 \mathbb{R}P^{k+1}/\mathbb{R}P^k, G/\mathrm{Top}_{(2)} ] \]
        is surjective.
        \item Let $k \in \{1, 5, 9\}$ and let $L \colon \mathbb{S}^{15} \to \mathbb{R}P^n/\mathbb{R}P^k$ be any map. For any 
\[ x \in [\Sigma \mathbb{R}P^{k+1}/\mathbb{R}P^k, \operatorname{Top}/O_{(2)}] \] 
and any extension $\overline{x} \in [\Sigma \mathbb{R}P^n/\mathbb{R}P^k, \operatorname{Top}/O_{(2)}]$ of $x$ along the map 
\[ i^* \colon [ \Sigma \mathbb{R}P^n/\mathbb{R}P^k, \operatorname{Top}/O_{(2)} ] \to [ \Sigma \mathbb{R}P^{k+1}/\mathbb{R}P^k, \operatorname{Top}/O_{(2)} ], \] 
there exists an element $y \in [ \Sigma \mathbb{R}P^n/\mathbb{R}P^k, \operatorname{Top}/O_{(2)} ]$ such that $i^*(y) = i^*(\overline{x}) = x$, and the induced map 
\[ (\Sigma L)^* \colon [ \Sigma \mathbb{R}P^n/\mathbb{R}P^k, \operatorname{Top}/O_{(2)} ] \to [ \mathbb{S}^{16}, \operatorname{Top}/O_{(2)} ] \] 
sends $y$ to zero.
    \end{enumerate}
    \item For $n \geq 9$, the map 
    \[ (\Sigma q_{n,6})^* : [ \Sigma \mathbb{R}P^n/\mathbb{R}P^{7}, \operatorname{Top}/O_{(2)} ] \longrightarrow [ \Sigma \mathbb{R}P^n/\mathbb{R}P^{6}, \operatorname{Top}/O_{(2)} ] \] 
    is surjective.
   \item Let $\overline{\pi}_{n,t}: \mathbb{R}P^{2n+1}/\mathbb{R}P^{2t} \to \mathbb{C}P^{n}/\mathbb{C}P^{t}$ be the map induced by the canonical projection $\pi: \mathbb{R}P^{2n+1} \to \mathbb{C}P^{n}$. Then
    \begin{enumerate}
        \item The map $(\Sigma \overline{\pi}_{4,3})^{*}: [ \Sigma \mathbb{C}P^4/\mathbb{C}P^3, \operatorname{Top}/O_{(2)} ] \to [ \Sigma \mathbb{R}P^9/\mathbb{R}P^6, \operatorname{Top}/O_{(2)} ]$ is an isomorphism.
        \item Let 
        \[ i^*: [ \Sigma \mathbb{R}P^{15}/\mathbb{R}P^6, \operatorname{Top}/O_{(2)} ] \to [ \Sigma \mathbb{R}P^{9}/\mathbb{R}P^6, \operatorname{Top}/O_{(2)} ] \] 
        be the map induced by the inclusion $i: \mathbb{R}P^{9}/\mathbb{R}P^6 \to \mathbb{R}P^{15}/\mathbb{R}P^6$. For any $x \in [\Sigma \mathbb{R}P^{9}/\mathbb{R}P^6, \operatorname{Top}/O_{(2)}]$, there exists an element $y$ in the image of 
        \[ (\Sigma \overline{\pi}_{7,3})^{*}: [\Sigma \mathbb{C}P^7/\mathbb{C}P^3, \operatorname{Top}/O_{(2)}] \to [\Sigma \mathbb{R}P^{15}/\mathbb{R}P^6, \operatorname{Top}/O_{(2)}] \] 
        such that $i^{*}(y) = x$ in $[\Sigma \mathbb{R}P^{9}/\mathbb{R}P^6, \operatorname{Top}/O_{(2)}]$.
    \end{enumerate}
\end{enumerate}
\end{lemma}
\begin{proof}
It follows from \cite{Sul67} that the $2$-local homotopy type of $G/\mathrm{Top}$ splits as a product of Eilenberg--MacLane spaces
\[ G/\mathrm{Top}_{(2)} \simeq \prod_{i \ge 1} K(\mathbb{Z}_{(2)}, 4i) \times \prod_{i \ge 1} K(\mathbb{Z}_2, 4i-2), \]
where the homotopy groups are given by
\[ \pi_n(G/\mathrm{Top}_{(2)}) \cong \begin{cases} \mathbb{Z}_{(2)} & n \equiv 0 \pmod 4 \\ \mathbb{Z}_2 & n \equiv 2 \pmod 4 \\ 0 & \text{otherwise.} \end{cases} \]
Therefore, for any CW complex $X$, there is an isomorphism:
\[ [ X, G/\mathrm{Top}_{(2)} ] \xrightarrow{\cong} \prod_{i \ge 1} H^{4i}(X; \mathbb{Z}_{(2)}) \times \prod_{i \ge 1} H^{4i-2}(X; \mathbb{Z}_2). \]
Now we consider the following commutative diagram 
\begin{equation}
\xymatrix@C=2pc{
[ \Sigma^2 \mathbb{R}P^n/\mathbb{R}P^k, G/\mathrm{Top}_{(2)} ]
\ar[r]^-{\cong} \ar[d]_{i^*}
&
\displaystyle \prod_{i \ge 1} H^{4i}(\Sigma^2 \mathbb{R}P^n/\mathbb{R}P^k; \mathbb{Z}_{(2)})
\times
\prod_{i \ge 1} H^{4i-2}(\Sigma^2 \mathbb{R}P^n/\mathbb{R}P^k; \mathbb{Z}_{2})
\ar[d]^{i_1^* \oplus i_2^*}
\\
[ \Sigma^2 \mathbb{S}^{k+1}, G/\mathrm{Top}_{(2)} ]
\ar[r]^-{\cong}
&
\displaystyle \prod_{i \ge 1} H^{4i}(\Sigma^2 \mathbb{S}^{k+1}; \mathbb{Z}_{(2)})
\times
\prod_{i \ge 1} H^{4i-2}(\Sigma^2 \mathbb{S}^{k+1}; \mathbb{Z}_{2})
}
\end{equation}
where $i_1^* = i_2^* = i^*$. Note that $i_{2}^*$ is the trivial map on the $H^{4i-2}$ components since $k+3 \equiv 0 \pmod 4$. From this diagram, to prove result (1)(a), it is enough to show that the map
\[ i_1^*: H^{k+3}(\Sigma^2 \mathbb{R}P^n/\mathbb{R}P^k; \mathbb{Z}_{(2)}) \to H^{k+3}(\Sigma^2 \mathbb{S}^{k+1}; \mathbb{Z}_{(2)}) \cong \mathbb{Z}_{(2)} \]
is surjective. Applying the long exact sequence of cohomology induced by the cofiber sequence
\[ 
\Sigma^{2} \mathbb{S}^{t-1} \to \Sigma^{2} (\mathbb{R}P^{t-1}/\mathbb{R}P^{k}) \to \Sigma^{2} (\mathbb{R}P^{t}/\mathbb{R}P^{k}) 
\]
for $k+3 \le t \le n$, we see that the induced map
\[ 
i^{*} \colon [ \Sigma^{2} (\mathbb{R}P^{t}/\mathbb{R}P^{k}), K(\mathbb{Z}_{(2)}, k+3) ] \to [ \Sigma^{2} (\mathbb{R}P^{t-1}/\mathbb{R}P^{k}), K(\mathbb{Z}_{(2)}, k+3) ] 
\] 
is an isomorphism. Additionally, for $t=k+2$, the map
\[ 
i^{*} \colon [ \Sigma^{2} (\mathbb{R}P^{k+2}/\mathbb{R}P^{k}), K(\mathbb{Z}_{(2)}, k+3) ] \to [ \Sigma^{2} \mathbb{S}^{k+1}, K(\mathbb{Z}_{(2)}, k+3) ] 
\]
is also an isomorphism, as $\mathbb{R}P^{k+2}/\mathbb{R}P^{k} \simeq \mathbb{S}^{k+2} \vee \mathbb{S}^{k+1}$ for odd $k$, and $[\mathbb{S}^{k+4}, K(\mathbb{Z}_{(2)}, k+3)] = 0$. Since the map $i_{1}^{*}$ is the successive composition of these isomorphisms, it follows that 
\[ 
i_{1}^{*} \colon H^{k+3}(\Sigma^{2} (\mathbb{R}P^{n}/\mathbb{R}P^{k}); \mathbb{Z}_{(2)}) \to H^{k+3}(\Sigma^{2} \mathbb{S}^{k+1}; \mathbb{Z}_{(2)}) 
\]
is an isomorphism, completing the proof of Part (1)(a). To prove Part (1)(b), consider the following commutative diagram:
\begin{equation}\label{gtopdiag}
\begin{CD}
[ \Sigma \mathbb{R}P^n/\mathbb{R}P^k, \Omega G/\mathrm{Top}_{(2)} ] @> i^* >> [ \Sigma \mathbb{R}P^{k+1}/\mathbb{R}P^k, \Omega G/\mathrm{Top}_{(2)} ] \\
@V \ell_* VV @VV \ell_* V \\
[ \Sigma \mathbb{R}P^n/\mathbb{R}P^k, \operatorname{Top}/O_{(2)} ] @> i^* >> [ \Sigma \mathbb{R}P^{k+1}/\mathbb{R}P^k, \operatorname{Top}/O_{(2)} ] \\
@V \psi_* VV @VV \psi_* V \\
[ \Sigma \mathbb{R}P^n/\mathbb{R}P^k, F/O_{(2)} ] @. [ \Sigma \mathbb{R}P^{k+1}/\mathbb{R}P^k, F/O_{(2)} ]
\end{CD}
\end{equation}
where $\ell \colon \Omega G/\mathrm{Top}_{(2)} \to \operatorname{Top}/O_{(2)}$ is the canonical map, and the columns are parts of exact sequences. The vertical map 
\[
\ell_* \colon [ \Sigma \mathbb{R}P^{k+1}/\mathbb{R}P^k, \Omega G/\mathrm{Top}_{(2)} ] \longrightarrow [ \Sigma \mathbb{R}P^{k+1}/\mathbb{R}P^k, \operatorname{Top}/O_{(2)} ]
\] 
is surjective, since $\pi_{k+2}(F/O_{(2)}) = 0$ for $k \in \{1,5,9\}$. By Part (1), the top horizontal map 
\[
i^* \colon [ \Sigma \mathbb{R}P^n/\mathbb{R}P^k, \Omega G/\mathrm{Top}_{(2)} ] \longrightarrow [ \Sigma \mathbb{R}P^{k+1}/\mathbb{R}P^k, \Omega G/\mathrm{Top}_{(2)} ]
\]
is also surjective. Let $\overline{x}$ be an extension of $x$ along the map $i^* \colon [ \Sigma \mathbb{R}P^n/\mathbb{R}P^k, \operatorname{Top}/O_{(2)} ] \to [ \Sigma \mathbb{R}P^{k+1}/\mathbb{R}P^k, \operatorname{Top}/O_{(2)} ]$. A diagram chase then implies the existence of an element $y \in \operatorname{Im}([ \Sigma \mathbb{R}P^n/\mathbb{R}P^k, \Omega G/\mathrm{Top}_{(2)} ] \xrightarrow{\ell_*} [ \Sigma \mathbb{R}P^n/\mathbb{R}P^k, \operatorname{Top}/O_{(2)} ])$ such that $i^*(y) = x$ under the middle horizontal map 
$i^* \colon [ \Sigma \mathbb{R}P^n/\mathbb{R}P^k, \operatorname{Top}/O_{(2)} ] \to [ \Sigma \mathbb{R}P^{k+1}/\mathbb{R}P^k, \operatorname{Top}/O_{(2)} ]$. This proves the first assertion of Part (1)(b). It remains to show that the map 
\[(\Sigma L)^* \colon [ \Sigma \mathbb{R}P^n/\mathbb{R}P^k, \operatorname{Top}/O_{(2)} ] \to [ \mathbb{S}^{16}, \operatorname{Top}/O_{(2)} ]\]
sends the element $y$ to zero. From the exactness of the left vertical column in Diagram~\eqref{gtopdiag}, the fact that $y \in \operatorname{Im}(\ell_*)$ implies that $y$ maps to zero under 
$\psi_* \colon [ \Sigma \mathbb{R}P^n/\mathbb{R}P^k, \operatorname{Top}/O_{(2)} ] \to [ \Sigma \mathbb{R}P^n/\mathbb{R}P^k, F/O_{(2)} ]$. 
Now, consider the following commutative diagram:
\begin{equation}\label{injective-diag}
\begin{CD}
[ \Sigma \mathbb{R}P^n/\mathbb{R}P^k, \operatorname{Top}/O_{(2)} ] @> (\Sigma L)^* >> [ \mathbb{S}^{16}, \operatorname{Top}/O_{(2)} ] \\
@V \psi_* VV @VV \psi_* V \\
[ \Sigma \mathbb{R}P^n/\mathbb{R}P^k, F/O_{(2)} ] @>>> [ \mathbb{S}^{16}, F/O_{(2)} ]
\end{CD}
\end{equation}
where the right vertical map $\psi_* \colon [ \mathbb{S}^{16}, \operatorname{Top}/O_{(2)} ] \to [ \mathbb{S}^{16}, F/O_{(2)} ]$ is injective. By the commutativity of Diagram~\eqref{injective-diag}, we have $\psi_* ((\Sigma L)^*(y)) = (\Sigma L)^*(\psi_*(y)) = 0$. Given the injectivity of $\psi_*$ on the right, it follows that $(\Sigma L)^*(y) = 0$. This completes the proof of Part (1)(b). We now turn to the proof of Part (2). Applying $[-, \operatorname{Top}/O_{(2)}]$ to the cofibration sequence
$\Sigma \mathbb{R}P^7/\mathbb{R}P^{6} \hookrightarrow \Sigma \mathbb{R}P^n/\mathbb{R}P^{6} \xrightarrow{q_{n,6}} \Sigma \mathbb{R}P^n/\mathbb{R}P^7$ yields the exact sequence
\begin{equation} \label{eq:exact1}
[\Sigma \mathbb{R}P^n/\mathbb{R}P^7, \operatorname{Top}/O_{(2)}] \xrightarrow{q_{n,6}^*} [\Sigma \mathbb{R}P^n/\mathbb{R}P^6, \operatorname{Top}/O_{(2)}] \xrightarrow{i^*} [\mathbb{S}^8, \operatorname{Top}/O_{(2)}] \cong \mathbb{Z}_2\{\epsilon\}. 
\end{equation}
From this sequence, to show that $q_{n,6}^*$ is surjective, it is enough to prove that the map $i^*: [\Sigma \mathbb{R}P^{n}/\mathbb{R}P^{6}, \operatorname{Top}/O_{(2)}] \to [\mathbb{S}^8, \operatorname{Top}/O_{(2)}]$ is trivial. Since this map factors through $[\Sigma \mathbb{R}P^{9}/\mathbb{R}P^{6}, \operatorname{Top}/O_{(2)}]$, it suffices to show that $i^*: [\Sigma \mathbb{R}P^{9}/\mathbb{R}P^{6}, \operatorname{Top}/O_{(2)}] \to [\mathbb{S}^8, \operatorname{Top}/O_{(2)}]$ is the trivial map. 
Note that $\mathbb{R}P^8/\mathbb{R}P^6 \simeq M(\mathbb{Z}_2, 7)$ and $\mathbb{R}P^9/\mathbb{R}P^6 \simeq M(\mathbb{Z}_2, 7) \cup_{i \circ \eta} \mathbb{S}^9$. Applying the functor $[-, \operatorname{Top}/O_{(2)}]$ to the cofibre sequences $\mathbb{S}^8 \xrightarrow{i} M(\mathbb{Z}_2, 8) \xrightarrow{q} \mathbb{S}^9$ and $\Sigma \mathbb{R}P^8/\mathbb{R}P^6 \simeq M(\mathbb{Z}_2, 8) \xrightarrow{i} \Sigma \mathbb{R}P^9/\mathbb{R}P^6 \to \mathbb{S}^{10}$, we obtain the exact sequences
\begin{equation} \label{eq:exact3}
0 \to [\mathbb{S}^9, \operatorname{Top}/O_{(2)}] \xrightarrow{q^{*}} [M(\mathbb{Z}_2, 8), \operatorname{Top}/O_{(2)}] \xrightarrow{i^*} [\mathbb{S}^8, \operatorname{Top}/O_{(2)}] \to 0
\end{equation}
and
\begin{equation} \label{eq:exact4}
\begin{aligned}
[M(\mathbb{Z}_2, 9), \operatorname{Top}/O_{(2)}] &\xrightarrow{(i \circ \eta)^*} [\mathbb{S}^{10}, \operatorname{Top}/O_{(2)}] \to [\Sigma \mathbb{R}P^9/\mathbb{R}P^6, \operatorname{Top}/O_{(2)}] \\
&\to [M(\mathbb{Z}_2, 8), \operatorname{Top}/O_{(2)}] \xrightarrow{(i \circ \eta)^*} [\mathbb{S}^9, \operatorname{Top}/O_{(2)}].
\end{aligned}
\end{equation}
Note from \cite[Lemma~3.1]{BKS25} that the map $\eta^* : [\mathbb{S}^8, \operatorname{Top}/O_{(2)}] \to [\mathbb{S}^9, \operatorname{Top}/O_{(2)}]$ is injective and $\eta^* : [\mathbb{S}^9, \operatorname{Top}/O_{(2)}] \to [\mathbb{S}^{10}, \operatorname{Top}/O_{(2)}]$ is surjective. Therefore, the latter long exact sequence \eqref{eq:exact4} reduces to
\begin{equation} \label{eq:exact5}
0 \to [\Sigma \mathbb{R}P^9/\mathbb{R}P^6, \operatorname{Top}/O_{(2)}] \xrightarrow{i^*} [M(\mathbb{Z}_2, 8), \operatorname{Top}/O_{(2)}] \xrightarrow{(i \circ \eta)^*} \mathbb{Z}_{2} \subset [\mathbb{S}^9, \operatorname{Top}/O_{(2)}] \to 0.
\end{equation}
Since $\eta^* : [\mathbb{S}^8, \operatorname{Top}/O_{(2)}] \to [\mathbb{S}^9, \operatorname{Top}/O_{(2)}]$ is injective, we have 
\begin{equation}
\begin{split}
\Ker\left([M(\mathbb{Z}_2, 8), \operatorname{Top}/O_{(2)}] \right. & \left. \xrightarrow{(i \circ \eta)^*} [\mathbb{S}^9, \operatorname{Top}/O_{(2)}]\right) \\
&= \Ker\left([M(\mathbb{Z}_2, 8), \operatorname{Top}/O_{(2)}] \xrightarrow{i^*} [\mathbb{S}^8, \operatorname{Top}/O_{(2)}]\right).
\end{split}
\end{equation}
Thus, the exact sequence \eqref{eq:exact5} implies that the image 
\begin{equation}\label{imker1}
\begin{split}
\operatorname{Im} \left( [\Sigma \mathbb{R}P^9/\mathbb{R}P^6, \operatorname{Top}/O_{(2)}] \right. & \left. \xrightarrow{i^{*}} [M(\mathbb{Z}_2, 8), \operatorname{Top}/O_{(2)}] \right) \\
&= \mathrm{Ker} \left( [M(\mathbb{Z}_2, 8), \operatorname{Top}/O_{(2)}] \xrightarrow{i^*} [\mathbb{S}^8, \operatorname{Top}/O_{(2)}] \right).
\end{split}
\end{equation}
This implies that the composite 
\[ [\Sigma \mathbb{R}P^9/\mathbb{R}P^6, \operatorname{Top}/O_{(2)}] \xrightarrow{i^*} [M(\mathbb{Z}_2, 8), \operatorname{Top}/O_{(2)}] \xrightarrow{i^*} [\mathbb{S}^8, \operatorname{Top}/O_{(2)}] \] 
is trivial. This finishes the proof of Part (2). 

For considering Part (3), we note that the composite $\mathbb{R}P^8/\mathbb{R}P^6\xrightarrow{i}\mathbb{R}P^9/\mathbb{R}P^6\xrightarrow{\overline{\pi}_{4,3}} \mathbb{C}P^4/\mathbb{C}P^3 \simeq \mathbb{S}^{8}$ is homotopic to the collapse map $q:\mathbb{R}P^8/\mathbb{R}P^6\to \mathbb{S}^{8}$ as a consequence of Lemma~3.2 of \cite{Muk00}, and induces the following commutative diagram 
\begin{equation}\label{imker2}
\xymatrix@C=2pc@R=3pc{
[ \Sigma \mathbb{C}P^4/\mathbb{C}P^3, \operatorname{Top}/O_{(2)} ] \cong [\mathbb{S}^9, \operatorname{Top}/O_{(2)}] \ar@{=}[r] \ar[d]_{(\Sigma \overline{\pi}_{4,3})^{*}} 
& [\mathbb{S}^9, \operatorname{Top}/O_{(2)}] \cong [ \Sigma \mathbb{R}P^8/\mathbb{R}P^7, \operatorname{Top}/O_{(2)} ] \ar[d]^{q^{*}} \\
[ \Sigma \mathbb{R}P^9/\mathbb{R}P^6, \operatorname{Top}/O_{(2)} ] \ar[r]^-{\qquad i^{*} \qquad} 
& [ \Sigma \mathbb{R}P^8/\mathbb{R}P^6, \operatorname{Top}/O_{(2)}]
}
\end{equation}
where the map $i^{*}$ is injective by \eqref{eq:exact5}, and the map $q^{*}$ is also injective by \eqref{eq:exact3}. Therefore, from the commutativity of the above diagram, the vertical map $$(\Sigma \overline{\pi}_{4,3})^{*}:[ \Sigma \mathbb{C}P^4/\mathbb{C}P^3, \operatorname{Top}/O_{(2)} ]\to [ \Sigma \mathbb{R}P^9/\mathbb{R}P^6, \operatorname{Top}/O_{(2)} ]$$ is injective. On the other hand, it follows from \eqref{imker1} and \eqref{eq:exact3} that 
\begin{equation}\label{imker3}
\begin{split}
[\Sigma \mathbb{R}P^9/\mathbb{R}P^6, \operatorname{Top}/O_{(2)}] 
&\cong \operatorname{Im} \left( [\Sigma \mathbb{R}P^9/\mathbb{R}P^6, \operatorname{Top}/O_{(2)}] \xrightarrow{i^*} [M(\mathbb{Z}_2, 8), \operatorname{Top}/O_{(2)}] \right) \\
&= \operatorname{Im} \left( [\mathbb{S}^9, \operatorname{Top}/O_{(2)}] \xrightarrow{q^*} [M(\mathbb{Z}_2, 8), \operatorname{Top}/O_{(2)}] \right) \\
&\cong \mathbb{Z}_{2} \oplus \mathbb{Z}_{2} \oplus \mathbb{Z}_{2}.
\end{split}
\end{equation}
This implies that the map $(\Sigma \overline{\pi}_{4,3})^{*}:[ \Sigma \mathbb{C}P^4/\mathbb{C}P^3, \operatorname{Top}/O_{(2)} ] \to [ \Sigma \mathbb{R}P^9/\mathbb{R}P^6, \operatorname{Top}/O_{(2)} ]$ is also surjective and, hence, an isomorphism. This proves Part (3)(a). For Part (3)(b), consider the following commutative diagram:
\begin{equation}\label{imker4}
\xymatrix@C=4pc@R=3pc{
[ \Sigma \mathbb{C}P^7/\mathbb{C}P^3, \operatorname{Top}/O_{(2)} ] \ar[r]^-{(\Sigma \overline{\pi}_{7,3})^{*}} \ar[d]_{i^{*}} & [ \Sigma \mathbb{R}P^{15}/\mathbb{R}P^6, \operatorname{Top}/O_{(2)} ] \ar[dd]^{i^{*}} \\
[ \Sigma \mathbb{C}P^6/\mathbb{C}P^3, \operatorname{Top}/O_{(2)} ] \ar[d]_{i^{*}} & \\
[ \Sigma \mathbb{C}P^4/\mathbb{C}P^3, \operatorname{Top}/O_{(2)} ] \ar[r]^-{(\Sigma \overline{\pi}_{4,3})^{*}}_{\cong} & [ \Sigma \mathbb{R}P^9/\mathbb{R}P^6, \operatorname{Top}/O_{(2)} ]
}
\end{equation}
where the bottom isomorphism $(\Sigma \overline{\pi}_{4,3})^{*}$ is given by Part (3)(a). Using $\mathbb{C}P^7/\mathbb{C}P^3 \simeq_{(2)} \mathbb{S}^{14} \vee \mathbb{C}P^6/\mathbb{C}P^3$, and the facts that $\pi_{12}(\operatorname{Top}/O_{(2)}) = 0$ and $\pi_{13}(\operatorname{Top}/O_{(2)}) = 0$, we apply the same arguments used to establish \eqref{cp6-1} to see that the induced map
\[
i^{*}: [ \Sigma \mathbb{C}P^6/\mathbb{C}P^3, \operatorname{Top}/O_{(2)} ] \to [ \Sigma \mathbb{C}P^4/\mathbb{C}P^3, \operatorname{Top}/O_{(2)} ] \oplus [ \Sigma \mathbb{S}^{10}, \operatorname{Top}/O_{(2)} ]
\]
is an isomorphism. Consequently, the first vertical composition
\[
i^{*}: [ \Sigma \mathbb{C}P^7/\mathbb{C}P^3, \operatorname{Top}/O_{(2)} ] \to [ \Sigma \mathbb{C}P^4/\mathbb{C}P^3, \operatorname{Top}/O_{(2)} ]
\]
is surjective. Combining this surjectivity with the commutativity of Diagram \eqref{imker4}, the assertion in Part (3)(b) follows immediately.
\end{proof}
\begin{lemma}\label{RP16-CP8-TopO}
\begin{itemize}
    \item[(1)] Let $f_{\mathbb{R}P^{16}}: \mathbb{R}P^{16} \to \mathbb{R}P^{16}/\mathbb{R}P^{15} \cong \mathbb{S}^{16}$ be the canonical quotient map. Then the induced map
    \[
    f_{\mathbb{R}P^{16}}^*: [\mathbb{S}^{16}, \operatorname{Top}/O] \longrightarrow [\mathbb{R}P^{16}, \operatorname{Top}/O]
    \]
    is injective.
    
    \item[(2)] Let $\pi \colon \mathbb{R}P^{17} \to \mathbb{C}P^{8}$ be the natural projection map. The induced map
\[
\pi^* \colon [\mathbb{C}P^{8}, \operatorname{Top}/O] \cong \mathbb{Z}_2\{(\eta^*)_{16}\} \oplus \mathbb{Z}_2\{\left(\overline{(\kappa)_{14}+\left(\overline{(\epsilon)_{8}}\right)_{14}}\right)_{16}\} \longrightarrow [\mathbb{R}P^{17}, \operatorname{Top}/O]
\]
is injective. Moreover, the composite $i^* \circ \pi^* \colon [\mathbb{C}P^{8}, \operatorname{Top}/O] \to [\mathbb{R}P^{16}, \operatorname{Top}/O]$ sends the generator $(\eta^*)_{16}$ to $f_{\mathbb{R}P^{16}}^*(\eta^*)$, where $i \colon \mathbb{R}P^{16} \to \mathbb{R}P^{17}$ is the inclusion map.
\end{itemize}
\end{lemma}
\begin{proof}
Since $[\mathbb{S}^{16}, \operatorname{Top}/O] \cong \mathbb{Z}_2\{\eta^*\}$, it is enough to work $2$-locally. Consider the following commutative diagram :
\begin{equation}\label{topquo}
\xymatrix@C=3pc{
[\Sigma \mathbb{R}P^{15}, \operatorname{Top}/O_{(2)}] \ar[r]^{(\Sigma p)^*} & [\mathbb{S}^{16}, \operatorname{Top}/O_{(2)}] \ar[r]^{f^*_{\mathbb{R}P^{16}}} & [\mathbb{R}P^{16}, \operatorname{Top}/O_{(2)}] \\
[\Sigma \mathbb{R}P^{15}/\mathbb{R}P^{1}, \operatorname{Top}/O_{(2)}] \ar[u]^{(\Sigma q_{15,0})^{*}} \ar[ur]_{(\Sigma p)^{*}\circ(\Sigma q_{15,0})^{*}} & & 
}
\end{equation}
where the top row is part of the long exact sequence induced by the cofibre sequence
$\mathbb{S}^{15} \xrightarrow{p} \mathbb{R}P^{15} \xrightarrow{i} \mathbb{R}P^{16} \xrightarrow{f_{\mathbb{R}P^{16}}} \mathbb{S}^{16}$,
and $q_{n,k}$ denotes the quotient map
$\mathbb{R}P^{n}/\mathbb{R}P^{k} \to \mathbb{R}P^{n}/\mathbb{R}P^{k+1}$, where
$q_{n,0} : \mathbb{R}P^{n} \to \mathbb{R}P^{n}/\mathbb{R}P^{1}$.
Since $[\Sigma \mathbb{R}P^{1}, \operatorname{Top}/O_{(2)}] = 0$, the induced map
\[
(\Sigma q_{15,0})^* : [\Sigma \mathbb{R}P^{15}/\mathbb{R}P^{1}, \operatorname{Top}/O_{(2)}]
\to
[\Sigma \mathbb{R}P^{15}, \operatorname{Top}/O_{(2)}]
\]
is surjective. It follows from the commutativity of Diagram~\eqref{topquo} and the exactness of the top row that, to show $f_{\mathbb{R}P^{16}}^*$ is injective, it suffices to prove that the composition
\[
(\Sigma p)^* \circ (\Sigma q_{15,0})^* :
[\Sigma(\mathbb{R}P^{15}/\mathbb{R}P^{1}), \operatorname{Top}/O_{(2)}]
\to
[\mathbb{S}^{16}, \operatorname{Top}/O_{(2)}]
\]
is trivial. Let $\varphi_{n,k}$ denote the composition
\[
q_{n,k-1} \circ \cdots \circ q_{n,1} \circ q_{n,0} \circ p \colon \mathbb{S}^{n} \to \mathbb{R}P^{n}/\mathbb{R}P^{k}.
\]
Note that $q_{n,k} \circ \varphi_{n,k} = \varphi_{n,k+1}$. Consider the following commutative diagram:
\begin{equation}\label{redu1}
\xymatrix@C=3.5pc@R=3pc{
{[\Sigma (\mathbb{R}P^{k+1}/\mathbb{R}P^{k}), \operatorname{Top}/O_{(2)}]} & \\
{[\Sigma (\mathbb{R}P^{15}/\mathbb{R}P^{k}), \operatorname{Top}/O_{(2)}]} \ar[u]^{i^*} \ar[r]^-{(\Sigma \varphi_{15,k})^*} & {[\mathbb{S}^{16}, \operatorname{Top}/O_{(2)}]} \\
{[\Sigma (\mathbb{R}P^{15}/\mathbb{R}P^{k+1}), \operatorname{Top}/O_{(2)}]} \ar[u]^{(\Sigma q_{15,k})^*} \ar[ru]_-{(\Sigma \varphi_{15,k+1})^*} &  
}
\end{equation}
where the first column is part of the long exact sequence induced by the cofiber sequence $\mathbb{R}P^{k+1}/\mathbb{R}P^{k} \hookrightarrow \mathbb{R}P^{15}/\mathbb{R}P^{k} \xrightarrow{q_{15,k}} \mathbb{R}P^{15}/\mathbb{R}P^{k+1}$. The commutativity of Diagram \eqref{redu1} immediately implies the following statement:
\begin{itemize}
    \item[($\star$)] Suppose $(\Sigma \varphi_{15,k+1})^* \colon [\Sigma (\mathbb{R}P^{15}/\mathbb{R}P^{k+1}), \operatorname{Top}/O_{(2)}] \to [\mathbb{S}^{16}, \operatorname{Top}/O_{(2)}]$ is the trivial map. Then the restriction of $(\Sigma \varphi_{15,k})^*$ to the image of $(\Sigma q_{15,k})^*$ is also trivial.
\end{itemize}
Now assume that $(\Sigma \varphi_{15,6})^* \colon [\Sigma (\mathbb{R}P^{15}/\mathbb{R}P^{6}), \operatorname{Top}/O_{(2)}] \to [\mathbb{S}^{16}, \operatorname{Top}/O_{(2)}]$ is the trivial map. For $k=5$, let $\overline{x} \in [\Sigma \mathbb{R}P^{15}/\mathbb{R}P^k, \operatorname{Top}/O_{(2)}]$ be any extension of an element \[x \in [\Sigma \mathbb{R}P^{k+1}/\mathbb{R}P^k, \operatorname{Top}/O_{(2)}]\] along the vertical map $i^*$. It follows from Lemma \ref{teccp8}(1)(b) with $L=\varphi_{15,k}$ that there exists an element $y \in [\Sigma \mathbb{R}P^{15}/\mathbb{R}P^k, \operatorname{Top}/O_{(2)}]$ such that $i^{*}(y) = i^{*}(\overline{x}) = x$ and $(\Sigma \varphi_{15,k})^*(y) = 0$.
Since $i^{*}(y - \overline{x}) = 0$, the exactness of the first column in Diagram \eqref{redu1} implies that $y - \overline{x}$ lies in the image of $(\Sigma q_{15,k})^*$. By statement ($\star$), we have $(\Sigma \varphi_{15,k})^*(y - \overline{x}) = 0$. Consequently,
\[
(\Sigma \varphi_{15,k})^*(\overline{x}) = (\Sigma \varphi_{15,k})^*(y) = 0.
\]
Thus, the map $(\Sigma \varphi_{15,k})^*$ sends any extension $\overline{x}$ of $x$ to zero. Combined with ($\star$), this implies that $(\Sigma \varphi_{15,k})^*$ is itself the trivial map for $k=5$.
This result for $k=5$ allows us to proceed further down the filtration.
For $k=2, 3, 4$, we have $[\Sigma \mathbb{R}P^{k+1}/\mathbb{R}P^k, \operatorname{Top}/O_{(2)}]=0$. By the exactness of the first column in Diagram \eqref{redu1}, this implies that the maps 
\[ (\Sigma q_{15,k})^* \colon [\Sigma (\mathbb{R}P^{15}/\mathbb{R}P^{k+1}), \operatorname{Top}/O_{(2)}] \to [\Sigma (\mathbb{R}P^{15}/\mathbb{R}P^{k}), \operatorname{Top}/O_{(2)}] \]
are surjective for $k=2, 3, 4$. Applying statement ($\star$) iteratively, starting from $k=4$, the triviality of $(\Sigma \varphi_{15,k+1})^*$ together with this surjectivity implies that $(\Sigma \varphi_{15,k})^*$ is trivial for $k=4,3,2$. Finally, for $k=1$, repeating the argument used for $k=5$ by invoking Lemma \ref{teccp8}(1)(b) with $L=\varphi_{15,k}$ and statement ($\star$), we conclude that $(\Sigma \varphi_{15,1})^*$ is also the trivial map. Therefore, instead of directly showing that $(\Sigma \varphi_{15,1})^* = (\Sigma p)^* \circ (\Sigma q_{15,0})^*:
[\Sigma(\mathbb{R}P^{15}/\mathbb{R}P^{1}), \operatorname{Top}/O_{(2)}]
\to
[\mathbb{S}^{16}, \operatorname{Top}/O_{(2)}]$ is trivial, it suffices to prove that 
\[ (\Sigma \varphi_{15,6})^* \colon [\Sigma \mathbb{R}P^{15}/\mathbb{R}P^6, \operatorname{Top}/O_{(2)}] \to [\mathbb{S}^{16}, \operatorname{Top}/O_{(2)}] \] 
is the trivial map. To this end, we first assume that the map 
\[ (\Sigma \varphi_{15,9})^* \colon [\Sigma \mathbb{R}P^{15}/\mathbb{R}P^9, \operatorname{Top}/O_{(2)}] \longrightarrow [\mathbb{S}^{16}, \operatorname{Top}/O_{(2)}] \] 
is trivial. Consider the part of the long exact sequence induced by the cofiber sequence $\mathbb{R}P^{9}/\mathbb{R}P^6 \xrightarrow{i} \mathbb{R}P^{15}/\mathbb{R}P^6 \xrightarrow{q} \mathbb{R}P^{15}/\mathbb{R}P^9$, where $q = q_{15,8} \circ q_{15,7} \circ q_{15,6}$:
\begin{equation} \label{redu67}
\dots \to [\Sigma \mathbb{R}P^{15}/\mathbb{R}P^9, \operatorname{Top}/O_{(2)}] \xrightarrow{q^*} [\Sigma \mathbb{R}P^{15}/\mathbb{R}P^6, \operatorname{Top}/O_{(2)}] \xrightarrow{i^*} [\Sigma \mathbb{R}P^{9}/\mathbb{R}P^6, \operatorname{Top}/O_{(2)}].
\end{equation}
Since $(\Sigma \varphi_{15,6})^{*} \circ q^* = (\Sigma \varphi_{15,9})^{*}$, which is trivial by assumption, the map $(\Sigma \varphi_{15,6})^*$ vanishes on $\operatorname{Im}(q^*)$. From this and the exact sequence \eqref{redu67}, it suffices to show that $(\Sigma \varphi_{15,6})^*$ also sends any extension of an element $x \in [\Sigma \mathbb{R}P^{9}/\mathbb{R}P^6, \operatorname{Top}/O_{(2)}]$ to zero, which establishes the triviality of $(\Sigma \varphi_{15,6})^*$.
Consider the following diagram:
\begin{equation}\label{reduconnct}
\begin{tikzcd}[column sep=3.5em, row sep=3em]
{[\Sigma \mathbb{R}P^{15}/\mathbb{R}P^{6}, \operatorname{Top}/O_{(2)}]}
\arrow[r, "{(\Sigma \varphi_{15,6})^*}"]
&
{[\mathbb{S}^{16}, \operatorname{Top}/O_{(2)}]}
\\
{[\Sigma \mathbb{C}P^{7}/\mathbb{C}P^3, \operatorname{Top}/O_{(2)}]
\cong
[\mathbb{S}^{15}, \operatorname{Top}/O_{(2)}] \oplus
[\Sigma \mathbb{C}P^{6}/\mathbb{C}P^3, \operatorname{Top}/O_{(2)}]}
\arrow[u, "{(\Sigma \overline{\pi}_{7,3})^{*}}"]
\arrow[ru, "{(\Sigma \varphi_{8})^* = (\Sigma \varphi_8)_1^{*} \oplus \eta^{*}}"']
&
\end{tikzcd}
\end{equation}
where $\varphi_8 = ((\varphi_8)_1, (\varphi_8)_2)$ with the first component $(\varphi_8)_1 \colon \mathbb{S}^{15} \to \mathbb{C}P^{6}/\mathbb{C}P^{3}$ and the second component $(\varphi_8)_2 = \eta \in \pi_{15}(\mathbb{S}^{14})$ (cf. \eqref{spil11}). The map $(\Sigma \varphi_{8})^*$ is trivial because both $\eta^{*} \colon [\mathbb{S}^{15}, \operatorname{Top}/O_{(2)}] \to [\mathbb{S}^{16}, \operatorname{Top}/O_{(2)}]$ and $(\Sigma \varphi_8)_1^{*} \colon [\Sigma \mathbb{C}P^{6}/\mathbb{C}P^3, \operatorname{Top}/O_{(2)}] \to [\mathbb{S}^{16}, \operatorname{Top}/O_{(2)}]$ are trivial (see \ref{conntcp7}). This, together with the commutativity of Diagram \eqref{reduconnct}, implies that the map $(\Sigma \varphi_{15,6})^*$ is trivial on the image $\operatorname{Im}((\Sigma \overline{\pi}_{7,3})^{*})$. Let $\overline{x} \in [\Sigma \mathbb{R}P^{15}/\mathbb{R}P^6, \operatorname{Top}/O_{(2)}]$ be an extension of $x \in [\Sigma \mathbb{R}P^{9}/\mathbb{R}P^6, \operatorname{Top}/O_{(2)}]$ along $i^*$. By Lemma \ref{teccp8}(3)(b), there exists an element 
\[ y \in \operatorname{Im} \left( (\Sigma\overline{\pi}_{7,3})^{*} \colon [\Sigma \mathbb{C}P^{7}/\mathbb{C}P^3, \operatorname{Top}/O_{(2)}] \longrightarrow [\Sigma \mathbb{R}P^{15}/\mathbb{R}P^{6}, \operatorname{Top}/O_{(2)}] \right) \] 
such that $i^*(y) = x$, and consequently $(\Sigma \varphi_{15,6})^{*}(y) = 0$.

On the other hand, the exactness of \eqref{redu67} implies the existence of an element $z \in [\Sigma \mathbb{R}P^{15}/\mathbb{R}P^9, \operatorname{Top}/O_{(2)}]$ such that $q^*(z) = \overline{x} - y$. Since $(\Sigma \varphi_{15,6})^{*}$ is trivial on $\operatorname{Im}(q^*)$, we have $(\Sigma \varphi_{15,6})^*(\overline{x} - y) = 0$, which implies $(\Sigma \varphi_{15,6})^*(\overline{x}) = (\Sigma \varphi_{15,6})^*(y) = 0$. 
Consequently, the map $(\Sigma \varphi_{15,6})^*$ is trivial on any extension of $x$, as well as on $\operatorname{Im}(q^*)$. This concludes the proof that $(\Sigma \varphi_{15,6})^* \colon [\Sigma \mathbb{R}P^{15}/\mathbb{R}P^6, \operatorname{Top}/O_{(2)}] \to [\mathbb{S}^{16}, \operatorname{Top}/O_{(2)}]$ is the trivial map, assuming the triviality of $(\Sigma \varphi_{15,9})^*$. To complete the proof of Part (1), it remains to verify the assumption that $(\Sigma \varphi_{15,9})^*:[\Sigma \mathbb{R}P^{15}/\mathbb{R}P^9, \operatorname{Top}/O_{(2)}] \longrightarrow [\mathbb{S}^{16}, \operatorname{Top}/O_{(2)}]$ is indeed the trivial map. Since $\mathbb{R}P^{15}/\mathbb{R}P^{9} \simeq \mathbb{R}P^{14}/\mathbb{R}P^{9} \vee \mathbb{S}^{15}$ (see, for example, \cite[Proof of Theorem 1·6, p.~237]{LR05}), the map $\varphi_{15,9} : \mathbb{S}^{15} \to \mathbb{R}P^{15}/\mathbb{R}P^{9}$ decomposes into two components, $\pi_0$ and $\pi_1$, where $\pi_1 : \mathbb{S}^{15} \to \mathbb{S}^{15}$ is a map of degree two and $\pi_0: \mathbb{S}^{15} \to \mathbb{R}P^{14}/\mathbb{R}P^{9}$ is the attaching map of the $16$-cell of $\mathbb{R}P^{16}/\mathbb{R}P^{9}$ onto the $14$-skeleton. Consequently, $(\Sigma \varphi_{15,9})^*:[\Sigma \mathbb{R}P^{15}/\mathbb{R}P^9, \operatorname{Top}/O_{(2)}]\to [\mathbb{S}^{16}, \operatorname{Top}/O_{(2)}]$ can be identified with the map
\[ (\Sigma\pi_0)^{*} \oplus (\Sigma\pi_1)^{*} : [\Sigma \mathbb{R}P^{14}/\mathbb{R}P^{9}, \operatorname{Top}/O_{(2)}] \oplus [\mathbb{S}^{16}, \operatorname{Top}/O_{(2)}] \to [\mathbb{S}^{16}, \operatorname{Top}/O_{(2)}]. \]
Since the component $(\Sigma\pi_1)^*$ corresponds to multiplication by $2$ and $[ \mathbb{S}^{16}, \operatorname{Top}/O_{(2)} ] \cong \mathbb{Z}_{2}$, the map $(\Sigma\pi_1)^*: [\mathbb{S}^{16}, \operatorname{Top}/O_{(2)}] \to [\mathbb{S}^{16}, \operatorname{Top}/O_{(2)}]$ is trivial. It remains to show that the map
\[ (\Sigma \pi_0)^{*} : [\Sigma \mathbb{R}P^{14}/\mathbb{R}P^9, \operatorname{Top}/O_{(2)}] \to [\mathbb{S}^{16}, \operatorname{Top}/O_{(2)}] \] 
is trivial. To establish this, we examine the following commutative diagram for $9 \leq k \leq 11$:
\begin{equation}\label{iterat9}
\xymatrix@C=4.5em@R=4em{
    [\Sigma \mathbb{R}P^{14}/\mathbb{R}P^{k}, \operatorname{Top}/O_{(2)}] \ar[r]^-{(\Sigma \pi_0)^{*}} & [\mathbb{S}^{16}, \operatorname{Top}/O_{(2)}] \\
    [\Sigma \mathbb{R}P^{14}/\mathbb{R}P^{k+1}, \operatorname{Top}/O_{(2)}] \ar[u]^{(\Sigma q_{14,k})^*} \ar[ru]_-{(\Sigma \pi_0)^{*} \circ (\Sigma q_{14,k})^*} &  
}  
\end{equation}
where $(\Sigma q_{14,k})^*$ is surjective for $k = 10, 11$, since $[\mathbb{S}^{12}, \operatorname{Top}/O_{(2)}] = [\mathbb{S}^{13}, \operatorname{Top}/O_{(2)}] = 0$. For $k=9$, we have the exact sequence
\begin{equation} \label{redexten}
\dots \to [\Sigma \mathbb{R}P^{14}/\mathbb{R}P^{10}, \operatorname{Top}/O_{(2)}] \xrightarrow{(\Sigma q_{14,9})^*} [\Sigma \mathbb{R}P^{14}/\mathbb{R}P^9, \operatorname{Top}/O_{(2)}] \xrightarrow{i^*} [\Sigma \mathbb{R}P^{10}/\mathbb{R}P^9, \operatorname{Top}/O_{(2)}].
\end{equation}
We first show that the composition
\[ (\Sigma \pi_{0})^{*} \circ (\Sigma q_{14,9})^* \circ (\Sigma q_{14,10})^* \circ (\Sigma q_{14,11})^* : [\Sigma \mathbb{R}P^{14}/\mathbb{R}P^{12}, \operatorname{Top}/O_{(2)}] \longrightarrow [\mathbb{S}^{16}, \operatorname{Top}/O_{(2)}] \]
is the trivial map. Note that $\Sigma \mathbb{R}P^{14}/\mathbb{R}P^{12} \simeq M(\mathbb{Z}_2, 14)$ and $\pi_{16}(M(\mathbb{Z}_2, 14)) \cong \mathbb{Z}_4\{\widetilde{\eta}\}$. The composition $q_{14,11} \circ q_{14,10} \circ q_{14,9} \circ \pi_0$ is the attaching map of the $16$-cell of $\mathbb{R}P^{16}/\mathbb{R}P^{12}$ onto the $14$-skeleton and is homotopic to $\widetilde{\eta}$ in $[\mathbb{S}^{16}, M(\mathbb{Z}_2, 14)]$, since $Sq^{2}$ acts nontrivially on $H^{14}(\mathbb{R}P^{16}/\mathbb{R}P^{12}; \mathbb{Z}_{2})$ (see, for example, \cite[p.~5]{GMS12}). Therefore, the composition 
\[ (\Sigma \pi_{0})^{*} \circ (\Sigma q_{14,9})^* \circ (\Sigma q_{14,10})^* \circ (\Sigma q_{14,11})^* \]
can be identified with the map $\widetilde{\eta}^* : [M(\mathbb{Z}_2, 14), \operatorname{Top}/O_{(2)}] \to [\mathbb{S}^{16}, \operatorname{Top}/O_{(2)}]$.
Applying $[-,X]$ to the cofiber sequence $\mathbb{S}^{14} \xrightarrow{2} \mathbb{S}^{14} \longrightarrow M(\mathbb{Z}_2,14) \xrightarrow{q} \mathbb{S}^{15}$, where \[X \in \{SF_{(2)}, F/O_{(2)}, \operatorname{Top}/O_{(2)}\},\] we obtain the short exact sequence
\begin{equation}\label{moorelemme}
0 \longrightarrow \pi_{15}(X)/2 \longrightarrow [M(\mathbb{Z}_2,14),X] \xrightarrow{i^*} \Ker\!\left(\pi_{14}(X)\xrightarrow{2} \pi_{14}(X)\right) \longrightarrow 0,
\end{equation}
where $A/2$ denotes the quotient $A/2A$. Let $a \in \Ker(\pi_{14}(X) \xrightarrow{2} \pi_{14}(X))$ and let $\overline{a} \in [M(\mathbb{Z}_2,14),X]$ be an extension, i.e., $i^*(\overline{a})=a$. By Lemma 2.3 of \cite[p.12]{CSS18}, \[ 2\overline{a} = [a \circ \eta] \in \pi_{15}(X)/2, \] 
where $[r]$ denotes the class of $r \in A$ in the quotient $A/2$. 

For $X = SF_{(2)}$, we have 
\[ \pi_{14}(SF_{(2)})=\mathbb{Z}_2\{\kappa\}\oplus \mathbb{Z}_2\{\sigma^2\}, \quad \pi_{15}(SF_{(2)})=\mathbb{Z}_2\{\eta\circ\kappa\}\oplus \mathbb{Z}_{32}\{\rho\} \quad (\text{see \cite{Tod62}}). \]
Thus, $\pi_{15}(SF_{(2)})/2 = \mathbb{Z}_2\{[\eta\circ\kappa]\}\oplus \mathbb{Z}_2\{[\rho]\}$. Let $\overline{\kappa}, \overline{\sigma^2} \in [M(\mathbb{Z}_2,14),SF_{(2)}]$ be extensions of $\kappa$ and $\sigma^2$, respectively. Then
\[ 2\overline{\kappa} = [\eta\circ\kappa] \neq 0, \quad 2\overline{\sigma^2} = [\eta\circ\sigma^2] = 0. \]
Hence, $\overline{\kappa}$ has order $4$, $\overline{\sigma^2}$ has order $2$, and $[\rho]$ provides an independent element of order $2$. Therefore,
\begin{equation}\label{m1}
[M(\mathbb{Z}_2,14),SF_{(2)}] \cong \mathbb{Z}_4\{\overline{\kappa}\} \oplus \mathbb{Z}_2\{\overline{\sigma^2}\} \oplus \mathbb{Z}_2\{[\rho]\}.
\end{equation}

For $X = F/O_{(2)}$, we have $\pi_{14}(F/O_{(2)})=\mathbb{Z}_2\{\kappa\}\oplus \mathbb{Z}_2\{\sigma^2\}$ and $\pi_{15}(F/O_{(2)})=\mathbb{Z}_2\{\eta\circ\kappa\}$. Thus, $\pi_{15}(F/O_{(2)})/2 = \mathbb{Z}_2\{[\eta\circ\kappa]\}$. Letting $\overline{\kappa}, \overline{\sigma^2}$ be extensions, we have $2\overline{\kappa} = [\eta\circ\kappa] \neq 0$ and $2\overline{\sigma^2} = 0$. Therefore,
\begin{equation}\label{m2}
[M(\mathbb{Z}_2,14),F/O_{(2)}] \cong \mathbb{Z}_4\{\overline{\kappa}\} \oplus \mathbb{Z}_2\{\overline{\sigma^2}\}.
\end{equation}

For $X = \operatorname{Top}/O_{(2)}$, we have $\pi_{14}(\operatorname{Top}/O_{(2)})=\mathbb{Z}_2\{\kappa\}$ and $\pi_{15}(\operatorname{Top}/O_{(2)})=\mathbb{Z}_2\{\eta\circ\kappa\} \oplus \mathbb{Z}_{8128}\{b\}$, where $b$ is the generator of $\mathit{bP}_{16} \subset \Theta_{15}$. Thus, $\pi_{15}(\operatorname{Top}/O_{(2)})/2 = \mathbb{Z}_2\{[\eta\circ\kappa]\} \oplus \mathbb{Z}_2\{[b]\}$. Letting $\overline{\kappa} \in [M(\mathbb{Z}_2,14),\operatorname{Top}/O_{(2)}]$ be an extension of $\kappa$, we have $2\overline{\kappa} = [\kappa\circ \eta] = [\eta\circ\kappa] \neq 0$. Consequently,
\begin{equation}\label{m3}
[M(\mathbb{Z}_2,14),\operatorname{Top}/O_{(2)}] \cong \mathbb{Z}_4\{\overline{\kappa}\} \oplus \mathbb{Z}_2\{[b]\}.
\end{equation}
Note that the composition $\mathbb{S}^{16} \xrightarrow{\widetilde{\eta}} M(\mathbb{Z}_2,14) \xrightarrow{q} \mathbb{S}^{15}$ is homotopic to $\eta$. Consequently, by \cite[Lemma 3.1]{BKS25}, the induced map 
\[ \eta^{*} \colon [\mathbb{S}^{15}, \operatorname{Top}/O_{(2)}] \longrightarrow [\mathbb{S}^{16}, \operatorname{Top}/O_{(2)}] \] 
is the trivial map. Since $\eta^* = \widetilde{\eta}^* \circ q^*$, the triviality of $\eta^*$ implies that the image of $q^*$ must lie in the kernel of $\widetilde{\eta}^*$. In view of \eqref{moorelemme}, where $X = \operatorname{Top}/O_{(2)}$, and \eqref{m3}, the generator $[b]$ is identified with the image of the generator of $[\mathbb{S}^{15}, \operatorname{Top}/O_{(2)}]$ under $q^*$. It follows immediately that 
\[ \widetilde{\eta}^* \colon [M(\mathbb{Z}_2,14), \operatorname{Top}/O_{(2)}] \cong \mathbb{Z}_4\{\overline{\kappa}\} \oplus \mathbb{Z}_2\{[b]\} \longrightarrow [\mathbb{S}^{16}, \operatorname{Top}/O_{(2)}] \] 
sends the generator $[b]$ to zero. It remains to show that the generator $\overline{\kappa}$ also maps to zero under $\widetilde{\eta}^*$. To this end, consider the following commutative diagram:
\begin{equation*}
\begin{aligned}
&[M(\mathbb{Z}_2,14), \operatorname{Top}/O_{(2)}]
\cong \mathbb{Z}_4\{\overline{\kappa}\} \oplus \mathbb{Z}_2\{[b]\}
\;\xrightarrow{\widetilde{\eta}^*}\;
[\mathbb{S}^{16}, \operatorname{Top}/O_{(2)}]
\cong \mathbb{Z}_2\{\eta^*\}
\\[8pt]
&\hspace{8.5em} \downarrow \psi_* \hspace{16em} \downarrow \psi_*
\\[8pt]
&[M(\mathbb{Z}_2,14), F/O_{(2)}]
\cong \mathbb{Z}_4\{\overline{\kappa}\} \oplus \mathbb{Z}_2\{\overline{\sigma^2}\}
\;\xrightarrow{\widetilde{\eta}^*}\;
[\mathbb{S}^{16}, F/O_{(2)}]
\cong \mathbb{Z}_{(2)} \oplus \mathbb{Z}_2\{\eta^*\}
\\[8pt]
&\hspace{8.5em} \uparrow \phi_* \hspace{16.5em} \uparrow \phi_*
\\[8pt]
&[M(\mathbb{Z}_2,14), SF_{(2)}]
\cong \mathbb{Z}_4\{\overline{\kappa}\} \oplus \mathbb{Z}_2\{\overline{\sigma^2}\} \oplus \mathbb{Z}_2\{[\rho]\}
\;\xrightarrow{\widetilde{\eta}^*}\;
[\mathbb{S}^{16}, SF_{(2)}]
\cong \mathbb{Z}_2\{\eta \circ \rho\} \oplus \mathbb{Z}_2\{\eta^*\}
\end{aligned}
\end{equation*}
where the extension $\overline{\kappa}$ of $\kappa$ can be chosen such that the vertical maps $\psi_*$ and $\phi_*$ in the first column send $\overline{\kappa}$ to $\overline{\kappa}$. By the commutativity of this diagram, instead of proving $\widetilde{\eta}^*(\overline{\kappa})=0$ in the first row, it suffices to show that the second row map $\widetilde{\eta}^*: [M(\mathbb{Z}_2,14), F/O_{(2)}] \to [\mathbb{S}^{16}, F/O_{(2)}]$ sends $\overline{\kappa}$ to zero. Since $\overline{\kappa}$ is an extension of $\kappa$ and $\widetilde{\eta}$ is a co-extension of $\eta$, we have $\widetilde{\eta}^*(\overline{\kappa}) = \overline{\kappa} \circ \widetilde{\eta} \in \langle \kappa, 2, \eta \rangle$. It follows from \cite[Lemma 15.2]{MT63} (see also \cite[p.~80]{Muk69}) that $\langle \kappa, 2, \eta \rangle \equiv 0 \pmod{\eta\circ \rho}$ in $[\mathbb{S}^{16}, SF_{(2)}]$. Since $\eta\circ \rho \in \operatorname{Im}([\mathbb{S}^{16}, O] \to [\mathbb{S}^{16}, SF])$, the Toda bracket $\langle \kappa, 2, \eta \rangle$ vanishes in $[\mathbb{S}^{16}, F/O_{(2)}]$. Consequently, the commutativity of the bottom square implies that $\widetilde{\eta}^*$ maps $\overline{\kappa}$ to zero in the middle row, which in turn ensures its triviality in the top row. Therefore, the composition 
\begin{equation}\label{compositiontrivial}
 (\Sigma \pi_{0})^{*} \circ (\Sigma q_{14,9})^* \circ (\Sigma q_{14,10})^* \circ (\Sigma q_{14,11})^* = \widetilde{\eta}^* \colon [M(\mathbb{Z}_2, 14), \operatorname{Top}/O_{(2)}] \longrightarrow [\mathbb{S}^{16}, \operatorname{Top}/O_{(2)}]   
\end{equation}
is the trivial map. Combining this result with the surjectivity of $(\Sigma q_{14,k})^*$ for $k = 10, 11$ (as shown in Diagram \eqref{iterat9}), we see that the composite 
\[ (\Sigma \pi_{0})^{*} \circ (\Sigma q_{14,9})^* \colon [\Sigma \mathbb{R}P^{14}/\mathbb{R}P^{10}, \operatorname{Top}/O_{(2)}] \longrightarrow [\mathbb{S}^{16}, \operatorname{Top}/O_{(2)}] \] 
is also trivial. This implies that $(\Sigma \pi_{0})^{*}$ is trivial on the image $\operatorname{Im}((\Sigma q_{14,9})^*)$. In view of the exact sequence \eqref{redexten} and following the same reasoning as before, to show that $(\Sigma \pi_{0})^{*}$ is trivial on its entire domain, it suffices to show that any extension in $[\Sigma \mathbb{R}P^{14}/\mathbb{R}P^{9}, \operatorname{Top}/O_{(2)}]$ of an element $x \in [\Sigma \mathbb{R}P^{10}/\mathbb{R}P^9, \operatorname{Top}/O_{(2)}]$ is mapped to zero under $(\Sigma \pi_{0})^{*}$. To this end, let $\overline{x} \in [\Sigma \mathbb{R}P^{14}/\mathbb{R}P^{9}, \operatorname{Top}/O_{(2)}]$ be an extension of $x \in [\mathbb{S}^{11}, \operatorname{Top}/O_{(2)}]$ along the induced map 
\[ i^* \colon [ \Sigma \mathbb{R}P^{14}/\mathbb{R}P^9, \operatorname{Top}/O_{(2)} ] \longrightarrow [ \Sigma \mathbb{R}P^{10}/\mathbb{R}P^9, \operatorname{Top}/O_{(2)} ]. \]
As established previously, by Lemma \ref{teccp8}(1)(b) with $L = \pi_0$, there exists an element $y \in [\Sigma \mathbb{R}P^{14}/\mathbb{R}P^9, \operatorname{Top}/O_{(2)}]$ such that $i^{*}(y) = i^{*}(\overline{x}) = x$ and $(\Sigma \pi_0)^*(y) = 0$. From the exactness of \eqref{redexten}, the difference $y - \overline{x}$ belongs to $\operatorname{Im}((\Sigma q_{14,9})^*)$. Since $(\Sigma \pi_{0})^{*}$ is trivial on $\operatorname{Im}((\Sigma q_{14,9})^*)$, it follows that 
\[ (\Sigma \pi_{0})^{*}(\overline{x}) = (\Sigma \pi_{0})^{*}(y) = 0. \]
Consequently, the map 
\[ (\Sigma \pi_{0})^{*} \colon [\Sigma \mathbb{R}P^{14}/\mathbb{R}P^{9}, \operatorname{Top}/O_{(2)}] \longrightarrow [\mathbb{S}^{16}, \operatorname{Top}/O_{(2)}] \] 
is indeed the trivial map, which concludes the proof of Part (1).
We turn to the proof of Part (2). By Remark~\ref{group}(D), we have 
\[
[\mathbb{C}P^{8},\operatorname{Top}/O]
\cong
\mathbb{Z}_{2}\{(\eta^{*})_{16}\}
\oplus
\mathbb{Z}_{2}\{\alpha_{16}\},
\]
where $\alpha_{16} = \left(\overline{(\kappa)_{14}+\left(\overline{(\epsilon)_{8}}\right)_{14}}\right)_{16}$. Consider the commutative diagram:
\begin{equation}\label{hopfrpn}
\xymatrix@C=3em@R=3em{
    [\mathbb{C}P^8,\operatorname{Top}/O] \ar[r]^-{\pi^*} \ar[rd]_-{p^*}
    &
    [\mathbb{R}P^{17},\operatorname{Top}/O] \ar[d]^{\overline{p}^*}
    \\
    &
    [\mathbb{S}^{17},\operatorname{Top}/O]
}
\end{equation}
where $\overline{p}:\mathbb{S}^{17}\to\mathbb{R}P^{17}$ is the quotient map and $p:\mathbb{S}^{17}\to\mathbb{C}P^8$ is the Hopf fibration. By Lemma \ref{cor-cp8pl}, $\mathrm{Ker}\!\left( [\mathbb{C}P^8,\operatorname{Top}/O]\xrightarrow{p^*}[\mathbb{S}^{17},\operatorname{Top}/O] \right) = \mathbb{Z}_2\{(\eta^*)_{16}\}$, and moreover $p^*(\alpha_{16})\neq 0$ (see \eqref{nontrivialge}). Since Diagram \eqref{hopfrpn} commutes, it follows that 
\begin{equation}\label{pi-alpha-nonzero}
\pi^*(\alpha_{16}) \neq 0 \quad \text{and} \quad \pi^*((\eta^*)_{16}+\alpha_{16}) \neq 0.
\end{equation}
From Lemma~3.2 of \cite{Muk00}, the composite $\mathbb{R}P^{16}\hookrightarrow\mathbb{R}P^{17} \xrightarrow{\pi}\mathbb{C}P^8 \xrightarrow{f_{\mathbb{C}P^8}} \mathbb{S}^{16}$ is homotopic to the collapse map $f_{\mathbb{R}P^{16}}:\mathbb{R}P^{16}\to \mathbb{S}^{16}$. Combining this with Part~(1), we obtain $\pi^*((\eta^*)_{16}) = f_{\mathbb{R}P^{16}}^*(\eta^*) \neq 0$. Now \eqref{pi-alpha-nonzero} shows that $\pi^*(\alpha_{16}) \neq \pi^*((\eta^*)_{16}+\alpha_{16})$. Thus the two generators $(\eta^*)_{16}$ and $\alpha_{16}$ have distinct non-trivial images under $\pi^*$. This implies that $\pi^*:[\mathbb{C}P^8,\operatorname{Top}/O]\longrightarrow[\mathbb{R}P^{17},\operatorname{Top}/O]$ is injective and $\pi^*((\eta^*)_{16})=f_{\mathbb{R}P^{16}}^*(\eta^*)$. This completes the proof of Part~(2).
\end{proof}
Since the complex conjugation 
\(C \colon \mathbb{C}P^{n} \longrightarrow \mathbb{C}P^{n}\) 
preserves the subspace \(\mathbb{C}P^{k}\), it induces a self–homotopy equivalence
\[
\widetilde{C}_{n,k} \colon \mathbb{C}P^{n}/\mathbb{C}P^{k} 
\longrightarrow 
\mathbb{C}P^{n}/\mathbb{C}P^{k}.
\]
Note that the restriction of \(\widetilde{C}_{n,k}\) to the bottom cell
\begin{equation}\label{degreeref}
\mathbb{S}^{2k+2} \longrightarrow \mathbb{S}^{2k+2}
\end{equation}
is homotopic to the identity map if \(k\) is odd, and to the reflection map if \(k\) is even.
\begin{theorem}\label{conju-cp8}
\begin{itemize}
    \item[(1)] The induced map $C^{*}\colon [\mathbb{C}P^{8},\operatorname{Top}/O] \to [\mathbb{C}P^{8},\operatorname{Top}/O]$ is the identity map, where
    \[
    [\mathbb{C}P^{8},\operatorname{Top}/O]
    =
    \mathbb{Z}_{2}\{(\eta^{*})_{16}\}
    \oplus
    \mathbb{Z}_{2}\left\{\left(\overline{(\kappa)_{14}+\left(\overline{(\epsilon)_{8}}\right)_{14}}\right)_{16}\right\}.
    \] 

    \item[(2)] The induced map
    \[
    (\widetilde{C}_{8,6})^{*}\colon 
    [\mathbb{C}P^{8}/\mathbb{C}P^{6},SF] \longrightarrow [\mathbb{C}P^{8}/\mathbb{C}P^{6},SF]
    \]
    is given by $(\widetilde{C}_{8,6})^{*}(x)=-x$, where $[\mathbb{C}P^{8}/\mathbb{C}P^{6},SF] = \mathbb{Z}_{4}\Bigl\{\bigl(\overline{(\sigma^{2})_{14}}\bigr)_{16}\Bigr\}$.

    \item[(3)] The induced map
    \[
    (\widetilde{C}_{4,2})^{*} \colon [\mathbb{C}P^{4}/\mathbb{C}P^{2}, SF] \longrightarrow [\mathbb{C}P^{4}/\mathbb{C}P^{2}, SF]
    \]
    is given by $(\widetilde{C}_{4,2})^{*}(x)=-x$. Here, $[\mathbb{C}P^{4}/\mathbb{C}P^{2}, SF] \cong \mathbb{Z}_4 \left\{ \left( \overline{(\nu^{2})_{6}} \right)_{8} \right\}$.
\end{itemize}
\end{theorem}
\begin{proof}
Consider the following commutative diagram:
\begin{equation}\label{ref-con}
\xymatrix@C=4.5em @R=4em{
[\mathbb{R}P^{16}, \operatorname{Top}/O] 
\ar[r]^{\mathrm{Id}}
& 
[\mathbb{R}P^{16}, \operatorname{Top}/O] 
\\
[\mathbb{R}P^{17}, \operatorname{Top}/O] 
\ar[u]^{i^*}
\ar[r]^{r^*}
& 
[\mathbb{R}P^{17}, \operatorname{Top}/O]
\ar[u]^{i^*}
\\
[\mathbb{C}P^{8}, \operatorname{Top}/O]
\ar[u]^{\pi^*}
\ar[r]^{C^{*}}
& 
[\mathbb{C}P^{8}, \operatorname{Top}/O]
\ar[u]^{\pi^*}
}
\end{equation}
where $r \colon \mathbb{R}P^{17} \to \mathbb{R}P^{17}$ is the reflection and 
$\pi \colon \mathbb{R}P^{17} \to \mathbb{C}P^{8}$ is the natural fibration.
The commutativity of the top square follows from the fact that any self-homotopy equivalence of $\mathbb{R}P^{16}$ is homotopic to the identity. The commutativity of the bottom square follows from the fact that the composition $C \circ p \colon \mathbb{S}^{17} \to \mathbb{C}P^{8}$ is homotopic to $(-1) \circ p \colon \mathbb{S}^{17} \to \mathbb{C}P^{8}$, where $(-1) \colon \mathbb{S}^{17} \to \mathbb{S}^{17}$ is the reflection map and $p \colon \mathbb{S}^{17} \to \mathbb{C}P^{8}$ is the Hopf fibration.
Recall from Remark~\ref{group}(D) that 
\[ [\mathbb{C}P^{8},\operatorname{Top}/O] \cong \mathbb{Z}_{2}\{(\eta^{*})_{16}\} \oplus \mathbb{Z}_{2}\{\alpha_{16}\}, \]
where $\alpha_{16} = \left(\overline{(\kappa)_{14}+\left(\overline{\epsilon}\right)_{14}}\right)_{16}$.
It follows from Remark~\eqref{actions}(1) that $C^{*}((\eta^*)_{16}) = (\eta^*)_{16}$.
Hence $C^{*}(\alpha_{16})$ must be either $\alpha_{16}$ or $(\eta^*)_{16} + \alpha_{16}$. We now show that the second possibility cannot occur.
By Lemma~\ref{RP16-CP8-TopO}, the composition
$i^* \circ \pi^* \colon [\mathbb{C}P^{8}, \operatorname{Top}/O] \longrightarrow [\mathbb{R}P^{16}, \operatorname{Top}/O]$
maps $(\eta^*)_{16}$ nontrivially; specifically,
$(i^* \circ \pi^*)((\eta^*)_{16}) = f_{\mathbb{R}P^{16}}^*(\eta^{*}) \neq 0$.
Consequently, 
\[ (i^* \circ \pi^*)(\alpha_{16}) \neq (i^* \circ \pi^*) \left( (\eta^*)_{16} + \alpha_{16} \right). \]
Applying this to the commutativity of Diagram~\eqref{ref-con}, a diagram chase shows that 
\[ C^*(\alpha_{16}) = \alpha_{16}. \]
Thus, $C^*$ fixes all the generators of $[\mathbb{C}P^{8}, \operatorname{Top}/O]$, which completes the proof of Part~(1). We proceed to the proof of Part (2). It follows from the proof of Lemma~\ref{eight-sub}(i) that
\[
\mathbb{C}P^{8}/\mathbb{C}P^{6} \simeq \mathbb{S}^{14} \cup_{\eta} \mathbb{D}^{16} \cong \Sigma^{12}\mathbb{C}P^{2},
\]
and 
\[
[\mathbb{C}P^{8}/\mathbb{C}P^{6}, SF] = \mathbb{Z}_{4}\Bigl\{\bigl(\overline{(\sigma^{2})_{14}}\bigr)_{16}\Bigr\},
\]
where $\bigl(\overline{(\sigma^{2})_{14}}\bigr)_{16}$ denotes an extension of $(\sigma^{2})_{14}$ under the restriction map
\[
i^{*}\colon [\mathbb{C}P^{8}/\mathbb{C}P^{6}, SF] \longrightarrow [\mathbb{C}P^{7}/\mathbb{C}P^{6}, SF] \cong [\mathbb{S}^{14}, SF].
\]
As noted in \eqref{degreeref}, the restriction of 
\[ 
\widetilde{C}_{8,6} \colon \mathbb{S}^{14} \cup_{\eta} \mathbb{D}^{16} \longrightarrow \mathbb{S}^{14} \cup_{\eta} \mathbb{D}^{16} 
\]
to the bottom cell is homotopic to the reflection map $(-1) \colon \mathbb{S}^{14} \to \mathbb{S}^{14}$. Consequently, the composition of the natural inclusion $i \colon \mathbb{S}^{14} \hookrightarrow \mathbb{S}^{14} \cup_{\eta} \mathbb{D}^{16}$ with $\widetilde{C}_{8,6}$ is homotopic to $i \circ (-1)$. 

Applying this fact, together with the observation that $\pi_{16}(\mathbb{S}^{14}) / \eta \pi_{16}(\mathbb{S}^{15}) = 0$, to Proposition~2.4(1) and Lemma~2.1(1) of \cite{KMNST01}, it follows that the self-map $\widetilde{C}_{8,6}$ is homotopic to
\[
-(\Sigma^{12}\iota_{\mathbb{C}}) + k\bigl(\widetilde{2\iota_{15}} \circ \Sigma^{12}f_{\mathbb{C}P^{2}}\bigr) \quad \text{for some } k \in \mathbb{Z},
\]
where $-(\Sigma^{12}\iota_{\mathbb{C}})$ denotes the additive inverse of the homotopy class of the identity map of $\Sigma^{12}\mathbb{C}P^{2}$ in the group $[\Sigma^{12}\mathbb{C}P^{2}, \Sigma^{12}\mathbb{C}P^{2}]$. Thus, we obtain the induced map:
\begin{equation}\label{iota-class}
(\widetilde{C}_{8,6})^{*} = -(\Sigma^{12}\iota_{\mathbb{C}})^{*} + k\bigl(\widetilde{2\iota_{15}} \circ \Sigma^{12}f_{\mathbb{C}P^{2}}\bigr)^{*}.
\end{equation}

We now demonstrate that the composite 
\[
\bigl(\widetilde{2\iota_{15}} \circ \Sigma^{12}f_{\mathbb{C}P^{2}}\bigr)^{*}\colon [\Sigma^{12}\mathbb{C}P^{2}, \operatorname{Top}/O] \xrightarrow{\;\widetilde{2\iota_{15}}^{*}\;} [\mathbb{S}^{16}, \operatorname{Top}/O] \xrightarrow{\;(\Sigma^{12}f_{\mathbb{C}P^{2}})^{*}\;} [\Sigma^{12}\mathbb{C}P^{2}, \operatorname{Top}/O]
\]
is trivial. Since $\bigl(\overline{(\sigma^{2})_{14}}\bigr)_{16}$ generates $[\Sigma^{12}\mathbb{C}P^{2}, \operatorname{Top}/O]$, we have
\[
\widetilde{2\iota_{15}}^{*} \left( \bigl(\overline{(\sigma^{2})_{14}}\bigr)_{16} \right) = \bigl(\overline{(\sigma^{2})_{14}}\bigr)_{16} \circ \widetilde{2\iota_{15}} \in \langle \sigma^{2}, \eta, 2 \rangle
\]
by \eqref{toda-eta}. This Toda bracket vanishes, as shown in the proof of Lemma~\ref{eight-sub}(i). This implies that the map
\[
\widetilde{2\iota_{15}}^{*}\colon [\Sigma^{12}\mathbb{C}P^{2}, \operatorname{Top}/O] \longrightarrow [\mathbb{S}^{16}, \operatorname{Top}/O]
\]
is trivial, which in turn ensures that the composite $\bigl(\widetilde{2\iota_{15}} \circ \Sigma^{12}f_{\mathbb{C}P^{2}}\bigr)^{*}$ is also trivial.

Substituting this result into \eqref{iota-class}, we conclude that
\[
(\widetilde{C}_{8,6})^{*} = -(\Sigma^{12}\iota_{\mathbb{C}})^{*} = -(\mathrm{Id})^{*},
\]
where $\mathrm{Id} \colon \mathbb{S}^{14} \cup_{\eta} \mathbb{D}^{16} \to \mathbb{S}^{14} \cup_{\eta} \mathbb{D}^{16}$ is the identity map. This completes the proof of Part (2). For the proof of Part (3), note that
\[
\mathbb{C}P^{4}/\mathbb{C}P^{2} \simeq \mathbb{S}^{6} \cup_{\eta} \mathbb{D}^{8}=\Sigma^{4}\mathbb{C}P^{2}.
\]

Consider the cofibre sequence
\[
\mathbb{S}^{7} \xrightarrow{\eta} \mathbb{S}^{6} \xrightarrow{\Sigma^{4}i_{\mathbb{C}}} \mathbb{S}^{6} \cup_{\eta} \mathbb{S}^{8} \xrightarrow{\Sigma^{4}f_{\mathbb{C}P^{2}}} \mathbb{S}^{8},
\]
where \(i_{\mathbb{C}}:\mathbb{S}^2\hookrightarrow \mathbb{C}P^{2}\) is the inclusion. Applying the functor $[-, SF]$ to this cofibre sequence, and using the facts that $\eta^{*}:[\mathbb{S}^{6},SF]\to [\mathbb{S}^{7},SF]$ is trivial and that the image of $\eta^{*}:[\mathbb{S}^{7},SF]\to [\mathbb{S}^{8},SF]$ is $\mathbb{Z}_{2}\{\eta \circ \sigma\}$ (since $\eta \circ \nu^2=0$ and $\eta \circ \sigma\neq 0$; see \cite{Tod62}), we obtain the following short exact sequence:
\begin{equation}\label{exct-cp4-cp2}
0 \longrightarrow  \mathbb{Z}_{2}\{x\} \xrightarrow{(\Sigma^{4}f_{\mathbb{C}P^{2}})^*} [\Sigma^{4}\mathbb{C}P^{2}, SF] \xrightarrow{(\Sigma^{4}i_{\mathbb{C}})^*} [\mathbb{S}^{6}, SF] \cong \mathbb{Z}_{2}\{\nu^{2}\} \longrightarrow 0,
\end{equation}
where $x\in \{\epsilon, \epsilon+\eta \circ \sigma\}$.
Let $\overline{\nu^2}\in [\Sigma^{4}\mathbb{C}P^{2}, SF]$ be an extension of $\nu^{2}\in [\mathbb{S}^{6}, SF]$ along the map $(\Sigma^{4}i_{\mathbb{C}})^{*}$. Applying the same technique as in \eqref{todaeq}, we obtain
\[
2\overline{\nu^2}=\nu^{2} \circ \overline{2 \iota_{6}} + \overline{\nu^{2}} \circ \widetilde{2 \iota_{7}} \circ \Sigma^{4} f_{\mathbb{C}P^{2}},
\]
where
\[
\nu^{2} \circ \overline{2 \iota_{6}} \in \langle \nu^{2}, 2 \iota_{6}, \eta \rangle \circ \Sigma^{4} f_{\mathbb{C}P^{2}},
\]
and
\[
\overline{\nu^{2}} \circ \widetilde{2 \iota_{7}} \in \langle \nu^{2}, \eta, 2 \iota_{7} \rangle.
\]
Since $\langle \nu^{2}, \eta, 2 \rangle=0$ and $\langle \nu^{2}, 2, \eta \rangle=\{\epsilon, \epsilon+\eta \circ \sigma\}$ by \cite[Theorem 2.1(ii), p.~68]{Muk69}, it follows that
\[
2\overline{\nu^2}=\nu^{2} \circ \overline{2 \iota_{6}}=x\circ \Sigma^{4} f_{\mathbb{C}P^{2}}=(\Sigma^{4}f_{\mathbb{C}P^{2}})^{*}(x),
\]
which is non-zero in $[\Sigma^{4}\mathbb{C}P^{2}, SF]$ by \eqref{exct-cp4-cp2}.
This shows that any extension $\overline{\nu^{2}}$ of $\nu^{2}$ has order $4$. Hence, by \eqref{exct-cp4-cp2},
\begin{equation}\label{cp4-cp2}
[\mathbb{C}P^{4}/\mathbb{C}P^{2}, SF]\cong [\Sigma^{4}\mathbb{C}P^{2},SF]\cong \mathbb{Z}_4\{\overline{\nu^{2}}\},
\end{equation}
where $2\overline{\nu^{2}} = (\Sigma^{4}f_{\mathbb{C}P^{2}})^{*}(x)$ for $x \in \{\epsilon, \epsilon+\eta \circ \sigma\}$.
Since the restriction of the map
\[
\widetilde{C}_{4,2} \colon \mathbb{S}^{6} \cup_{\eta} \mathbb{S}^{8} \longrightarrow \mathbb{S}^{6} \cup_{\eta} \mathbb{S}^{8}
\]
to the bottom cell is homotopic to the reflection map $\mathbb{S}^{6} \to \mathbb{S}^{6}$, we may apply the same arguments as in Part (2). Thus, $\widetilde{C}_{4,2}$ is homotopic to
\[
-(\Sigma^{4}\iota_{\mathbb{C}}) + s\bigl(\widetilde{2\iota_{7}} \circ \Sigma^{4}f_{\mathbb{C}P^{2}}\bigr), \qquad s \in \mathbb{Z},
\]
and hence
\begin{equation}\label{iota-class1}
(\widetilde{C}_{4,2})^{*} = -(\Sigma^{4}\iota_{\mathbb{C}})^{*} + s\bigl(\widetilde{2\iota_{7}} \circ \Sigma^{4}f_{\mathbb{C}P^{2}}\bigr)^{*}.
\end{equation}
Since $\overline{\nu^{2}}$ generates $[\Sigma^{4}\mathbb{C}P^{2}, SF]$ and is an extension of $\nu^{2}$, applying the same arguments as in Part (2) gives
\[
\widetilde{2\iota_{7}}^{*}\left(\overline{\nu^{2}}\right)
= \overline{\nu^{2}} \circ \widetilde{2\iota_{7}}
\in \langle \nu^{2}, \eta, 2 \rangle = 0.
\]
Therefore, the map
\[
\widetilde{2\iota_{7}}^{*} \colon [\Sigma^{4}\mathbb{C}P^{2}, SF] \to [\mathbb{S}^{8}, SF]
\]
is trivial. Hence, the composition
\[
s\bigl(\widetilde{2\iota_{7}} \circ \Sigma^{4}f_{\mathbb{C}P^{2}}\bigr)^{*}
\]
is also trivial. Substituting this into \eqref{iota-class1}, we obtain
\[
(\widetilde{C}_{4,2})^{*} = -(\Sigma^{4}\iota_{\mathbb{C}})^{*} = -(\mathrm{Id})^{*}.
\]
Together with \eqref{cp4-cp2}, this completes the proof of Part (3), since the map 
\[
f_{\mathbb{C}P^4 / \mathbb{C}P^2} : \mathbb{C}P^4 / \mathbb{C}P^2 \simeq \mathbb{S}^{6} \cup_{\eta} \mathbb{D}^{8} \longrightarrow \mathbb{S}^{8}
\]
is homotopic to \(\Sigma^{4} f_{\mathbb{C}P^{2}}\), and the inclusion 
\[
i : \mathbb{S}^{6} \hookrightarrow \mathbb{S}^{6} \cup_{\eta} \mathbb{D}^{8}
\]
is homotopic to \(\Sigma^{4} i_{\mathbb{C}}\).
\end{proof}
\begin{remark}\label{pl-top}\rm
It follows from \eqref{equ21} and \eqref{equ22} that the induced maps $\beta_{*} \colon [\mathbb{C}P^{m}, PL] \to [\mathbb{C}P^{m}, PL/O]$ and $F_{*} \colon [\mathbb{C}P^{m}, PL/O] \to [\mathbb{C}P^{m}, \operatorname{Top}/O]$ are isomorphisms. Consequently, the induced maps $C^{*} \colon [\mathbb{C}P^{m}, PL] \to [\mathbb{C}P^{m}, PL]$ and $C^{*} \colon [\mathbb{C}P^{m}, PL/O] \to [\mathbb{C}P^{m}, PL/O]$ satisfy the same formulas as the map $C^{*} \colon [\mathbb{C}P^{m}, \operatorname{Top}/O] \to [\mathbb{C}P^{m}, \operatorname{Top}/O]$ provided in Theorem~\ref{main2} and Theorem~\ref{conju-cp8}(1) for $4 \leq m \leq 8$.
\end{remark}
For $\alpha \in \mathcal{C}(\mathbb{C}P^{m})$, we denote by $\left[\alpha\right]$ the (both oriented and unoriented) diffeomorphism class of $\alpha$. If $\alpha, \gamma \in \mathcal{C}(\mathbb{C}P^{m})$, then $\alpha \cong_{+} \gamma$ and $\alpha \cong \gamma$ denote that $\alpha$ and $\gamma$ are orientation-preserving diffeomorphic and diffeomorphic as unoriented manifolds, respectively. 

With these notational conventions, applying Theorem~\ref{main2} and Theorem~\ref{conju-cp8}(1) to the action in \eqref{act1}, together with \eqref{iso2}, \eqref{forget}, and Remark~\eqref{group}(A) and (D), we obtain the following description of the orientation-preserving diffeomorphism classes.
\begin{theorem}\label{main3}
\begin{itemize}
\item[(i)]
For odd $m$, the map
\[
F_{\mathcal{C}} \colon \mathcal{C}(\mathbb{C}P^{m}) \longrightarrow \mathcal{M}^{+}_{\mathrm{Diff}}(\mathbb{C}P^{m})
\]
is bijective. In particular, for $m=5$, the set is given by
\[
\mathcal{M}^{+}_{\mathrm{Diff}}(\mathbb{C}P^{5}) 
= \mathbb{Z}_2\left\{\left[(\eta \circ \mu)_{10}\right]\right\} 
\oplus \mathbb{Z}_3\left\{\left[(\beta_1)_{10}\right]\right\} 
\oplus \mathbb{Z}_2\left\{\left[\left(\overline{(\epsilon)_{8}}\right)_{10}\right]\right\},
\]
where
\[
\mathbb{Z}_2\left\{\left[(\eta \circ \mu)_{10}\right]\right\} 
\oplus \mathbb{Z}_3\left\{\left[(\beta_1)_{10}\right]\right\}
=
\left\{\left[\mathbb{C}P^{5} \# \Sigma \right] \mid \Sigma \in \Theta_{10}\right\}.
\]
For $m=7$, we have
\[
\mathcal{M}^{+}_{\mathrm{Diff}}(\mathbb{C}P^{7}) 
= \mathbb{Z}_{2}\left\{[(\kappa)_{14}]\right\}
\oplus
\mathbb{Z}_{2}\left\{\left[\left(\overline{(\epsilon)_{8}}\right)_{14}\right]\right\},
\]
where
\[
\mathbb{Z}_{2}\left\{[(\kappa)_{14}]\right\}
=
\left\{\left[\mathbb{C}P^{7} \# \Sigma \right] \mid \Sigma \in \Theta_{14}\right\}.
\]
\item[(ii)]
\[
\mathcal{M}^{+}_{\mathrm{Diff}}(\mathbb{C}P^{6})
=
\mathbb{Z}_{2}\left\{\left[\left(\overline{(\epsilon)_{8}}\right)_{12}\right]\right\}
\oplus 
\mathbb{Z}_{2}\left\{\left[\left(\overline{(\beta_{1})_{10}}\right)_{12}\right]\right\}
\]
as sets, where
\[
-\left(\overline{(\beta_{1})_{10}}\right)_{12} \cong_{+} \left(\overline{(\beta_{1})_{10}}\right)_{12}.
\]
\item[(iii)]
\[
\mathcal{M}^{+}_{\mathrm{Diff}}(\mathbb{C}P^{8})
=
\mathbb{Z}_{2}\left\{\left[(\eta^{*})_{16}\right]\right\}
\oplus
\mathbb{Z}_{2}\left\{
\left[
\left(\overline{(\kappa)_{14} + \left(\overline{(\epsilon)_{8}}\right)_{14}}\right)_{16}
\right]
\right\}
\]
as sets, where
\[
\mathbb{Z}_{2}\left\{(\eta^{*})_{16}\right\}
=
\left\{\left[\mathbb{C}P^{8} \# \Sigma^{16}\right] \mid \Sigma^{16} \in \Theta_{16}\right\}.
\]
\end{itemize}
\end{theorem}
\begin{remark}\label{rema-un}\rm
Let $\mathcal{M}^{\pm}_{\mathrm{Diff}}(\mathbb{C}P^{m})$ denote the set of unoriented diffeomorphism classes of smooth manifolds homeomorphic to $\mathbb{C}P^{m}$. Then
\[
\mathcal{M}^{\pm}_{\mathrm{Diff}}(\mathbb{C}P^{m})
=
\mathcal{M}^{+}_{\mathrm{Diff}}(\mathbb{C}P^{m}) / \{\pm 1\},
\]
where the group $\mathbb{Z}_{2} = \{\pm 1\}$ acts on $\mathcal{M}^{+}_{\mathrm{Diff}}(\mathbb{C}P^{m})$ by orientation reversal.
\begin{itemize}
    \item[(1)] Since $\mathbb{C}P^{2n}$ does not admit a self-homeomorphism that reverses orientation, the action of $\mathbb{Z}_{2} = \{\pm 1\}$ on $\mathcal{M}^{+}_{\mathrm{Diff}}(\mathbb{C}P^{2n})$ is trivial. Hence,
    \[
    \mathcal{M}^{\pm}_{\mathrm{Diff}}(\mathbb{C}P^{2n})
    =
    \mathcal{M}^{+}_{\mathrm{Diff}}(\mathbb{C}P^{2n}).
    \]

    \item[(2)] Note that the mapping class group $MCG^{\pm}(\mathbb{C}P^{2n+1}) = \{[\mathrm{Id}], [C]\} \cong \mathbb{Z}_2$ acts on the concordance set $\mathcal{C}(\mathbb{C}P^{2n+1})$ by
    \[
    [C] \cdot [N, f] = [N, C \circ f].
    \]
    Moreover, the forgetful map
    \[
    F_{\mathcal{C}} : \mathcal{C}(\mathbb{C}P^{2n+1}) \longrightarrow \mathcal{M}^{\pm}_{\mathrm{Diff}}(\mathbb{C}P^{2n+1}),
    \]
    defined by $[f : N \to M] \mapsto [N]$ (the unoriented diffeomorphism class), induces a bijection 
    \[\mathcal{C}(\mathbb{C}P^{2n+1}) / MCG^{\pm}(\mathbb{C}P^{2n+1})
    \;\cong\;
    \mathcal{M}^{\pm}_{\mathrm{Diff}}(\mathbb{C}P^{2n+1}).\]
It follows from \eqref{for} that
\begin{equation}
[N, C \circ f] = C^*([N, f]).
\end{equation}
\end{itemize}
\end{remark}
Combining Remark~\ref{rema-un} and Theorem~\ref{main2}, we obtain the following result.
\begin{theorem}\label{main4}
\begin{itemize}

\item[(i)]
For $m$ even,
\[
\mathcal{M}^{\pm}_{\rm Diff}(\mathbb{C}P^{m})
=
\mathcal{M}^{+}_{\rm Diff}(\mathbb{C}P^{m}).
\]

\item[(ii)]
\[
\mathcal{M}^{\pm}_{\mathrm{Diff}}(\mathbb{C}P^{5})
=
\mathbb{Z}_{2}\{[(\eta\circ\mu)_{10}]\}
\oplus \mathbb{Z}_{2}\{[(\beta_{1})_{10}]\}
\oplus
\mathbb{Z}_{2}\{[\left(\overline{(\epsilon)_{8}}\right)_{10}]\}, \text{ as sets.}
\]
Here $(\eta\circ\mu)_{10}$ represents the unoriented diffeomorphism class
$[\mathbb{C}P^{5}\#\Sigma_{\eta\circ\mu}]$, where
$\Sigma_{\eta\circ\mu}\in\Theta_{10}$ is the unique element of order $2$.
The generator $(\beta_{1})_{10}$ represents the class
$[\mathbb{C}P^{5}\#\Sigma_{\beta_{1}}]$, where
$\Sigma_{\beta_{1}}\in\Theta_{10}$ has order $3$. $(\beta_{1})_{10}\cong -(\beta_{1})_{10}$ 
\item[(iii)]
\[
\mathcal{M}^{\pm}_{\mathrm{Diff}}(\mathbb{C}P^{7})
=
\mathbb{Z}_{2}\{[(\kappa)_{14}]\}
\oplus
\mathbb{Z}_{2}\{[\left(\overline{(\epsilon)_{8}}\right)_{14}]\}, \text{ as sets,}
\]
where $\mathbb{Z}_{2}\{(\kappa)_{14}\}
=\{[\mathbb{C}P^{7}\#\Sigma^{14}]\mid \Sigma^{14}\in \Theta_{14}\}.$
\end{itemize}
\end{theorem}
\begin{remark}\rm
\begin{enumerate}
    \item Theorem~\ref{main4}(i) and Theorem~\ref{main3}(i) determine the set $\mathcal{M}^{\pm}_{\mathrm{Diff}}(\mathbb{C}P^{m})$ completely for $m=6$ and $m=8$.
    
    \item It follows from Theorem~\ref{main3}(i) and Theorem~\ref{main4}(ii) that the manifolds $\mathbb{C}P^{5} \# \Sigma^{10}$ and $\mathbb{C}P^{5} \# (-\Sigma^{10})$ are diffeomorphic, but not orientation-preservingly diffeomorphic, where $\Sigma^{10}$ corresponds to the element $\beta_{1} \in \Theta_{10} \cong \mathbb{Z}_{2}\{\eta \circ \mu\} \oplus \mathbb{Z}_{3}\{\beta_{1}\}$.
\end{enumerate}
\end{remark}
\section{The smooth tangential structure set \texorpdfstring{$\mathcal{S}^{t}_{\mathrm{Diff}}(\mathbb{C}P^{m})$}{}}\label{4}

In this section, the smooth tangential structure set \(\mathcal{S}^t_{\mathrm{Diff}}(\mathbb{C}P^{m})\) (see \cite[p.~103]{CH15} for definition and details) is determined for \(m=7\) and \(8\). The main tool for computing the smooth tangential structure set \(\mathcal{S}^t_{\mathrm{Diff}}(M)\) of an \(n\)-dimensional simply connected closed smooth manifold \(M\) with \(n\ge 5\) is the surgery exact sequence
\begin{equation}\label{tsurgery}
L_{n+1}(e)\xrightarrow{\;\theta\;}\mathcal{S}^t_{\mathrm{Diff}}(M)
\xrightarrow{\;\eta^t\;}[M,SF]\xrightarrow{\;s\;}L_n(e),
\end{equation}
where \(s\colon[M,SF]\to L_{n}(e)\) is the tangential surgery obstruction map (see \cite[\S6]{CH15}). Recall that there is a forgetful map \(F_{t}\colon \mathcal{S}^t_{\mathrm{Diff}}(M)\longrightarrow \mathcal{S}_{\rm Diff}(M)\), where \(\mathcal{S}_{\rm Diff}(M)\) is the homotopy surgery smooth structure set of \(M\). Let \(\mathcal{M}^{\pm}_{\mathrm{hDiff}}(M)\) denote the set of diffeomorphism classes of smooth manifolds homotopy equivalent to \(M\). This set can be identified with the quotient \(\mathcal{S}_{\mathrm{Diff}}(M)/\mathrm{haut}(M)\), where \(\mathrm{haut}(M)\) is the group of homotopy classes of self-homotopy equivalences of \(M\), acting on \(\mathcal{S}_{\mathrm{Diff}}(M)\) by post-composition.

Since \(L_{\mathrm{odd}}(e)=0\) \cite{Wal99}, the tangential surgery exact sequence \eqref{tsurgery} for \(M=\mathbb{C}P^{m}\) takes the form
\begin{equation}\label{taur}
0\longrightarrow \mathcal{S}^t_{\mathrm{Diff}}(\mathbb{C}P^{m})
\longrightarrow [\mathbb{C}P^{m},SF]\xrightarrow{\;s\;}L_{2m}(e),
\end{equation}
where \(s\colon[\mathbb{C}P^{m},SF]\to L_{2m}(e)\) is trivial for \(m\) even (see \cite[Lemma~I.5(i)]{Bru71}). The surgery obstruction maps \(s\colon[\mathbb{C}P^{m},SF]\to L_{2m}(e)\) were computed by Brumfiel \cite[Lemma~I.5 and I.6]{Bru71} for \(m\le 6\); in particular, \(s\) is an isomorphism for \(m=3\), while \(s=0\) for \(m=4, 5, 6\). Applying these results to the exact sequence \eqref{taur}, we obtain:

\begin{proposition}[Brumfiel {\cite{Bru71,Bru68}}]\label{lowbru}
\indent
\begin{itemize}
    \item[(i)] \(\mathcal{S}^t_{\mathrm{Diff}}(\mathbb{C}P^{3})\cong \{[\mathbb{C}P^{3}]\}\).
    \item[(ii)] For \(m=4, 5, 6, 8\), the normal invariant \(\eta^{t}\colon\mathcal{S}^t_{\mathrm{Diff}}(\mathbb{C}P^{m})\longrightarrow [\mathbb{C}P^{m},SF]\) is bijective.
\end{itemize}
\end{proposition}
Using the computations of \([\mathbb{C}P^{m}, SF]\) provided in Lemma \eqref{lem:Bru-seq}, we describe the generators of \([\mathbb{C}P^{m}, SF]\) for \(3 \le m \le 6\), adopting the notation from Remark~\ref{group}:
\begin{equation}\label{equtan_final}
\begin{gathered}
[\mathbb{C}P^{3}, SF] = \mathbb{Z}_{2}\{(\nu^{2})_{6}\}, \quad  [\mathbb{C}P^{4}, SF] = \mathbb{Z}_{4}\{\left(\overline{(\nu^{2})_{6}}\right)_{8}\} \\[1.5ex]
\quad  
[\mathbb{C}P^{5}, SF] = \mathbb{Z}_2\{(\eta \circ \mu)_{10}\} \oplus \mathbb{Z}_3\{(\beta_1)_{10}\} \oplus \mathbb{Z}_2\{\left (\overline{(\epsilon)_{8}}\right )_{10}\} \\[1.5ex]
[\mathbb{C}P^{6}, SF] = \mathbb{Z}_3\{\left(\overline{(\beta_1)_{10}}\right )_{12}\} \oplus \mathbb{Z}_2\{\left (\overline{(\epsilon)_{8}}\right )_{12}\}.
\end{gathered}
\end{equation}
For the case \(m=7\), we have the following result.
\begin{proposition}\label{sur-obst}
The surgery obstruction map
\[
s\colon [\mathbb{C}P^{7},SF]\cong \mathbb{Z}_{2}\{\left (\overline{(\epsilon)_{8}}\right )_{14}\} \oplus \mathbb{Z}_{2}\{(\kappa)_{14}\} \oplus \mathbb{Z}_{2}\{(\sigma^{2})_{14}\}\longrightarrow \mathbb{Z}_2
\]
is surjective. In particular,
\[
\Ker\left([\mathbb{C}P^{7}, SF]\xrightarrow{s}L_{14}(e)\cong \mathbb{Z}_{2}\right) = \mathbb{Z}_2\{(\kappa)_{14}\} \oplus \mathbb{Z}_2\{\left (\overline{(\epsilon)_{8}}\right )_{14}\}.
\]
\end{proposition}

\begin{proof}
The group of smooth homotopy \(14\)-spheres is \(\Theta_{14}\cong \mathbb{Z}_2\{\kappa\}\), and the surgery obstruction map
\[
s\colon \pi_{14}(F/O) \cong \mathbb{Z}_{2}\{\kappa\} \oplus \mathbb{Z}_{2}\{\sigma^{2}\} \longrightarrow \mathbb{Z}_2
\]
maps both \(\sigma^{2}\) and \(\sigma^{2}+\kappa\) to \(1 \in \mathbb{Z}_2\) (see \eqref{kappagenrator}); hence, it is surjective. It follows from the commutativity of Diagram~\eqref{digram31} and  \eqref{equ3} that the composite map
\[
s\colon [\mathbb{C}P^{7}, SF] \hookrightarrow [\mathbb{C}P^{7}, F/O] \longrightarrow L_{14}(e) \cong \mathbb{Z}_2
\]
is also surjective. From Corollary~\ref{cp7-sf} and \eqref{cp7-top}, we have:
\[
[\mathbb{C}P^{7},SF]=\mathbb{Z}_{2}\{\left (\overline{(\epsilon)_{8}}\right )_{14}\} \oplus \mathbb{Z}_{2}\{(\kappa)_{14}\} \oplus \mathbb{Z}_{2}\{(\sigma^{2})_{14}\}
\]
and
\[
[\mathbb{C}P^{7},\operatorname{Top}/O]=\mathbb{Z}_2\{(\kappa)_{14}\} \oplus \mathbb{Z}_2\{\left (\overline{(\epsilon)_{8}}\right )_{14}\}.
\]
Moreover, by \eqref{equ2} and \eqref{equ3}, we have the inclusion
\[
[\mathbb{C}P^{7},\operatorname{Top}/O]\cong \operatorname{Im}\left([\mathbb{C}P^{7},\operatorname{Top}/O]\xrightarrow{\psi_{*}}[\mathbb{C}P^{7},F/O]\right) \subseteq [\mathbb{C}P^{7}, SF].
\]
Since \([\mathbb{C}P^{7},\operatorname{Top}/O]\subset \mathcal{S}_{\mathrm{Diff}}(\mathbb{C}P^{7})\) (see, for example, \cite[Theorem 3.1]{Ram17}), it follows that the image
\[
\operatorname{Im}\left([\mathbb{C}P^{7},\operatorname{Top}/O]\xrightarrow{\psi_{*}}[\mathbb{C}P^{7},F/O]\right) \subseteq \Ker\left([\mathbb{C}P^{7}, SF]\xrightarrow{s}L_{14}(e)\cong \mathbb{Z}_{2}\right).
\]
Therefore, the kernel of $s$ is precisely \(\mathbb{Z}_2\{(\kappa)_{14}\} \oplus \mathbb{Z}_2\{\left (\overline{(\epsilon)_{8}}\right )_{14}\}\). This completes the proof.
\end{proof}
\begin{theorem}\label{main5}
\begin{itemize}
\item[(i)] The normal invariant \(\eta^{t}\colon\mathcal{S}^t_{\mathrm{Diff}}(\mathbb{C}P^{7})\longrightarrow [\mathbb{C}P^{7},SF]\) is injective and its image is \(\mathbb{Z}_2\{(\kappa)_{14}\}
\oplus
\mathbb{Z}_2\{\left (\overline{(\epsilon)_{8}}\right )_{14}\}\subset [\mathbb{C}P^{7},SF]\).
\item[(ii)] The normal invariant \(\eta^{t}\colon\mathcal{S}^t_{\mathrm{Diff}}(\mathbb{C}P^{8})\longrightarrow [\mathbb{C}P^{8},SF]\) is bijective, where
\[
[\mathbb{C}P^{8},SF]=
\mathbb{Z}_{4}\{\left(\overline{(\sigma^{2})_{14}}\right)_{16}\}
\oplus \mathbb{Z}_2\{\left(\overline{(\kappa)_{14}+(\overline{(\epsilon)_{8}})_{14}}\right)_{16}\}.
\]
\end{itemize}
\end{theorem}
\begin{proof}
Statement \textup{(i)} follows from the exact sequence \eqref{taur} using Proposition~\ref{sur-obst}. Since the surgery obstruction \(s\colon[\mathbb{C}P^{8},SF]\to \mathbb{Z}\) is trivial, Statement \textup{(ii)} follows immediately from the exact sequence \eqref{taur} and Theorem~\ref{eight-sub2}\textup{(i)}.
\end{proof}
Recall that the surgery exact sequence appearing in the bottom row of Diagram~\eqref{digram31} for \(\mathbb{C}P^{m}\) yields an inclusion
\[
\mathcal{S}_{\mathrm{Diff}}(\mathbb{C}P^{m})\subset [\mathbb{C}P^{m},F/O],
\]
and identifies \(\mathcal{S}_{\mathrm{Diff}}(\mathbb{C}P^{m})\) with the kernel
\[
\Ker\!\left([\mathbb{C}P^{m},F/O]\xrightarrow{\;s\;}L_{2m}(e)\right).
\]
There is also an exact sequence
\begin{equation}\label{bex}
0\longrightarrow [\mathbb{C}P^{m},SF]\xrightarrow{\;\phi_{*}\;}[\mathbb{C}P^{m},F/O]
\xrightarrow{\;i_{*}\;}\widetilde{KO}^{0}(\mathbb{C}P^{m})
\longrightarrow [\mathbb{C}P^{m},BF],
\end{equation}
induced by the fibration sequence
\[
SF\longrightarrow F/O\longrightarrow BSO\longrightarrow BSF,
\]
where \(\widetilde{KO}^{0}(\mathbb{C}P^{m})\) denotes reduced real \(K\)-theory, and if
\(f\colon M\longrightarrow \mathbb{C}P^{m}\) represents an element of
\(\mathcal{S}_{\mathrm{Diff}}(\mathbb{C}P^{m})\), then its image under the map
\(i_{*}\colon[\mathbb{C}P^{m},F/O]\longrightarrow \widetilde{KO}^{0}(\mathbb{C}P^{m})\) is given by
\[
i_{*}\bigl([(f,M)]\bigr)
=(f^{-1})^{*}\nu_{M}-\nu_{\mathbb{C}P^{m}},
\]
see \cite[\S2]{Bru68}. Thus, given \([(f,M)]\in \mathcal{S}_{\mathrm{Diff}}(\mathbb{C}P^{m})\), we have from the exact sequence \eqref{bex} that \(f^{*}(\nu_{\mathbb{C}P^{m}})=\nu_{M}\) if and only if
\((f^{-1})^{*}\nu_{M}-\nu_{\mathbb{C}P^{m}}=0\) in \(\widetilde{KO}^{0}(\mathbb{C}P^{m})\), or, equivalently, if and only if the element \([(f,M)]\) lies in \([\mathbb{C}P^{m},SF]\). We say that \([M]\) and \([N]\in \mathcal{M}^{\pm}_{\mathrm{hDiff}}(\mathbb{C}P^{m})\) have the same Pontryagin classes if there exists a homotopy equivalence \(h\colon M\longrightarrow N\) such that
\(h^{*}\bigl(p_{i}(\nu_{N})\bigr)=p_{i}(\nu_{M})\) for all \(i\). With all these observations, we have the following.
\begin{proposition}\label{num}
Let $[N] \in \mathcal{M}^{\pm}_{\mathrm{hDiff}}(\mathbb{C}P^{m})$ with $m \ge 3$. Then the elements in $\mathcal{M}^{\pm}_{\mathrm{hDiff}}(\mathbb{C}P^{m})$ with the same Pontryagin classes as $[N]$ are in one-to-one correspondence with the elements of $\Ker\left([\mathbb{C}P^{m}, SF] \xrightarrow{s} L_{2m}(e)\right)$.
\end{proposition}

\begin{proof}
Let $h \colon M \to N$ be a homotopy equivalence such that $h^{*}\bigl(p_{i}(\nu_{N})\bigr) = p_{i}(\nu_{M})$ for all $i$. Consider the elements $[(f,M)]$ and $[(g,N)] \in \mathcal{S}_{\mathrm{Diff}}(\mathbb{C}P^{m})$ representing the diffeomorphism classes $[M]$ and $[N]$ in $\mathcal{M}^{\pm}_{\mathrm{hDiff}}(\mathbb{C}P^{m})$, respectively. From this, we have:
\begin{align*}
p_{i}\bigl((f^{-1})^{*}(\nu_{M})\bigr)
&=(f^{-1})^{*}\bigl(p_{i}(\nu_{M})\bigr)
=(f^{-1})^{*}\bigl(h^{*}(p_{i}(\nu_{N}))\bigr)\\
&=(f^{-1})^{*} \circ h^{*}\bigl(p_{i}(\nu_{N})\bigr)
=\bigl(h \circ f^{-1}\bigr)^{*}\bigl(p_{i}(\nu_{N})\bigr).
\end{align*}
By \cite[Theorem~8(i)]{Sul67}, $h \circ f^{-1}$ is homotopic either to $g^{-1}$ or to $g^{-1} \circ C$, where $C \colon \mathbb{C}P^{m} \longrightarrow \mathbb{C}P^{m}$ is the conjugation map. Since
$C^{*}\colon \widetilde{KO}^{0}(\mathbb{C}P^{m}) \longrightarrow \widetilde{KO}^{0}(\mathbb{C}P^{m})$
is the identity (see \cite[p.~35]{Bru68}), we obtain:
\[
\bigl(h \circ f^{-1}\bigr)^{*}\bigl(p_{i}(\nu_{N})\bigr)
=(g^{-1})^{*}\bigl(p_{i}(\nu_{N})\bigr)
=p_{i}\bigl((g^{-1})^{*}(\nu_{N})\bigr).
\]
Hence, $p_{i}\bigl((f^{-1})^{*}(\nu_{M})\bigr) = p_{i}\bigl((g^{-1})^{*}(\nu_{N})\bigr)$.
Now, by \cite[Lemma~2.25]{Bro68}, the difference $(f^{-1})^{*}(\nu_{M}) - (g^{-1})^{*}(\nu_{N})$ lies in the torsion subgroup of $\widetilde{KO}^{0}(\mathbb{C}P^{m})$. Since the image
\[
\operatorname{Im}\Bigl([\mathbb{C}P^{m},F/O] \xrightarrow{\;i_{*}\;} \widetilde{KO}^{0}(\mathbb{C}P^{m})\Bigr)
\]
is a free subgroup of $\widetilde{KO}^{0}(\mathbb{C}P^{m})$ (see \cite[Corollary~2.1]{Bru68}), it follows that $(f^{-1})^{*}(\nu_{M}) = (g^{-1})^{*}(\nu_{N})$ in $\widetilde{KO}^{0}(\mathbb{C}P^{m})$. Therefore, the elements $[(f,M)]$ and $[(g,N)]$ have the same image in $\widetilde{KO}^{0}(\mathbb{C}P^{m})$. It follows from the exact sequence \eqref{bex} that the difference $[(f,M)] - [(g,N)]$ lies in $\Ker\left([\mathbb{C}P^{m}, SF] \xrightarrow{s} L_{2m}(e)\right)$. Conversely, any element $$[h,W] \in \Ker\left([\mathbb{C}P^{m}, SF] \xrightarrow{s} L_{2m}(e)\right)$$ is represented by a difference $[(f,M)] - [(g,N)]$, where $[(f,M)] \in \mathcal{S}_{\mathrm{Diff}}(\mathbb{C}P^{m})$. This implies that $[(f,M)]$ and $[(g,N)]$ have the same image in $\widetilde{KO}^{0}(\mathbb{C}P^{m})$, showing that $[N]$ and $[M]$ have the same Pontryagin classes in $\mathcal{M}^{\pm}_{\mathrm{hDiff}}(\mathbb{C}P^{m})$. This completes the proof.
\end{proof}
Applying Proposition \ref{lowbru}, Proposition \ref{sur-obst} and Theorem \ref{main5} to Proposition~\ref{num}, we can completely determine the number of elements in $\mathcal{M}^{\pm}_{\mathrm{hDiff}}(\mathbb{C}P^{m})$ with fixed Pontryagin classes for $3 \leq m \leq 8$. It follows from Diagram 6.7 of \cite[p.~107]{CH15} using \eqref{equ1} that the natural map $F_{t}:\mathcal{S}^{t}(\mathbb{C}P^{m})\longrightarrow\mathcal{S}_{\rm Diff}(\mathbb{C}P^{m})$ is injective. The next result describes that the image of $F_{t}$ may be identified with the set of elements in $\mathcal{S}_{\rm Diff}(\mathbb{C}P^{m})$ with the same Pontryagin classes as $\mathbb{C}\mathbb{P}^{m}$.
\begin{lemma}\label{pont}
Let $M$ be a smooth manifold and $f \colon M \longrightarrow \mathbb{C}P^{m}$ be a homotopy equivalence. If $f$ preserves integral Pontryagin classes or is a homeomorphism, then $f$ is a tangential homotopy equivalence.
\end{lemma}
\begin{proof}
By Novikov's theorem, rational Pontryagin classes are topological invariants. Since $H^{*}(M; \mathbb{Z})$ is torsion-free, it follows that any homeomorphism $f \colon M \to \mathbb{C}P^{m}$ preserves integral Pontryagin classes. Therefore, following the argument in the proof of Proposition~\ref{num}, we obtain $(f^{-1})^{*}(\nu_M) - \nu_{\mathbb{C}P^{m}} = 0$ in $\widetilde{KO}^0(\mathbb{C}P^{m})$. This implies that the map $f \colon M \longrightarrow \mathbb{C}P^{m}$ is a tangential homotopy equivalence.
\end{proof}
Denote by $\epsilon^{t}(M)$ the group of tangential self-equivalences of $M$, i.e., bundle maps of the stable normal bundle of $M$, up to homotopy as bundle maps; let $\mathcal{M}^{\pm}_{\mathrm{tDiff}}(M)$ denote the set of diffeomorphism classes of smooth manifolds tangentially homotopy equivalent to $M$. Clearly, $\epsilon^{t}(M)$ acts on $\mathcal{S}^t_{\mathrm{Diff}}(M)$ via composition. The forgetful map \[\mathscr{F}_{t}\colon \mathcal{S}^t_{\mathrm{Diff}}(M) \to \mathcal{M}^{\pm}_{\mathrm{tDiff}}(M)\] induces a bijection from the orbit space $\mathcal{S}^t_{\mathrm{Diff}}(M) / \epsilon^{t}(M)$ to $\mathcal{M}^{\pm}_{\mathrm{tDiff}}(M)$. 

For $M = \mathbb{C}P^{m}$ ($m \geq 3$), it follows from the exact sequence in \cite[p.~161]{Bro65} that $\epsilon^{t}(\mathbb{C}P^{m}) \cong \mathrm{haut}(\mathbb{C}P^{m}) \cong \mathbb{Z}_{2}\{[Id], [C]\}$. For $4 \leq m \leq 8$, by Proposition~\ref{lowbru}, Proposition \ref{sur-obst} and Theorem~\ref{main5}, the action of $\epsilon^{t}(\mathbb{C}P^{m})$ on $\mathcal{S}^{t}(\mathbb{C}P^{m})$ is determined by the normal invariant composition formula:
\begin{equation}\label{act2}
\eta^{t}(C \circ f, \widehat{C} \circ \widehat{f}) = C^{*}(\eta^{t}(f, \widehat{f}))
\end{equation}
since the forgetful map $F_{t}\colon \mathcal{S}^{t}(\mathbb{C}P^{m}) \to \mathcal{S}_{\mathrm{Diff}}(\mathbb{C}P^{m})$ is injective, and the map \[\phi_{*}\colon [\mathbb{C}P^{m}, SF] \to [\mathbb{C}P^{m}, F/O]\] is also injective, implying $\eta^{t}(C, \widehat{C}) = 0$. Here, $C^{*}\colon [\mathbb{C}P^{m}, SF] \to [\mathbb{C}P^{m}, SF]$ is the induced map. We now prove the following.
\begin{theorem}\label{main6}
Let $C\colon \mathbb{C}P^{m} \to \mathbb{C}P^{m}$ denote the complex conjugation map.
\begin{itemize}
    \item[(i)] The map $C^{*} \colon [\mathbb{C}P^{4}, SF] \to [\mathbb{C}P^{4}, SF]$ is given by $C^{*}(x) = -x$, where $[\mathbb{C}P^{4}, SF] \cong \mathbb{Z}_{4}\{(\overline{(\nu^{2})_{6}})_{8}\}$.
    \item[(ii)] The map $C^{*} \colon [\mathbb{C}P^{7}, SF] \to [\mathbb{C}P^{7}, SF]$ is the identity map.
    \item[(iii)] On the group $[\mathbb{C}P^{5}, SF] \cong \mathbb{Z}_{2}\{(\eta \circ \mu)_{10}\} \oplus \mathbb{Z}_{3}\{(\beta_{1})_{10}\} \oplus \mathbb{Z}_{2}\{(\overline{(\epsilon)_{8}})_{10}\}$, the map $C^{*}$ is given by $C^{*}(x, y, z) = (x, -y, z)$.
    \item[(iv)] On the group $[\mathbb{C}P^{6}, SF] \cong \mathbb{Z}_{3}\{(\overline{(\beta_{1})_{10}})_{12}\} \oplus \mathbb{Z}_{2}\{(\overline{(\epsilon)_{8}})_{12}\}$, the map $C^{*}$ is given by $C^{*}(x, y) = (-x, y)$.
    \item[(v)] On the group $[\mathbb{C}P^{8}, SF] \cong \mathbb{Z}_{4}\{(\overline{(\sigma^{2})_{14}})_{16}\} \oplus \mathbb{Z}_{2}\{(\overline{(\kappa)_{14}+(\overline{(\epsilon)_{8}})_{14}})_{16}\}$, the map $C^{*}$ is given by $C^{*}(x, y) = (-x, y)$.
\end{itemize}
\end{theorem}
\begin{proof}
Consider the following commutative diagram:
\begin{equation}\label{cpn-pl}
\begin{gathered}
\xymatrix@C=3em@R=3.2em{
[\mathbb{C}P^{m}, PL] \ar[r]^-{J_{PL}} \ar[d]_-{C^{*}} & [\mathbb{C}P^{m}, SF] \ar[d]^-{C^*} \\
[\mathbb{C}P^{m}, PL] \ar[r]^-{J_{PL}} & [\mathbb{C}P^{m}, SF]
}
\end{gathered}
\end{equation}
where the $J$-homomorphism $J_{PL} \colon [\mathbb{C}P^{m}, PL] \to [\mathbb{C}P^{m}, SF]$ is injective by \eqref{equ23}. Since $[\mathbb{C}P^{m}, PL] \cong [\mathbb{C}P^{m}, \operatorname{Top}/O]$, it follows from the group structures given in Remark~\ref{group} and \eqref{equtan_final} that $J_{PL}$ is an isomorphism for $m=5, 6$. Applying these to Diagram~\eqref{cpn-pl} and using Remark~\ref{pl-top} and Theorem~\ref{main2}(i) and (ii), we conclude that Part~(iii) and Part~(iv) hold. 

For Part~(i), we have $[\mathbb{C}P^{4}, SF] =\mathbb{Z}_{4}\{\left(\overline{(\nu^{2})_{6}}\right)_{8}\}$ by \eqref{equtan_final}. Now consider the diagram:
\begin{equation}\label{cpn-pl2}
\begin{gathered}
\xymatrix@C=3em@R=3.2em{
[\mathbb{C}P^{4}/\mathbb{C}P^{2}, SF] \ar[r]^-{q^{*}} \ar[d]_-{(\widetilde{C}_{4,2})^{*}} & [\mathbb{C}P^{4}, SF] \ar[d]^-{C^*} \\
[\mathbb{C}P^{4}/\mathbb{C}P^{2}, SF] \ar[r]^-{q^{*}} & [\mathbb{C}P^{4}, SF]
}
\end{gathered}
\end{equation}
where the map \( q^{*} : [\mathbb{C}P^{4}/\mathbb{C}P^{2}, SF] \to [\mathbb{C}P^{4}, SF] \) is surjective since \( [\mathbb{C}P^{2}, SF] = 0 \), and the map \( (\widetilde{C}_{4,2})^{*} = -\mathrm{Id} \) on \( [\mathbb{C}P^{4}/\mathbb{C}P^{2}, SF] \) follows from Theorem~\ref{conju-cp8}(3). There exists an element \( y \in [\mathbb{C}P^{4}/\mathbb{C}P^{2}, SF] \) such that \( q^{*}(y) = (\overline{(\nu^{2})_{6}})_{8} \in [\mathbb{C}P^{4}, SF] \). This, together with the commutativity of Diagram \eqref{cpn-pl2}, implies that the map \( C^{*} : [\mathbb{C}P^{4}, SF] \to [\mathbb{C}P^{4}, SF] \) reverses the generator \( (\overline{(\nu^{2})_{6}})_{8} \). Therefore, \( C^{*} = -\mathrm{Id} \) on \( [\mathbb{C}P^{4}, SF] \), which completes the proof of Part (i). 

We now turn to the proof of Part (ii). By Remark~\ref{pl-top} and Theorem~\ref{main2}(iii), the map $C^{*} \colon [\mathbb{C}P^{7}, PL] \to [\mathbb{C}P^{7}, PL]$ is the identity map. From Diagram \eqref{cpn-pl} for $m=7$, the map $C^{*} \colon [\mathbb{C}P^{7}, SF] \to [\mathbb{C}P^{7}, SF]$ fixes all elements of the image
\[
\operatorname{Im}\Bigl([\mathbb{C}P^{7}, PL] \xrightarrow{J_{PL}} \mathbb{Z}_{2}\{\overline{(\epsilon)_{8}}\}_{14} \oplus \mathbb{Z}_{2}\{(\kappa)_{14}\} \oplus \mathbb{Z}_{2}\{(\sigma^{2})_{14}\}\Bigr) \cong \mathbb{Z}_{2}\{\overline{(\epsilon)_{8}}\}_{14} \oplus \mathbb{Z}_{2}\{(\kappa)_{14}\},
\]
as given by \eqref{cp7-pl} and Corollary~\ref{cp7-sf}. On the other hand, by Remark~\ref{actions}(1), the map $C^{*}$ also fixes the generator $(\sigma^{2})_{14}$. Combining these, we conclude that $C^{*} \colon [\mathbb{C}P^{7}, SF] \to [\mathbb{C}P^{7}, SF]$ is the identity map. This completes Part (ii). 

For the proof of Part (v), as in the previous cases, it follows from Theorem~\ref{conju-cp8}(1) that the map $C^{*} \colon [\mathbb{C}P^{8}, PL] \to [\mathbb{C}P^{8}, PL]$ is the identity map. Applying this to Diagram~\eqref{cpn-pl} for $m=8$, we find that the map $C^{*} \colon [\mathbb{C}P^{8}, SF] \to [\mathbb{C}P^{8}, SF]$ fixes all elements of the image:
$$
\operatorname{Im}\left( [\mathbb{C}P^{8}, PL] \xrightarrow{ J_{PL}} \mathbb{Z}_{4}\left\{\left(\overline{(\sigma^{2})_{14}}\right)_{16}\right\} \oplus \mathbb{Z}_2\left\{\left(\overline{(\kappa)_{14}+\left(\overline{(\epsilon)_{8}}\right)_{14}}\right)_{16}\right\} \right)
$$
$$
\cong \mathbb{Z}_{2}\left\{2\left(\overline{(\sigma^{2})_{14}}\right)_{16}\right\} \oplus \mathbb{Z}_{2}\left\{\left(\overline{(\kappa)_{14} + \left(\overline{(\epsilon)_{8}}\right)_{14}}\right)_{16}\right\}
$$
where $2\left(\overline{(\sigma^{2})_{14}}\right)_{16}=(z)_{16}$ for $z\in \{\eta^{*},\eta^{*}+\eta \circ \rho\}\subset \pi^{S}_{16}$, as provided by \eqref{pl8} and Theorem~\ref{eight-sub2}(1). To complete the proof of Part (v), it remains to show that the map $C^{*} \colon [\mathbb{C}P^{8}, SF] \to [\mathbb{C}P^{8}, SF]$ sends the generator $\left(\overline{(\sigma^{2})_{14}}\right)_{16}$ to $-\left(\overline{(\sigma^{2})_{14}}\right)_{16}$. Consider the diagram:
\begin{equation}\label{cp8-pl3}
\begin{gathered}
\xymatrix@C=3em@R=3.2em{
[\mathbb{C}P^{8}/\mathbb{C}P^{6}, SF] \ar[r]^-{q^{*}} \ar[d]_-{(\widetilde{C}_{8,6})^{*}} & [\mathbb{C}P^{8}, SF] \ar[d]^-{C^*} \\
[\mathbb{C}P^{8}/\mathbb{C}P^{6}, SF] \ar[r]^-{q^{*}} & [\mathbb{C}P^{8}, SF]
}
\end{gathered}
\end{equation}
where the map $q^{*}\colon [\mathbb{C}P^{8}/\mathbb{C}P^{6}, SF] \cong \mathbb{Z}_{4}\left\{\left(\overline{(\sigma^{2})_{14}}\right)_{16}\right\} \to [\mathbb{C}P^{8}, SF]\cong \mathbb{Z}_{4}\left\{\left(\overline{(\sigma^{2})_{14}}\right)_{16}\right\} \oplus \mathbb{Z}_2\left\{\left(\overline{(\kappa)_{14}+\left(\overline{(\epsilon)_{8}}\right)_{14}}\right)_{16}\right\}$ maps the generator $\left(\overline{(\sigma^{2})_{14}}\right)_{16}$ to itself (see \eqref{ext3} and the proof of Lemma~\ref{eight-sub}(1)), and the map $(\widetilde{C}_{8,6})^{*} = -\mathrm{Id}$ on $[\mathbb{C}P^{8}/\mathbb{C}P^{6}, SF]$ follows from Theorem~\ref{conju-cp8}(2). Applying these observations to Diagram~\eqref{cp8-pl3} and using its commutativity, we obtain $C^{*}(\left(\overline{(\sigma^{2})_{14}}\right)_{16})=-\left(\overline{(\sigma^{2})_{14}}\right)_{16}$. This proves Part (v). 
\end{proof}
Recall from \cite[p.~194, Definition]{Sch87} that a tangential PL smoothing of $M$ is a triple $(V, t, \widehat{t})$, where $V$ is a smooth manifold, $t \colon V \to M$ is a piecewise differentiable homeomorphism, and $\widehat{t} \colon \nu_{V} \to \nu_{M}$ is a vector bundle isomorphism covering $t$. By Theorem~3.8 of \cite{Sch87}, there is a bijection between the group of homotopy classes $[M, PL]$ and the set of concordance classes of tangential PL smoothings of $M$. Note that there is a forgetful map $H \colon [M, PL] \to \mathcal{S}^{t}(M)$ that takes a tangential PL smoothing to a tangential homotopy smoothing. For $M = \mathbb{C}P^{m}$, it follows from \cite[Theorem~3.9]{Sch87} that the map $H \colon [\mathbb{C}P^{m}, PL] \to \mathcal{S}^{t}(\mathbb{C}P^{m})$ is injective, since $[\Sigma \mathbb{C}P^{m}, F/PL] = 0$ and $L_{\mathrm{odd}}(e) = 0$. Since the maps $\beta_{*} \colon [\mathbb{C}P^{m}, PL] \to [\mathbb{C}P^{m}, PL/O]$ and $F_{*} \colon [\mathbb{C}P^{m}, PL/O] \to [\mathbb{C}P^{m}, \operatorname{Top}/O]$ are isomorphisms, there is a well-defined injective map $E = H \circ \beta_{*}^{-1} \circ F_{*}^{-1} \colon [\mathbb{C}P^{m}, \operatorname{Top}/O] \to \mathcal{S}^{t}(\mathbb{C}P^{m})$ sending $(N, f)$ to $(N, f_{PL}, \widehat{f_{PL}})$, where $f \colon N \to \mathbb{C}P^{m}$ is topologically concordant to a piecewise differentiable homeomorphism $f_{PL} \colon N \to \mathbb{C}P^{m}$, and $\widehat{f_{PL}} \colon \nu_{N} \to \nu_{\mathbb{C}P^{m}}$ is a vector bundle isomorphism covering $f_{PL}$. It also follows that the composition
\begin{equation}\label{normalinva}
[\mathbb{C}P^{m}, \operatorname{Top}/O] \xrightarrow{E} \mathcal{S}^{t}(\mathbb{C}P^{m}) \xrightarrow{\eta^{t}} [\mathbb{C}P^{m}, SF] \subset [\mathbb{C}P^{m}, F/O]
\end{equation}
is precisely the canonical map 
\(\psi_{*} : [\mathbb{C}P^{m}, \operatorname{Top}/O] \longrightarrow [\mathbb{C}P^{m}, F/O]\).
Note that the maps $\beta_{*}$ and $F_{*}$ induce well-defined bijective maps between orbit spaces:
\[
[\mathbb{C}P^{m}, PL]/\{[Id, \widehat{Id}], [C, \widehat{C}]\} \to [\mathbb{C}P^{m}, PL/O]/\{[Id], [C]\}
\]
and
\[
[\mathbb{C}P^{m}, PL/O]/\{[Id], [C]\} \to [\mathbb{C}P^{m}, \operatorname{Top}/O]/\{[Id], [C]\},
\]
respectively. Also, the map $H$ induces a well-defined injective map between orbit spaces:
\[
[\mathbb{C}P^{m}, PL]/\{[Id, \widehat{Id}], [C, \widehat{C}]\} \to \mathcal{S}^{t}(\mathbb{C}P^{m})/\{[Id, \widehat{Id}], [C, \widehat{C}]\}.
\]
Under these identifications, the map $E$ induces a well-defined injective map $\widetilde{E}$ between orbit spaces 
\begin{equation}\label{forgettan}
 \mathcal{M}^{\pm}_{\mathrm{Diff}}(\mathbb{C}P^{m}) \to \mathcal{M}^{\pm}_{\mathrm{tDiff}}(\mathbb{C}P^{m})   
\end{equation}
by sending the diffeomorphism class of $N$ in the homeomorphism type of $\mathbb{C}P^{m}$ to the diffeomorphism class of $N$ in the tangential homotopy type of $\mathbb{C}P^{m}$.
\begin{theorem}\label{difftang}
\begin{itemize}
    \item[(1)] The image of the map \(\widetilde{E} \colon \mathcal{M}^{\pm}_{\mathrm{Diff}}(\mathbb{C}P^{4}) \longrightarrow \mathcal{M}^{\pm}_{\mathrm{tDiff}}(\mathbb{C}P^{4})\) is \(\mathbb{Z}_{2}\{[(\epsilon)_{8}]\}\), where
    \[ 
    \mathcal{M}^{\pm}_{\mathrm{tDiff}}(\mathbb{C}P^{4}) \cong \mathbb{Z}_{4}\{(\overline{(\nu^{2})_{6}})_{8}\} / \sim,
    \] 
    with the relation \((\overline{(\nu^{2})_{6}})_{8} \sim -(\overline{(\nu^{2})_{6}})_{8}\). Furthermore, \((\overline{(\nu^{2})_{6}})_{8} \cong -(\overline{(\nu^{2})_{6}})_{8}\), \(\mathbb{Z}_{2}\{(\epsilon)_{8}\} = \{[\mathbb{C}P^{4} \# \Sigma^{8}] \mid \Sigma^{8} \in \Theta_{8}\}\), and \(2(\overline{(\nu^{2})_{6}})_{8} = (\epsilon)_{8}\).

    \item[(2)] For \(m=5, 6, 7\), the map \(\widetilde{E} \colon \mathcal{M}^{\pm}_{\mathrm{Diff}}(\mathbb{C}P^{m}) \longrightarrow \mathcal{M}^{\pm}_{\mathrm{tDiff}}(\mathbb{C}P^{m})\) is a bijection. In particular:
    \begin{itemize}
        \item For \(m=5\), the set is given by
        \[
        \mathcal{M}^{\pm}_{\mathrm{tDiff}}(\mathbb{C}P^{5}) = \mathbb{Z}_{2}\{[(\eta\circ\mu)_{10}]\} \oplus \mathbb{Z}_{2}\{[(\beta_{1})_{10}]\} \oplus \mathbb{Z}_{2}\{[\left(\overline{(\epsilon)_{8}}\right)_{10}]\},
        \]
        where \((\eta\circ\mu)_{10}\) represents the unoriented diffeomorphism class \([\mathbb{C}P^{5}\#\Sigma_{\eta\circ\mu}]\) (\(\Sigma_{\eta\circ\mu}\in\Theta_{10}\) is the unique element of order \(2\)), and the generator \((\beta_{1})_{10}\) represents \([\mathbb{C}P^{5}\#\Sigma_{\beta_{1}}]\) (\(\Sigma_{\beta_{1}}\in\Theta_{10}\) has order \(3\)), with \((\beta_{1})_{10} \cong -(\beta_{1})_{10}\).
        
        \item For \(m=6\), 
        \[
        \mathcal{M}^{\pm}_{\mathrm{tDiff}}(\mathbb{C}P^{6}) = \mathbb{Z}_{2}\left\{\left[\left(\overline{(\epsilon)_{8}}\right)_{12}\right]\right\} \oplus \mathbb{Z}_{2}\left\{\left[\left(\overline{(\beta_{1})_{10}}\right)_{12}\right]\right\},
        \]
        where \(-\left(\overline{(\beta_{1})_{10}}\right)_{12} \cong \left(\overline{(\beta_{1})_{10}}\right)_{12}\).
        
        \item For \(m=7\),
        \[
        \mathcal{M}^{\pm}_{\mathrm{tDiff}}(\mathbb{C}P^{7}) = \mathbb{Z}_{2}\{[(\kappa)_{14}]\} \oplus \mathbb{Z}_{2}\{[\left(\overline{(\epsilon)_{8}}\right)_{14}]\},
        \]
        where \(\mathbb{Z}_{2}\{(\kappa)_{14}\} = \{[\mathbb{C}P^{7}\#\Sigma^{14}] \mid \Sigma^{14} \in \Theta_{14}\}\).
    \end{itemize}

    \item[(3)] The image of the map \(\widetilde{E} \colon \mathcal{M}^{\pm}_{\mathrm{Diff}}(\mathbb{C}P^{8}) \longrightarrow \mathcal{M}^{\pm}_{\mathrm{tDiff}}(\mathbb{C}P^{8})\) is 
    \[
    \mathbb{Z}_{2}\{[2(\overline{(\sigma^{2})_{14}})_{16}]\} \oplus \mathbb{Z}_{2}\{[\left(\overline{(\kappa)_{14} + (\overline{(\epsilon)_{8}})_{14}}\right)_{16}]\},
    \] 
    where
    \[ 
    \mathcal{M}^{\pm}_{\mathrm{tDiff}}(\mathbb{C}P^{8}) \cong A \oplus \mathbb{Z}_2\{[\left(\overline{(\kappa)_{14} + (\overline{(\epsilon)_{8}})_{14}}\right)_{16}]\},
    \] 
    and \(A = \{ [\mathbb{C}P^{8}], [(\overline{(\sigma^{2})_{14}})_{16}], [2(\overline{(\sigma^{2})_{14}})_{16}] \}\). Here, the element \(2(\overline{(\sigma^{2})_{14}})_{16}\) represents the unoriented diffeomorphism class of \([\mathbb{C}P^{8} \# \Sigma]\) for the unique exotic \(16\)-sphere \(\Sigma \in \Theta_{16}\), and \((\overline{(\sigma^{2})_{14}})_{16} \cong -(\overline{(\sigma^{2})_{14}})_{16}\).
\end{itemize}
\end{theorem}
\begin{proof}
For $4 \leq m \leq 8$, it follows from Proposition~\ref{lowbru}, Proposition~\ref{sur-obst}, and Theorem~\ref{main5} that the set $\mathcal{M}^{\pm}_{\mathrm{tDiff}}(\mathbb{C}P^{m})$ is obtained by computing the action defined in \eqref{act2}. Therefore, the computations of the conjugation action
\[
C^{*}:[\mathbb{C}P^{m},SF]\to [\mathbb{C}P^{m},SF]
\]
given in Theorem~\ref{main6} immediately imply the structure of the set $\mathcal{M}^{\pm}_{\mathrm{tDiff}}(\mathbb{C}P^{m})$ as stated in the theorem.

On the other hand, the image of the map
\[
\widetilde{E} \colon \mathcal{M}^{\pm}_{\mathrm{Diff}}(\mathbb{C}P^{m}) \longrightarrow \mathcal{M}^{\pm}_{\mathrm{tDiff}}(\mathbb{C}P^{m})
\]
is determined using the definition and injectivity of $\widetilde{E}$ as in \eqref{forgettan}, together with the computations of $\mathcal{M}^{\pm}_{\mathrm{Diff}}(\mathbb{C}P^{m})$ given in Theorem~\ref{main4}.
\end{proof}
Recall from Proposition~\ref{lowbru}(i) that every closed smooth $6$-manifold $M$ that is tangentially homotopy equivalent to $\mathbb{C}P^{3}$ is diffeomorphic to $\mathbb{C}P^{3}$. For $4 \leq m \leq 8$, by combining Theorem~\ref{difftang}, Theorem~\ref{main4}, and Theorem~\ref{main3}, we obtain the following :
\begin{corollary}\label{TE}
\begin{enumerate}
    \item[(1)] For $m=5, 6, 7$, a closed smooth manifold $M$ is homeomorphic to $\mathbb{C}P^{m}$ if and only if $M$ is tangentially homotopy equivalent to $\mathbb{C}P^{m}$. 
    \item[(2)] For $m=4$, there exists a unique smooth manifold (up to diffeomorphism) that is tangentially homotopy equivalent to $\mathbb{C}P^{4}$ but is not homeomorphic to $\mathbb{C}P^{4}$.
    \item[(3)] For $m=8$, there are exactly two smooth manifolds (up to diffeomorphism) that are tangentially homotopy equivalent to $\mathbb{C}P^{8}$ but are not homeomorphic to $\mathbb{C}P^{8}$.
\end{enumerate}
\end{corollary}
Recall that the conjugation actions given in Theorem~\ref{act2} reverse the signs on the $\mathbb{Z}_{3}$-summands for $m=5, 6$ and the $\mathbb{Z}_{4}$-summand for $m=4,8$, while the remaining summands are fixed. Consequently, the generator notations in $\mathcal{M}^{\pm}_{\mathrm{tDiff}}(\mathbb{C}P^{m})$ provided by Theorem~\ref{difftang} represent the tangential normal invariants corresponding to the diffeomorphism classes of those generators up to a sign. Therefore, the diffeomorphism classes are determined by their normal invariants up to a sign, leading to the following result.
\begin{theorem}\label{Main7}
Let $4 \le m \le 8$ and let 
$\eta^{t} \colon \mathcal{S}^t_{\mathrm{Diff}}(\mathbb{C}P^{m}) \longrightarrow [\mathbb{C}P^{m}, SF]$ 
be the tangential normal invariant map. Suppose that 
$[M, f, \widehat{f}]$ and $[N, g, \widehat{g}] \in 
\mathcal{S}^t_{\mathrm{Diff}}(\mathbb{C}P^{m})$.
\begin{itemize}
\item[(1)] For $m=4, 5, 6, 8$, the smooth manifolds $M$ and $N$ are diffeomorphic if and only if 
$\eta^{t}([M, f, \widehat{f}]) = \pm \eta^{t}([N, g, \widehat{g}])$.
\item[(2)] For $m=7$, the smooth manifolds $M$ and $N$ are diffeomorphic if and only if 
$\eta^{t}([M, f, \widehat{f}]) = \eta^{t}([N, g, \widehat{g}])$.
\end{itemize}
\end{theorem}

\end{document}